\documentclass[12pt]{article}
\usepackage[utf8x]{inputenc}
\usepackage{comment}
\usepackage{amsthm}
\usepackage{multirow}
\usepackage{array}
\newcolumntype{P}[1]{>{\centering\arraybackslash}p{#1}}
\usepackage{longtable}
\usepackage[top=2cm,bottom=2cm,left=2cm,right=2cm]{geometry}
\usepackage{amsmath,amssymb}
\usepackage[mathscr]{euscript}
\usepackage[usenames,dvipsnames]{color}
\usepackage{booktabs}
\usepackage[colorlinks]{hyperref}
\usepackage{pstricks}
\usepackage{graphicx,caption}
\usepackage{dsfont}
\usepackage[all]{xy}
\usepackage{tabularx}
\usepackage{bbm}
\usepackage{todonotes}
\usepackage{hyperref}
\usepackage[french,english]{babel}
\usepackage{fancyhdr}
\usepackage{mathtools}

\usepackage{tikz}
\usepackage{tikz-3dplot}
\usepackage{amsmath}
\usetikzlibrary{arrows.meta,decorations.markings}
\usetikzlibrary{3d,decorations.markings}

\tikzset{
  ->-/.style={decoration={
    markings,
    mark=at position #1 with {\arrow{>}}},postaction={decorate}},
  ->-/.default=0.5
}

\newcommand{\dd}{\mathrm{d}}

\renewcommand{\geq}{\geqslant}
\renewcommand{\leq}{\leqslant}

\newcommand{\N}{ \mathbb{N} } 
\newcommand{\R}{ \mathbb{R} }

\newcommand{\G}{\mathcal{G}}

\newcommand{\Tr}{ \mathrm{Tr} }
\newcommand{\Ad}{ \mathrm{Ad} }

\newcommand{\g}{\mathfrak{g}}

\newtheorem{remark}{Remark}
\numberwithin{equation}{section}

\newtheorem{thm}{Theorem}[section]
\newtheorem{proposition}[thm]{Proposition}
\newtheorem{corollary}[thm]{Corollary}

\newtheorem{definition}[thm]{Definition}
\newtheorem{example}[thm]{Example}
\newtheorem{lemma}[thm]{Lemma}

\newcommand{\EN}[1]{{\color{blue} #1}}

\title{The Yang--Mills measure on compact surfaces as a \\ universal scaling limit of lattice gauge models}
\author{Nguyen Viet {\sc{Dang}\footnote{IRMA, Université de Strasbourg, 7 rue René Descartes, 67084 Strasbourg Cedex, France AND Institut Universitaire de France, Paris, France. Email : nvdang@unistra.fr} } and Elias \sc{Nohra}\footnote{Sorbonne Université, Université Paris Cité, CNRS, Laboratoire de Probabilités, Statistique et Modélisation, LPSM, F-75005, France. Email : elias.nohra@sorbonne-universite.fr}}
\date{\today}

\begin{document}
\maketitle
\thispagestyle{empty}
\begin{abstract}
\par In this article, we study the 2–dimensional Yang--Mills measure on compact surfaces from a unified continuum and discrete perspective. We construct the Yang--Mills measure as a random distributional $1$-form on surfaces of arbitrary genus equipped with an arbitrary smooth area form, using the analytic concept of pseudo-coordinates. Our approach yields a canonical noise–flat decomposition of the measure, reflecting the topology of the surface.
\par We prove a universality theorem stating that the continuum Yang--Mills measure arises as the scaling limit of a wide class of lattice gauge theories --- including Wilson, Manton, and Villain actions --- on any compact surface. We study the convergence in natural spaces of distributions with anisotropic regularity. As further consequences, we obtain a new intrinsic construction of the Yang--Mills measure, independent of the previous constructions in the literature, and prove the convergence of correlation functions and Segal amplitudes on all compact surfaces. \\ \\ 
Keywords: Anisotropic Sobolev and H\"older spaces,  holonomy process, lattice gauge theory, Morse theory, Yang--Mills measure, random distributional $1-$forms, second micro-localization, scaling limit, tightness,  universality. \\ \\
\end{abstract}

\tableofcontents

\section{Introduction}

\par Yang--Mills theory is a quantum field theory that aims to describe interactions between fundamental particles. In Feynman's path integral formulation, the study of an Euclidean quantum field theory proceeds as follows. Given an Euclidean space-time $M$--- for instance $\R^d$ or a more general manifold--- and a space of fields ---for instance functions, or generalized functions--- over $M$, we associate two quantities of interest:
\begin{itemize}
    \item observables, which describe the physically relevant aspects of the fields;
    \item an action functional $S$, which represents a cost associated with a field's configuration.
\end{itemize}

Roughly speaking, at least for our concern, describing a quantum field theory amounts to computing the expected values of observables, where the field $\Phi$ is chosen at random according to the probability measure
\begin{equation} \label{formalmeasure}
    \mu(\dd \Phi)\coloneq \frac{1}{Z}\exp(-S(\Phi))\, \dd \Phi
\end{equation}
defined on the space of all possible fields. In (\ref{formalmeasure}), the constant $Z$ denotes the normalization factor ensuring that $\mu$ is a probability measure; and the measure $\dd \Phi$ is the formal Lebesgue measure on the space of all fields.
\par The main mathematical difficulty lies in the rigorous construction of this probability measure, since the space of all possible fields is typically infinite-dimensional. The first issue is that the Lebesgue measure $\dd \Phi$ is, in general, ill-defined. The second issue is that sometimes, the normalization constant $Z$ appears to be infinite, even at a formal level. The third issue is that the fields $\Phi$ sampled by the measure  $\mu$ are expected to be too irregular for the quantity $S(\Phi)$ to be well-defined.

\par A typical example arises when one attempts to construct quantum mechanics as a quantum field theory. In the simplest setting, the spacetime is $M=[0,1]$, the fields are functions $\Phi:[0,1]\rightarrow \R$, the observables are the joint values of $\Phi$ at $0\leq t_1<\cdots <t_n\leq 1$, and the action functional is given by
\[
S(\Phi)=\frac12\int_{[0,1]}\vert \Phi'(t)\vert ^2\,\dd t.
\]
There is little doubt that the rigorous construction of the measure $\mu$ in this case is the Wiener measure, so that a field sampled at random according to this measure is a standard Brownian motion $(B_t)_{t\in [0,1]}$. However, we emphasize that $S\big((B_t)_{t\in [0,1]}\big)$ is ill-defined, since $B$ is almost surely nowhere differentiable.
\par A natural question is then how one can determine in which sense this measure is indeed the correct realization of $\mu$. Several results support this identification.
\begin{itemize}
    \item The Cameron--Martin theorem, which states that the only translations of $B$ that preserve the measure up to absolute continuity are given by functions with finite action.
    \item Schilder's theorem, which establishes a large deviation principle and relates the action functional  $S$ to the Wiener measure.
    \item The well-known fact that Brownian motion can be obtained as the scaling limit of discrete (finite dimensional) approximations of $\mu$.
\end{itemize}
\par In the case of Yang--Mills theory, the situation is somewhat analogous. The spacetime is a $d$-dimensional Riemannian manifold $M$ equipped with a volume measure that we denote by $\mathrm{vol}$. The fields are differential $1$-forms valued in the Lie algebra $\g$ of a compact Lie group $G$, that we will assume without loss of generality matricial. The choice of the Lie group $G$ depends on the type of interactions one wishes to describe\footnote{For example, $G=U(1)$ for quantum electrodynamics, or $G=SU(3)$ for quantum chromodynamics.}. 

The action of a field $A\in \Omega^1(M,\g)$ is defined as half the $L^2$ norm of its curvature:
\[
S_{YM}(A)\coloneq \frac12\int_{M}\|\dd A+A\wedge A\|^2\,\dd\mathrm{vol}.
\label{YMaction}\]
Let us describe an important class of observables of Yang--Mills theory, called the Wilson loops. Given a field $A\in \Omega^1(M,\g)$, and a closed loop $c:[0,1]\to M$, the holonomy of $A$ along $c$, that we denote by $\mathrm{Hol}(A,c)$, is the solution to the ordinary differential equation
\begin{eqnarray*}
    \begin{cases}
        \overset{.}{y}_t=-y_tA_{c_t}(\overset{.}{c}_t) \\
        y_0=1_G
    \end{cases}
\end{eqnarray*}
taken at time one. The Wilson loop observable associated with $c$ is then 
$
\Tr\big(\mathrm{Hol}(A,c)\big). \label{hol}
$
\footnote{In the case of electrodynamics, this quantity represents the change in the phase of the wave function of a particle as it travels along $c$.} 
\par A fundamental physical feature of Yang--Mills theory is its gauge invariance. This invariance asserts that many seemingly different $1$-forms $A\in\Omega^1(M,\g)$ actually represent the same physical field. Indeed, any $1$-form of the form
\[
g\cdot A \coloneq  g^{-1}Ag - g^{-1}\dd g, \label{GaugeTransform}
\]
where $g\in \mathcal{G}\coloneq C^\infty(M,G)$\label{GaugeGroup}, represents the same physical field as $A$. In particular, they yield the same Wilson loop observables for all loops in $M$. Two such $1$-forms are said to be gauge equivalent. Moreover, they satisfy
$
S_{YM}(A)=S_{YM}(g\cdot A).
$
This suggests that the Yang--Mills measure can be formally factorized as
\begin{equation} \label{factor}
\mu_{YM}(\dd A) = \frac{1}{Z}\left(\exp(-S_{YM}([A]))\,\dd [A]\right)\otimes \mathcal{D}g,
\end{equation}
where the first factor is interpreted as a measure on $A\,\mathrm{mod}\,\mathcal{G}$, and $\mathcal{D}g$ denotes the (formal) Haar measure on the infinite-dimensional group $\mathcal{G}$.

\par We can now state the main mathematical difficulties involved in the construction of this measure, two of which were already present in the case of Brownian motion:
\begin{itemize}
    \item there is no Lebesgue measure on the space of $1$-forms;
    \item one expects typical fields $A$ to be so irregular that neither $\dd A\label{ExtDiff}$ is defined pointwise, nor the product $A\wedge A$ is well-defined.
\end{itemize}
Of course, in the present setting these issues are more severe than in the case of Brownian motion. Indeed, since the Brownian action is quadratic, it falls within the framework of Gaussian measures, for which a general and well-developed theory exists. In contrast, the Yang--Mills action contains the quartic term coming from $A\wedge A$, and the resulting measure is therefore no longer Gaussian. Another difficulty, specific to gauge theories, is that the factorization suggested by (\ref{factor}) formally leads to $Z=\infty$, introducing an additional divergence that must be addressed.

\par Due to these mathematical complications, a rigorous construction of the Yang--Mills measure in three or more dimensions is currently unknown. However, over the last $50$ years or so, the two-dimensional Yang--Mills measure has been successfully constructed by several authors, and its properties have been extensively studied.

\par In the physics literature, the study of two-dimensional Yang--Mills theory dates back to the work of Migdal~\cite{Migdal} in 1975 on $\mathbb{T}^2$, whose approach was discrete, in the spirit of lattice gauge theories. This construction was later extended to surfaces of arbitrary genus in the groundbreaking work of Witten~\cite{Witten1, Witten2}, where deep connections with many other areas of mathematics were revealed.

\par From a mathematical perspective, what makes the two-dimensional case far more tractable is the fact that the theory is locally free, meaning that for any $1$-form $A$, one can find a gauge-equivalent $1$-form $B$ such that locally, $B\wedge B=0$.
This observation was central to the work of Driver~\cite{driver1989ym2}, whose construction for $M=\R^2$ is sufficiently important for the present work that it is recalled in Section~\ref{ss:DriverPlane} and adapted to the cylinder in Section~\ref{ss:DriverCylinder}. Sengupta subsequently generalized this approach to the sphere~\cite{Sengupta92}, and then to all compact surfaces~\cite{Sengupta97}.
\par These constructions led to the so-called Driver--Sengupta formula, which describes the joint law of the holonomies along a finite collection of loops. This, in turn, gave rise to the holonomy process, a $G$-valued stochastic process indexed by loops on the surface, first introduced in all generality by Lévy in his PhD thesis~\cite{Levyphd}. 
\par In~\cite{Levyphd, LevyMarkov}, the construction yields a holonomy process that can be interpreted as a random geometric connection, in the sense that it creates one random object that associates a parallel transport to each curve on the surface. However, \textbf{no analytic connection was constructed}: in particular, \textbf{no random element $A\in \Omega^1(M,\g)$} whose parallel transport generates the holonomy process was obtained.
\par In 2018~\cite{Chevyrev_2019}, in a remarkable advance, Chevyrev constructed for the first time a genuine Yang--Mills measure directly on distributional connections on the flat torus $\mathbb{T}^2$ as a scaling limit of discrete random connections whose holonomies satisfy the Driver--Sengupta formula. This result provided a rigorous realization of the Yang--Mills measure at the level of random connections, rather than merely at the level of holonomies. A key feature of Chevyrev's approach is that it produces a random $1$--form with isotropic regularity $0^{-}$. More precisely, the resulting object is a random distributional $1$--form on the torus,
\[
A = A_1\,\mathrm{d}x + A_2\,\mathrm{d}y,
\]
where both components $A_1$ and $A_2$ have regularity $0^{-}$. 

By contrast, Driver’s construction leads to an object with the same total regularity ($0^{-} = \tfrac{1}{2}^{-} - \tfrac{1}{2}^{-}$), but distributed heterogeneously. In the sequel, we will describe such regularity as \textbf{anisotropic} in the sense the regularity is not the same in every direction. In fact,
 there is no longer a $\dd y$ component and $A_1$ has regularity $\tfrac{1}{2}^{-}$ in the $y$ direction and regularity $-\tfrac{1}{2}^{-}$ in the $x$ direction. This asymmetry prevents the direct application of the solution theory based on regularity structures. Chevyrev’s isotropic construction avoids this issue as both components have the same regularity $0^-$, and the solution theory can be applied separately and consistently to each of them. 

\par This breakthrough stimulated a series of important developments within the stochastic partial differential equations community, aimed at understanding the stochastic quantization of gauge theories. Notable contributions in this direction include the works of Chandra--Chevyrev--Hairer--Shen~\cite{CCHS1,CCHS2}, Cao--Chatterjee~\cite{CC1,CC2}, Bringmann--Cao~\cite{BringCao1,BringCao2}, Chevyrev--Shen~\cite{ChevShen}, Shen--Smith--Zhu~\cite{SSZ1}, and Shen--Zhu--Zhu~\cite{SZZ1}, and the very interesting work~\cite{CKM} in preparation by Chevyrev--Klose--Mohamed.

\par We also mention that it would be interesting to investigate potential connections between our approach and some other recent works on 2-dimensional gauge theories, like for instance the works of Park--Pfeffer--Sheffield--Yu~\cite{PPSY}, Sauzedde~\cite{IS}, or Dahlqvist--Lemoine~\cite{LD1,LD2}. 

\subsection{Contribution of the paper} \label{ss:Contributions}
\par Let us start by outlining that the first author, together with his coauthors in~\cite{BCDRT}, introduces on a general surface a random distributional $1$-form whose holonomies generate Lévy’s random holonomy process.  They investigate the problem of constructing the Yang--Mills measure as a random distributional $1-$form from a complementary perspective to ours, most notably through the lens of dynamical systems, leading to results of independent interest.  They construct the measure directly in the continuum and study the associated holonomy process using a novel gauge : the Morse gauge. The latter, viewed as an analytic object in its own right, is studied in depth in~\cite{BCDRT,BCDRT2}, where it is developed within the frame of classical gauge theory.
\paragraph{Separation of probabilistic and geometrical considerations. } The first contribution of the present work is to present a new version of the construction from~\cite{BCDRT} made possible by the introduction of \textbf{pseudo-coordinates} (see Section~\ref{s:globalresolutionsurface}). The construction of these pseudo-coordinates relies crucially on the recent microlocal approach to 
Morse--Smale (and more generally Axiom A) flows developed by the first author with G.~Rivière~\cite{DR16, DRWitten} which was inspired by several results on the microlocal analysis of hyperbolic dynamical systems~\cite{FaureSjostrand,DZ16,Baladibook}. 
Thanks to these tools, an arbitrary surface can be represented by a cylinder (see Section~\ref{s:fundamentalformula}), which allows all probabilistic arguments to be carried out on the cylinder and only afterwards pulled back to the original surface. This yields a complete separation between the analytical and geometrical aspects of the problem, on the one hand, and its probabilistic aspects, on the other.
\par From an analytical point of view, the focus lies on the problem of pulling back currents on the cylinder to the original surface, as well as gluing distributions. The proof relies on a mixture of Morse theory with microlocal and harmonic analysis techniques. From a probabilistic point of view, all steps are performed directly on the cylinder.

\paragraph{A universal scaling limit.}  Taken together with~\cite{BCDRT}, the present paper proposes, for the first time, a construction of a candidate Yang--Mills measure on the space of distributional $1$--forms on a surface of arbitrary genus and smooth area form. We show 
 that the 2-dimensional Yang--Mills measure arises as the scaling limit of a wide range of lattice gauge theories (see Sections~\ref{s:CLT} and~\ref{s:functspaces}). These lattice gauge theories can be interpreted as discretizations of the continuum measure on a sequence of well chosen lattices. The approach relies on a discrete analogue of the Morse gauge, and the convergence is established towards the same gauge fixed Yang--Mills measure constructed in~\cite{BCDRT}. We also analyze some aspects of this convergence in Besov spaces. Beyond its foundational role in the present work, the Morse gauge provides a powerful and conceptually elegant framework to state and prove universality.
\par We further observe that this universal scaling limit provides a new, intrinsic construction of the measure, independent of any of the previous constructions in the literature.  Moreover, it is possible to define the parallel transport associated with the limiting rough object using the Wong--Zakai theorem, which gives back the holonomy process. We will now state our first main result, after describing the setting of the present work. 
\par Let $\Sigma$ be a closed compact surface of genus $g$,  endowed with an area measure $\sigma$. Let $G$ be a compact Lie group, with Lie algebra $\g$ always assumed to be equipped with an invariant scalar product. Let  $(Q_t\dd g)_{t>0}$  be a family of probability measures on $G$, arising from the Manton, Wilson, or Villain action. 
\par Using the flow lines and level sets of a Morse function on $\Sigma$, we construct a sequence $(\mathcal{T}_N\coloneq (\mathbb{V}_N,\mathbb{E}_N,\mathbb{F}_N))_{N\geq 0}$ of increasingly fine lattice approximations of $\Sigma$ (see Section~\ref{s:Morselattice} for the construction of the lattice, and 
Section~\ref{s:morse} for a recollection on Morse functions).  
\par For each $N\geq 0$, we define a piecewise smooth random $\mathfrak{g}$-valued $1$-form $A_N$ on $\Sigma$ such that the corresponding holonomy process $(\mathrm{Hol}(A_N,c))_{c\in \mathrm{Loops}(\mathcal{T}_n)}$, verifies the Driver--Sengupta formula associated to $(Q_t\dd g)_{t>0}$, meaning that $(\mathrm{Hol}(A_N,c))_{c\in \mathrm{Loops}(\mathcal{T}_n)}$ is equal to the holonomy process associated with the lattice measure $\mu_N$ defined by 
\begin{equation}\label{eq:DSformula}
\dd\mu_N\big((g_e)_{e\in\mathbb{E}_N}\big)=\prod_{F\in \mathbb{F}_N}Q_{\sigma(F)}\big(\operatorname{Hol}(g,\partial F)\big)\underset{e\in \mathbb{E}_N}{\bigotimes} \dd g_e.
\end{equation}

All along our paper, whenever we will mention the Driver--Sengupta formula for a lattice approximation $\mathcal{T}$ of a surface $\Sigma$ (which could be closed or with boundary), we mean a measure on $G^{\mathbb{E} } $ given by the above equation (\ref{eq:DSformula}).

Then, the main theorem is as follows.
\begin{thm}[Universal scaling limit of lattice gauge theories] \label{thm:mainthmintro}
  The sequence $(A_N)_{N\geq 0}$ converges in law, in the H\"older space $C_\mathrm{loc}^{-(\frac{5}{2}+\varepsilon)}(\Sigma\setminus f^{-1}(\max (f));\g)$ of distributional $\mathfrak{g}$-valued $1$-forms, to the Morse-gauge-fixed Yang--Mills measure on $\Sigma$.
Moreover outside the union of some finite number of closed curves, the sequence $A_N$ converges to $A$ in local anisotropic Hölder spaces of regularity $(-\frac{1}{2}^-, \frac 1 2^-)$.
\end{thm}

 In fact, we prove more detailed regularity results for the limiting random connection on the surface $\Sigma$ depending on the region of the surface $\Sigma$. We refer the curious reader to paragraph~\ref{sss:regimeofregularity} where we detail these different 
regions and the corresponding regularities. 
Let us outline the main steps leading to this theorem. The first step consists in blowing up the surface at the maximum of the Morse function, thereby obtaining a new surface $\mathcal{S}$ with one boundary component. Next, we encode this surface using pseudo-coordinates $(r, \theta)$ that behave like polar coordinates: the Morse function serves as the radial (or height) coordinate, while the base angle provides the angular one. This construction (see Section~\ref{s:globalresolutionsurface}) yields an almost-diffeomorphism between the surface and a cylinder, together with a well-defined continuous pullback of currents from the cylinder to the surface.
\par The second step is probabilistic. We transfer the lattice structure onto the cylinder and define, as before, a sequence of random $\mathfrak{g}$-valued $1$-forms $A_N$ realizing this time the \emph{free-boundary} lattice Yang--Mills measure, meaning that there are no constraints on the bonds located along the boundary of the cylinder. 

\par For topological reasons discussed in Subsection~\ref{pant}, each $A_N$ decomposes into the sum of two terms,
\[
A_N = A_N^{\mathrm{noise}} + A_N^{\mathrm{flat}},
\]
where $A_N^{\mathrm{flat}}$ is a random flat connection and $A_N^{\mathrm{noise}}$ is the noise component. This comes from the fact that one expects the measure $\exp(-S_{YM}(A))$ to look like random perturbations around the minimizers of $S_{YM}(A)$, and the above decomposition of $A_N$ says that it is the sum of two independent term: a random flat connection, that is a random minimizer of $S_{YM}$, plus a noise term.
\par We then establish the convergence of the free lattice measure in suitable functional spaces on the cylinder. Finally, the continuity of the pullback map leads the result on $\mathcal{S}$, as stated in the following proposition.
\begin{proposition}[Bulk--singular decomposition of the free boundary Yang--Mills measure]\label{prop:boundarycase}
The sequences $A_N^{\mathrm{bulk}}$ and $A_N^{\mathrm{sing}}$ are independent. Moreover, the following convergence statements hold.
\begin{itemize}
    \item The noise component $A_N^{\mathrm{noise}}$ converges in law, in anisotropic Hölder spaces of regularity $(-\frac{1}{2}^-, \frac 1 2^-)$, to a de Rham primitive of a $\mathfrak g$--valued white noise on $\mathcal{S}$.
    \item The flat component $A_N^{\mathrm{flat}}$is a constant random $1$--current of the form $
        \sum_{a=1}^{2g} U_a\left(\log g_a\right),
    $
    where $(U_a)_{0\leq a\leq 2g}$ are de Rham currents of degree $1$ supported on smooth curves in $\mathcal{S}$, 
    and $(g_a)_{1\leq a\leq 2g}$ are independent Haar--distributed random variables on $G$.
\end{itemize}
\end{proposition}
 To recover the measure on the closed surface, we condition the holonomy on the boundary component to be $1_G$ (see Section~\ref{ss:closing}).  We summarize this procedure in  Figure~\ref{fig:steps}.
  \begin{figure}[t]
      \centering
    \includegraphics[]{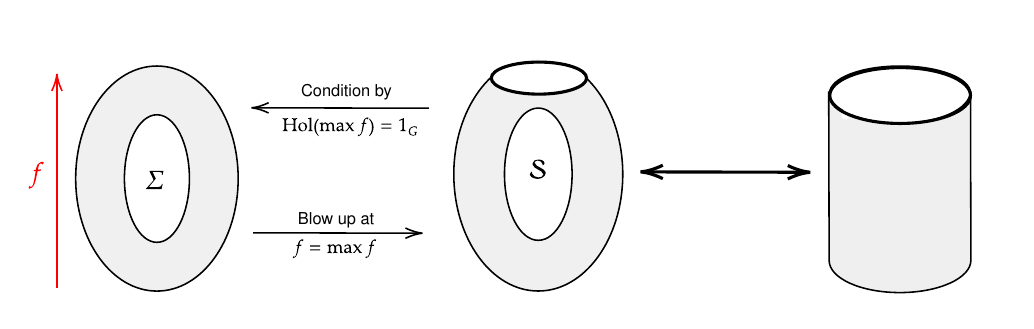}
      \caption{Main steps.}
      \label{fig:steps}
  \end{figure}
\par Let us make three remarks.
\begin{itemize}
    \item     We have the convergence for a much wider family of actions which are recalled in Section~\ref{ss:classicalactions}; however, for the sake of clarity, we leave the details to the core of the paper.
    \item Our result is in fact stronger. It allows one to work on surfaces with boundary, to impose a prescribed constraint on the conjugacy class of the boundary holonomy, and to construct a sequence of discrete $1$--forms satisfying this constraint, whose continuum limit can be identified. The proof of Theorem~\ref{thm:mainthmintro} actually relies on a disintegration of the free $1-$forms of Proposition~\ref{prop:boundarycase} along the boundary holonomy, followed by an analysis of the resulting sequence of conditioned measures.  This approach naturally leads to the study of the associated partition function, which yields results of independent interest, discussed later in this introduction.
    \item A first universality result was established by Sengupta in~\cite{Sengupta92}. He showed that, for any reasonable action (for instance the Manton or Wilson action), on the sphere, the law of the Wilson loop associated with a fixed base lattice converges to its law under the Villain action, provided the lattice is successively refined into finer and finer lattices whose faces all have the same area. Our result, by contrast, proves the convergence of the random $1$-form as a whole single object. In this sense, Sengupta’s result can be viewed as a “convergence of finite-dimensional marginals”, whereas ours plays the role of the “tightness” part in a classical Donsker-type invariance principle. Moreover, our result removes the equal-area assumption on faces and extends the universality statement to arbitrary compact surfaces.  
    \item    A second universality result for the Yang–Mills measure was proven in a groundbreaking paper by Chevyrev and Shen~\cite{ChevShen} when $\Sigma$ is the \emph{flat two-dimensional torus}, in the context of SPDEs. The present result provides an alternative proof of this theorem, notably without SPDEs, and generalizes it to surfaces of \emph{arbitrary genus with arbitrary smooth area form} and also controls the convergence of the partition function and Segal amplitudes. 
\end{itemize}

\paragraph{A rigorous link between the action and Lévy's holonomy process.}
\par In Lévy's constructions~\cite{Levyphd, LevyMarkov}, the Yang--Mills action -- which is a fundamental quantity in the theory -- appears only at a formal level. The main ingredient in this construction is, however, the heat kernel. Through the Driver--Sengupta formula, it is used as a way of expressing the Yang--Mills action of a discrete connection. Although Migdal~\cite{Migdal96} argued that the heat kernel is the most physically relevant choice, its use can also be considered somewhat arbitrary. In fact, many other actions appear in the literature, and this will be discussed in Section~\ref{actions}. 

\par This raises the natural question of whether the measure constructed by Lévy is indeed the 2D Yang--Mills measure. An argument supporting this identification is the large deviation principle for the Yang--Mills measure established by Lévy and Norris~\cite{LevyNorris}, which relates the measure to the Yang--Mills action. Other results in the same directions include~\cite{Chevyrev_2019, ChevyrevGarban}.

\par Our universality result (Sections~\ref{s:CLT} and~\ref{s:functspaces}), can be interpreted as further evidence for this link. Specifically, it shows that under any reasonable discretizations of the Yang--Mills action, the measure  coincides, in the scaling limit, with the Yang--Mills measure in the sense of~\cite{Levyphd, LevyMarkov}. This supports from a mathematical point of view the arguments of Migdal.

\paragraph{Convergence of correlation and partition functions  on all compact surfaces.} As a corollary of the key result of Theorem~\ref{convck}, we obtain the convergence of correlation functions and Segal amplitudes of lattice gauge theories on all compact surfaces. This generalizes the results of Balaban~\cite{Balaban} from the planar case to surfaces of arbitrary genus. Let us give a precise statement. 
\par Consider a compact Lie group $G$. We use conventional notations of representation theory of compact Lie groups, that we recall in Appendix~\ref{symb}. Consider a family of probability measures $(\mu_t = \rho_t(\,\cdot\,)\,dg)_{t>0}$ on $G$ satisfying the conditions of Definition~\ref{def:CLTplus}. Let $\Sigma$ be a compact surface of genus $g$ with $k$ boundary components $\partial\Sigma_1,\dots,\partial\Sigma_k$, endowed with an area measure $\sigma$. For a graph $\mathcal{T}=(\mathbb{V},\mathbb{E},\mathbb{F})$ on $\Sigma$, the partition function of the lattice gauge measure on $\mathcal{T}$ associated with $(\mu_t)_{t>0}$, and with prescribed boundary conditions
\[
\mathrm{Hol}(\partial \Sigma_1) = [g_1],\dots,\mathrm{Hol}(\partial \Sigma_k) = [g_k], \quad g_1,\dots,g_k\in G,
\]
is given by
\[
Z_{\mathcal{T}}(g_1,\dots,g_k)
\coloneq  \int_{G^{E}} 
 \prod_{i=1}^k \delta_{[g_i]}\!\bigl(\mathrm{Hol}(\partial\Sigma_i)\bigr)
\prod_{F\in \mathbb{F}} \rho_{\sigma(F)}\!\bigl(\mathrm{Hol}(\partial F)\bigr) 
\prod_{e\in \mathbb{E}} \dd g_e.
\] \label{SegalAmp}
We can give two formulas to compute $Z_{\mathcal{T}}$ . The first one is the integral formula 
\[Z_{\mathcal{T}}(g_1,\dots,g_k)=\int_{G^{2g}}\left (\underset{F\in \mathbb{F}}{\bigstar}\rho_{\sigma(F)}\right)([a_1,b_1]\cdots [a_g,b_g]c_1g_1c_1^{-1}\cdots c_kg_kc_k^{-1})\dd a \dd b \dd c .\] 
The second one is the character expansion of the latter, and it is given by 
\[Z_{\mathcal{T}}(g_1,\dots,g_k)
= \sum_{\lambda\in\widehat{G}} 
\left( \prod_{F\in \mathbb{F}} \frac{\widehat{\mu}_{\sigma(F)}(\lambda)}{d_\lambda} \right)
\frac{\chi_\lambda(g_1)\cdots \chi_\lambda(g_k)}{d_\lambda^{\,2g-2+k}}.
\]
In the special case where $\rho$ is the heat kernel $p$ on $G$, thanks to the semi-group property, the quantity $Z_{\mathcal{T}}$ becomes independent of the graph $\mathcal{T}$. It is therefore an intrinsic property of the surface denoted by $Z_{\Sigma}$, and it reads
\begin{align*}
Z_\Sigma(g_1,\dots,g_k) =\int_{G^{2g}} p_{\sigma(\Sigma)}([a_1,b_1]\cdots [a_g,b_g]c_1g_1c_1^{-1}\cdots c_kg_kc_k^{-1})\otimes_{i=1}^g\dd a \dd b \dd c  \\
\hspace{-3cm}=\sum_{\lambda\in\widehat{G}} 
e^{-\sigma(\Sigma)c_2(\lambda)}
\frac{\chi_\lambda(g_1)\cdots \chi_\lambda(g_k)}{d_\lambda^{2g-2+k}}.    
\end{align*}
These formulas, and some of their applications are studied in a work in progress of the second author and Thibaut Lemoine~\cite{NL}
We are ready to state the second main theorem. 
\begin{thm}[Convergence of Segal amplitudes]\label{ConvSegalAmp}
Let $(\mathcal{T}_n\coloneq (\mathbb{V}_n,\mathbb{E}_n,\mathbb{F}_n))_{n\ge 0}$ be a sequence of triangulations of $\Sigma$ such that there exist constants $a,A>0$ for which 
\[
\liminf_{n\to\infty} 
\frac{\Bigl\lvert \bigl\{f\in \mathbb{F}_n : \tfrac{a^2}{\lvert \mathbb{F}_n\rvert}\leq \sigma(f)\leq \tfrac{A^2}{\lvert \mathbb{F}_n\rvert}\bigr\}\Bigr\rvert}{\lvert \mathbb{F}_n\rvert} > 0.
\]
Then, for all $(g_1,\dots,g_k)\in G^k$,
$Z_{\mathcal{T}_n} \underset{n\to\infty}{\xrightarrow{C^\infty(G^k)}}Z_{\Sigma}$ in the 
 $C^{\infty}(G^k)$ topology.
\end{thm}

\subsection{Structure of the paper}\label{ss:Structure}
\par First, in Section~\ref{s:Driver} we revisit Driver's construction on the plane and adapt it to the cylinder. We then explain how this construction can be intuitively generalized to arbitrary surfaces, and we identify thereby the technical challenges we need to take care of.
\par In Section~\ref{s:morse}, we recall the  preliminaries from Morse theory and the microlocal analysis of Morse--Smale flows that we need in Section~\ref{s:globalresolutionsurface} to construct a system of polar pseudo-coordinates on a surface with one outgoing boundary component. This covers every such a surface by a cylinder, and lets us transfer the construction from the cylinder to the surface. Justifying this transfer requires us to 
introduce many functional norms and relies on tools from harmonic and functional analysis.
In particular, we recover in Section~\ref{s:fundamentalformula} the main formula for the Yang--Mills random $1$-form as it appears in~\cite{BCDRT}.
\par In Section~\ref{s:Morselattice} we use Morse theory to construct the sequence of well-chosen lattice approximations of surfaces that shows up in the statement of Proposition \ref{prop:boundarycase}. 
\par In Section~\ref{TwoCLT}, we show a local limit theorem, crucial for the study of the measure on closed surfaces which involves singular conditioning. However, as a by product, we show the corollary about the convergence in $C^\infty$ topology of the partition functions and Segal amplitudes, which is a result of independent interest. 
\par In Section~\ref{LieGroupRW},
we establish some technical estimates for Lie groups-valued random walks that we will need in identifying the scaling limit of lattice gauge theories in Section~\ref{s:CLT}. 
\par In Section~\ref{s:functspaces}, we investigate the convergence of lattice models in more refined functional spaces. 
\par Finally, in Section~\ref{ProveMainTH}, we close the surface by means of conditioning, and conclude the proof of the main theorem. 
\par In the appendix, we gather several important technical results we need along the way whose proof would interrupt the flow of the article. 

\par For a reference, the reader can refer to Figure \ref{summary} for an illustrative diagram describing the structure of the paper and the link between the sections.
\begin{figure}
    \centering
    \includegraphics[width=1\linewidth]{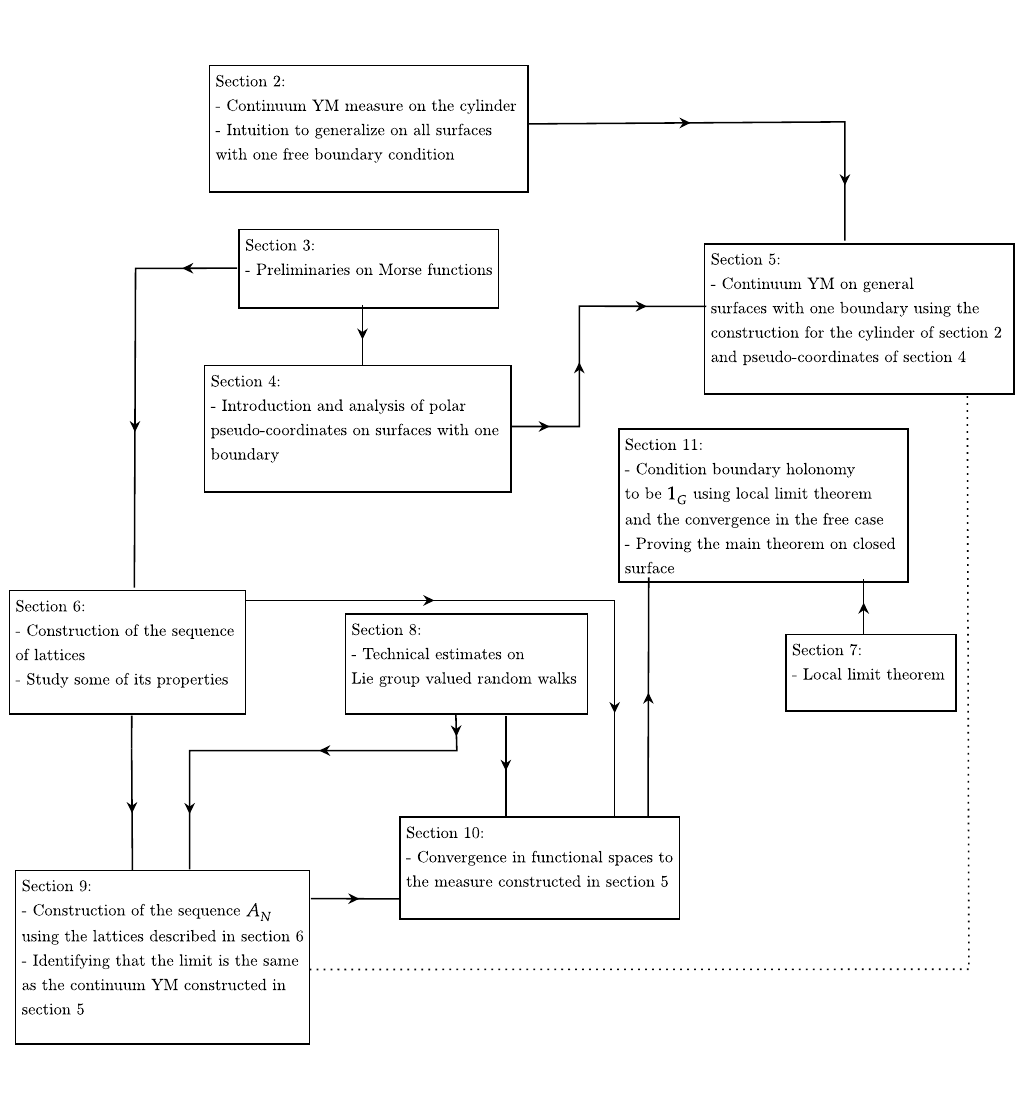}
    \caption{Structure of the paper}
    \label{summary}
\end{figure}
\section*{Acknowledgments}
\par E.N. wishes to express his deep gratitude to his PhD advisor, Thierry Lévy, for asking him the main question of this paper, for countless long and productive discussions, and for carefully reading parts of this work. He also thanks David García-Zelada and Thibaut Lemoine for many suggestions and close help, and Luis Zegarra for introducing him to Morse theory. Furthermore, he thanks Vladimir Bošković, Nicolas Fournier, François Jacopin, Thomas Jaffard, Thomas Le Guerch, Abdulwahab Mohamed, Sami Mustapha, Breki Pálsson, Maxence Petit, Damien Simon, and Lorenzo Zambotti for many interesting discussions. Finally, he acknowledges funding from the “Fondation CFM pour la Recherche” and thanks the foundation for providing excellent working conditions during his PhD.
\par N.V.D would like to express his deep gratitude to Rongchan and Xiangchan Zhu for explaining to him the results of Driver, this allowed him to start working on the subject and is the main source of inspiration. He also warmly thanks
his collaborators Yannick Guedes--Bonthonneau, Reda Chhaibbi, T\^o T\^at Dat, Gabriel Rivière for many discussions around gauge theories, probability and Yang--Mills, and also Ismael Bailleul, Laurent Charles, Ilya Chevyrev, Benoit Estienne, Léonard Ferdinand, Vladimir Fock, Colin Guillarmou, Frédéric Hélein, Paul Laurain, Thibault Lefeuvre, Thibaut Lemoine, Jiasheng Lin, Julien Marché, Antoine Mouzard, Phan Thanh Nam, Tristan Robert, Scott Smith, Nikolay Tzvetkov, Martin Vogel for interesting questions, remarks when we gave talks or just discussed the present work, which motivated us to pursue.
He also thanks Benjamin Melinand for giving wonderful references on fractional Sobolev spaces.
N.V.D acknowledges the support of the Institut Universitaire de France.

\section{Revisiting Driver's construction of $\mathrm{YM}^2$} \label{s:Driver}
In this section we revisit first Driver's argument~\cite{driver1989ym2} for the construction of the Yang--Mills measure on the plane. We then generalize it to the cylinder  with one outgoing free boundary and 
discuss possible generalization to all surfaces.  
\par We consider in this section  a compact Lie group $G$ with Lie algebra $\mathfrak{g}$, equipped with a bi-invariant scalar product.
\subsection{The case of the plane}\label{ss:DriverPlane}
 The Yang--Mills measure on $\R^2$ is the probability measure on the set of $1-$forms modulo gauge transformations $\Omega^1(\R^2,\g)/ \G $ with formal expression $d\mu = \frac{1}{Z}\exp(-S_{YM}(A))\dd A,$
where 
\begin{enumerate}
    \item the constant $Z$, also called \emph{partition function}, is a normalizing constant,
    \item the quantity $S_{YM}(A)$ is the Yang--Mills action of $A$, i.e. half of the squared $L^2$ norm of the curvature $F$ of $A\coloneq A_1\dd x+A_2\dd y$ defined as
    \begin{eqnarray*}
        F(A)&\coloneq & \dd A+A\wedge A=\left(\frac{\partial A_2}{\partial x}-\frac{\partial A_1}{\partial y}+[A_1,A_2]_{\g}\right)\dd x \wedge \dd y,
    \end{eqnarray*}
    \item the measure $dA$ is the formal Lebesgue measure on $\Omega^1(\R^2,\g)/\G$.
\end{enumerate}
This measure is ill-defined for several reasons. One way of giving a rigorous meaning to this measure was achieved in the seminal work of ~\cite{driver1989ym2}. In this article, it is discussed that every connection on $\R^2$ is gauge equivalent to a connection of the form 
\[A \dd x \ \  \text{ such that }\ \  A(\cdot ,0)=0, \]
and that the map 
\[{\rm{ax}} : \Omega^1(\R^2,\g) / \G \rightarrow \mathcal{A}_{ax}\coloneq \{A \, \dd x \, ; A(\cdot ,0)=0\}\]
has a constant (formal) Jacobian (or Fadeev--Popov determinant). 
For a connection $[A]\in \Omega^1(\R^2,\g) / \G$, the differential form ${\rm{ax}}{([A])}$ is said to be an axial gauge representation of $[A]$. The core idea is that the Yang--Mills action becomes a quadratic functional  
when restricted to connections in axial gauge.  Indeed 
\[S_{YM}(A\dd x) = \left\|\frac{\partial A}{\partial y}\right\|_{L^2(\R^2)}^2.\]
The Yang--Mills measure induced  on the slice $\mathcal{A}_{ax}$ has therefore the formal expression 
\[d\mu_{ax} \coloneq  \frac{1}{Z_{ax}}\exp\left(-\left\|\frac{\partial A}{\partial y}\right\|_{L^2(\R^2)}^2\right).\]
This means that in this representation, $\frac{\partial A}{\partial y}$ is a two dimensional white noise. Formally, $A$ can be constructed by taking a white noise $\xi_{\mathbb{R}^2}$ in $\R^2$, and defining 
\[A(x,y) = \int_0^y \xi_{\mathbb{R}^2}(x,u)\dd u ,\]
which is not straightforward to define mathematically. In fact, the white noise belongs to $C^{(-1)^-}(\R^2)$, and cannot be integrated along lines without further justification.
\par One way to do it goes as follows. 
Write the planar white noise $\xi_{\mathbb{R}^2}$ as a random series 
$\xi_{\mathbb{R}^2}= \sum_{n\geq 0} \xi^{\mathfrak{g}}_{n}(y)e_n(x),$ where $(\xi^{\mathfrak{g}}_n)_{n\geq 0}$ are i.i.d $\g-$ valued white noises on the real line $\mathbb{R}$ and $(e_n)_{n\geq 0}$ is some orthonormal basis of $L^2(\mathbb{R})$ (like for instance the Hermite polynomials times Gaussian, the Haar basis or any orthonormal wavelet basis of $L^2(\mathbb{R})$). The random series converges in $\mathcal{S}^\prime(\mathbb{R}^2)$ and the integral along the $y$ axis can be defined as
\[  \sum_{n\geq 0} \left(\int_0^y\xi^{\mathfrak{g}}_n(u)du \right) e_n(x).\]
It is natural to recognize in the above expression that 
$\left(\int_0^y\xi^{\mathfrak{g}}_n(u)du  \right)$
is a $\mathfrak{g}$--valued Brownian motion $B^{\mathfrak{g}}_n(y)$ started at $0$ and therefore to propose the random series
\[A\coloneq  \sum_{n\geq 0} B^{\mathfrak{g}}_n(y) e_n(x)\dd x\] as a natural candidate for the random YM connection on the plane $\mathbb{R}^2$. The reader will recognize that our candidate is nothing but the cylindrical Brownian motion on $L^2_x(\mathbb{R})$ where $y$ plays the role of time for the cylindrical Brownian motion.

Another way to derive a sensible object in this special gauge where the curvature is exactly $\dd A$, is to use Stokes theorem to formally define $A$ as a Schwartz distribution as follows.  Let $A_{\rm{ax}}\coloneq T\dd x$ where
\[\forall \phi \in S(\R^2),  \langle T, \phi \rangle =\left \langle \xi, (x,y) \mapsto \int_{\R} \phi(x,z)[1_{0\leq y\leq z}-1_{z\leq y\leq 0}] \dd z\right \rangle .\]
Then, Driver defines a parallel transport for the random connection $A$ thus constructed using stochastic calculus, and computes the joint law of the parallel transport for a finite set of curves. This is known as the Driver--Sengupta formula on the plane. We will revisit the main ideas but in the slightly different case of the cylinder. 

\subsection{The case of the cylinder}\label{ss:DriverCylinder}

\par Let us see how we can mimic Driver's construction on a closed disk, seen as a manifold with boundary, that can also represent a spherical cap, or a cylinder with one outgoing boundary component.  Let us use polar coordinates to describe our disc $D$. Concretely, it means that we have a smooth submersion $[0,1]_r \times \mathbb{S}^1\mapsto D$, where $\mathbb{S}^1$ is the circle parametrized by the angle $[0,2\pi[$. Let $\sigma$ be a smooth area measure on $D$. The cylinder is obtained by blowing up the disc at the origin $(0,0)$, the thing is that the induced area form on the cylinder, still denoted by $\sigma$, vanishes at $r=0$. On the blow--up space, the one form $\dd r$ is defined globally, as well as the corresponding vector field $\partial_r$.    \par Distributions and currents on the cylinder are defined as the topological dual of the space of smooth forms that vanish near $1$ and are smooth up to $r=0$. 
Conceptually, it means our distributions are extendible near $r=0$ but not near $r=1$.

\par Based on the previous discussion, 
let us consider a white noise $\xi_\sigma$ on $(D,\sigma)$ corresponding to the area form $\sigma$, that we will see as a random distributional $2-$form. We will try to define a primitive $A$ of $\xi_\sigma$ such that $A_{|r=0}=0$, and such that $i_{\partial r}A =0$. This means that in polar coordinates, $A$ does not have a component along $\dd r$.  For some $\theta \in [0,2\pi)$, let us denote by $0 \xrightarrow{\theta} r$ the curve whose polar parametrization is given by $c(t)=(t,\theta), 0\leq t\leq r$. 

\begin{definition}[Definition-Proposition]
    Let $A$ be defined as follows by duality.
    \[\forall\,  \alpha \in  \Omega^1(D,\g),  \langle A, \alpha \rangle \coloneq \left \langle \xi_\sigma  , (r,\theta) \mapsto -\int_{r \xrightarrow{\theta} 1} \alpha \right \rangle.\]
The well-defined random de Rham current $A$ is called the axial-gauge Yang--Mills distributional $1-$form. 

We can also define $A$ by some explicit formula: 
\begin{align*}
A\coloneq  \sum_{n\geq 0} \left( \int_0^r \sqrt{\sigma(s,\theta)} \xi_n^{\mathfrak{g}}(s)ds \right)   e_n(\theta)\dd \theta, 
\end{align*}
where $(\xi_n^{\mathfrak{g}})_{n\geq 0}$ are i.i.d $\mathfrak{g}$--valued white noise.
\end{definition}
\begin{proof}
For $a \in  \Omega^1(D,\g)$, the function
\[ (r,\theta) \mapsto -\int_{r \xrightarrow{\theta} 1} \alpha\] 
is smooth and vanishes at all orders at the boundary circle $\{r=1\}$ (since $\alpha$ does). Moreover, the map 
\[\alpha \in \Omega^1(D,\g)\mapsto \left\{ (r,\theta) \mapsto -\int_{r \xrightarrow{\theta} 1} \alpha \in C^{\infty}(D,\g) \right\}\] 
is continuous for the topology of test functions that vanish up to all order at $\{r=1\}$.

To get  another formula, we shall use another representation of the $2D$ white noise in polar coordinates.
Rewrite $\xi_\sigma$ as a random series using Fourier decomposition in $\theta$:
\begin{align*}
\xi_\sigma= \sum_{n\geq 0}  \left(\sigma(r,\theta) \right)^{\frac{1}{2}} \xi^\mathfrak{g}_n(r)e_n(\theta)\dd r \wedge \dd\theta    
\end{align*}
where $\xi_n^{\mathfrak{g}}$ are i.i.d $\mathfrak{g}$--valued white noise and $(e_n(.))_{n\geqslant 0}$ is any ONB of $L^2([0,2\pi])$. The factor $\left(\sigma(r,\theta) \right)^{\frac{1}{2}}$ ensures we get the correct white noise normalized w.r.t. the area form $\sigma$. Then contract the two form $\xi_\sigma$ with the radial vector field $\partial_r$ and take the primitive in the radial direction.
This is represented as the random series
\begin{align*}
    \sum_{n\geq 0} \underbrace{\left( \int_0^r\left(\sigma(s,\theta) \right)^{\frac{1}{2}} \xi^\mathfrak{g}_n(s)ds \right)} e_n(\theta)\dd\theta 
\end{align*}
where the reader has to think of the terms underbraced as independent reparametrized $\mathfrak{g}$--valued Brownian motions. 
The above series yields a de Rham primitive of $\xi_\sigma $ which is the correct random distribution $A$.
\end{proof}

By the correct we mean that it is possible to define parallel transport of this object, and that these parallel transports verify the Driver--Sengupta formula for a wide class of graphs. This is what we will do in the remaining of this section.

\paragraph{Defining line integrals.}

We will now define a natural parallel transport for $A$ around curves of the form 
$\theta_1\xrightarrow{r}\theta_2 $
for $r\in (0,1]$ and $0\leq \theta_1 <\theta_2 \leq 2\pi$.
Consider our random $1-$form and the curve $\theta_1\xrightarrow{r}\theta_2$ defined as the curve parametrized in the polar coordinates as $c(t)=(r,t); \theta_1\leq t \leq t_2$.
The first thing to note is that $\dd A=\xi$.
\begin{proposition}
    The equation $\dd A=\xi$ holds in the sense of currents. 
\end{proposition}
\begin{proof}
For $\phi \in C^\infty(D,\g)$, we have 
    \[  \langle \dd A, \phi \rangle =   \langle A, \dd\phi \rangle = \left \langle \xi, (r,\theta) \mapsto \int_{[r,1]\times\{\theta\}} \dd\phi \right \rangle= \left \langle \xi, (r,\theta) \mapsto -\phi(1,\theta)+\phi(r,\theta)\right  \rangle = \langle \xi, \phi \rangle \] 
where we used the fact that our test function $\phi$ vanishes near $r=0$.
In the proof of the Lemma, the duality pairing $\left\langle \psi_1,\psi_2 \right\rangle $ means 
$\int_D \left\langle\psi_1\wedge \psi_2\right\rangle_{\mathfrak{g}}$ where we used the exterior product and 
the natural inner product in the Lie algebra $\mathfrak{g}$.
\end{proof}

Then, Stokes theorem let us define the line integral of $A$ along $\theta_1\xrightarrow{r}\theta_2$ as follows. Since $i_{\partial r} A =0$, in whatever way we define this integral, it is natural for    
\[\int_{0\xrightarrow{r}\theta} A
 = \int_{\partial \square(0,r,0,\theta)} A \]
 to hold where $\square(r_1,r_2,\theta_1,\theta_2)$ is the unique rectangle in the cylinder with vertices $(r_1,\theta_1)$, $(r_2,\theta_2)$, $(r_1,\theta_2)$, and $(r_2,\theta_1)$, see Figure \ref{rec1}.
\begin{figure}
    \centering
    \includegraphics[width=1.2\linewidth, trim={3cm 1cm 0 0}]{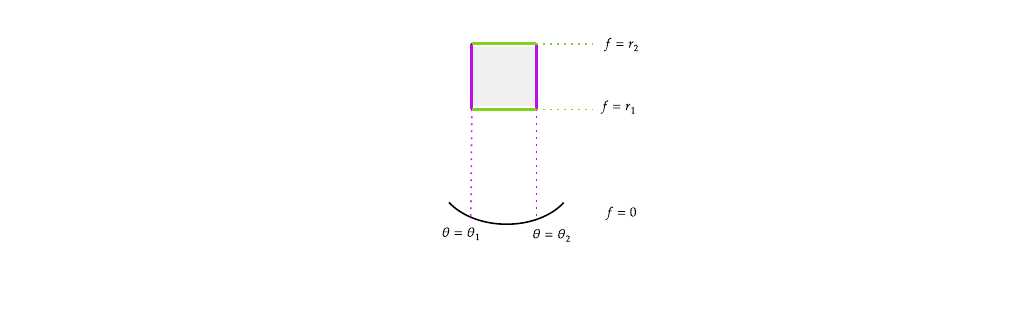}
    \caption{The rectangle $\square(r_1,r_2,\theta_2,\theta_2)$ in the cylinder.}
    \label{rec1}
\end{figure}
 Then, we use Stokes theorem to define 
 \[\int_{\partial \square(0,r,0,\theta)} A\coloneq \int_{\square(0,r,0,\theta)}\dd A =\xi(\mathds{1}_{\square(0,r,0,\theta)}).\] We get therefore the following definition which describes completely the probability distribution of line integrals.

 \begin{definition}
     Let $r\in (0,1]$ and $0\leq \theta_1 <\theta_2 \leq 2\pi$. The line integral of $A$ along $\theta_1\xrightarrow{r}\theta_2$ is defined as 
     \[\int_{\theta_1\xrightarrow{r}\theta_2}A = \xi(\mathds{1}_{\square(0,r,\theta_1,\theta_2)}).\]

   The pairing $\xi(\mathds{1}_{\square(0,r,\theta_1,\theta_2)})$ is well-defined since the indicator function 
   $\mathds{1}_{\square(0,r,\theta_1,\theta_2)} \in L^2$ hence it belongs to the Hilbert space indexing 
   the white noise.
 \end{definition}

 This definition is a bit restrictive, in the sense we only have line integrals on level curves $\theta_1\xrightarrow{r}\theta_2$. We can actually define the line integral on some slightly more general class of curves. This is treated in detail in the paper~\cite{BCDRT}.
We can sketch the main idea here. Assume we are given a piece of arc $\gamma$ : $$\gamma: \theta\in [a,b]\subset [0,2\pi]\mapsto (\theta, r(\theta)) $$ where $r\in C^\infty([a,b], \mathbb{R}_{\geqslant 0})$, 
that we can represent as a piece of vertical graph on the cylinder. Then we can define the 
line integral over $\gamma$ either as the random series
\begin{equation}
  \int_{a}^b \gamma^*A:= \sum_{n\geq 0} \int_a^b \left( \int_0^{r(\theta)} \sqrt{\sigma(s,\theta)} \xi_n^{\mathfrak{g}}(s)ds \right)   e_n(\theta)\dd \theta
\end{equation}
where the series converges almost surely by some Martingale argument or by the pairing
\begin{equation}
\xi\left( \mathds{1}_{\square} \right) \text{ where }\square:=\{ (u,\theta); \theta\in [a,b], u\in [0,r(\theta)] \}.    
\end{equation}

\paragraph{From line integrals to stochastic parallel transport.}
 
 To define the parallel transport of $A$ along horizontal curves, let us recall that if $\alpha$ is a connection, the holonomy of $\alpha$ along the curve $c:[0,1]\rightarrow M$ is given by the solution at time $1$ of an ODE that can be written as 
\begin{equation} \label{PT}
\begin{cases}
       \dd X_t = X_t \dd \int_{0}^t\alpha_s(c'(s))ds \\ 
     X_0=1
\end{cases}.
\end{equation}
The idea is that  the holonomy equation for $A$ can be interpreted as a Stratonovich SDE thanks to the following proposition.  
\begin{proposition} \label{sheet}
    Let $r\in (0,1]$. The stochastic process 
    \[(r,\theta) \mapsto \int_{0\xrightarrow{r}\theta} A\]
    is a time changed $\g-$valued Brownian sheet with time changing function
   $(r,\theta) \mapsto \sigma(\square(0,r,0,\theta)).$
\end{proposition}
\begin{proof} 
     The process 
    $(r,\theta) \mapsto \xi(\mathds{1}_{\square(0,r,0,\theta)})$ 
    is gaussian, and for $\theta ,\theta' \in [0,2\pi)$,  \[\mathbb{E}\left[\xi(1_{\square(0,r,0,\theta)})\xi(1_{\square(0,r,0,\theta')})\right ] = \sigma (\square(0,r,0,\theta))\land \sigma (\square(0,r,0,\theta)).\]
\end{proof}
\begin{definition}
    Let $r\in (0,1]$ and $0\leq \theta_1 \leq \theta_2 \leq 2\pi$. The parallel transport of $A$ along $\theta_1\xrightarrow r \theta_2$ is defined as the strong solution of the Stratonovich SDE 
      \[\label{Strato}
  \begin{cases}
      \dd X_t = X_t \circ \dd \int_{\theta_1\xrightarrow{r}\theta_1 + t}A \\
      X_0 = 1
  \end{cases}
  \]
  at time $\theta_2-\theta_1$. It is denoted by 
  \[\mathrm{Hol}(A,\theta_1\xrightarrow r \theta_2).\]
\end{definition}
The main properties of this parallel transport are regrouped in the next proposition.
\begin{proposition} \label{DriverSenForA}
    The following hold.
    \begin{enumerate}
        \item For fixed $r\in (0,1]$, $\mathrm{Hol}(A,\theta \xrightarrow r \theta +\cdot)$ is a $G-$valued time-changed Brownian motion. The time changing function is 
   $t\mapsto \sigma (\square (0,r,\theta,\theta+t)).$
        \item For fixed $ 0\leq \theta_1 \leq \theta_2 \leq 2\pi$,  
        $r\mapsto \mathrm{Hol}(A,(\theta_1 \xrightarrow r \theta_2)$ is $G-$valued time-changed Brownian motion. The time changing function is 
   $t\mapsto \sigma (\square (0,t,\theta_1,\theta_2-\theta_1)).$ 
        \item For any $0 = \theta_1 <\cdots < \theta_n =2\pi $,  the family of processes $(r\mapsto \mathrm{Hol}(A,(\theta_i \xrightarrow r \theta_{i+1}))_{1\leq i \leq n}$ are independent. 
    \end{enumerate}
\end{proposition}
\begin{proof}
For the first point, the parallel transport $\mathrm{Hol}(A,(.)$ is the solution of 
          \[
  \begin{cases}
      \dd X_t = X_t \circ \dd \int_{\theta_1\xrightarrow{r}\theta_1 + t}A \\
      X_0 = 1
  \end{cases}.
  \]
  However, from Proposition~\ref{sheet}, the process $\int_{\theta_1\xrightarrow{r}\theta_1 + \cdot}A$ is a time changed Brownian motion, so that the solution has the same law as the solution of $\dd Z_t = Z_t\sqrt{F'(t)}\circ \dd B_t$, 
where $B$ is a Brownian motion and 
$F(t)=\sigma(\square(0,r,\theta_1,\theta_1+t)).$
The density of the law of the solution at time $t$ is hence 
$p_{\sigma(\square(0,t,\theta_1,\theta_2 ))}(g)\dd g.$ However, the solution to~\ref{PT} is strong and it is adapted to $\left(\int_{\theta_1\xrightarrow{\rho}\theta_1+t} A\right)_{\substack {0\leq \rho \leq r \\0\leq t \leq \theta_2-\theta_1 }}$. 
\par For the second point, fix $0\leq \theta_1 \leq \theta_2 \leq 2\pi$, and $0\leq r_1 <\cdots < r_n \leq 1$. Let us calculate the joint law of 
  \[(X^1,\dots,X^n)\coloneq ({\rm Hol}_A\{\theta_1\xrightarrow{r_1} \theta_2 \},\dots,{\rm Hol}_A\{\theta_1\xrightarrow{r_n} \theta_2 \}).\]
   The vector $(X^1,\dots,X^n)$ is the strong solution to 
  \[
  \begin{cases}
      \dd X^1_t = X^1_t \circ \dd \int_{\theta_1\xrightarrow{r_1}\theta_1+t}A \,\,; X^1_0=1 \\
      \vdots \\
     \dd X^n_t = X^n_t \circ \dd \int_{\theta_1\xrightarrow{r_n}\theta_1+t}A \,\,; X^n_0=1
  \end{cases}.
  \]
  However, we have for $i=2,\dots,n$, 
  \[\int_{\theta_1\xrightarrow{r_i}\theta_1 + t}A = \xi(\mathds{1}_{\square(0,r_i,\theta_1,\theta_1+t)})=\xi(\mathds{1}_{\square(r_{i-1},r_i,\theta_1,\theta_1+t)})+\xi(\mathds{1}_{\square(0,r_{i-1},\theta_1,\theta_1+t)}),\]

and we have the following equality in law 
\[\left( \dd \int_{\theta_1\xrightarrow{r_i}\theta_1+t}A \right)_{i=1,\dots, n} = \left(F_1(t)\circ \dd B^1_t+\cdots +F_i(t)\circ \dd B^i_t\right)_{i=1\dots,n},\]  
where $(B^i)_{i=1,\dots, n}$ are independent Brownian motions and 
\[F_i(t) = \frac{\dd}{\dd t} \sigma(\square(r_{i-1},r_i,\theta_1,\theta_1 +t)).\]
The vector $(X^1,\dots,X^n)$ has thus the same law as the solution of 
  \[
  \begin{cases}
      \dd  X^i_t = X^i_t \left(F_1(t)\circ \dd B^1_t+\cdots +F_i(t)\circ \dd B^i_t\right) \\
      X^i_0 = 1  \\
      1\leq i\leq n
  \end{cases},
  \]
Let's calculate $\dd \{X^i (X^{i-1}_t)^{-1}\}$. To simplify notations, we will call $Y\coloneq X^i$ and $X\coloneq X^{i-1}$. We have 
\[\dd(Y_t X_t^{-1})= \dd  Y_t X_t^{-1} + Y_t \dd  X_t^{-1} + \dd  Y_t \dd  X_t^{-1}.\]
To compute $\dd  X_t^{-1}$ we write 
\[0=\dd(X_tX_t^{-1})=\dd  X_t X_t^{-1} + X_t \dd  X_t^{-1} + \dd  X_t \dd  X_t^{-1},\]
which gives 
\[\dd  X_t^{-1}=-X_t^{-1}\dd  X_t X_t^{-1} - X_t^{-1}\dd  X_t \dd  X_t^{-1}.\]
This equation tells us that the local-martingale part of $\dd  X_t^{-1}$ is contained in $-X_t^{-1}\dd  X_t X_t^{-1}$, the second term containing only processes of finite variations. Therefore, 
\[\dd  X_t^{-1}=-X_t^{-1}\dd  X_t X_t^{-1} + X_t^{-1}\dd  X_t X_t^{-1}\dd  X_t X_t^{-1}.\]
Now, we have 
\begin{eqnarray*}
    &&\dd(Y_t X_t^{-1})\\
    &=& \dd  Y_t X_t^{-1} + Y_t \dd  X_t^{-1} + \dd  Y_t \dd  X_t^{-1} \\
   &=& Y_t (Y_t^{-1}\dd  Y_t) X_t^{-1} - Y_t(X_t^{-1}\dd  X_t) X_t^{-1} + Y_t(X_t^{-1}\dd  X_t)(X_t^{-1}\dd  X_t)X_t^{-1}- Y_t(Y_t^{-1}\dd  Y_t)(X_t^{-1}\dd  X_t)X_t^{-1}\\
   &=&Y_t \left(F_1(t)\circ \dd B^1_t+\cdots +F_i(t)\circ \dd B^i_t\right) X_t^{-1} - Y_t\left( F_1(t)\circ \dd B^1_t+\cdots +F_{i-1}(t)\circ \dd B^{i-1}_t\right) X_t^{-1} \\
   &-& Y_t\left(F_1(t)\circ \dd B^1_t+\cdots +F_i(t)\circ \dd B^i_t\right)\left( F_1(t)\circ \dd B^1_t+\cdots +F_{i-1}(t)\circ \dd B^{i-1}_t\right)X_t^{-1} \\
   &+& Y_t\left( F_1(t)\circ \dd B^1_t+\cdots +F_{i-1}(t)\circ \dd B^{i-1}_t\right)\left(F_1(t)\circ \dd B^1_t+\cdots +F_{i-1}(t)\circ \dd B^{i-1}_t\right)X_t^{-1} \\
   &=&Y_t \left(F_i(t)\circ \dd B^i_t\right) X_t^{-1} \\
   &=&(Y_tX_t^{-1}) F_i(t)X_t(\circ \dd B^i_t)X_t^{-1}.
\end{eqnarray*}
where we have used that $\circ \dd B^i_t \circ \dd B^{i-1}_t = 0$ since $B^i$  and $B^{i-1}$ are independent.
Therefore, the vector $\left(X^1_t, X^2_t(X_t^{1})^{-1},\dots,X^n_t(X_t^{n-1})^{-1}\right)$ has the same solution as the vector $(Z^1,\dots,Z^n)$, solution of 
  \[
  \begin{cases}
      \dd  Z^1_t = Z^1_t F_1(t)\circ \dd B^1_t\\
      \dd Z^i_t=Z^i F_i(t) \Ad_{\{Z^i_tZ^{i-1}_t\cdots Z^1_t \}}\left(\circ \dd B^i_t\right) \, \, ;i=2,\dots, n \\
      Z^i_0 = 1 \,\, ; i=1,\dots, n  
  \end{cases}.
  \]
Moreover, 
\[\left\{\circ \dd B^1_t, \Ad_{\{ Z^1_t \}}\left(\circ \dd B^i_t\right), \cdots ,\Ad_{\{Z^i_tZ^{i-1}_t\cdots Z^1_t \}}\left(\circ \dd B^i_t\right)\right\}\] 
is a family of independent white noises driving independent Brownian motions. The vector 
$$\left(X^1_t, X^2_t(X_t^{1})^{-1},\dots,X^n_t(X_t^{n-1})^{-1}\right)$$ 
has finally the same law as the solution to 
  \[
  \begin{cases}
      \dd  Z^i_t = Z^i_t F_i(t)\circ \dd W^i_t  \\
      Z^i_0 = 1 \\
    i=1,\dots, n  
  \end{cases},
  \]
where $(W^i)_{i}$ are independent Brownian motions. Therefore, $\left\{X^1_t, X^2_t(X_t^{1})^{-1},\dots,X^n_t(X_t^{n-1})^{-1}\right\}$  is an independent family and  
\begin{itemize}
    \item the law of $X^1_{ {  \theta_2-\theta_1} }$ is
    \[ \mathbb{P} \left( X^1_{\theta_2-\theta_1}\in dg \right)  = p_{\sigma(\square(0,r_1,\theta_1,\theta_2))}(g)\dd g ,\]
    \item the law, for some $i$, of $X^i_{{  \theta_2-\theta_1}}(X_{{  \theta_2-\theta_1}}^{i-1})^{-1}$ is 
    \[ \mathbb{P} \left( X^i_{\theta_2-\theta_1}(X_{\theta_2-\theta_1}^{i-1})^{-1}\in dg \right) = p_{\sigma(\square(r_{i-1},r_i,\theta_1,\theta_2))}(g)  \dd g.\]
\end{itemize}
which proves the second point. 

\begin{figure}[t]
    \centering
    \includegraphics[width=\linewidth]{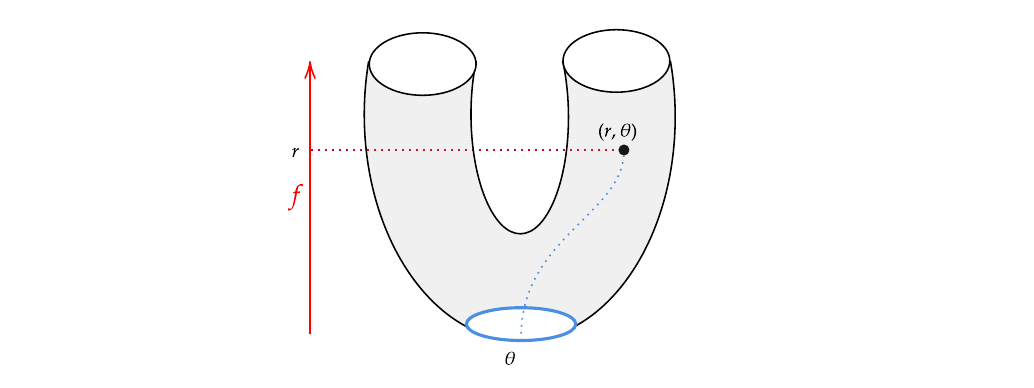}
    \caption{Morse coordinates.}
    \label{Morsecoord}
\end{figure}

\par For the third point, and since all the solutions are strong, it is enough to note from the proof of the first point that the process $\mathrm{Hol}(A,(\theta_i\xrightarrow \cdot \theta_{i+1} )$ is measurable with respect to 
\[\left\{\xi(f); f\in L^2(D), \mathrm{supp} f \subset (\theta_{i},\theta_{i+1})\times[0,1]\right\}.\]
\end{proof}
As a direct consequence of the previous proposition, we get the following corollary. 
\begin{corollary}
    \begin{enumerate}
        \item (Segal Amplitudes) the stochastic process $({\rm Hol}_A\{0\xrightarrow{r} 2\pi \})_{r \in [0,1]}$ has the same law as the restriction of Lévy's holonomy process to the family of curves $\{0\xrightarrow{r} 2\pi \}, r\in [0,1]$.
        \item (Driver--Sengupta formula) For any graph $\Lambda$ whose edges are horizontal or vertical lines, the process
        $(\mathrm{Hol}(A,c))_{c\in \mathrm{Loops}(\Lambda)}$
        verifies the Driver--Sengupta formula~\ref{eq:DSformula}.
    \end{enumerate}
\end{corollary}

We refer to \cite[Thm 6.4 p.~592 and Thm 6.6 p.~595]{driver1989ym2} for the original planar version of this formula due to Driver and to \cite[p.~289]{levysurvey}.

 \subsection{The case of a general surface} \label{pant}
Now that we have recalled Driver’s construction on the disk, let us adapt this framework to the case of a disk sector, such as the one illustrated in Figure \ref{disk}, viewed as a surface with boundary. Because of the axial gauge and the imposed free boundary conditions, repeating Driver’s procedure in this setting produces a Yang–Mills measure with free boundary condition along the circular arc and zero boundary condition along the boundary radial edges.
\begin{figure}
    \centering
    \includegraphics[width=\linewidth]{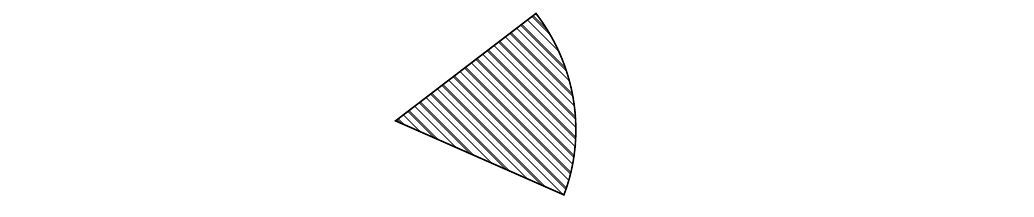}
    \caption{An angular disk sector.}
    \label{disk}
\end{figure}
Next, observe that any oriented surface with boundary can be decomposed into such angular sectors, glued together along their radial edges, so that the resulting boundary of the surface corresponds to the union of the circular boundaries of the sectors (see Figure \ref{gluingtorus}, which illustrates the construction of a torus with one boundary component by gluing together colored edges). The natural question is whether the random connection obtained by gluing the Driver random $1$-forms of each sector yields the free‑boundary Yang–Mills measure on the entire surface. The answer is no. This becomes clear when one computes the distribution of the boundary holonomy: the result does not coincide with the one given by the Driver--Sengupta measure. The discrepancy arises because the gluing operation alters the topology of the domain.
\begin{figure}
    \centering
    \includegraphics[width=\linewidth]{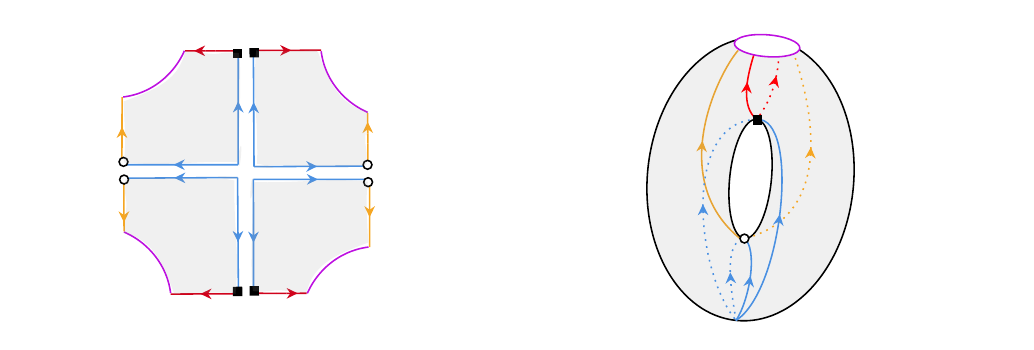}
    \caption{A torus from gluing pieces of angular sector.}
    \label{gluingtorus}
\end{figure}
Indeed, gluing changes the fundamental group $\pi_1(\Sigma)$, and therefore modifies the set of minimizers of the Yang--Mills action $S_{\mathrm{YM}}$. According to the Laplace principle, the measure $\exp(-S_{\mathrm{YM}}(A))\,\mathrm{d}A$ should concentrate equally around all such minimizers, a phenomenon neglected when one glues the constructions sector-by-sector; the patched field constructed directly from the disk sectors concentrates only around the trivial connection $[0]$.
\par To account for this, one needs to describe explicitly the set $\{[B] : S_{\mathrm{YM}}(B)=0\}$, which, in the case of free boundary conditions, can be identified with $G^{2g}$, where $g$ is the genus of $\Sigma$. Each minimizer corresponds to a random flat connection determined by $2g$ independent Haar-distributed elements $(g_a)_{1\le a\le 2g}$. 
\par The main idea is therefore to select one minimizer $B$ at random and superpose it with the fluctuation field constructed by patching. To pick one minimizer $B$ at random, we need to be able to construct a flat connection that assigns prescribed holonomies to the generator of $\pi_1(\Sigma)$. 
\par An easy way to do this is by means of distributional currents. These integration currents (we refer to~\cite[Appendix D]{DRconorm} and \cite[Appendix A p.~62]{DRDuke} for quick and efficient recollections on these notions), denoted by $[c]$ for a given smooth oriented curve $c$ can be integrated against transverse curves as follows~: take a curve $\gamma$ which is everywhere transverse to $c$. Then the integral $\int_\gamma [c]$ of the current $[c]$ over $\gamma$ can be interpreted as the distributional pairing 
$ \int_\Sigma [\gamma]\wedge [c]=\pm 1 $
depending on the orientations at intersection points of $\gamma$ with $c$. Therefore, the current $[c]\log g$ assigns a holonomy of $g$ or $g^{-1}$ for some curves that crosses $c$, depending on the direction~\cite[Lemma 6.2 p.~1836]{DRWitten}.
\par To adapt this to our case, look at Figure \ref{disk}. We would like to construct, for given group elements $g_1$ and $g_2$, a connection that assigns $g_1$ for the holonomy of the small blue loop, and $g_2$ for the big blue loop. Since the orange and red loops cross the blue loops transversally, we can simply set 
\[[\mathrm{orange}]\log g_1 +[\mathrm{red}]\log g_2.\]

\par This leads to an explicit expression for the free-boundary Yang--Mills connection on $\Sigma$:
\[
A_{\mathrm{YM}, \Sigma, \mathrm{Free}}
= A_{\mathrm{noise}} 
+ \sum_{1 \le a \le 2g} [c_a]\log g_a,
\]
where $c_a$ are loops transverse to basic loops of $\pi_1(\Sigma)$ that could for example be unstable curves of some Morse function on $\Sigma$ such that $\nabla f$ satisfies the Smale transversality condition. 
The first term represents the noise contribution, and the second term encodes a random flat connection.
\par Let us compute the holonomy of the boundary component of $A_{\mathrm{YM}, \Sigma, \mathrm{Free}}$ where $\Sigma$ is of genus $g$. It is given by the product
\[\mathrm{Hol}(\partial \Sigma)=X_1U_1\cdots X_{2g}U_{2g}X_{2g+1}U_1^{-1}\cdots X_{2g}U_{2g}^{-1},\]
where $(U_i)_{1\leq i\leq 2g}$ are independent and Haar distributed $G-$valued random variables, and the $X_i$ are as before, the holonomy of the De Rham primitive of the white noise in the sector within $U_i, U_{i+1}$. 
Let $f$ be a measurable function. We have 
\begin{eqnarray*}
    \mathbb{E}[f(\mathrm{Hol}(\partial \Sigma))] 
    &=&\int_{G^{6g}} f(x_1u_1\cdots x_{2g}u_{2g}x_{2g+1}u_1^{-1}\cdots x_{4g}u_{2g}^{-1})p_{\sigma_1}(x_1)\cdots p_{\sigma_{4g}}(x_{4g})\dd x \dd u,
\end{eqnarray*}
where $\sigma_i$ is the are of the corresponding sector.
\begin{eqnarray*}
        \mathbb{E}[f(\mathrm{Hol}(\partial \Sigma))] &=&\int_{G^{6g}} f(\alpha)p_{\sigma_1}(\alpha u_{2g}x_{4g}^{-1}\cdots u_{1}x_{2g+1}^{-1} u_{2g}^{-1}x_{2g}^{-1}\cdots x_{2}^{-1} u_1^{-1})p_{\sigma_2}(x_2)\cdots p_{\sigma_{4g}}(x_{4g}) \dd x \dd u \\
    &=&\int_{G} f(\alpha) d\alpha \int_{G^{6g-1}}p_{\sigma(\Sigma)}(\alpha[a_1,b_1]\cdots [a_{2g},b_{2g}]) \dd a \dd b.
\end{eqnarray*}
Therefore, 
\[\mathbb{P}(\mathrm{Hol}(\partial \Sigma)\in \dd g) = \dd g \int_{G^{6g-1}}p_{\sigma(\Sigma)}(g[a_1,b_1]\cdots [a_{2g},b_{2g}]) \dd a \dd b,\]
which is the same as what we would have obtained from the Driver--Sengupta formula. 
\par There are some technical difficulties we need to face to define this measure as describes here. Namely, we need to understand how to glue distributions. This will be done in subsequent sections.

\section{Preliminaries on Morse theory} \label{s:morse}

We gather the main definitions and results that we will use, namely on how to decompose surfaces using Morse theory.
\par Let $\Sigma$ be a smooth, compact, oriented and closed surface.
We recall that $f:\Sigma\rightarrow\R$ is a \emph{Morse function} if all its critical points, the set of which we denote by $\mathrm{Crit}(f)$, are non-degenerate. It can be shown that the set of Morse functions is open and dense in the $\mathcal{C}^\infty$ topology. A Morse function is said to be perfect if it has exactly $2g+2$ critical points, where $g$ is the genus of the surface $\Sigma$, and if all its critical values are distinct. Note that the set of perfect Morse functions is also open (and nonempty) in the $\mathcal{C}^\infty$ topology.

\par A fundamental property of Morse functions is the following so-called Morse Lemma.
\begin{lemma}[Morse Lemma] Let $f:\Sigma\rightarrow \mathbb{R}$ be a Morse function. Then, for any $a\in \operatorname{Crit}(f)$, there exists a smooth chart $x\in U\subset M\mapsto (x_1(x),x_2(x))\in \R^2$ centered at $a$, called a \emph{Morse Chart near $a$}\label{MorseChart}, in which the function $f$ reads
$$
f(x)=f(a)+\frac {1}{2}\left(\varepsilon_1x_1^2+\varepsilon_2 x_2^2\right),
$$
with $\varepsilon_1,\varepsilon_2\in\{-1,1\}$. Moreover, 
\begin{itemize}
    \item If $\varepsilon_1=\varepsilon_2=1$, we say that $a$ has index $0$, or is a local minimum.
    \item If $\varepsilon_1=1,\varepsilon_2=-1$, we say that $a$ has index $1$, or is a saddle point.
    \item If $\varepsilon_1=\varepsilon_2=-1$, we say that $a$ has index $2$, or is a local maximum. 
\end{itemize}  
\end{lemma}
With this Lemma at hand, we introduce the notion of adapted  metrics~\cite[\S2]{HarveyLawson}.

\begin{definition}[Adapted Metric] Let $f:\Sigma\rightarrow \mathbb{R}$ be a Morse function. We fix a Morse chart $x_a$
near every point in $a\in \operatorname{Crit}(f)$. We say that a metric $g$ on $\Sigma$ is adapted to $f$ if it reads 
\[
g= \dd x_{a,1}^2+ \dd x_{a,2}^2.
\] 
in every Morse chart $x_a$.\label{AdaptedMetric}
\end{definition}
\begin{figure}[t]
    \centering
    \includegraphics[width=\linewidth]{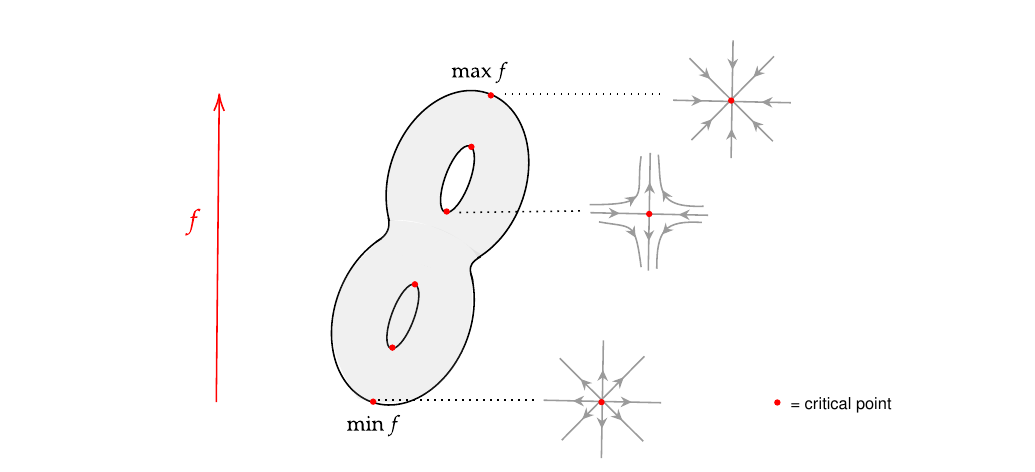}
    \caption{A perfect Morse function.}
    \label{fig:perfectMorse}
\end{figure}
Such a metric is flat near the critical locus of $f$.
By a partition of unity argument, it is not hard to verify that such metrics exist. 
We now fix $f:\Sigma\rightarrow \mathbb{R}$ a perfect Morse function once and for all. Given a metric $g$, one can define the corresponding gradient vector field $\nabla_gf$ by duality
$$
\forall x\in\Sigma,\quad \dd_xf=g_x(\nabla_gf(x),.).
$$
This induces a complete and smooth flow on $\Sigma$ that we denote by $\varphi_f^t:\Sigma\rightarrow \Sigma$. One can verify that
\begin{equation}\label{e:lyapunov-property}
 \forall t_1,t_2\in\R,\ \forall x\in\Sigma,\quad
 f\circ\varphi_f^{t_2}(x)-f\circ\varphi_f^{t_1}(x)=\int_{t_1}^{t_2}\|\dd_{\varphi_f^t(x)}f\|_{g^*(x)}^2\dd t,
\end{equation}
where $g^*$ is the metric induced by $g$ on $T^*\Sigma$. In particular, $f$ is non-decreasing along the flow lines of $\varphi_f^t$. A key property of gradient flows is that, for any $x\in M$, there exists $x_-$ and $x_+$ in $\mathrm{Crit}(f)$ such that
\begin{equation}\label{e:limitset}
 \lim_{t\rightarrow\pm\infty}\varphi_f^t(x)=x_\pm.
\end{equation}

We say that the pair $(f,g)$ has the Morse-Smale \label{MorseSmale} property if there is no gradient line connecting two distinct saddle points of $f$. It is proved in~\cite[Th.~14.4]{HarveyLawson}\cite{Laudenbach}\cite[Thm 2.2.5 p.~40]{AudinDamian} that there exists Morse-Smale pair with $g$ being an adapted metric and $f$ a perfect Morse function.
This condition is essential for defining the celebrated Morse--Witten complex.
A typical example of a perfect Morse function is shown in~\ref{fig:perfectMorse}. 

\emph{From this point on, we will always assume that the pair $(f,g)$ is Morse-Smale, the function $f$ is perfect and the metric $g$ is adapted.} 
Given a critical point $a$ of $f$, we define the unstable and stable manifolds of $a$ as \label{StableUnstable}
$$
W^u(a)\coloneq \left\{x\in\Sigma:\ \lim_{t\rightarrow-\infty}\varphi_f^t(x)=a\right\}\ \text{ and } \
W^s(a)\coloneq \left\{x\in\Sigma:\ \lim_{t\rightarrow+\infty}\varphi_f^t(x)=a\right\}.
$$
These are smooth embedded curves of $\Sigma$ which are respectively  diffeomorphic to $\R^{2-\mathrm{ind}(a)}$ and to $\R^{\mathrm{ind}(a)}$. In particular when $a$ is a saddle point of $f$, these are gradient flow lines connecting the saddle point to either the minimum or maximum of $f$; check~\ref{fig:stableunstable}.
\begin{figure}
    \centering
    \includegraphics[width=1.1\linewidth, trim={1cm 0 0 0}]{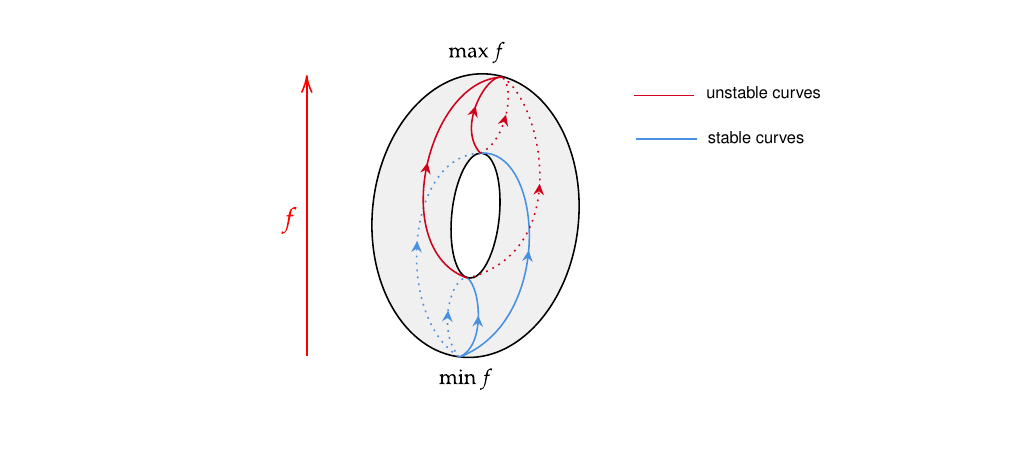}
    \caption{Stable and Unstable curves.}
    \label{fig:stableunstable}
\end{figure}
\par As we are working with compact surfaces, one can verify that these sub-manifolds induce de Rham currents, i.e. the unstable currents $U_a$ defined as~: \label{UnstableCurrents}
$$
\forall\psi\in\Omega^{2-\mathrm{ind}(a)}(\Sigma),\quad\langle U_a ,\psi\rangle\coloneq \int_{W^u(a)}\psi,
$$
and the stable currents $S_a$ defined as
$$
\forall\psi\in\Omega^{\mathrm{ind}(a)}(\Sigma),\quad\langle S_a ,\psi\rangle\coloneq \int_{W^s(a)}\psi.
$$
Thanks to the local expression of the vector field, one has

\begin{enumerate}
\item near the maximum
$U_a=\delta_0(x_1,x_2)\dd x_1\wedge \dd x_2;$
\item near the saddle points
$U_a=\delta_0(x_2)\dd x_2;$
\item near the minimum $U_a=1$. In fact, as $f$ has a single minimum, one has globally $U_a=1$ in the case where $\mathrm{ind}(a)=0$.
\end{enumerate}

In the sequel, we will denote by $N^*W^{u}(a)\coloneq \{(x;\xi); x\in W^u(a), \xi\in \left(T_xW^u(a)\right)^\perp \}$ (resp $N^*W^{s}(a)$) the conormal bundle of the unstable (resp stable) manifold $W^u(a)$ (resp $W^s(a)$).

\section{A global resolution of a surface by a cylinder}
\label{s:globalresolutionsurface}
Let us state informally the goal of this section. Resolving the surface
$\Sigma$ allows us to identify it with some cylinder. This lets us transfer our computations in~\ref{ss:DriverCylinder} from the cylinder to general surfaces. Furthermore, our resolution aims to provide a transparent discussion of the unstable currents that appear in the expression of the Yang--Mills measure, as in figures~\ref{pant} and~\ref{fig:stableunstable}.
\par Consider a triple $(\Sigma,f,g)$ of a surface $\Sigma$ equipped with a Morse function $f$ and an adapted metric $g$ as above.
 Blow up the surface $\Sigma$ at both $\min(f)$ and $\max(f)$. This produces a  surface $\mathcal{S}$ with two boundary components, as shown in~\ref{fig:blowup}. We will call the one obtained by blowing up the minimum ingoing : $\partial\mathcal{S}_{\mathrm{in}}$;  and the other one, outgoing :  $\partial\mathcal{S}_{\mathrm{out}}$. We still denote by $W^{u/s}(a)$ the lift of the unstable/stable curves to the blow--up surface $\mathcal{S}$.
\begin{figure}
    \centering
    \includegraphics[width=0.9\linewidth, trim={0 2cm 0 0}]{ 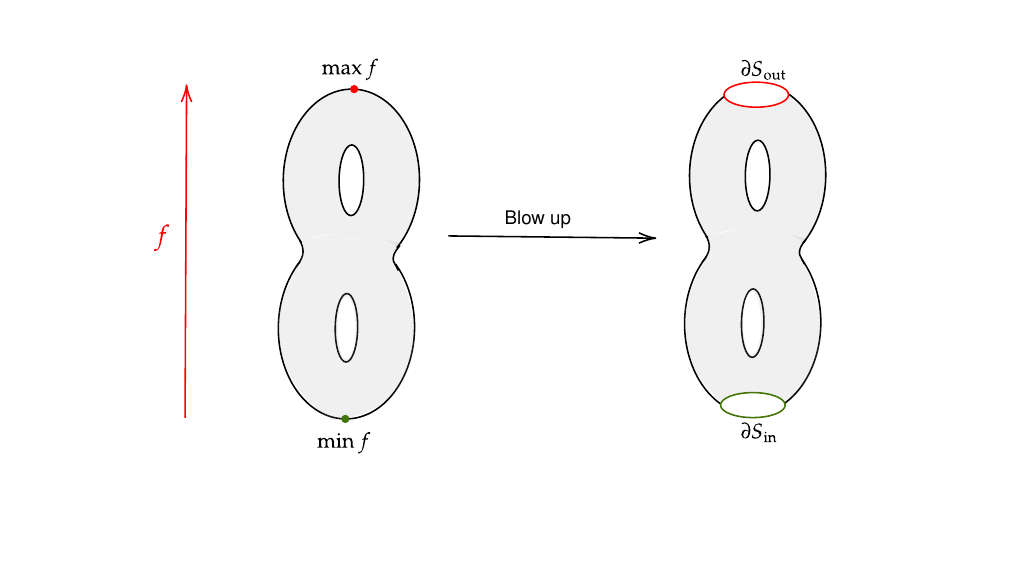}
    \caption{Blown up surface.}
    \label{fig:blowup}
\end{figure}

\subsection{A global angular form and period coordinates}

Concretely, blowing up the surface at $\min(f)$ consists in working in polar coordinates. For $(x_1,x_2)$ a Morse chart at $\min(f)$, 
\[\nabla f=x_1\partial_{x_1}+x_2\partial_{x_2}.\]
Set 
\[r=\sqrt{x_1^2+x_2^2}\,\, , \,\,\,\, x_1=r\cos(\theta), \text{ and }\,\, x_2=r\sin(\theta);\]
and the map
\[(r,\theta)\in [0,\varepsilon)\times \mathbb{S}^1_\theta\mapsto  x_1=r\cos(\theta),x_2=r\sin(\theta) \]
realizes the blow--up. The preimage of $\min(f)=(0,0)$ is the circle $\{0\}\times \mathbb{S}^1$ which is also given in polar coordinates by the equation $r=0$.
The vector field $\nabla f$ lifts automatically to some $b$--vector field still denoted by $V$ (the subscript $b$ stands for boundary) on the resolved surface $\mathcal{S}$, the corresponding flow is still denoted by $(\varphi^t_f)_{t\in \mathbb{R}}$. The $b$-vector fields are in particular always tangent to the boundary.
We refer to~\cite{Grieserbcalc},~\cite{Vasy18},~\cite[2.1 p.~10]{Hintzbcalc} for more discussions on $b$--vector fields.
In polar coordinates near $r=0$, the vector field $V$ reads
$V=r\partial_r$.
Informally, the goal of what comes next is to extend these angular variables $\theta$, which is only defined near $\min(f)$, to the whole surface $\mathcal{S}$ in such a way that the level sets of $\theta$ are flowlines of $V$.

Let us try to give some integral formula,. We work in polar coordinates $(r,\theta)$ near the blown-up minimum and where the vector field $V$ writes $V=r\partial_r$. 
Start from any function $\chi$ which equals $1$ near the minimum and vanishes outside $r\leq 3$ say.
Then observe that 
\[\lim_{T\rightarrow +\infty} \chi\circ \varphi^{-T}_f=1\]
where the convergence holds almost everywhere, since every point 
\[x\in \mathcal{S}\setminus \overline{\cup_{a\in \mathrm{Crit}(f)_1} W^u(a) }\]
is attracted by $\min(f)$ under the backward flow.
Therefore, using the fundamental Theorem of calculus and the definition of the following definition of the Lie derivative \[\frac{\dd}{\dd t}F\circ \varphi^{t}_f=\left(\mathcal{L}_VF\right)\circ \varphi^{t}_f,\] 
we may write 
\begin{align*}
\chi\circ \varphi^{-T}_f=\chi+\chi\circ  \varphi^{-T}_f-\chi=\chi+\int_0^T \left(-\mathcal{L}_V\chi\right)\circ  \varphi^{-s}_f\dd s.    
\end{align*}
Now we let $T\rightarrow +\infty$ and set $\psi=-\mathcal{L}_V\chi$ which is smooth on $\Sigma$ since it vanishes near $\min(f)$. This yields
a continuous partition of unity adapted to the gradient flow which writes
\begin{align*}
 1=\chi+\int_0^\infty \psi\circ  \varphi^{-s}_f\dd s.   
\end{align*}

The above suggests to define, at least at the formal level, our global angular form $\overline{\alpha}$ on $\mathcal{S}$ by the integral formula
\begin{align*}
   \overline{\alpha}= \chi(r)\dd \theta +\int_0^\infty \left(\psi \dd \theta\right)\circ  \varphi^{-t}_f \dd t
\end{align*}
where  the element $\left(\psi \dd \theta\right)$ is well--defined. Indeed, $\dd \theta$ is well--defined on the support of $\psi$ which happens to be contained in the disc $\{r\leq 3\}$.

\begin{lemma}\label{lem:angvar1}
The above $1$-form $\overline{\alpha}$ is well--defined, smooth on $\mathcal{S}\setminus 
\overline{\cup_{a\in \mathrm{Crit}(f)_1} W^u(a)}$, 
and satisfies the following.
\begin{enumerate}
\item  For any $x$ in the domain of definition of the chart $(r,\theta)$, we have
$\overline{\alpha}=\dd \theta$. This means that $\overline{\alpha}$ is an angular form,
\item The form $\overline{\alpha}$ is invariant by the flow $\mathcal{L}_V   \overline{\alpha}=0,$ and is closed in the sense of currents $
\dd \overline{\alpha}=0$, 
with periods in $2\pi\mathbb{Z}$.
\end{enumerate}
\end{lemma}

\begin{proof}
Step 1. The first step is to prove that our formula gives a $1$--form that coincides with the angular form $\dd \theta$ near $\min(f)$. This follows
from the partition of unity formula satisfied by the pair $(\chi,\psi)$ and the fact that $\dd \theta$ is locally invariant by the flow.

Step 2. The second step consists in showing that
\[\int_0^\infty \left(\psi(r) \dd \theta\right)\circ \varphi^{-t}_f \dd t\]
exists, meaning that it converges in  appropriate spaces.
The $1$--form $\left(\psi(r) \dd \theta\right)$ is smooth. By a result of the first author and Rivière~\cite{DR16}, \cite{DRWitten}, there exists an anisotropic Sobolev space $\mathcal{H}^m\left( \Sigma\right)$ of currents adapted to the dynamics ~\cite[section 4 p.~1421]{DR16}, with $C^\infty\left( \Sigma\right)\subset \mathcal{H}^m\left( \Sigma\right) \subset \mathcal{D}^\prime(\Sigma)$, and where $m\in S^0(T^*M)$ is a symbol of degree $0$ acting as a variable order for the Sobolev space~\cite[4.1.2 p.~1422 and 4.1.3 p.~1423]{DR16}, such that
the pull--back $\left(\psi(r) \dd \theta\right)\circ \varphi^{-t}_f$ converges exponentially fast to the equilibrium in $\mathcal{H}^m\left( \Sigma\right)$~\cite[Prop 5.7 p.~1429]{DR16}. Quantitatively, 
\begin{align*}
\left(\psi(r) \dd \theta\right)\circ \varphi^{-t}_f=\sum_{a\in \mathrm{Crit}(f)} \left(\int_{W^s(a)} \psi(r) \dd \theta \right) U_a+\mathcal{O}_{\mathcal{H}^m}(e^{-t})  
=\mathcal{O}_{\mathcal{H}^m}(e^{-t}),
\end{align*}
where 
we used
the fact that the $1$--form $\psi(r) \dd \theta $ belongs to $\ker(\iota_V)$. Therefore, its integral over all stable curves $W^s(a)$ vanishes : $\int_{W^s(a)}\psi(r) \dd \theta=0$.
The remainder converges to $0$ exponentially fast in the appropriate anisotropic Sobolev space $\mathcal{H}^m$, namely
\[ \Vert \left(\psi(r) \dd \theta\right)\circ \varphi^{-t}_f\Vert_{\mathcal{H}^m(\Sigma)}\leq Ce^{-Kt} .\]
The partial conclusion is that $\int_0^\infty \left(\psi(r) \dd \theta\right)\circ \varphi^{-t}_f \dd t$ converges in the anisotropic Sobolev space $\mathcal{H}^m$ and therefore $\overline{\alpha}$ is well--defined in $\mathcal{H}^m$.

Recall given a closed conic set $\Gamma\subset T^*\text{int}(\mathcal{S})$, we denote by $\mathcal{D}^\prime_{\Gamma}$ the currents whose wave front set is contained in the conic set $\Gamma$.
We refer the reader to paragraph~\ref{ss:WF} of the appendix for recollections on the notion of wave front set of a current.
By the results of~\cite[section 7.2 p.~1841]{DRWitten}, the discussion of \cite[section 3.1.2 p.~1810]{DRWitten} ensures that once $\overline{\alpha}$ belongs to $\mathcal{H}^m$ for some given order function $m$, it belongs to  $\mathcal{H}^{Lm}$ for all $L\geqslant 1$, in other words one can scale the order functions. Then arguing as in 
~\cite[section 7.2 p.~1841]{DRWitten}, we can ensure that the intersection of all 
$\mathcal{H}^{Lm}$ for all $L\geqslant 1$ is contained in $\mathcal{D}^\prime_{\Gamma_1}$ for 
$\Gamma_1=\overline{\cup_{a\in \mathrm{Crit}(f)_1} N^*W^u(a) }$.
The wave front set of $\overline{\alpha}$ is contained in 
\[\overline{\cup_{a\in \mathrm{Crit}(f)_1} N^*W^u(a) },\] and therefore 
$\overline{\alpha}$ is smooth outside 
\[\overline{\cup_{a\in \mathrm{Crit}(f)_1} W^u(a)}.\]
Step 3. We calculate $\mathcal{L}_V\overline{\alpha}$. 
Observe that, by construction, we have the identity 
$\mathcal{L}_V\chi+\psi=0.$
We must first choose some order function $m\in S^0(T^*M)$ in such a way that in the space
$\mathcal{H}^{m-1}$, we have 
\[\Vert \left(\psi(r) \dd \theta\right)\circ \varphi^{-t}_f\Vert_{ \mathcal{H}^{m-1}}=\mathcal{O}(e^{-Ct}) \]
for some $C>0$. 
We have 
\begin{align*}
 &\mathcal{L}_V\overline{\alpha}=\mathcal{L}_V\left( \chi \dd \theta +\int_0^T \left(\psi \dd \theta\right)\circ  \varphi^{-t}_f\dd t  \right) =-\psi \dd \theta-\int_0^T  \frac{\dd}{\dd t} \left(\psi \dd \theta\right)\circ  \varphi^{-t}_f \dd t\\
&=-\psi \dd \theta-(\left(\psi \dd \theta\right)\circ \varphi^{-t}_f-\psi \dd \theta)=-\left(\psi \dd \theta\right)\circ  \varphi^{-T}_f.
\end{align*}
We use that $\Vert \left(\psi \dd \theta\right)\circ \varphi^{-T}_f\Vert_{\mathcal{H}^{m-1}} \lesssim e^{-CT}$, which goes to $0$ when $T\rightarrow +\infty$. 
We deduce from the above that
\[\mathcal{L}_V\overline{\alpha}=\lim_{T\rightarrow +\infty} -\left(\psi \dd \theta\right)\circ \varphi^{-T}_f=0\in \mathcal{H}^{m-1}.\]

Step 4. We calculate $\dd\overline{\alpha}$. Near $\{r=0\}$, since $\overline{\alpha}=\dd \theta$ we have immediately that $\dd\overline{\alpha}=\dd^2\theta=0$ so $\overline{\alpha} $ is locally closed near $\{r=0\}$.
Now we use the fact that $\overline{\alpha}\in \ker(\mathcal{L}_V)$ to propagate the closedness property to 
\[\Sigma\setminus \overline{\cup_{a\in \mathrm{Crit}(f)_1} W^u(a)}. \]
Since $\overline{\alpha}\in \ker(\mathcal{L}_V)$, then $\implies \varphi^{-T*}_f \overline{\alpha}=\overline{\alpha}$ for all $T>0$.
Therefore for all $T>0$,
\begin{align*}
\dd\overline{\alpha}=\dd \varphi^{-T*}_f \overline{\alpha}=  \varphi^{-T*}_f \dd\overline{\alpha}.  
\end{align*}

Outside $\overline{\cup_{a\in \mathrm{Crit}(f)_1} W^u(a)} $, choosing $T$ large enough allows to conclude that $\dd\overline{\alpha}=0$. 
In fact $\dd\overline{\alpha}$ is closed when tested against \emph{smooth forms which vanish} at $r=0$. 
We need to decompose
\begin{align*}
\dd\overline{\alpha}=\dd\left(\chi(r)\dd \theta \right)+ \dd\int_0^\infty \left(\psi(r)\dd \theta\right)\circ  \varphi^{-T}_f \dd t \\
=\dd\left(\chi(r)\dd \theta \right)+ \int_0^\infty \left(\dd\left(\psi(r)\dd \theta\right) \right)\circ \varphi^{-T}_f\dd t\\
=\left(\dd\chi \wedge \dd \theta \right)+ \int_0^\infty \left(\left(\dd\psi \wedge \dd \theta\right) \right)\circ \varphi^{-T}_f\dd t=0
\end{align*}
where we could invert the integral and $\dd$ since we have absolute convergence of the integral in the appropriate anisotropic Sobolev space. Then using the identity
\[ 0=\dd\chi+\int_0^\infty (\dd\psi)\circ \varphi^{-s}_f\dd s  \]
and the property of the supports of $(\dd\psi)\circ \varphi^{-s}_f$.

The key point is that $\overline{\alpha}$ is closed but non exact but with periods in $2\pi\mathbb{Z}$ by construction. 
For any closed cycle $\gamma$ supported on $\Sigma\setminus \min(f)$ and such that the support is contained in the polar coordinate chart, then the claim is obvious, we have 
\[\int_\gamma \overline{\alpha}=\int_\gamma \dd \theta=2i\pi k\] where
$k=\mathrm{deg}_{\gamma}(\min(f))$ is the degree of the closed curve $\gamma$ around $\min(f)$.
Consider any closed oriented cycle $\gamma$ which is transverse to the unstable curves and whose support does not meet $\max(f),\min(f)$. We denote by $[\gamma]$ the corresponding current. Choose some oriented flowline $\gamma_0$ going from $\min(f)$ to $\max(f)$ and consider its lift, still denoted by $\gamma_0$, to the blow--up surface $\mathcal{S}$; see Figure \ref{CurveGamma}.
\begin{figure}[t]
    \centering
    \includegraphics[width=\linewidth, trim={0 1cm 1cm 0}]{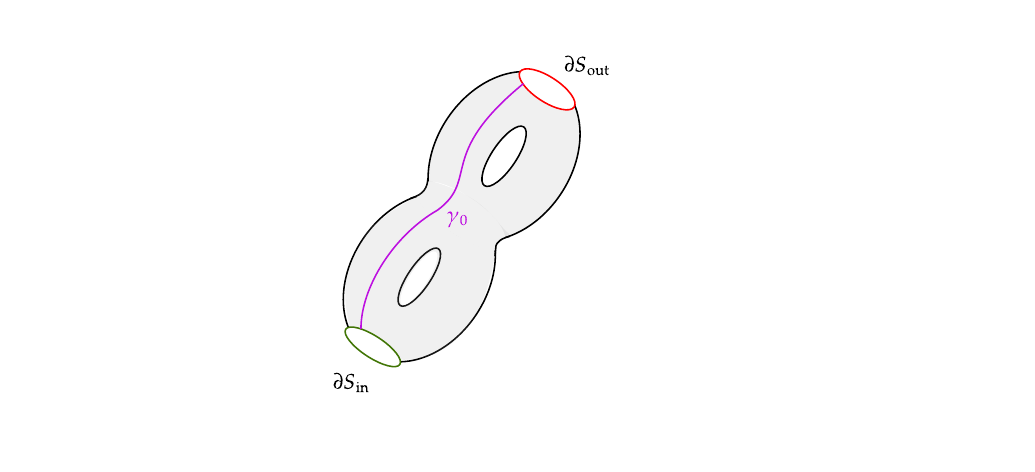}
    \caption{Thee curve $\gamma_0$.}
    \label{CurveGamma}
\end{figure}

By Lemma~\ref{lem:convergencecycles} that we prove in appendix,
for any smooth closed curve $\gamma$ whose support does not meet $\partial\mathcal{S}$ and is everywhere transverse to $V$, we get that 
$$\varphi^{T*}_f[\gamma]\underset{T\rightarrow +\infty}{\longrightarrow}  \sum_{a\in \mathrm{Crit}(f)}\left(\int_\Sigma U_a\wedge [\gamma]\right)S_a+ \left(\int_{\Sigma} [\gamma_0]\wedge [\gamma] \right) \delta_{\{0\}}^{\mathbb{R}}(r)\dd r $$
where the current $\delta_{\{0\}}^{\mathbb{R}}(r)\dd r$ is just the integration current $[\partial \mathcal{S}_{\mathrm{in}}]$ on the incoming boundary of $\mathcal{S}$. 

Here away from $\partial\mathcal{S}_{in} $, the convergence of $\varphi^{T*}_f[\gamma] \rightarrow  \sum_{a\in \mathrm{Crit}(f)}\left(\int_\Sigma U_a\wedge [\gamma]\right)S_a$ holds in the space $\mathcal{D}^\prime_{\Gamma_2}\left(\text{int}(\mathcal{S}) \right)$ for $\Gamma_2=\overline{\cup_{a\in \text{Crit}(f)_1} N^* W^s(a) }$ by the results of~\cite{DR16}~\footnote{In fact it holds true in the dual anisotropic space $\mathcal{H}^{-m}_{loc}$, it is as if we worked in the closed surface $\Sigma$} and near $\partial\mathcal{S}_{in} $, the convergence
$$\varphi^{T*}_f[\gamma]\underset{T\rightarrow +\infty}{\longrightarrow}  \sum_{a\in \mathrm{Crit}(f)}\left(\int_\Sigma U_a\wedge [\gamma]\right)S_a+ \left(\int_{\Sigma} [\gamma_0]\wedge [\gamma] \right) \delta_{\{0\}}^{\mathbb{R}}(r)\dd r$$
holds true only in the weak topology by Lemma~\ref{lem:convergencecycles}. 
Then by flow invariance, we have~:
\begin{equation}\label{eq:pairingangular}
\int_\gamma \overline{\alpha}=\int_{\mathcal{S}} \varphi^{T*}_f[\gamma]\wedge \overline{\alpha}\rightarrow   \left(\int_{\Sigma} [\gamma_0]\wedge [\gamma] \right) \int_{\partial \mathcal{S}_{in}} \dd \theta \in 2\pi\mathbb{Z}   
\end{equation}
where we used the fact that $\dd \theta$ integrates as zero on the preimages $W^s(a)$ of the stable curves, we also rely on the bound $WF\left(\overline{\alpha}\right)\subset\Gamma_1$ on the wave front set of $\overline{\alpha}$ and $\varphi^{T*}_f[\gamma]$ converges in $\mathcal{D}^\prime_{\Gamma_2}$ where $\Gamma_1,\Gamma_2$ \textbf{are transverse} which justifies the convergence of the above pairing in equation~\ref{eq:pairingangular} by the hypocontinuity of the wedge product of currents whose wave front sets are transverse recalled in Lemma~\ref{lem:prodWF}.
\end{proof}
So the $1$--form $\overline{\alpha}$ is our global angular form.

\begin{definition}
Denote by $\theta:\mathcal{S}\mapsto \mathbb{R}/2\pi\mathbb{Z}$ the circle valued function which is the unique de Rham primitive of $\overline{\alpha}$ which belongs to $\ker(\mathcal{L}_V)$ and coincides with the actual polar coordinates $\theta$ in the chart near $\min(f)$.
\end{definition}

The reader should be aware that
the circle valued function
$\theta$ is only piecewise smooth  
with singularities along unstable curves. This is illustrated in~\ref{fig:discont}.
\begin{figure}
    \centering
    \includegraphics[width=1.1\linewidth, trim={0 1cm 0cm 1cm}]{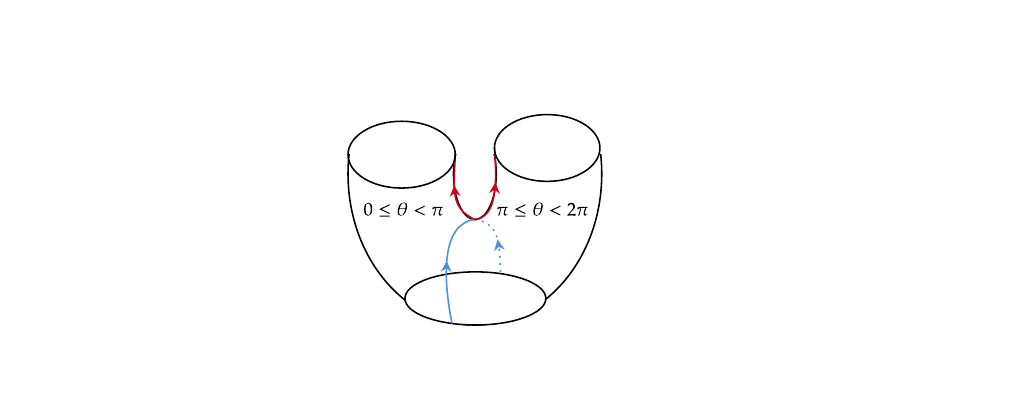}
    \caption{Singularities along unstable curves.}
    \label{fig:discont}
\end{figure}

\paragraph{General surfaces with boundary.}

In general, given any bordered surface $\Sigma$ assumed to have both, in and out boundary components, with smooth area form $\sigma$, we can  choose a vector field 
$V$ which points inward along $\partial\Sigma_{\mathrm{in}}$ and points outward
along $\partial\Sigma_{\mathrm{out}}$.
Then, up to changing the vector field $V$ by a conformal factor 
 ($V_b=\rho V$ where $\rho\geq 0$ is some function 
which vanishes at order $1$ near $\partial\Sigma$ and is $>0$ in the interior), one can easily establish that
the above formula 
\begin{align*}
\overline{\alpha}\coloneq \chi(r)\dd \theta +\int_0^\infty \left(\psi(r) \dd \theta\right)\circ \varphi^{-t}_f \dd t  
\end{align*}
also defines a closed angular form in the interior of $\Sigma$ with periods in $2\pi\mathbb{Z}$. Then, the de Rham primitives $\theta$ of the above closed form $\overline{\alpha}$ 
yields global angular coordinates on $\Sigma$.

\subsection{Pseudo coordinates}

Once we have constructed the global angular variable $\theta$, 
the pair $(f,\theta)$ yields global functions on $\mathcal{S}$ 
which act as coordinate functions outside $\cup_{a\in \mathrm{Crit}(f)_1}W^u(a)$. In fact $f$ is smooth everywhere and is nondegenerate except at critical values of $f$, and $\theta$ is piecewise smooth with discontinuities along $\overline{\cup_{a\in \mathrm{Crit}(f)_1} W^u(a)}$.

\begin{proposition}[Regularity of the pseudo-coordinates]\label{prop:regpseudo}
The pseudo-coordinate $\theta:\mathcal{S}\mapsto \mathbb{S}^1$ is piecewise $C^\infty$ on $\mathcal{S}$ with jumps along the unstable manifolds. Both $(f,\theta)$ are
regular coordinate functions
on 
\[\mathcal{S}\setminus \overline{  \cup_{a\in \mathrm{Crit}(f)_1}W^u(a)}.\] The $1-$form $\dd \theta$ vanishes at order $1$ at $\mathrm{Crit}(f)_1$ and $\dd f$ vanishes at order $1$ at $\mathrm{Crit}(f)_1\cup \partial\mathcal{S}$.
\end{proposition}
\begin{proof}
By propagation of singularities, it is enough to study the regularity of $\theta$ near a saddle point, then use the dynamics to propagate the singularities.
Use coordinates $(x,y)\in [-1,1]\times [-1,1]$ near a saddle point $a\coloneq (0,0)$ where the flow reads
$(e^{-t}x,e^ty)$ and the unstable (resp. stable) manifold $W^u(a)$ (resp. $W^s(a)$) reads locally $\{x=0\}$ (resp $\{y=0\}$).
In the hyperbolic box $[-1,1]_x\times [-1,1]_y$, the in face for the dynamics reads $\{\pm 1\}_x\times [-1,1]_y$ any flowline that enters the box must intersect the in face once; the out face reads
$[-1,1]_x\times \{\pm 1\}_y $, every flow-line entering the box must escape via the out face by intersecting it exactly once; see~\ref{fig:inoutdynamics}.
\begin{figure}
    \centering
    \includegraphics[width=\linewidth, trim={0 0 1cm 0}]{ 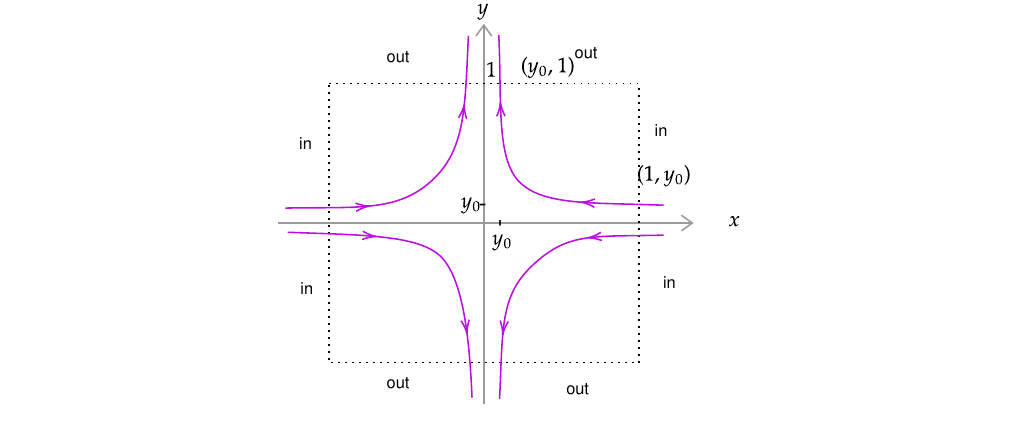}
    \caption{Dynamics of in and out faces.}
    \label{fig:inoutdynamics}
\end{figure}
We know $\theta$ has smooth Cauchy data from the in face
$\theta(\pm 1,y)=F_\pm(y)$ where $F_\pm$ is smooth since we are coming from the minimum. 
Then by flow invariance 
\[\theta(x,y)= 1_{\mathbb{R}_{\geq 0}}(x) F_+(xy)+1_{\mathbb{R}_{\geq 0}}(-x) F_-(xy) \] is a piecewise smooth function of the pair $(x,y)$ with 
a discontinuity exactly along the $y$--axis which is the piece of unstable manifold  $\{x=0\}$.

Calculate $\dd \theta$ in a quadrant, say $x>0$ reads
\[\dd \theta= F_+^\prime(xy) (y\dd x+x\dd y) \] and let $x\rightarrow 0^+$ yields  $\dd \theta= F_+^\prime(0) y\dd x $ which vanishes at order $1$ at $(x,y)=(0,0)$.
\end{proof}

We next introduce a crucial concept which will be useful for us in the sequel.

\subsection{Resolution of the surface $\mathcal{S}$ by a cylinder with defect lines}

The pair $(f,\theta)$
defines a global map
\begin{equation}\label{eq:globalPsi}
 \Psi: (f,\theta):\mathcal{S} \mapsto [\min(f),\max(f)]\times \mathbb{S}^1_\theta .
\end{equation} 

Let us describe a bit the geometry of both unstable and stable curves $W^{u/s}(a), a\in \mathrm{Crit}(f)_1$ when lifted to $\mathcal{S}$. The $2g$ unstable (resp stable) curves intersect transversely the blow up circle at $\max(f)$ (resp $\min(f)$) at exactly $4g$ points.
In the angular variables, each $W^u(a)\cup W^s(a)$ corresponds to some distinguished level 
$\theta(a)$ so that the saddle point $a$ is given in pseudo-coordinates by $(f(a),\theta(a))$.
Remove the $2g$ unstable curves and the $2g$ stable curves, this yields a union of $4g$ hexagons  with $6$ edges which are described as follows:
\begin{enumerate}
\item there is one edge coming from the blow--up $\max(f)$,
\item there is one edge coming from the blow--up $\min(f)$,
\item there are two unstable edges,
\item there are two stable edges,
\item there are two vertices which are identified with saddle points.
\end{enumerate}

\begin{definition}[Resolution of $\mathcal{S}$ abstract version]
These $4g$ hexagons are interiors of $4g$ surfaces with smooth corners $H_1,\dots,H_{4g}$ which are smoothly embedded inside the 
surface $\mathcal{S}$, obtained by gluing smoothly these $4g$ hexagons along well chosen edges. 
The pair $(f,\theta)$ induces smooth maps
\begin{equation}
(f,\theta):H_i\subset \mathcal{S} \longmapsto  [\min(f),\max(f)]\times I_i ,  
\end{equation}
where the collection $I_1,\dots,I_{4g}$ are $4g$--intervals on the unit circle with disjoint interiors and only endpoints can coincide.
So the cylinder $[\min(f),\max(f)]\times \mathbb{S}^1_\theta$ writes as a union of $4g$ rectangles
$ \left([\min(f),\max(f)]\times I_i\right)_{i=1}^{4g} $ which are glued along lines
$\left([\min(f),\max(f)]\times \partial I_i\right)_{i=1}^{4g}$. 
These lines are the defect lines and are broken flowlines made of a stable curve glued with an unstable curve.

The \emph{bands of zero area} are just diffeomorphic to tubular neighborhoods of the above lines and are obtained by thickening these lines as shown in~\ref{fig:zeroAreas}.
\begin{figure}
    \centering
    \includegraphics[width=0.8\linewidth]{ 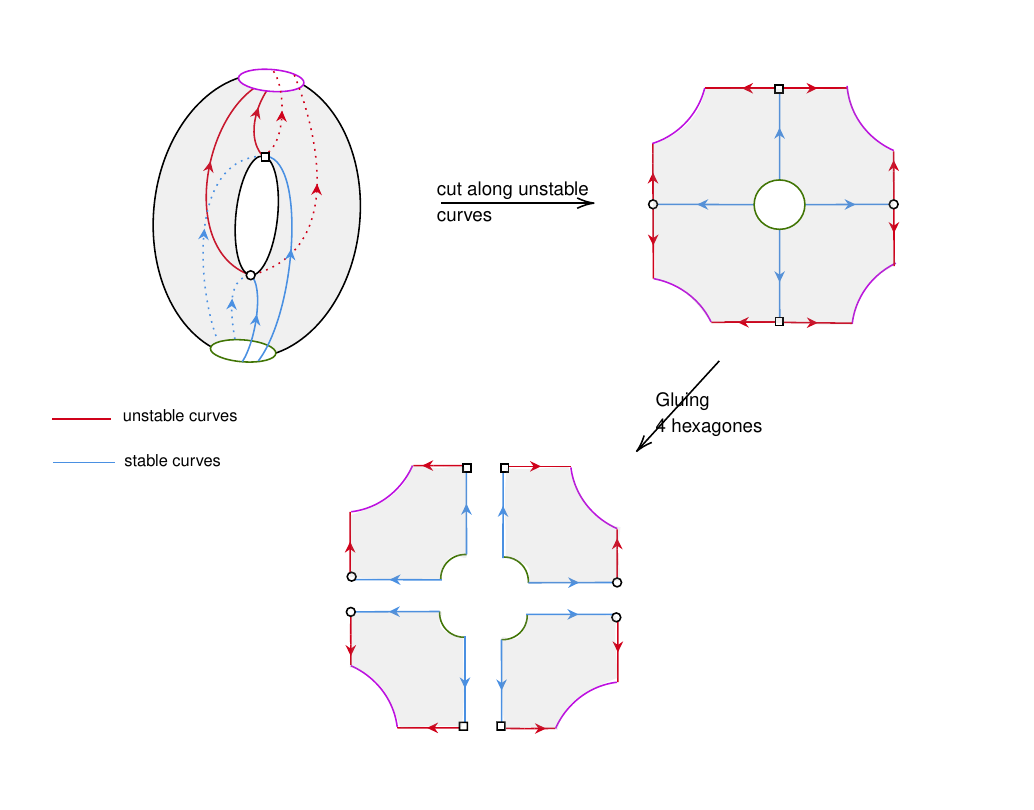}
    \caption{Zero area bands.}
    \label{fig:zeroAreas}
\end{figure}
The restriction of the pair $(f,\theta)$ to every hexagon $H_i, i\in \{1,\dots, 4g\}$ is smooth.
\end{definition}

The consequence of Proposition~\ref{prop:regpseudo} is that for every hexagon $H_i$, the map $\Psi:H_i\setminus \mathrm{Crit}(f)_1 \longmapsto  \Psi \left( H_i\setminus \mathrm{Crit}(f)_1 \right) $ is a diffeomorphism up to the edges of the hexagon, the only singularities occurs at the corners corresponding to saddle points.

In practice and for the next part of the article, we will work on each of these rectangles with the pushforward of the
area form $\sigma$ via the map $(f,\theta)$.
All estimates and random currents are constructed  on the rectangles $\left([\min(f),\max(f)]\times I_i\right)_{i=1}^{4g} $ first, then using the pair $(f,\theta)$ we will pull--back these random objects on the initial blow--up surface $\mathcal{S}$.

\subsection{Regularity of area function}

The pushforward area $\Psi_*\sigma$ form reads
$\sigma(f,\theta)\dd f\wedge \dd \theta$, where the density $\sigma$ is expected to blow up near critical points. This is going to affect area estimates near critical points.

\paragraph{A further modification of the angular form $\overline{\alpha}$.}

In our investigation, we need a further
modification of the angular form $\overline{\alpha}$ so that is has a very specific form near the saddle point. This specific form is very important in our application.
Let us be more precise.
In a hyperbolic box in the local Morse chart, up to shifting coordinates by constants and up to some smooth change of variables, we may assume that
$f=x^2-y^2$, and $\theta=xy$.
To justify this choice, we just mimick the proof of
Proposition~\ref{prop:regpseudo}.

\begin{proposition}\label{prop:angvar2}
There exists a global angular form $d\theta$ on $\mathcal{S}$ such that for every saddle point $a\in \mathrm{Crit}(f)_1$, there is a neighborhood $\Omega_a$ of $a$ contained in the Morse chart, in which $d\theta|_{\Omega_a}= \pm d(xy)$ where $(x,y)$ is the Morse chart corresponding to $a$.    
\end{proposition}

\begin{proof}
Assume some angular form $\theta$ is already constructed.

Let us number the saddle points as $a_1,\dots, a_{2g}$ where $g$ is the genus of the initial surface $\Sigma$. Near each saddle point $a_i$, $i\in \{1,\dots,2g\}$, there is a local Morse coordinate system
$(x_i,y_i)\in [-1,1]^2$ where $a_i=(0,0)$ and the dynamics is totally linear of the form
$(e^{-t}x_i,e^ty_i)$.  Recall that $\{\pm 1\}_{x_i}\times [-1,1]_{y_i} $ is the in face of the hyperbolic box around $a_i$. 

For every $\varepsilon\in (0,1)$ and every $i\in \{1,\dots, 2g\}$, we define the ribbon of thickness $\varepsilon>0$ the following set~:
\begin{align*}
R_i(\varepsilon):= \overline{\{\varphi^{-t}_f( \{\pm 1\}_{x_i}\times [-\varepsilon,\varepsilon]_{y_i}  ), t\geqslant 0  \}} \end{align*}
which contains all the flow lines starting from the blown-up minimum $\partial\mathcal{S}_{in}$ and hitting the 
small intervals  $\{\pm 1\}_{x_i}\times [-\varepsilon,\varepsilon]_{y_i}$  contained in the in face of the hyperbolic box around $a_i$. Geometrically, this is the union of two very thin rectangles; see Figure \ref{Rib1}.
\begin{figure}[t]
    \centering
    \includegraphics[width=\linewidth]{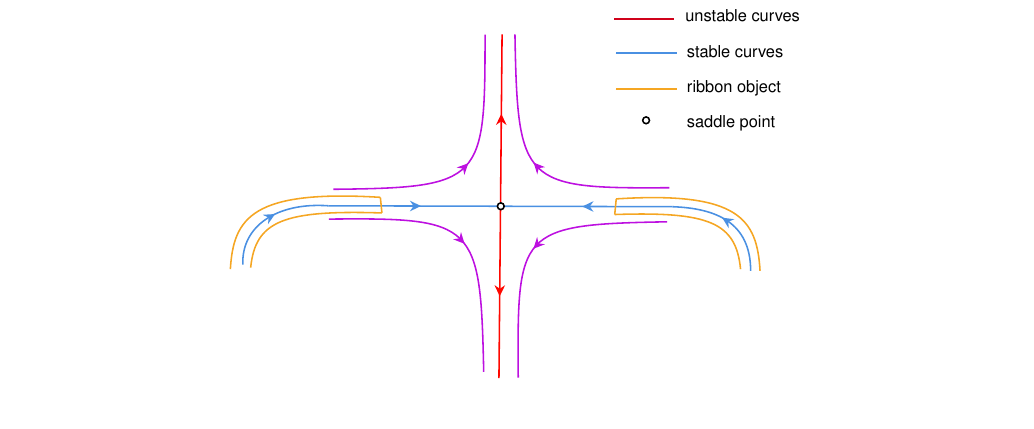}
    \caption{Ribbon object.}
    \label{Rib1}
\end{figure}

We make the following claim, there exists $\varepsilon>0$ small enough such that all the ribbons of thickness $\varepsilon$ 
$R_i(\varepsilon), i\in \{1,\dots,2g\}$ are two by two disjoint; see Figure \ref{Rib3}.
\begin{figure}[t]
    \centering
    \includegraphics[width=1.3\linewidth, trim={3cm 0 0 0}]{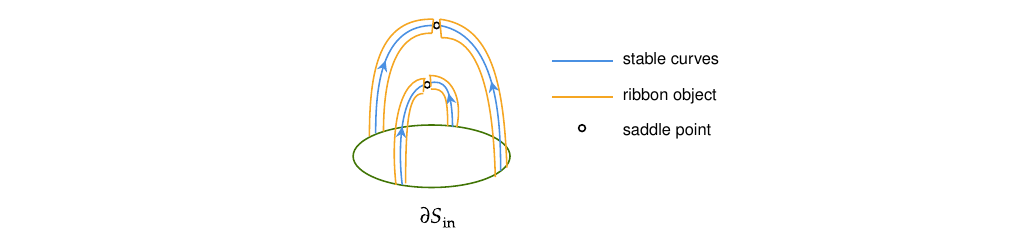}
    \caption{Ribbon object (bis).}
    \label{Rib3}
\end{figure}
To prove the claim, let us argue by contradiction.  Up to extraction we can always assume there is a given pair $i< j\in \{1,\dots,2g\}^2 $ and a sequence $y_i(n)\rightarrow 0^+$ when $n\rightarrow +\infty$ such that the flowline $ \varphi^{[-\infty,0]}_f\left( (1,y_i(n))\right)$ ending at the in face of $a_i$, gets at distance $\leqslant \frac{1}{n}$ from the critical point $a_j$. Then letting $n\rightarrow +\infty$, the sequence of curves $ \varphi^{[-\infty,0]}_f\left( (1,y_i(n))\right)_n, ,\geqslant 0$ should converge in $C^0$ topology to a broken flowline using the compactness of the space of broken gradient flowlines~\cite[section 3.2 p.~57]{AudinDamian}. We deduce that there exists a broken flowline connecting the saddle points $a_i$ and $a_j$ which contradicts Smale's transversality.

Therefore we have $4g$ two by two disjoint rectangles that connect in faces of every saddle points and the circle $\partial\mathcal{S}_{in}$. This decomposes the circle $\partial\mathcal{S}_{in}$ as a union of $4g$ two by two disjoint intervals $I_1,\dots,I_{4g}$ which are in bijection with $4g$ pieces of in faces that we will denote $J_1,\dots,J_{4g}$ in such a way that $I_i=\overline{\varphi^{[-\infty,0]}_f(J_i)}\cap \partial\mathcal{S}_{in}$.
We denote by $\mathcal{P}_i: J_i\mapsto I_i $ the natural smooth diffeomorphism mapping $J_i$ to $I_i$ and which is induced by the backward flow: $\mathcal{P}(u)=v$ if and only if $v=\lim_{t\rightarrow +\infty}\varphi^{-t}_f(u)$; see Figure \ref{Rib4}.
\begin{figure}
    \centering
    \includegraphics[width=\linewidth]{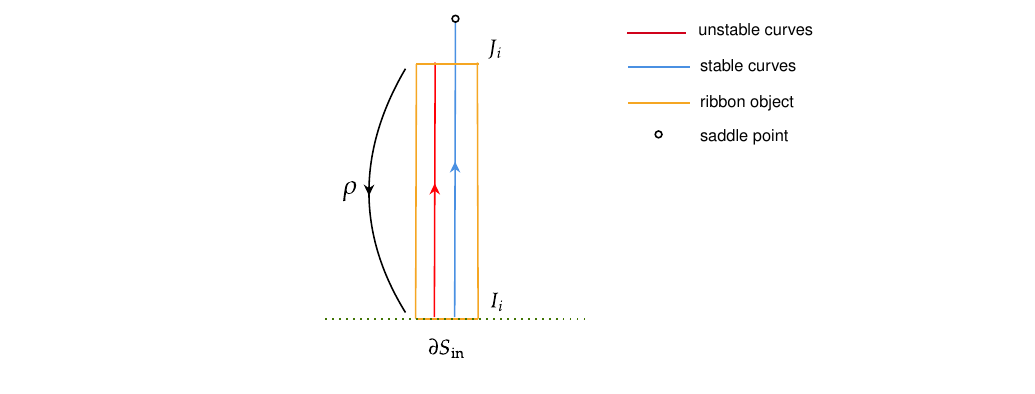}
    \caption{Ribbon object (bis bis).}
    \label{Rib4}
\end{figure}
Now we modify the angular variable $\theta$ on $\partial\mathcal{S}_{in}$ in such a way that on each interval $I_j=\partial\mathcal{S}_{in}\cap R_i(\varepsilon)$, it coincides with the $1$--form 
\begin{equation}\label{eq:constraintang}
d\theta=\varepsilon_i\mathcal{P}_*(d(x_iy_i)), \varepsilon_i\in \{\pm 1\}
\end{equation}
which intuitively represents the $1$-form $\pm d(x_iy_i)$ 
defined on the in face of $a_i$ and then
transported back by the flow.
We choose $\varepsilon_i$ so that each piece of $1$-form $\varepsilon_i\mathcal{P}_*(d(x_iy_i))$ has positive orientation on $\mathbb{S}^1$.
We also allow ourselves to make the intervals describing the in faces $J_i, i=1,\dots,4g$ smaller. 
At the end of this step, we basically claim we can smoothly interpolate some $1$ form defined locally on segments of $\mathbb{S}^1$ and we need to be careful since $\theta$ must be a coordinate function on the circle and has exact period $2\pi$ which is non trivial to achieve.

Let us justify more precisely this step. This reduces to solving the following exercise : we are given $4g$ $1$-forms $du_1,\dots, du_{4g}$ on two by two disjoint subintervals   $I_1,\dots,I_{4g} $ of the circles $\mathbb{S}^1$, the intervals are centered at points $\theta_1,\dots,\theta_{4g}$ respectively. These $4g$ points correspond to the intersection of the stable curves with $\partial\mathcal{S}_{in}$. We assume each $du_i$ is non degenerate and $du_i(\partial_\theta)>0$ on each $I_i$ (this is always possible thanks to choosing the coefficients $\varepsilon_i,i=1,\dots,4g$ carefully). We just need to prove there exists a global non degenerate closed form $d\tilde{\theta}$ which coincides with $du_i$ near each $\theta_i$ and such that $\int_{\mathbb{S}^1}d\tilde{\theta}=2\pi$. Consider a sequence $\chi_{i,n}\in C^\infty_c(I_i)$ of functions such that $\chi_{i,n}\geqslant 0$, $\chi_{i,n}=1$ near $\theta_i$ and $\chi_{i,n}=0$ outside $\frac{2}{n}$ neighborhood of $\theta_i$. 
Then observe that for $n$ large enough so that $\int_{\mathbb{S}^1}\sum_{i=1}^{4g} \chi_{i,n} du_i<2\pi$, we can always find some $\lambda_n>0$ s.t. $d\tilde{\theta}=\sum_{i=1}^{4g} \chi_{i,n} du_i+\lambda_n (1-\sum_{i=1}^{4g} \chi_{i,n} ) d\theta$ has period $2\pi$. Also observe that $d\tilde{\theta}$ never vanishes by construction.

Once we modified the angular function in a new angular function $\tilde{\theta}$ which satisfies the above constraint equation~\ref{eq:constraintang} on each of these intervals, we repeat the construction of the angular variable by applying to $\tilde{\theta}$ the procedure described in the proof of Lemma~\ref{lem:angvar1}. 
\end{proof}

Thus, 
$\dd f(x,y)\wedge \dd \theta(x,y)=2(x^2+y^2)\dd x\wedge \dd y.$
Therefore, we have the explicit inverse relation: 
\begin{equation}
\Psi_*\sigma= \tilde{\sigma}(f,\theta)  \dd x(f,\theta)\wedge \dd y(f,\theta)=  \tilde{\sigma}(f,\theta)(f^2+4\theta^2)^{-\frac{1}{2}}\dd f\wedge \dd \theta
\end{equation}
where $\tilde{\sigma}$ is bounded and the above yields the asymptotics of $\Psi_*\sigma$ near singular points in $\Psi\left(\mathrm{Crit}(f)_1 \right)$. Therefore for every domain $\square\subset [\min(f),\max(f)]\times \mathbb{S}^1$ of finite area with respect to $drd\theta$, $\int_\square \Psi_*\sigma$ is \emph{integrable}.

\end{proof}
\end{comment}

\section{The Yang--Mills connection in the continuum}\label{s:fundamentalformula}

\subsection{A formula for the Yang--Mills connection}

In the parallel work~\cite{BCDRT}, we give a new formula for a 
random  connection in Morse gauge under the free boundary Yang--Mills measure.
In the notations of the present paper, this formula reads formally
\begin{equation}\label{eq:BCDRT}
A:=\int_0^\infty \varphi_f^{-t*}\left(\iota_V\xi_\sigma\right) dt +\sum_{a\in \mathrm{Crit}(f)_1} \log(g_a)U_a
\end{equation}
where  $\xi_\sigma$ is a $\mathfrak{g}$--valued white noise viewed as random $\mathfrak{g}$--valued current of degree $2$ associated to the area form $\sigma$, 
$\iota_V$ is the contraction with the gradient vector field $V=\nabla f$, $U_a:=[W^u(a)]$ are the currents of integration on unstable curve $W^u(a)$ and where the second sum runs over critical points of index $1$ (saddle points).
In~\cite{BCDRT}, we actually prove that equation~\ref{eq:BCDRT} defines an actual random current of degree $1$ of Sobolev regularity $<-1$ by solving a random cohomological equation.

In the current paper,
we will rather use our pseudo-coordinates to write another \emph{global formula} for the 
random connection $A$ from equation ~\ref{eq:BCDRT}.
For every $r,\theta)\in [\min(f),\max(f)]\times [0,2\pi),$ we define
$\square(r,\theta)$ as the unique rectangle obtained by connecting
$(0,\theta), (r,\theta), (0,0), (r,0) $ by straight lines.
Inspired by our discussion of Driver's result, we can write two explicit formulas for the Yang--Mills connection. Recall we use the letter $\xi$ for the white noise on the cylinder adapted to the pushforward area form $\Psi_*\sigma$, it is a random current of degree $2$. Then we may \emph{write formally} the random connection $A$ as~:
\begin{equation} \label{YMFormula}
A= \Psi^*\left(\partial_\theta\left\langle \xi, 1_{\square(r,\theta)} \right\rangle  \dd \theta\right)+\sum_{a\in \mathrm{Crit}(f)_1} \log(g_a)U_a,
\end{equation}
where as above, $U_a:=[W^u(a)]$ are the currents of integration on unstable curve $W^u(a)$ and where the second sum runs over critical points of index $1$ (saddle points).
Let us comment on the formula of equation~(\ref{YMFormula}).
The function $1_{\square(r,\theta)}$ is in $L^2(\Psi_*\sigma)$ therefore it is immediate that the Brownian sheet $W\coloneq \left\langle \xi,  1_{\square(r,\theta)}\right\rangle$, reparametrized by the pushforward area $\Psi_*\sigma$, is well--defined globally as process on the cylinder. We can also define $\left( \partial_\theta W \right)\dd\theta $ globally as a current of degree $1$ on the cylinder, however it is not obvious whether we can pull--back this random current of degree $1$ as $\Psi^*\left( \left( \partial_\theta W \right)\dd\theta\right)$ globally on $\mathcal{S}$, it is only immediate outside unstable curves: $\Psi^*\left( \left( \partial_\theta W \right)\dd\theta\right)\in \mathcal{D}^\prime\left(\mathcal{S}\setminus \overline{  \cup_{a\in \text{Crit}(f)_1} W^u(a)} \right)$. Therefore the global extension is still subtle and still require some non trivial work. This is the topic of the next paragraphs.

We immediately give a second \emph{formal} formula which involves some random series. This second formula will also suggest the correct regularity estimates for $A$. The rigorous estimation of correct regularities will be done later.  
\begin{equation}\label{eq:continuum2}
A= \Psi^*\left(\sum_{n\geq 0} \left(\int_0^r\sqrt{\sigma(s,\theta)}\xi_n^{\mathfrak{g}}(s)ds\right)   e_n(\theta) \dd \theta \right) +   \sum_{a\in \mathrm{Crit}(f)_1} \log(g_a)U_a
\end{equation}
where $\xi_n^{\mathfrak{g}}$ are i.i.d $\mathfrak{g}$ valued white noise, $(e_n)_n$ is a ONB of functions on the circle $\mathbb{S}^1$. The goal of the whole section is to give a precise mathematical meaning to the above two formulas.
The proof of the correct regularity will be done when we give a rigorous meaning to the above formula. 
In fact, we observe that we will also recover these regularity estimates when we study scaling limit of connections coming from the discrete gauge theory, we refer the reader to section \ref{s:functspaces}. 
The random series $A$ will be defined on $\mathcal{S}$ by the singular pull--back and the detailed study of the regularity of derivatives in $\theta$ of the area functional. But formula (\ref{eq:continuum2}) already suggests that
outside $\mathrm{Crit}(f)$, the random connection $A$ is expected to have regularity $\frac{1}{2}-$ along the flow.

We start by summarizing what we know and what are the main difficulties.

\subsection{The motivation to deal with singularities}

Let us explain what is going on and why we need the results of the present subsection. We start by defining some random distribution $\tilde{A}$ on the cylinder 
$\mathbf{Cyl}:=[\min(f),\max(f)]\times \mathbb{S}^1$ in a way very similar to Driver's 
construction using the push--forward area $\Psi_*\sigma\in \Omega^2\left( \mathbf{Cyl}\right)$ induced from the surface area $\sigma\in \Omega^2(\mathcal{S})$. Since $\Psi$ fails to be a diffeomorphism along unstable curves, the pushforward area form $\Psi_*\sigma$ is singular at images of saddle points $\Psi\left(\mathrm{Crit}(f)_1 \right)$ by $\Psi$.
\begin{definition}[The random connection $\tilde{A}$ upstairs ]
The  random distribution $\tilde{A}$ defined upstairs on the cylinder 
$\mathbf{Cyl}:=[\min(f),\max(f)]\times \mathbb{S}^1$
is given by the formula
\begin{equation}\label{eq:tildeA}
 \tilde{A}:=\sum_{n\geq 0} \left(\int_0^r\sqrt{\sigma(s,\theta)}\xi_n^{\mathfrak{g}}(s)ds\right)   e_n(\theta) \dd \theta = \partial_\theta \left\langle \xi,1_{\square(r,\theta)} \right\rangle  d\theta    
\end{equation}
in the coordinates of the cylinder and where $\xi$ is the $\mathfrak{g}$-valued white noise w.r.t. to the area form $\Psi_*\sigma$.
\end{definition}

Then we can
integrate $\tilde{A}(r,\theta)\dd\theta$ globally on the cylinder $\mathbf{Cyl}$ 
along the $\theta$ direction
which 
defines an element denoted by $W(r,\theta)$ that we control in certain anisotropic H\"older spaces outside $\Psi(\mathrm{Crit}(f)_1)$.
However, the singularities of $\Psi_*\sigma$ imply that we can control $W(r,\theta)$ only
in some weighted anisotropic spaces where the weight controls 
singularities when we approach the points in $\Psi(\mathrm{Crit}(f)_1)$. These weighted anisotropic spaces are defined by scalings directly on the cylinder.
The second problem we need to handle is that $\Psi:\mathcal{S}\mapsto [\min(f),\max(f)]\times \mathbb{S}^1$ fails to be a diffeomorphism along unstable curves, so it is not obvious that the pull--back of any current by $\Psi$ should be well--defined. What we will prove is that if we can decompose the 
random connection $\tilde{A}$ as sums of pieces and glue together these pieces along the unstable curves except at saddle points, then we will
scale near saddle points and this allows us to define the pull--back $\Psi^*\left(  \tilde{A}d\theta\right)$ globally on $\mathcal{S}$. 
We next define precisely the functional spaces involved in our discussion.

\subsection{Weighted anisotropic Sobolev and H\"older norms}\label{DistSpaces}

The formula from equation \ref{eq:continuum2}
suggests that the random connection $A$ has \emph{regularity $\frac{1}{2}-$ along the flow direction and $-\frac{1}{2}-$ along the levels sets} of the Morse function $f$.
So we expect the regularity of $A$ to be \emph{anisotropic} which motivates the introduction of anisotropic Banach spaces of distributions of low regularity to describe precisely our random current $A$.  
Our Banach spaces are somewhat reminiscent of the anisotropic Banach spaces appearing in the functional analysis of hyperbolic dynamical systems, specially the spaces from the works of Baladi--Tsujii~\cite{BaladiTsujii1,BaladiTsujii2}. We refer the reader to the book~\cite{Baladibook} and the extensive survey~\cite{Baladisurvey} and the references therein.

\subsubsection{Non weighted anisotropic semi-norms}

We start by some semi-norm on the cylinder.
\begin{definition}[Anisotropic semi-norms on cylinder]\label{def:anisocyl}
Set $\alpha\in (0,1)$.
Outside $\Psi(\mathrm{Crit}(f)_1)$, on some product of intervals $I\times J\subset [\min(f),\max(f)]_r\times \mathcal{S}^1_\theta $ which avoids $\Psi(\mathrm{Crit}(f)_1)$, we define these local semi-norms on smooth connections as~:
\begin{equation}\label{eq:weightedSobo}
\Vert \tilde{A}\Vert_{\mathcal{W}^{\alpha,\alpha-1;p}_{r,\theta}(I\times J) }\coloneq    \Vert W\Vert_{\mathcal{W}^{\alpha,\alpha;p}_{r,\theta}(I\times J) }= 
\left(
\int_{I^2} \frac{\Vert W(r_1,.)-W(r_2,.) \Vert^p_{W^{\alpha,p}_{\theta,J}}}{\vert r_1-r_2\vert^{1+\alpha p}} \dd r_1\dd r_2 \right)^{\frac{1}{p}} + \Vert W \Vert_{L^p(I\times J)}  \end{equation}
for $\tilde{A}(r,u) \coloneq \partial_\theta W(r,\theta) $ and $W(.,0)=0$~\footnote{This acts like some half--Dirichlet condition on $W$.} for $J=[0,a]$.

The H\"older version reads
\begin{equation}\label{eq:anisoHolder}
 \Vert \tilde{A}\Vert_{\mathcal{C}^{\alpha,\alpha-1}_{r,\theta}(I\times J) }\coloneq    \Vert W\Vert_{\mathcal{C}^{\alpha,\alpha}_{r,\theta}(I\times J) }= 
\sup_{ r_1\neq r_2\in  I^2} \frac{\Vert W(r_1,.)-W(r_2,.) \Vert^p_{\mathcal{C}^{\alpha}_{\theta,J}}}{\vert r_1-r_2\vert^{\alpha}}+ \sup_{(I\times J)}\vert W\vert.    
\end{equation}
\end{definition}

We consider the completion of $C^\infty$ connections $A$, with half--Dirichlet condition $A(0,.)=0$ for the above seminorms where $I\times J$ runs over all product intervals avoiding $\Psi(\mathrm{Crit}(f)_1)$. This defines local anisotropic Sobolev spaces (resp H\"older spaces).

These anisotropic semi-norms of fractional Sobolev type are defined in the spirit of the work of Gagliardo and Slobodeckij~\cite[Def 1.12 p.~13]{Leoni}, \cite[equation (2.2) p.~524]{Nezzafractional} and involve integral formulas in position space as opposed to the Fourier based definitions. This will be very convenient for our probabilistic applications: estimating the regularities of our random connections and proving tightness of sequences of measures on piecewise affine connections. 

We can also define a similar semi-norm but this time on $\mathcal{S}$~:
\begin{definition}[Anisotropic semi-norms on $ \mathcal{S}$ ]
Set $\alpha\in (0,1)$.
On every flowbox $\square \simeq [0,1]^2 $ \emph{avoiding the unstable curves} with local coordinates $(r,\theta)$ where $V=\partial_r$, we define these local semi-norms on smooth connections as~:
\begin{equation}\label{eq:weightedSobo}
\Vert A\Vert_{\mathcal{W}^{\alpha,\alpha-1;p}_{r,\theta}(\square) }\coloneq  \left( \int_{I^2} \frac{\Vert W(r_1,.)-W(r_2,.) \Vert^p_{W^{\alpha,p}_{\theta,J}}}{\vert r_1-r_2\vert^{1+\alpha p}} \dd r_1\dd r_2 \right)^{\frac{1}{p}}+\Vert W\Vert_{L^p(\square)}  \end{equation}
for $\tilde{A}(r,u) \coloneq \partial_\theta W(r,\theta) $ and $W(.,0)=0$.
\end{definition}

After completion w.r.t the above semi-norms, 
this defines a \emph{local anisotropic Sobolev} space outside the saddle points.
A similar completion process w.r.t. the anisotropic H\"older semi-norms allows to define \emph{local anisotropic H\"older} norms.

We state an easy, yet essential lemma. It plays an essential
role in the sequel.
\begin{lemma}
For $\alpha\in (0,1)$ in the setting of definition~\ref{def:anisocyl}, any distribution $T\in \mathcal{C}_{r,\theta}^{\alpha,\alpha-1}([0,1]\times [0,1])$ can be viewed as a H\"older continuous function of $r$ valued in distributions in the $\theta$-variable in $\mathcal{D}^\prime((0,1))$.    
\end{lemma}

For $\alpha\in (0,\frac{1}{2})$, we have continuous injections $$\mathcal{W}^{\alpha,\alpha-1,p}\hookrightarrow  \mathcal{C}^{\alpha-\frac{1}{p},\alpha-1-\frac{1}{p}} \text{ for } \alpha>\frac{1}{p}$$ and 
$$ \mathcal{C}^{\alpha,\alpha-1} \hookrightarrow \mathcal{C}^{\alpha-1} $$
whose proof is done in Proposition~\ref{prop:continuousinjection} of the appendix. In the appendix, we also establish the necessary compact injections $\mathcal{C}^{\alpha+\varepsilon,\alpha+\varepsilon-1} \hookrightarrow \mathcal{C}^{\alpha,\alpha-1} $, $\forall\varepsilon>0$ which are needed 
to show tightness of the sequence of measures on these anisotropic spaces.

\begin{remark}
An important remark is that for every $\alpha>-1$ and every interval $[a,b]\subset \mathbb{R}$, it makes perfect sense to think about $\mathcal{C}^{\alpha}([a,b])$ as the space of H\"older distribution which are supported on $[a,b]$. Moreover, we prove in Lemma~\ref{lem:restrictionBesov} in the appendix that any distribution $T$ in $\mathcal{C}^\alpha(\mathbb{R})$ can be multiplied with the indicator function $1_{[a,b]}$ to define
$T1_{[a,b]}$ which is the canonical extension of $T\in \mathcal{D}^\prime((a,b))$.
However the multiplication with the indicator function produces some loss of regularity which is discussed in Lemma~\ref{lem:restrictionBesov}.
\end{remark}

Moreover, we will also often use the following result: for every $\alpha\in (0,1)$, we can always integrate anisotropic distributions $T\in \mathcal{C}^{\alpha,\alpha-1}(\mathbf{Cyl})$ in the $\theta$ direction and define 
$W:=\int_0^\theta T(r,u)du$.

\subsubsection{The weighted anisotropic spaces by scalings}

Recall that the map 
$\Psi$ has both discontinuities along 
unstable curves and 
singularities at the saddle points which also induces singularities of the pushforward area $\Psi_*\sigma$ at $\Psi\left(\mathrm{Crit}(f)_1 \right)$. Therefore we can only define $\tilde{A}$ outside $\Psi\left(\mathrm{Crit}(f)_1 \right)$ and we will need some functional spaces that will control how $\tilde{A}$ behaves as we zoom in closer and closer to points in $\Psi\left(\mathrm{Crit}(f)_1 \right)$. This is the main motivation
to introduce some new weighted norms near the critical points and their image by $\Psi$.
 So there are two types of spaces we will introduce: on the cylinder which resolves $\mathcal{S}$ and on the surface $\mathcal{S}$ itself.  
 Our ideas are very close to those appearing in the work of Bony on second micro-localization~\cite[b) page 7]{Bony1},~\cite[Eq 2.12 Def 2.5 p.~18]{Bony2} and also Meyer's work revisiting certain aspects of second micro-localization~\cite{MeyerWaveletsVibrationsScalings}. As usual, we rely 
on scalings and dyadic decompositions to define our spaces. 

We next use the previous semi-norms together with local scalings to define weighted spaces of
distributions.
\begin{definition}[Scalings and weighted semi-norms]
Set $\alpha\in (0,1)$.
On the cylinder $[\min(f),\max(f)]_r \times \mathbb{S}^1_\theta $ near any element $(r_0,\theta_0)$,
we define local scaling centered at $(r_0,\theta_0)$ as~:
$$ \mathcal{S}^{\lambda}_{r_0,\theta_0}(r,\theta)\coloneq (\lambda(r-r_0)+r_0,\lambda(\theta-\theta_0)+\theta_0  ), \forall \lambda\in \mathbb{R}_{>0} . $$

We define weighted semi-norms a singular point $(r_0,\theta_0)\in \Psi(\mathrm{Crit}(f)_1) \subset [\min(f),\max(f)]_r \times \mathbb{S}^1_\theta$ in the cylinder.
We require that for a certain scaling exponent $s<2\alpha-1$,
\begin{equation}\label{eq:weightedHolder}
\Vert \tilde{A}\Vert_{\mathcal{W}^{\alpha,\alpha-1;p;s}_{r,\theta}(\mathrm{Cyl})}^p\coloneq   \sum_{(r_0,\theta_0)\in \Psi(\mathrm{Crit}(f)_1)} \sum_{I\times J} \sum_{n=0}^\infty \left(2^{n(s-\frac{2}{p})}\Vert \mathcal{S}^{2^{-n} *}_{(r_0,\theta_0)} \tilde{A}\Vert_{\mathcal{W}^{\alpha,\alpha-1;p}_{r,\theta} (I\times J) }\right)^p + \sum_{I_2\times J_2} \Vert \tilde{A} 
\Vert^p_{\mathcal{W}^{\alpha,\alpha-1;p}_{r,\theta}(I_2\times J_2)} 
 \end{equation}
where we take the first sum over some \emph{finite cover} of the form $I_r\times J_\theta$ of a certain corona of the form $\{ m\in \Sigma; \mathrm{dist}(m,\Psi(a))\in [1,2] \}$ centered near a singular point $\Psi(a)=(r_0,\theta_0)$ and the second sum $\sum_{I_2\times J_2} \Vert \tilde{A} \Vert^p_{\mathcal{W}^{\alpha,\alpha-1,p}_{I_2\times J_2}}$ runs over a finite cover of $\{m, \mathrm{dist}\left(m,  \Psi\left(\mathrm{Crit}(f)_1\right)\right)\geqslant 1 \}$.
\end{definition}
The index $\alpha$ indicates the regularity in the $r$ variable, this means regularity along the flow, 
the index $\alpha-1$ indicates the regularity in the $\theta$ variable, this means transversal to the flow,
the index $p\geq 2$ tells us we used $L^p$ and $\ell^p$ norms to define our Sobolev spaces and finally the
exponent $s\in \mathbb{R}$ is a scaling exponent which indicates how our connection $\tilde{A}$ blows up as measured in Sobolev norms when we approach the critical point $a$; check~\ref{fig:Corona}.

\begin{figure}[t]
    \centering
    \includegraphics[width=1.1\linewidth]{ 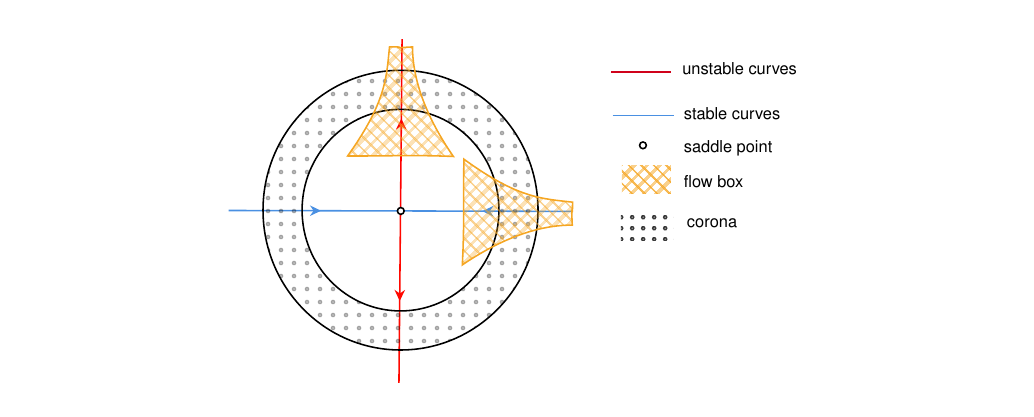}
    \caption{A corona around a critical point.}
    \label{fig:Corona}
\end{figure}
We also define
such weighted norms for anisotropic H\"older norms, for $s<2\alpha-1$ 
\begin{equation}
\Vert \tilde{A}\Vert_{\mathcal{C}^{\alpha,\alpha-1;s}(\mathrm{Cyl}) }  \coloneq \sup_{(r_0,\theta_0)\in \Psi(\mathrm{Crit}(f)_1)} \sup_{I\times J} \sup_{n\geq 0} 2^{ns}\Vert \mathcal{S}^{2^{-n} *} A\Vert_{\mathcal{C}^{\alpha,\alpha-1}_{r,\theta}(I\times J) }. 
 \end{equation}

\subsection{Regularity estimates in anisotropic spaces}

We estimate the regularity of 
$\tilde{A}\in \mathcal{D}^\prime([\min(f),\max(f)] \times \mathbb{S}^1_\theta )$ defined by equation~(\ref{eq:tildeA}) in global coordinates, start from the integrated object
$W(r,\theta)=\int_0^\theta \tilde{A}(r,u)du$.
Recall we would like to control the anisotropic norm of 
$W$ on some product of intervals $I_1\times I_2$ as follows~:
\begin{align*}
\Vert W \Vert^{2p}_{\mathcal{W}^{\alpha,\alpha;2p}_{r,\theta}(I_1\times J_1)}\coloneq \int_{I_1^2\times I_2^2} \frac{\vert W(r_1,\theta_1)-W(r_2,\theta_1)- W(r_1,\theta_2)+W(r_2,\theta_2)  \vert_{\mathfrak{g}}^{2p} }{ \vert r_1-r_2 \vert^{1+2\alpha p}\vert \theta_1-\theta_2 \vert^{1+2\alpha p}  }  \dd r_1\dd r_2\dd\theta_1\dd\theta_2    
\end{align*}
Then we get an estimate of the form~:
\begin{align*}
\mathbb{E}\left[\Vert W \Vert^p_{\mathcal{W}^{\alpha,\alpha;p}_{r,\theta}(I_1\times I_2)}\right] & =\int_{I_1^2\times I_2^2} \frac{\mathbb{E}\left[ \vert W(r_1,\theta_1)-W(r_2,\theta_1)- W(r_1,\theta_2)+W(r_2,\theta_2)  \vert_{\mathfrak{g}}^{2p} \right] }{ \vert r_1-r_2 \vert^{1+2\alpha p}\vert \theta_1-\theta_2 \vert^{1+2\alpha p}  }  \dd r_1\dd r_2\dd\theta_1\dd\theta_2 \\
&\leq C_p\int_{I_1^2\times I_2^2} \frac{\mathbb{E}\left( \vert W(r_1,\theta_1)-W(r_2,\theta_1)- W(r_1,\theta_2)+W(r_2,\theta_2)  \vert_{\mathfrak{g}}^{2} \right)^p }{ \vert r_1-r_2 \vert^{1+2\alpha p}\vert \theta_1-\theta_2 \vert^{1+2\alpha p}  }  \dd r_1\dd r_2\dd\theta_1\dd\theta_2\\
&\leq C_p\int_{I_1^2\times I_2^2} \frac{ \left(\int_{\square(r_1,r_2,\theta_1,\theta_2)} \left(\Psi_*\sigma\right)(r,\theta)\dd rd\theta \right) ^p }{ \vert r_1-r_2 \vert^{1+2\alpha p}\vert \theta_1-\theta_2 \vert^{1+2\alpha p}  }  \dd r_1\dd r_2\dd\theta_1\dd\theta_2
\end{align*}
where we used Fubini, then hypercontractivity of $W$ since it is Gaussian where $\square(r_1,r_2,\theta_1,\theta_2) $ is the rectangle bordered by $(r_1,\theta_1),(r_2,\theta_1),(r_1,\theta_2),(r_2,\theta_2)$ then the above is bounded for $p$ large enough as soon as $\alpha<\frac{1}{2}$. If we specialize our estimate to rectangular regions $I_1\times I_2$ in $(r,\theta)$ that avoid $\Psi\left(\mathrm{Crit}(f)_1\right)$, then
we have upper bounds
of the form $\int_{\square(r_1,r_2,\theta_1,\theta_2)} \left(\Psi_*\sigma \right) (r,\theta)\dd r\dd\theta =\mathcal{O}(\vert r_1-r_2\vert \vert \theta_1-\theta_2 \vert)$.

Now we deal with neighborhoods of singular points in $\Psi\left(\mathrm{Crit}(f)_1\right)$. We
assume without loss of generality that 
the singular point reads $(0,0)$ in coordinates of the cylinder.
We just need to scale the previous estimate taking into account that the 
area form blows up at the singular point $(0,0)$:
\begin{align*}
    \mathbb{E}\left[\Vert  \mathcal{S}^{\lambda *}_{r,\theta} W \Vert^{2p}_{\mathcal{W}^{\alpha,\alpha;2p}_{r,\theta}}  \right] &\lesssim C_p\int_{I_1^2\times I_2^2} \frac{ \left(\int_{\square(\lambda r_1, \lambda r_2, \lambda \theta_1, \lambda \theta_2)}  \left(\Psi_*\sigma \right)(r,\theta)\dd r\dd\theta \right) ^p }{ \vert r_1-r_2 \vert^{1+2\alpha p}\vert \theta_1-\theta_2 \vert^{1+2\alpha p}  }  \dd r_1\dd r_2\dd\theta_1\dd\theta_2\\
    \lesssim  C_p\int_{I_1^2\times I_2^2} &\frac{ \left(\int_{\square(\lambda r_1, \lambda r_2, \lambda \theta_1, \lambda \theta_2)} \left(\Psi_*\sigma \right)(r,\theta)\dd r\dd\theta \right) ^p }{ \vert r_1-r_2 \vert^{1+2\alpha p}\vert \theta_1-\theta_2 \vert^{1+2\alpha p}  }  \dd r_1\dd r_2\dd\theta_1\dd\theta_2\\
    = C_p&\int_{I_1^2\times I_2^2} \frac{ \left(\int_{\lambda r_1}^{\lambda r_2} \int_{\lambda \theta_1}^{\lambda \theta_2} \left(\Psi_*\sigma \right)(r,\theta)\dd r\dd\theta \right) ^p }{ \vert r_1-r_2 \vert^{1+2\alpha p}\vert \theta_1-\theta_2 \vert^{1+2\alpha p}  }  \dd r_1\dd r_2\dd\theta_1\dd\theta_2\\
= &C_p\int_{I_1^2\times I_2^2} \frac{ \left(\int_{ r_1}^{ r_2} \int_{\theta_1}^{ \theta_2} \left(\Psi_*\sigma \right)(\lambda r, \lambda \theta) \lambda^2 \dd r\dd\theta \right) ^p }{ \vert r_1-r_2 \vert^{1+2\alpha p}\vert \theta_1-\theta_2 \vert^{1+2\alpha p}  }  \dd r_1\dd r_2\dd\theta_1\dd\theta_2\\
&\lesssim C_p\lambda^p\int_{I_1^2\times I_2^2} \frac{ \left(\int_{\square(r_1,r_2,\theta_1,\theta_2)} \dd rd\theta \right) ^p }{ \vert r_1-r_2 \vert^{1+2\alpha p}\vert \theta_1-\theta_2 \vert^{1+2\alpha p}  }  \dd r_1\dd r_2\dd\theta_1\dd\theta_2.
\end{align*}
where we use the bound $\Psi_*\sigma(r,\theta)=\mathcal{O}( (r^2+4\theta^2)^{-\frac{1}{2}} ) $
which implies that 
$$\Psi_*\sigma( \lambda r, \lambda \theta)|_{(r,\theta)\in \square(r_1,r_2,\theta_1,\theta_2)}=\mathcal{O}\left( \lambda^{-1} \right) $$
hence $ \left(\int_{ r_1}^{ r_2} \int_{\theta_1}^{ \theta_2} \left(\Psi_*\sigma \right)(\lambda r, \lambda \theta) \lambda^2 \dd r\dd\theta \right)^p=\mathcal{O}(\lambda^p(\vert r_1-r_2\vert \vert \theta_1-\theta_2 \vert)^p)$. 
From the above bound, for all $\alpha\in (0,\frac{1}{2})$, for all $\beta<\frac{1}{2}$, there exists $p$ large enough s.t. ~:
\begin{align*}
\sum_{n\geqslant 1} 2^{n(\beta-\frac{2}{2p}) 2p} \mathbb{E}\left(\Vert  \mathcal{S}^{2^{-n} *}_{r,\theta} W \Vert^{2p}_{\mathcal{W}^{\alpha,\alpha;2p}_{r,\theta}}  \right)  
\leqslant \mathbb{E}\left(\sum_{n\geqslant 1} 2^{n(\beta-\frac{2}{2p} )2p}\Vert  \mathcal{S}^{2^{-n} *}_{r,\theta} W \Vert^{2p}_{\mathcal{W}^{\alpha,\alpha;2p}_{r,\theta}}  \right)\\
\lesssim C_p \int_{I_1^2\times I_2^2} \frac{ (\vert r_1-r_2\vert \vert \theta_1-\theta_2 \vert)^p }{ \vert r_1-r_2 \vert^{1+2\alpha p}\vert \theta_1-\theta_2 \vert^{1+2\alpha p}  }  \dd r_1\dd r_2\dd\theta_1\dd\theta_2 \sum_{n\geqslant 1} 2^{n(\beta-\frac{2}{2p} -  \frac{1}{2})2p}<+\infty
\end{align*}
where the conclusion holds true because $I_1\times I_2$ never meets $(0,0)$.
By the commutation relation
$\partial_\theta\mathcal{S}^{2^{-n}*}_{r_0,\theta_0}=2^{-n}\partial_\theta\mathcal{S}^{2^{-n}*}_{r_0,\theta_0}$,
$$ \Vert \mathcal{S}^{2^{-n}*}\tilde{A} \Vert_{\mathcal{W}^{\alpha,\alpha-1,p}}=\Vert \mathcal{S}^{2^{-n}*}\partial_\theta W \Vert_{\mathcal{W}^{\alpha,\alpha-1,p}}=2^n\Vert \partial_\theta \mathcal{S}^{2^{-n}*} W  \Vert_{\mathcal{W}^{\alpha,\alpha-1,p}}=2^n\Vert  \mathcal{S}^{2^{-n}*} W  \Vert_{\mathcal{W}^{\alpha,\alpha,p}},$$
combining with the above estimate on $\sum_{n\geqslant 1} 2^{n(\beta-\frac{2}{2p}) 2p} \mathbb{E}\left(\Vert  \mathcal{S}^{2^{-n} *}_{r,\theta} W \Vert^{2p}_{\mathcal{W}^{\alpha,\alpha;2p}_{r,\theta}}  \right)  $, we get
\begin{align*}
\mathbb{E}\left( \sum 2^{n(s-\frac{2}{2p})2p}    \Vert \mathcal{S}^{2^{-n}*}\tilde{A} \Vert_{\mathcal{W}^{\alpha,\alpha-1,2p}}^{2p} \right) =
\mathbb{E}\left( \sum 2^{n(s-\frac{2}{2p})2p}   2^{2np}\Vert  \mathcal{S}^{2^{-n}*} W  \Vert_{\mathcal{W}^{\alpha,\alpha,2p}}^{2p} \right)\lesssim \sum 2^{n(s-\frac{2}{2p})2p}   2^{np}
\end{align*}
since $\mathbb{E}\left(\Vert  \mathcal{S}^{2^{-n}*} W  \Vert_{\mathcal{W}^{\alpha,\alpha,2p}}^{2p} \right)=\mathcal{O}(2^{-np})$ and the series converges as soon as $s<-\frac{1}{2}$ for large enough $p$.

We deduce the following Lemma:
\begin{lemma}[Regularity of $\tilde{A}$]
For any $s<-\frac{1}{2}$ and for the random connection $\tilde{A}$ given by equation (\ref{eq:tildeA})~:
\begin{equation}
\mathbb{E} \left(\Vert   \tilde{A} \Vert^p_{\mathcal{W}^{\alpha,\alpha-1,p,s}}  \right)<+\infty.  
\end{equation}
\end{lemma}

When we calculate the expectations $\mathbb{E}\left[\Vert  \mathcal{S}^{\lambda *}_{r,\theta} W \Vert^{2p}_{\mathcal{W}^{\alpha,\alpha;2p}_{r,\theta}}  \right] $, we need to justify that 
the integrand itself $\Vert  \mathcal{S}^{\lambda *}_{r,\theta} W \Vert^{2p}_{\mathcal{W}^{\alpha,\alpha;2p}_{r,\theta}}$ is measurable. A simple way to justify this is to go back to the white noise $\xi$ on $\mathbf{Cyl}$ and to replace $\xi$ with the finite sum
 $ \xi_N:=\sum_{i=0}^N c_i e_i  $ 
where $(c_i)_{i\in \mathbb{N}} $ is an i.i.d sequence of random variables distributed as $\mathcal{N}(0,1)$,
$(e_i)_{i\in \mathbb{N}}$ is any orthonormal basis of $L^2\left(\mathbf{Cyl},\Psi_*\sigma, \mathfrak{g} \right)$. This yields a sequence of random connection 
$$\tilde{A}_N:= \Psi^*\left( \partial_\theta \left\langle \xi_N, 1_{\square(r,\theta)} \right\rangle  d\theta \right) +\sum_{a\in \mathrm{Crit}(f)_1} \log(g_a)U_a $$ 
and $W_N:=\int_0^\theta \tilde{A}_N(r,u)du$. Then we can repeat all the calculations of the present paragraph with $\tilde{A}_N$ and $W_N$ instead of $\tilde{A}$ and $W$, all objects are measurable since $\tilde{A}_N$ and $W_N$ depend on a finite number of random variables. Now we can let $N\rightarrow +\infty$ in all terms since $W_N$ is a Martingale bounded in $L^2$. We refer the reader to the companion paper~\cite{BCDRT} for more details on this approach.

\subsection{Extending near defect lines and reducing to neighborhoods of saddle points}
\label{ss:singpullback}

\paragraph{Clarifying various notions of distributions on some domain with smooth corners.} We state very precisely what we mean by some distribution, or current, defined on some manifold with corner $H$ whose role will be played by the Hexagons $(H_i)_{i=1}^{4g}$. Let $H$ be a surface with smooth corners which is smoothly embedded in an closed compact surface $\mathcal{S}$: $H\hookrightarrow \mathcal{S}$. Then there are mostly four classes of distributions one can define on such surface with smooth corners, each definition corresponds to some choice for the class of test functions:
\begin{enumerate}
    \item If the test functions are in $C^\infty_c(\mathrm{int}(H))$ means their support does not meet the boundary, then the dual space reads $\mathcal{D}^\prime(\mathrm{int}(H))$, this is the largest possible space of distributions defined on the open set $\mathrm{int}(U)$.
    \item If the test functions belong to the ideal $\mathcal{S}(H)$ (by analogy with Schwartz class functions) of $C^\infty(H)$ that vanish at infinite order at the boundary $\partial H$, then the dual $\mathcal{S}^\prime(H)$ consists of all distributions in $\mathcal{D}^\prime(\mathrm{int}(H))$ which are extendible on $\mathcal{S}$. One can also think of them as distributions obtained by restricting elements in $\mathcal{D}^\prime(\mathcal{S})$ to $H$, this is some kind of extrinsic viewpoint, or as distributions in $\mathcal{D}^\prime(\mathrm{int}(H))$ with moderate growth when one approaches the boundary $\partial H$ a point of view pioneered by Kashiwara in ~\cite[section 3 p.~332]{KashiwaraRH} and also described in detail in~\cite[section 2 to 5]{DangHers}.
    The terminology Schwartz is inspired from \cite{AizenbudGourevitchNash} and \cite{Casselman}. Given an extendible distribution $T\in \mathcal{D}^\prime(\mathrm{int}(H))$, any extension $\overline{T}\in \mathcal{D}^\prime(\mathcal{S})$ is \emph{not necessarily unique}. 
    \item If the test functions are the smooth functions in $C^\infty(H)$ which are smooth up to the boundary. In this case, the topological dual can be realized as distributions in $\mathcal{D}^\prime(\mathcal{S})$ which are supported
    in $H$ (the rigorous proof uses Whitney's extension Theorem for smooth functions). This space is denoted by $\mathcal{D}^\prime_H(\mathcal{S})$.
    \item If the test functions are $\mathcal{C}^\alpha(H)$ functions  for $\alpha\in (0,1)$ and given $\overline{T}\in \mathcal{C}^\alpha(H)^\prime$, we identify its restriction to $\mathrm{int}(H)$ as an element $T\in \mathcal{D}^\prime(\mathrm{int}(H))$, and one can recover $\overline{T}$ as an element in $\mathcal{D}^\prime_K(\mathcal{S})$ from $T$ by a limiting procedure as follows~:
    $$\overline{T}:=\lim_{\varepsilon\rightarrow 0^+}\chi_\varepsilon T $$
    for $(\chi_\varepsilon)_\varepsilon$ a family of smooth functions in $C^\infty_c(\mathrm{int}(H))$ s.t. $\chi_\varepsilon=1$ outside some $\varepsilon$--neighborhood of $\partial H$. This is the most restricted class of distributions of the four classes of distributions we just described and this is the class of distributions that we will mostly encounter in this work.
\end{enumerate}

For a pedagogical presentation of the three first class of distributions with more references on these topics, we recommend to look at ~\cite[section 2 to 5]{DangHers} and the references therein, we also recommend to look at the papers~\cite{AizenbudGourevitchNash}, \cite[p.~157--165]{Casselman} for the relation with real algebraic and analytic geometry.
The discussion for currents is almost verbatim, the reader just has to replace the word test functions with test forms.

\paragraph{The three steps extensions.} We would like to prove the following.
\begin{proposition}\label{ContPullBack}
     Let $\alpha\in (0,\frac{1}{2})$ and $s\in (-2,0)$.
Let $\tilde{A}\dd\theta$ be any current in the weighted anisotropic space $\mathcal{C}^{\alpha,\alpha-1,s}\left( \mathbf{Cyl}\setminus \Psi(\mathrm{Crit}(f)_1) \right)$ of the cylinder $\mathbf{Cyl}:=[\min(f),\max(f)]_r\times \mathbb{S}^1_\theta$. 
Then for $\varepsilon$ small enough, the pull--back $\Psi^*\left(\tilde{A}d\theta\right)$ 
by the map $\Psi$ from equation~\ref{eq:globalPsi} which is well--defined on 
\[\mathrm{int}\left(\mathcal{S}\right)\setminus \overline{\cup_{a\in \mathrm{Crit}(f)_1} W^u(a)}\]
\emph{extends uniquely to some distribution} in $ \mathcal{C}^{-\beta-2}\left( \mathcal{S} \right) $, for all $\beta$ such that $\beta+\alpha-1>0$ and the pull--back operator followed by the extension map 
$$\Psi^*:\mathcal{C}^{\alpha,\alpha-1,s}\left(  [\min(f),\max(f)]\times \mathbb{S}^1 \setminus \Psi\left(\mathrm{Crit}(f)_1\right)  \right) \longmapsto  \mathcal{C}^{-\beta-2}(\mathcal{S}) $$ is linear continuous.
\end{proposition}

The proof of this extension is done in three steps, and is summarized in Figure \ref{threestepext}
\begin{figure}[t]
    \centering
    \includegraphics[width=\linewidth]{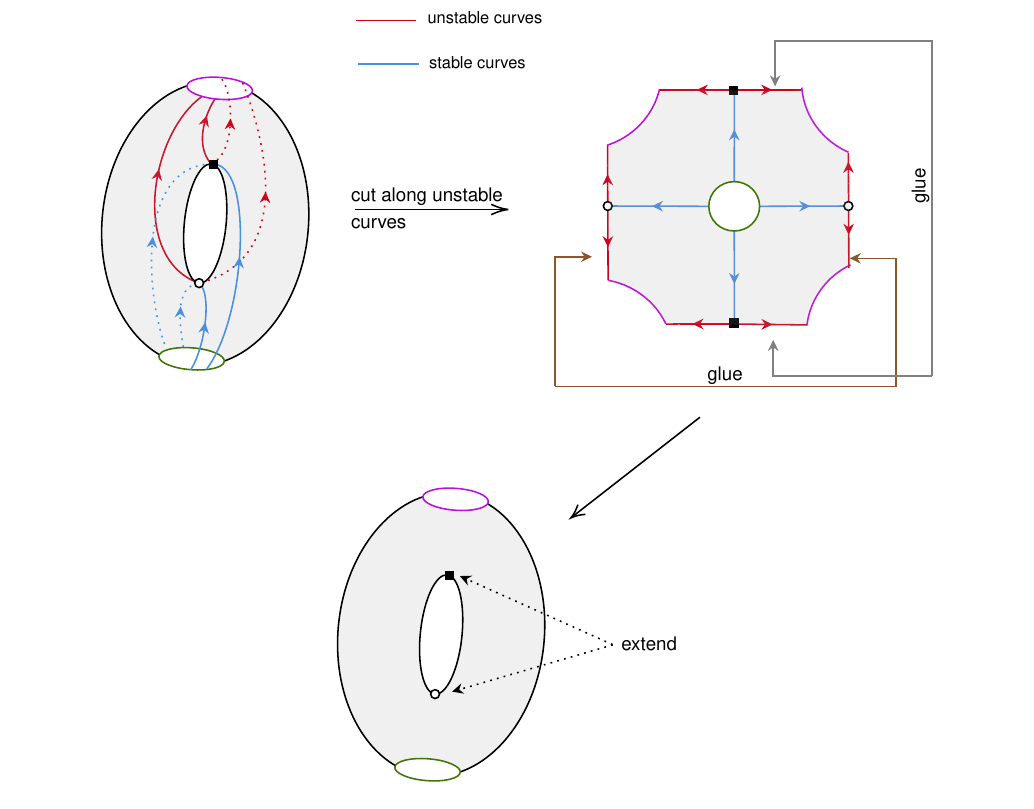}
    \caption{Three-step extension.}
    \label{threestepext}
\end{figure}

\begin{enumerate}
    \item First, since $\Psi$ is discontinuous along unstable curves argue that the pull--back $\Psi^*\left(\tilde{A}d\theta\right)$ extends across unstable curves and defines a distribution in 
    $\mathcal{S}\setminus \mathrm{Crit}(f)_1$. We deal with this part in Example~\ref{ex:pullbackgluing} and Lemma~\ref{lem:extcurves}. 
    \item Second, our distribution $ \tilde{A}d\theta$ is a current in the weighted anisotropic space 
    \[\mathcal{C}^{\alpha,\alpha-1,s}\left( \mathbf{Cyl}\setminus \Psi(\mathrm{Crit}(f)_1) \right)\] hence it is only well--defined on the pointed space $\mathbf{Cyl}\setminus \Psi(\mathrm{Crit}(f)_1)$ where we deleted images under $\Psi$ of all saddle points. So we need to prove that 
    the scaling of $ \tilde{A}d\theta$ near the singular points in $\Psi(\mathrm{Crit}(f)_1)$ which is controlled by the exponent $s$ from the weighted norm, allows us to extend  $ \tilde{A}d\theta$ canonically as a current in the topological dual $\mathcal{C}^{\beta}(\mathbf{Cyl})^\prime$ for all $\beta$ such that $\beta+\alpha-1>0$. This is done in Lemma~\ref{lem:extcylinder}.
    \item Third and finally, once $\tilde{A}d\theta$ is well--defined globally on $\mathbf{Cyl}$ and 
    also that we are allowed to define the singular pull--back $\Psi^*\left(\tilde{A}d\theta\right)$ on the pointed surface 
    $\mathcal{S}\setminus \mathrm{Crit}(f)_1$, we use all methods from the previous two steps to construct the extension of $\Psi^*\left(\tilde{A}d\theta\right)$ near all saddle points in $\mathrm{Crit}(f)_1$. This uses both ideas from gluing and also relies on the scaling behaviour of 
    $\Psi^*\left(\tilde{A}d\theta\right)$ as defined in the first step near each saddle point. This is done in Lemma~\ref{lem:extsurface}.
\end{enumerate}

\paragraph{First extensions across edges of the hexagons $H_i$.}

Recall each $H_i$ is our hexagon in the decomposition of $\mathcal{S}$.
Here the main technical problem consists in gluing distributions
defined on each hexagon $H_i$ into some global object defined on the surface $\mathcal{S}$ or conversely given some global distribution $T$ defined on the surface $\mathcal{S}$, how can we localize on each piece $H_i$ by multiplying with indicators $1_{H_i}$ of $H_i$.

We begin by illustrating the whole difficulty with an example which contains all the crux of the problem.
\begin{figure}
    \centering
    \includegraphics[width=\linewidth]{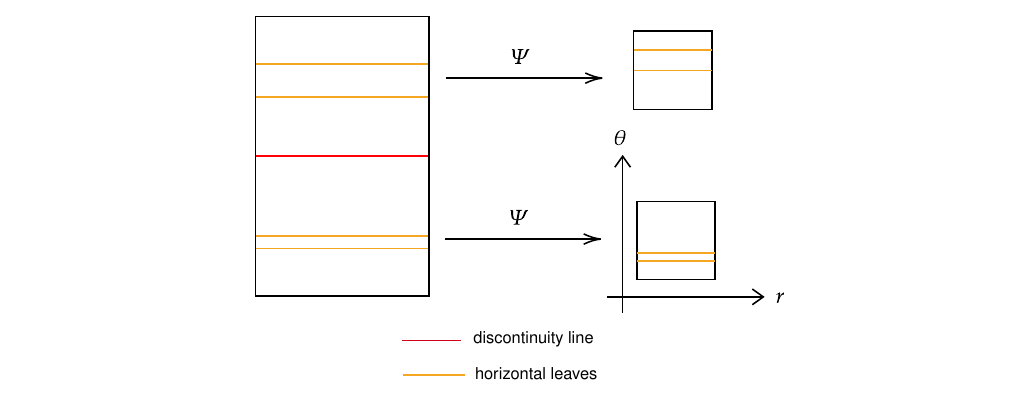}
    \caption{Discontinuity lines.}
    \label{disclines}
\end{figure}
\begin{example}[Pull--back by discontinuous map followed by gluing]\label{ex:pullbackgluing}
Let $\alpha\in (0,\frac{1}{2})$ and $\mathbb{R}^2$ with coordinates $(x,y)$.
Consider the two squares $\square_1:=[-1,1]_x\times [-1,0]_y $ and $\square_2:=[-1,1]_x\times [0,1]_y$ and a map
$$\Psi:[-1,1]_x\times [-1,1]_y \mapsto M $$ which is smooth on $\square_1$ and $\square_2$, in the sense of a smooth map from a differentiable manifold with corners with value into a smooth manifold $M$. The map $\Psi$ is eventually discontinuous on the wall $[-1,1]_x\times\{0\}$ in the sense that for all $x$, $\lim_{y\rightarrow 0^+}\Psi(x,y) \in \mathbb{R}^2$ and $\lim_{y\rightarrow 0^-}\Psi(x,y) \in \mathbb{R}^2 $ 
might differ.
We furthermore assume that $\Psi$ maps the horizontal foliation $\cup_{y} [-1,1]_x\times \{y\}$
to the foliation $\cup_{\theta \in \mathbb{R}} \mathbb{R}_r\times \{\theta\} $. The images $\Psi(\square_1)$ and $\Psi(\square_2)$
are disjoint smooth squares in the target space $\mathbb{R}^2$; see Figure \ref{disclines}.

Assume that on the target space $\mathbb{R}^2$, we have a global distribution $T\in \mathcal{D}^\prime(\mathbb{R}^2)$ which happens to be a derivative $T=\partial_\theta W$ in the sense of distributions, where $W\in \mathcal{C}^{\alpha,\alpha}_{r,\theta}(\mathbb{R}^2)$.   
Our  goal is to give a rigorous meaning to the pull--back $T_1=\Psi^*(T 1_{\Psi(\square_1)} )$
and $T_2:=\Psi^*(T 1_{\Psi(\square_2)} )$ and justify that we glue the two pieces $T_1$ and $T_2$ as the sum $T_1+T_2$.
 
 First observe that $T\in \mathcal{C}^{\alpha,\alpha-1}$ implies that the restrictions 
$(T 1_{\Psi(\square_1)} ) $ and $(T 1_{\Psi(\square_2)} )$ are both well--defined by Lemma~\ref{lem:restrictionBesov} from the appendix. 

 Second, a crucial observation is that since $\Psi$ is smooth on $\square_i$ up to the boundary, using a classical result of Seeley~\cite{Seeley} or the Whitney extension Theorem~\cite[Thm 4.1 p.~10]{Malgrange}, we may extend $\Psi$ smoothly to some neighborhood $\tilde{\square}_i$ of each $\square_i$, each extension is denoted by $\tilde{\Psi}_i $ and satisfies the crucial identity $\tilde{\Psi}_i|_{\square_i}=\Psi|_{\square_i}$,
in such a way that the image of extensions $\Psi(\tilde{\square}_1)$ and $\Psi(\tilde{\square}_2)$ are disjoint.
Since $d\Psi|_{\partial \square_i}$ is never vanishing, we know that up to taking a smaller $\tilde{\square}_i$ we may assume that the extended map $\tilde{\Psi}_i: \tilde{\square}_i \mapsto \tilde{\Psi}( \tilde{\square}_i )$ is a diffeomorphism on its image; see Figure \ref{extpsi}.
\begin{figure}
    \centering
    \includegraphics[width=\linewidth]{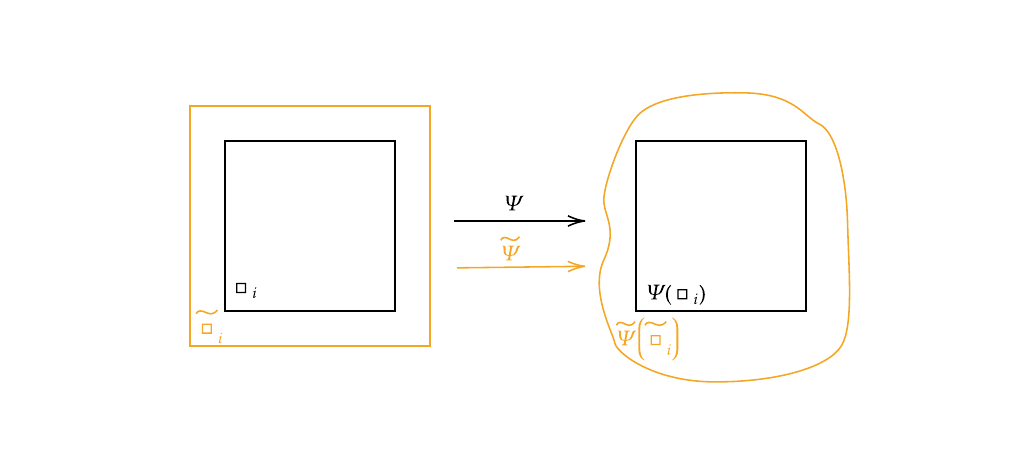}
    \caption{Extending $\Psi$ near $\square_i$ by $\tilde{\Psi}$.}
    \label{extpsi}
\end{figure}
 Third, each pull--back $\tilde{\Psi}_i^*T$ belongs to the space $\mathcal{C}^{\alpha,\alpha-1}_{x,y}$. Using the continuous injection $ \mathcal{C}^{\alpha,\alpha-1}_{x,y}\hookrightarrow \mathcal{C}^{\alpha-1}_{x,y}$ established in Lemma~\ref{lem:aniso_inj_holder}  from the appendix, this implies 
  $\tilde{\Psi}_i^*T$ also belongs to $\mathcal{C}^{\alpha-1}_{x,y}(\mathrm{int}(\tilde{\square}_i))$. Hence by Lemmas~\ref{lem:pairing} and \ref{lem:restrictionBesov}, this implies that $\tilde{\Psi}_i^*T$
 can be restricted to $\square_i$ by multiplication with the indicator function $1_{\square_i}$, hence $T_i:=\left(\tilde{\Psi}_i^*T \right) 1_{\square_i}$ is well--defined in $\mathcal{D}^\prime(\mathrm{int}(\tilde{\square}_i)  )$ with support contained in $\square_i$.
 Another approach to this third step involves \emph{leafwise extensions}. Since $\tilde{\Psi}_i^*T$ is $\alpha$-H\"older in $x$ valued in distributions of the $y$ variable, we may consider the leafwise  restriction $\tilde{\Psi}_i^*T(x,.)\in \mathcal{C}^{\alpha-1}_y$ for every $x\in [-1,1]$; see Figure \ref{leafwise}.

 The leaf is just the interval $ \{x\}\times  [-1,1]_y$.
 Since $\alpha-1>0$, Lemma~\ref{lem:restrictionBesov} allows us to multiply $\tilde{\Psi}_1^*T(x,.)$ (resp $\tilde{\Psi}_2^*T(x,.)$) with $1_{[-1,0]}(y)$ (resp $1_{[0,1]}(y)$) which yields distributions
 $\tilde{\Psi}_1^*T(x,.)1_{[-1,0]}(.)$ (resp $\tilde{\Psi}_2^*T(x,.)1_{[0,1]}(.)$) supported on $[-1,0]$ (resp $[0,1]$) for all $x\in [-1,1]$. By continuity of the product with indicators, we still work with continuous functions in $x$ valued in distributions of $y$.

 Finally, the sum $T_1+T_2$ answers our question.
\end{example}
 \begin{figure}
     \centering
     \includegraphics[width=\linewidth]{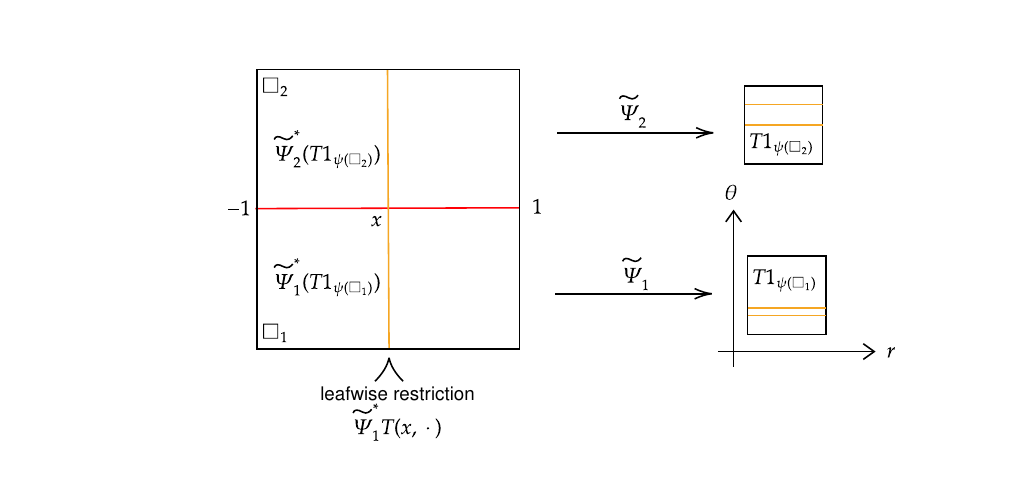}
     \caption{Gluing pull backs and leaf wise restrictions.}
     \label{leafwise}
 \end{figure}

Once we discussed the above key example, we can move on to our extension procedure along unstable curves.

\begin{lemma}[Extension along unstable curves]\label{lem:extcurves}
Let $\alpha=\frac{1}{2}-\varepsilon$.
Let $\tilde{A}d\theta$ be any current in the anisotropic space $\mathcal{C}^{\alpha,\alpha-1}_{I\times J}$ for any $I\times J\subset [\min(f),\max(f)]\times \mathbb{S}^1$ avoiding $\Psi\left(\mathrm{Crit}(f)_1\right)$. 

Then for $\varepsilon$ small enough, the pull--back 
$\Psi^*\left(\tilde{A}d\theta\right)$ extends uniquely on $\mathcal{S}\setminus \mathrm{Crit}(f)_1$.
\end{lemma}

\begin{proof}
In fact we can almost reduce the above proof to our example~\ref{ex:pullbackgluing}.
The reader who just wants to get a flavour of the proof can skip what we wrote below and refer to our example~\ref{ex:pullbackgluing} and look at the pictures.
First observation, the map
\[\Psi:\mathrm{int}(H_i)\mapsto [\min(f),\max(f)]\times \mathbb{S}^1_\theta \]
is such that its differential $d\Psi$ is nondegenerate along $\partial H_i\setminus \mathrm{Crit}(f)_1$ and smooth up to the boundary minus the saddle points $\partial H_i\setminus \mathrm{Crit}(f)_1$. 
This follows from Proposition~\ref{prop:regpseudo}.

Because of the definition of $\Psi$ and the fact that $\tilde{A}\dd\theta$ is $\alpha=\frac{1}{2}-\varepsilon$ regular in the $r$--direction, its pull--back $\Psi^*\left( \tilde{A}d\theta \right)$ 
is continuous in the flow direction, valued into distributions of regularity $\alpha-1$ in the $\theta$ variable. The continuity in the flow direction being precisely captured by the anisotropic spaces.

Therefore, we use Lemma~\ref{lem:restrictionBesov} which shows that the leafwise restriction $\Psi^*\left(\tilde{A}\dd\theta\right)|_{\{f=\mathrm{constant}\}\cap \mathrm{int}\left( H_i\right) }$ has regularity $-\frac{1}{2}-\varepsilon$ along level sets $\{f=\mathrm{constant}\}\cap \mathrm{int}\left( H_i\right) $ which shows there exists a unique extension to $\{f=\mathrm{constant}\}$ where we extend by $0$ outside $\{f=\mathrm{constant}\}\cap H_i$.
\end{proof}

\paragraph{Various notions of scalings.}

To describe extensions of distributions defined on pointed spaces, we measure their singular behaviour with scalings.
This is why we need to recall the various notions of scalings that appear in our problem.
Fix an open set $U\subset \mathcal{S}$ near a saddle point and a system of Morse coordinates   $(x,y)$  on $U$, the scaling is a local diffeomorphism which writes~:
$$ \mathcal{S}^\lambda_{x_0,y_0}: (x,y) \in U\mapsto (\lambda (x-x_0)+x_0,\lambda( y-y_0)+y_0)\in U, \lambda\in (0,1]  .$$
The subtle point in our discussion is that there are two scalings, the one $$\mathcal{S}^\lambda_{x_0,y_0}: (x,y) \in U\mapsto (\lambda (x-x_0)+x_0,\lambda( y-y_0)+y_0)\in U, \lambda\in (0,1]$$ in the nice Morse  coordinates near the saddle point $a$, and the scaling 
$$\mathcal{S}^\lambda_{r_0,\theta_0}:(r,\theta) \mapsto (\lambda (r-r_0)+r_0,\lambda (\theta-\theta_0)+\theta_0), \lambda\in (0,1]$$ which is defined in coordinates $(r,\theta)$ in the cylinder $\mathbf{Cyl}$. 
These scalings do not match because $\Psi$ fails exactly to be a diffeomorphism at $\mathrm{Crit}(f)_1$ and
a point where we need to be careful is that the weighted 
norms were defined using scaling on the cylinder.

\paragraph{Extension of $\tilde{A}d\theta$ near $\Psi\left(\mathrm{Crit}(f)_1\right)$ on the cylinder $\mathbf{Cyl}$.}

\begin{lemma}\label{lem:extcylinder}
    Let $\alpha=\in (0,\frac{1}{2})$ and $s\in (-2,0)$.
Let $\tilde{A}\dd\theta$ be any current in the weighted anisotropic space $\mathcal{C}^{\alpha,\alpha-1,s}\left( \mathbf{Cyl}\setminus \Psi(\mathrm{Crit}(f)_1) \right)$ of the cylinder $[\min(f),\max(f)]_r\times \mathbb{S}^1_\theta$. 
Then $\tilde{A}d\theta$ extends uniquely as distribution in $\mathcal{C}^{-\beta-2}(\mathbf{Cyl})$ for all $\beta $ such that $\beta+\alpha-1>0$.
\end{lemma}
\begin{proof}
To extend the current $\tilde{A}\dd\theta$ near a singular point $a\in \Psi\left( \mathrm{Crit}(f)_1 \right)$, we will use the scaling  
of $\left(\tilde{A}d\theta\right)$ near $a \in \Psi\left( \mathrm{Crit}(f)_1 \right)$. Very similar ideas to what we are doing can be found in~\cite[Chapter 2 p.~43--53]{MeyerWaveletsVibrationsScalings} and also in ~\cite[Thm 4.4 p.~832]{DangIHP}.

Without loss of generality, $a$ is given by $(0,0)$ in coordinates $(r,\theta)$ of the cylinder.
Start from any function $\psi_0$ which equals $1$ near $a=(0,0)$, $\psi_0(r,\theta)=1 $ when $\vert (r,\theta)\vert\leqslant 1$ and $\psi_0(r,\theta)=0$ when $\vert (r,\theta)\vert \geqslant 2$.
Observe that the function $1-\psi_0(2^n.)$ equals $1$ on the complement of a ball of radius $2^{-n+1}$ near $(0,0)$ and vanishes on a ball of radius $2^{-n}$ so that the product $(1-\psi_0(2^n.)) \tilde{A}d\theta$ is well--defined globally in $\mathcal{D}^\prime(\mathbf{Cyl})$ and coincides with $\tilde{A}d\theta $ on the complement of a ball of radius $2^{-n+1}$.
Therefore for every test $1$--form $\varphi\in C^\infty_c(\mathbf{Cyl}\setminus \Psi\left( \mathrm{Crit}(f)_1 \right) )$, 
$\left\langle \tilde{A}d\theta, \varphi \right\rangle= \left\langle (1-\psi_0(2^n.)) \tilde{A}d\theta, \varphi \right\rangle $ for all $n\geqslant N$ for $N$ large enough, since $\mathrm{supp}(\psi_0(2^n.))\cap \mathrm{supp} (\varphi)=\emptyset$ for all $n\geqslant N$.
Then we would like to understand the convergence of $\lim_{n\rightarrow +\infty}(1-\psi_0(2^n.)) \tilde{A}d\theta $ in $\mathcal{D}^\prime(\mathbf{Cyl})$. 
For that purpose, we are going to decompose $1-\psi_0(2^n.)$ as a telescopic series of functions each of them is supported on some annular domain avoiding $(0,0)$ of smaller and smaller dyadic size.
Denote by $K$ the annular domain centered around a singular point $ a\in \Psi\left( \mathrm{Crit}(f)_1 \right) $, $K=\{  (r,\theta); \Vert (r,\theta) \Vert_2\in [\frac{1}{2},4] \}$. Following~\cite[p.~48-49]{MeyerWaveletsVibrationsScalings},
the main idea is to use a spatial version of the Littlewood--Paley 
partition of unity, we write 
$1=(1-\psi_0)+\sum_{n=1}^\infty \psi_n$, where for all $n\geqslant 1$, $\psi_n:=\psi_0(2^{n+1}.)-\psi_0(2^n.)$, $1-\psi_0$ vanishes near $(0,0)$ hence we do not need to control $(1-\psi_0)\tilde{A}d\theta$, $\psi_1$ is supported in $K$ and where each $\psi_n$ is supported in a concentric corona centered at $(0,0)$ of radius $\in [2^{-n-1},2^{-n+2}]$.
\begin{figure}
    \centering
    \includegraphics[width=0.8\linewidth]{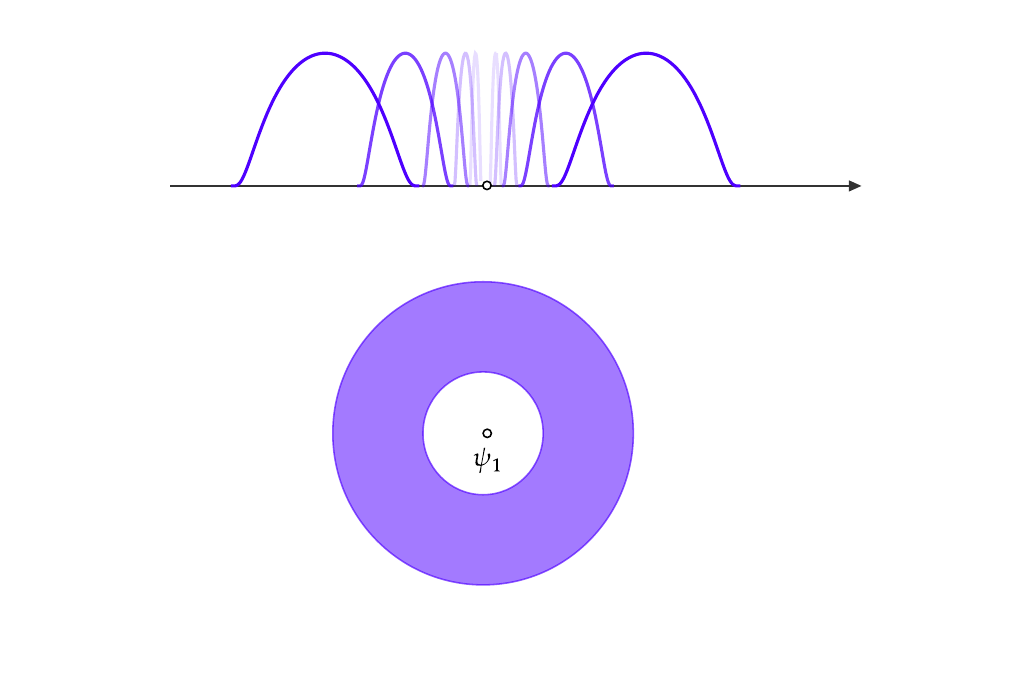}
    \caption{Littlewood--Paley decomposition in space, annular support of $\psi_1$.  }
    \label{rec}
\end{figure}

What is important is that $a\notin \mathrm{supp}(\psi_n)$ for all $n\geqslant 1$. 
Also note that for every test $1$--form $\varphi\in C^\infty_c(\mathbf{Cyl}\setminus \Psi\left( \mathrm{Crit}(f)_1 \right) )$, the sum
$\sum_{n=1}^\infty \psi_n\varphi $ reduces in fact to some finite sum $\sum_{n=1}^N \psi_n\varphi$  since for $N$ large enough, $\mathrm{supp}(\psi_n)\cap \mathrm{supp} (\varphi)=\emptyset$ for all $n\geqslant N$.
So we are left with the task to control the convergence of the series $\sum_{n=1}^\infty \psi_n \tilde{A}d\theta $ in $\mathcal{D}^\prime(\mathbf{Cyl})$.  
For any test $1$-form $\varphi\in \Omega^1(\mathbf{Cyl})$, we get
\begin{align}
 \left\langle \sum_{n=0}^\infty \psi_n \tilde{A}d\theta,\varphi\right\rangle= \sum_{n=0}^\infty \left\langle \mathcal{S}^{2^{-n}*}_{(r_0,\theta_0)}\left(\psi_n \tilde{A}d\theta \right) , \mathcal{S}^{2^{-n}*}_{(0,0)}\varphi \right\rangle   
\end{align}
where the right hand side follows from the change of variables formula.
We drop the $(0,0)$ subscript in the scaling to simplify the notations.

To study the convergence of the series in spaces of distributions, let us study an individual term $\left\langle \mathcal{S}^{2^{-n}*}\left(\psi_n \tilde{A}d\theta \right) , \mathcal{S}^{2^{-n}*}\varphi \right\rangle   $.
By the upper bound~:
\begin{align*}
\vert \left\langle \mathcal{S}^{2^{-n}*}\left(\tilde{A}d\theta \right),\psi_1 \mathcal{S}^{2^{-n}*}\varphi \right\rangle \vert \leqslant 2^{-2n} \Vert  \mathcal{S}^{2^{-n}*}\tilde{A} \Vert_{\mathcal{C}^{\alpha,\alpha-1}}  \Vert\psi_1 \mathcal{S}^{2^{-n*}}\varphi   \Vert_{\mathcal{C}^\beta} \lesssim 2^{-2n} 2^{-ns}\Vert \tilde{A}\Vert_{\mathcal{C}^{\alpha,\alpha-1,s}} \Vert \varphi\Vert_{\mathcal{C}^\beta(B)},
\end{align*}
using the continuous injections $\mathcal{C}^{\alpha,\alpha-1}(K)\hookrightarrow \mathcal{C}^{\alpha-1}(K)$ proved in Proposition~\ref{prop:continuousinjection}, $ \mathcal{C}^\beta\hookrightarrow \left(\mathcal{C}^{\alpha-1}\right)^\prime $ for all $\beta+\alpha-1>0$
which follows from the criterion on Young products of H\"older--Besov distributions, see \ref{lem:young} in the appendix, the definition of the weighted anisotropic norms, the family $\Vert\psi\mathcal{S}^{2^{-n}*}\varphi\Vert_{\mathcal{C}^\beta(K)}, n\geqslant 1$ is bounded by the $\mathcal{C}^\beta$ norm of $\varphi$ on a certain ball $B$ centered at $(0,0)$ and the $2^{-2n}$ in factor comes from $\mathcal{S}^{2^{-n}*}d\theta=2^{-n}d\theta$, $\mathcal{S}^{2^{-n}*}dr=2^{-n}dr$. So the series 
$ \sum_{n=0}^\infty \psi_n \tilde{A}d\theta$ converges in $\mathcal{C}^{\beta}(\mathbf{Cyl})^\prime$
for all $\beta+\alpha-1>0$ and $s>-2$ and we conclude using the embedding $(\mathcal{C}^{\beta})^\prime \hookrightarrow \mathcal{C}^{-\beta-2-\varepsilon}$, $\forall \varepsilon>0$ proved in Lemma~\ref{lem:continjectiondualHolder} in the appendix. 
\end{proof}

\paragraph{Extension of $\Psi^*\left(\tilde{A}d\theta \right)$ to the whole surface $\mathcal{S}$}

\begin{lemma}\label{lem:extsurface}
Under the assumptions of Lemma~\ref{lem:extcurves}, the pull--back   $\Psi^*\left(\tilde{A}d\theta \right)$ defined on the pointed surface $\mathcal{S}\setminus \mathrm{Crit}(f)_1$ extends uniquely as an element in $\mathcal{C}^{-\beta-2}(\mathcal{S})$ for all $\beta$ such that $\beta+\alpha-1>0$.  
\end{lemma}

\begin{proof}
Step 1,
now $\Psi^*\left(\tilde{A}d\theta\right)$ is well--defined on $\mathcal{S}\setminus \mathrm{Crit}(f)_1$ thanks to Lemma~\ref{lem:extcurves}. 

Step 2,
to extend near a critical point, we will use the scaling  
of $\Psi^*\left(\tilde{A}d\theta\right)$ at a critical point. 

Step 3, 
to define the pull--back $\Psi^*\left(\tilde{A}d\theta \right) $ even at the critical points, the key is to control the pull--back of the convergent series appearing in the proof of Lemma~\ref{lem:extcylinder}. This suggests to study $ \sum_{n=1}^\infty\Psi^*\left( \psi_n \tilde{A}d\theta\right)$. 
But there is a problem since $\Psi$ is only piecewise smooth. This is why we need to use indicator functions to localize in regions where $\Psi$ is a well--defined smooth diffeomorphism exactly in the spirit of example~\ref{ex:pullbackgluing} and Lemma~\ref{lem:extcurves}. We consider the indicator function $1_{\Psi\left(H_i\right)}$ of image under $\Psi$ of one of the hexagon $H_i\subset \mathcal{S}$ where $\Psi^{-1}:\Psi(H_i)\mapsto H_i$ is well--defined and smooth. 
We make the crucial observation that the domain $\Psi\left(H_i\right)\subset \mathbf{Cyl}$ is stable by the scaling flow: $ \mathcal{S}^{2^{-n}}\left( \Psi\left(H_i\right)\right)\subset \Psi\left(H_i\right) $ for all $n\geqslant 0$, since in local coordinates each domain $\Psi\left(H_i\right)$ is defined in terms of \emph{linear inequalities} involving $r,\theta$. In fact,
in the local Morse chart in a hyperbolic box, up to shifting coordinates by constants and up to some smooth change of variables, we can always assume that
$f=x^2-y^2$, $\theta=xy$ and we work on the quadrant $H_i=\{x\geq 0, y\geq 0\}$.
Exactly for the same reason as in example~\ref{ex:pullbackgluing} and Lemma~\ref{lem:extcurves}~\footnote{We cannot pull--back by a partially smooth map some functional object supported exactly on the domain of the pull--back, we need some slight enlargement}, we need another extra ingredient, namely the existence of certain smooth extensions of our piecewise smooth diffeomorphism $\Psi$ in some neighborhood of every hexagon $H_i,i=1,\dots,4g$.
 As in example~\ref{ex:pullbackgluing}, we need to slightly extend the diffeomorphism 
  $\Psi:H_i\mapsto \Psi(H_i)$ to some slightly larger neighborhood $E_i$ of $H_i$ in such a way that the extended map denoted by $\tilde{\Psi}_i:E_i\mapsto \tilde{\Psi}_i(E_i)$ is still a diffeomorphism.
We used the fundamental fact that both $\Psi^*$ and $\tilde{\Psi}_i^*$ both coincide on distributions supported by $\Psi(H_i)$.

 So we will study instead the series $\sum_{i=1}^{4g}\sum_n\tilde{\Psi}_i^*\left(1_{\Psi(H_i)} \psi_n \tilde{A}d\theta\right)$
where we localized, decomposed in dyadic annular domains and then resum on $i=1,\dots, 4g$.

We rewrite the general term of the series using the following series of identities, for any test $1$--form $\varphi\in \Omega^1(\mathcal{S})$~:
\begin{align*}
&\left\langle \tilde{\Psi}_i^*\left(1_{\Psi(H_i)} \psi_n \tilde{A}d\theta\right), \varphi \right\rangle=
\left\langle  \tilde{\Psi}_i^*\mathcal{S}_{r_0,\theta_0}^{2^{n}*}\left(\psi\mathcal{S}_{r_0,\theta_0}^{2^{-n}*}\left(1_{\Psi(H_i)}\tilde{A} \dd \theta \right)\right), \varphi \right\rangle\\
&=\left\langle  \tilde{\Psi}_i^*\mathcal{S}_{r_0,\theta_0}^{2^{n}*}\left(\psi 1_{\Psi(H_i)}\mathcal{S}_{r_0,\theta_0}^{2^{-n}*}\left(\tilde{A} \dd \theta \right)\right), \varphi \right\rangle=
\left\langle \left(\psi\mathcal{S}_{r_0,\theta_0}^{2^{-n}*}\left(1_{\Psi(H_i)}\tilde{A} \dd \theta \right)\right),  \mathcal{S}_{r_0,\theta_0}^{2^{-n}*}\left(\tilde{\Psi}_i^{-1*} \varphi \right) \right\rangle
\end{align*}
where in the second line we use the scaling flow stability of $\Psi(H_i)$.

 The singular point reads $(0,0)$. Then we have the relation
\begin{align*}
x= \sqrt{\frac{r+(r^2+4\theta^2)^{\frac{1}{2}}}{2}}, y= \sqrt{ \frac{ (r^2+4\theta^2)^{\frac{1}{2}}-r}{2}}. 
\end{align*}

This implies that $\Psi^{-1}$ is locally given in coordinates
by
\begin{align*}
\Psi^{-1}:(r,\theta)\longmapsto   \left( \sqrt{\frac{r+(r^2+4\theta^2)^{\frac{1}{2}}}{2}}, \sqrt{ \frac{ (r^2+4\theta^2)^{\frac{1}{2}}-r}{2}} \right).
\end{align*}
Observe that the above formulas extend to arbitrary values 
of $(r,\theta)\in [-1,1]^2$ where 
we recognize immediately the only non smooth point at $(r,\theta)=(0,0)$.
We need to analyze the growth of the $1$-form $\mathcal{S}_{0,0}^{2^{-n}*}\left(\Psi^{-1*}\left(\varphi\right) \right)$ in $C^\beta(K)$ (recall $K$ is the annular domain $K=\{  (r,\theta); \Vert (r,\theta) \Vert_2\in [\frac{1}{2},4] \}$) for large $n$ for $\beta+\alpha-1 > 0 $ and we proceed exactly as above~:
\begin{align*}
  &\mathcal{S}_{0,0}^{2^{-n}*}\left(\Psi^{-1*}\left(\varphi\right) \right)=\varphi_1  \left( 2^{-n}\sqrt{\frac{r+(r^2+4\theta^2)^{\frac{1}{2}}}{2}}, 2^{-n}\sqrt{ \frac{ (r^2+4\theta^2)^{\frac{1}{2}}-r}{2}} \right) 2^{-n}\dd x\\
  &+\varphi_2  \left( 2^{-n}\sqrt{\frac{r+(r^2+4\theta^2)^{\frac{1}{2}}}{2}}, 2^{-n}\sqrt{ \frac{ (r^2+4\theta^2)^{\frac{1}{2}}-r}{2}} \right) 2^{-n}\dd y 
\end{align*}
where $x,y$ are viewed as implicit functions of $(r,\theta)$. The key idea is that $(r,\theta)$ belong to the annular domain $K$ and therefore all derivatives of $\sqrt{\frac{r+(r^2+4\theta^2)^{\frac{1}{2}}}{2}}, \sqrt{ \frac{ (r^2+4\theta^2)^{\frac{1}{2}}-r}{2}} $ in $(r,\theta)$ are bounded uniformly on $K$.

So we have the decay estimate 
$
 \Vert \mathcal{S}_{0,0}^{2^{-n}*}\left(\Psi^{-1*}\left(\varphi\right) \right) \Vert_{C^\beta(K)}\leq C2^{-n}\Vert \varphi\Vert_{C^\beta(B)} $ where as above $\beta$ is chosen in such a way that $\beta+\alpha-1>0$. 
Therefore
\begin{align*}
\left\vert    \sum_{i=1}^{4g} \left\langle \tilde{\Psi}_i^* \left(\sum_{n=1}^\infty \psi_n 1_{\Psi\left(H_i\right)} \tilde{A} \dd \theta \right),  \varphi \right\rangle \right\vert \leq \sum_{n=1}^\infty  C2^{-2n}2^{-ns} \Vert \tilde{A}\Vert_{\mathcal{C}^{\alpha,\alpha-1,s}}  \Vert \varphi\Vert_{\mathcal{C}^\beta(B)}   
\end{align*}
where we used again the fact that $\tilde{A}\in \mathcal{C}^{\alpha,\alpha-1,s}$.
So for all $s\in (-2,0)$, the pull--back $\Psi^* \left( \tilde{A}d\theta\right)$ has a unique extension in $\mathcal{C}^{\beta}(\mathcal{S})^\prime$ for all $\beta+\alpha-1>0$.
Finally we conclude again by the continuous embedding of Lemma~\ref{lem:continjectiondualHolder}.
\end{proof}

Now, the formal definition given in~\ref{YMFormula} gives perfect sense as a random distributional $1-$form on $\mathcal{S}$. The next theorem shows that the parallel transport generated by this $1-$form verifies the Driver--Sengupta formula~(\ref{eq:DSformula}).

\begin{thm}[Driver--Sengupta Formula]
    The random $1-$form $A$ defined in~\ref{YMFormula} induces via stochastic differential equations, on any graph on $\Sigma$ whose edges are either flow lines or level sets, holonomies verifying the Driver--Sengupta formula. 
\end{thm}
\begin{proof}
    Consider a graph $\Lambda$ on $\Sigma$ whose edges are either flow lines of the Morse function, or level sets. Such a graph can be seen for instance on Figure \ref{MLatt}. The idea is to show that the form $A$ defined in~\ref{YMFormula} verifies 
    \[\big(\mathrm{Hol}(A,c)\big)_{c\in\mathrm{Loop}_{\min f}(\Lambda)}\]
    verifies the Driver--Sengupta formula \ref{eq:DSformula}. 
    \par Note first that there is no restriction in supposing that faces are entirely contained within two consecutive unstable curves. In fact, one only needs to add the unstable curves as edges. 
    \par Then, choose a generator of the free group $\mathrm{Loop}_{\min f}(\Lambda)$ as follows :
    \begin{itemize}
        \item A set of loops $(f_i)_{i\in \mathrm{Faces}(\Lambda)}$  based at $\min(f)$ and entangling the faces.
        \item A set of loops $(l_k)_{1\leq k\leq 2g}$ that generate $\pi_1(\Sigma)$ that we can choose to be the stable curves $W^s(a), a\in \textbf{Crit}(f)_1$~\footnote{The fact that either the stable or unstable curves generate the $\pi_1(\mathcal{S})$ can be proved for instance with the Seifert--Van Kampen Theorem, see \cite{BCDRT} for more on this topic}.
    \end{itemize}
    In the free boundary context, the Driver--Sengupta formula says that the joint law of $\big(\mathrm{Hol}(f_i),\mathrm{Hol}(l_k))_{k,i}$ is such that these random variables are independent, and 
    \[\mathrm{Hol}(f_i) \sim p_{\sigma(f_i)}(g)\dd g \ \text { and } \ \mathrm{Hol}(l_k) \sim \dd g.\]
    This completely characterizes their law, and since these loops generate   $\mathrm{Loop}_{\min f}(\Lambda)$, it completely characterizes the law of the stochastic process $\big(\mathrm{Hol}(c)\big)_{c\in\mathrm{Loop}_{\min f}(\Lambda)}$ under the Driver--Sengupta--Lévy measure. 
    \par Therefore, the only thing we need to show is that 
     \[\mathrm{Hol}(A,f_i) \sim p_{\sigma(f_i)}(g)\dd g \ \text { and } \ \mathrm{Hol}(A,l_k) \sim \dd g,\]
     and that all these random variables are independent. This can be done in each region between two consecutive unstable lines in the same way as it was done for the cylinder case. The second Uniform part comes from the adjunction of the unstable curves. 

    We would like to relate our random connection with the one appearing in~\cite{BCDRT}.
   For any smooth oriented curve $\gamma$ everywhere transverse to the flow, we need to show that
    \begin{equation}\label{eq:integralAovergamma}
    \mathbf{W}_\gamma:= \int_\gamma \Psi^*\left(\tilde{A}d\theta \right)   
    \end{equation}
    is well--defined as a $\mathfrak{g}$-valued normal random variable with variance the area of the rectangle $\square:=\{\varphi_f^{-t}(x); x\in \gamma, t\geqslant 0\}$ defined in terms of the curve $\gamma$.
    A key consequence from its definition is that $\Psi$ preserves the orientation.
    Beware that the curve $\gamma$ might intersect some unstable curve transversally hence its image $\Psi(\gamma)\subset \mathbf{Cyl}$ under $\Psi$ is only piecewise smooth~: $\Psi(\gamma)=\cup_{i=1}^k \gamma_i$ as a disjoint union of smooth curves. By transversality with the flow and since this transversality is preserved under mapping by $\Psi$, $\gamma_i\subset \mathbf{Cyl}$ is transverse with $\partial_r$ and  we can describe each curve $\gamma_i$ as some graph: $\theta\in I_i\subset [0,2\pi]\mapsto (\theta,r_i(\theta))$ where $r_i$ is a smooth function. A key consequence from its definition is that $\Psi$ preserves the orientation this is why we get the same orientation for all curves $\gamma_i$, they run counterclockwise if the initial curve $\gamma$ is oriented counterclockwise. Another consequence of transversality is that the open intervals
    $\mathrm{int}(I_i)_{i=1,\dots,k}$ are two by two disjoint.
     We decompose $\Psi_*[\gamma]=\sum_{i=1}^k [\gamma_i]$ where the r.h.s. is a finite sum of currents of integration of degree $1$ on the curves $\gamma_i, i=1,\dots,k$.
    The above integral in equation \ref{eq:integralAovergamma} can only be defined probabilistically as follows~: 
\begin{align*}
\int_\gamma \Psi^*\left(\tilde{A}\dd\theta \right)=\left\langle \Psi_*[\gamma], \tilde{A}\dd\theta \right\rangle_{\mathbf{Cyl}}=\sum_{i=1}^k\xi( 1_{\square_i} )   
\end{align*}
where each domain $\square_i$ is defined as follows 
 $$\square_i:=\{ (r,\theta);0\leqslant r\leqslant r_i(\theta), \theta\in I_i  \}\subset \mathbf{Cyl} $$
 and each term $\xi( 1_{\square_i} ) $ is well--defined since each $1_{\square_i} \in L^2(\mathbf{Cyl}, \Psi_*\sigma)$ belongs to the Cameron Martin space of $\xi$.
Then from the above identities
\begin{align*}
\mathbb{E}\left( \int_\gamma \Psi^*\left(\tilde{A}\dd\theta \right)\otimes \int_\gamma \Psi^*\left(\tilde{A}\dd\theta \right) \right)=\mathbb{E}\left((\sum_{i=1}^k \xi(1_{\square_i}))\otimes(\sum_{i=1}^k \xi(1_{\square_i})) \right)=\sum_{i=1}^k \sigma(\square_i) \mathrm{Id}^{\mathfrak{g}\otimes \mathfrak{g}}  
\end{align*}
where we used the crucial fact that all mutual intersections $\square_i\cap \square_j, i\neq j$ have measure $0$. Now observe that $\sum_{i=1}^k \sigma(\square_i)$ is the area on $\mathcal{S}$ of the set $\{x\prec \gamma, x\in \mathcal{S}  \}:=\overline{\{\varphi_f^{-t}(x); x\in \gamma\}}$ which implies we recover the result
that $\mathbf{W}_\gamma$ has normal law with variance equals to the area $\sigma\left(\{x\prec \gamma  \} \right)$ which is one of the main result of \cite{BCDRT} on the Gaussian part of the free boundary Yang--Mills measure.
\end{proof}

We use the notion of strong solutions for Stratonovich SDE, when we reparametrize our driving Brownian motion by some $C^{1-}$ function, this notion of strong solution still exists and the area function is only $C^{1-}$ so there is no problem here. We refer the reader to a separate note for the detailed proof~\cite{NotesRDELiegroups} or the appendix of \cite{BCDRT}.

\section{Morse lattice  from flow lines and level sets} \label{s:Morselattice}
We introduce a sequence of dyadic decompositions of the surface into rectangles whose sides are either flow lines of $ \nabla f $ or level sets of $ f $. This construction is simplified by the existence of global pseudo-coordinates $(r,\theta)$ on $ \Sigma $. 
\par 
Let $ c_1 < \dots < c_{2g+2} $ denote the $ 2g+2 $ critical levels of the Morse function $ f $. 
Each interval $ [c_i, c_{i+1}] $ is subdivided into $ 2^n $ levels of the form 
\[
[c_i + \tfrac{k}{2^n}(c_{i+1} - c_i)], \quad k \in \{0, \dots, 2^n\}.
\]
Hence, we obtain $(2g+1)2^n$ level sets of $ f $ given by 
\[
f^{-1}\!\Big(c_i + \tfrac{k}{2^n}(c_{i+1} - c_i)\Big), \quad i \in \{1, \dots, 2g+1\},\; k \in \{0, \dots, 2^n\}.
\]
These levels correspond, under the map $ \Psi $, to a level decomposition of the cylinder 
\[
[\min(f), \max(f)] \times \mathbb{S}^1.
\]

We then blow up the surface $ \Sigma $ at the minimum $ a_1 $, so that the blown-up space becomes a circle parametrized by an angular coordinate $\theta \in [0,2\pi]$. Let 
\[
0 < \theta_1 < \dots < \theta_{4g} < 2\pi
\]
be the $4g$ intersection points of the stable curves of the flow $(\Phi^t)_{t \in \mathbb{R}}$ with this blow-up circle. These define the intervals 
\[
I_1 = [\theta_1, \theta_2],\; I_2 = [\theta_2, \theta_3],\; \dots,\; I_{4g} = [\theta_{4g}, \theta_1].
\]

Next, for each $ i = 1, \dots, 4g $, we introduce $ 2^n $ directions
\[
\theta_i + \tfrac{k}{2^n}(\theta_{i+1} - \theta_i), \quad k \in \{0, \dots, 2^n - 1\},
\]
thus decomposing every interval $ I_i $ into $ 2^n $ dyadic subintervals. We then obtain $2^n \cdot 4g$ level sets of the angular variable $\theta$, which correspond to the same number of integral curves of the flow $\Phi^t$; these intersect the $(2g+1)2^n$ level sets of the Morse function $f$. For each $n$, the resulting structure defines a graph $\Lambda^n$ on $\Sigma$, compatible with the Morse function and referred to as the grid of resolution $n$.

Under the map $\Psi$, this grid is canonically identified with the decomposition
\[
\Psi(H_i) = [\min(f), \max(f)] \times I_i.
\]
Its cells are bounded either by flow lines of $\Phi^t$ or by smooth arcs orthogonal to the flow and contained within level sets of $f$. Our lattice gauge model will be defined on this Morse grid. An example is shown in Figure~\ref{fig:MorseGrid}.

\begin{figure}[t]
    \centering
    \includegraphics[width=1\linewidth, trim={-1cm 0 0 0}]{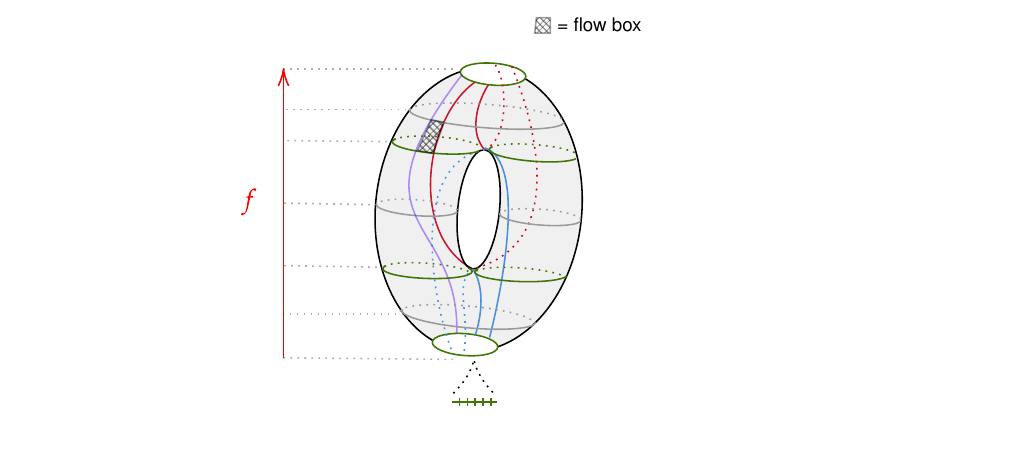}
    \caption{Morse lattice.}
    \label{fig:MorseGrid}
\end{figure}

Consider an integer $N$, and the graph $\Lambda_N$. For simplicity, we assume that the critical points have integer heights, so that in the discretizations $\Lambda^N$ of $\Sigma$, the variable $r$ takes values in 
\[
R^N \coloneq \Big\{ c + \tfrac{k}{2^N} \,;\, 0 \le k \le 2^N - 1,\; c \in \{0, \dots, 2g+1\} \Big\},
\]
and the angular variable $\theta$ in
\[
\Theta^N \coloneq \Big\{ c + \tfrac{k}{2^N} \,;\, 0 \le k \le 2^N - 1,\; c \in \{0, \dots, 2g-1\} \Big\}.
\]
We denote by 
\[
R^N_- \coloneq R^N \setminus \{2g + 2\}, \qquad \Theta^N_- \coloneq \Theta^N \setminus \{2g\}
\]
the sets of all but the last levels and flow lines, respectively.  

Vertices are labeled by their coordinates $(r, \theta)$. The horizontal edge between two vertices $(r, \theta_1)$ and $(r, \theta_2)$ is denoted $\theta_1 \xrightarrow{r} \theta_2$, while a vertical edge is written $r_1 \xrightarrow{\theta} r_2$. Moreover, for $r \in R^N_-$ and $\theta \in \Theta^N_-$, we set
\[
r^{(N)}_+ \coloneq r + 2^{-N}, \qquad \theta^{(N)}_+ \coloneq \theta + 2^{-N}.
\]
When the resolution $N$ is clear from context, we simply write $r_+$ and $\theta_+$ to lighten notation.

\par The important fact is that for any compact that does not meet the saddle points, we have 
that the area of a dyadic face $\simeq 2^{-2N}$ (of course the constant becomes worst when we approach the critical points). 
\begin{lemma}[Areas of small squares and dyadic corona decomposition]\label{lem:dyadiccorona}
We denote by $F^N$ the faces of $\Lambda^N$. 
Fix an arbitrary $\varepsilon>0$, then there exists $c_\varepsilon<C_\varepsilon$ such that for every $N$, for any face $\square$ of the grid $\Lambda^N$ that does meet $\cup_{a\in \mathrm{Crit}(f)_1} B(a,\varepsilon)$, we have
\begin{equation}
c_\varepsilon 2^{-2N}\leq \int_\square \sigma \leq C_\varepsilon 2^{-2N}.     
\end{equation}

Moreover, for all $N$, for all squares that are at approximate distance
$2^{-j}$ (measured in the coordinates $(r,\theta)$) of some given singular point in $\Psi\left(\mathrm{Crit}(f)_1 \right)$ $j=1,\dots,N$,  
\begin{equation}
c2^{j}2^{-2N} \leqslant
 \int_\square \sigma\leqslant C 2^j 2^{-2N}.   
\end{equation}
where the constants $c,C$ do not depend on $j,N$.
\end{lemma}
Beware that it is important in the previous Lemma that $\varepsilon$ does not depend on $N$.
\begin{proof}
This is an immediate consequence of the fact that $\Psi_*\sigma$ fails to be smooth
at saddle points. Let us estimate the asymptotic of the worst dyadic faces of the decomposition
which are close to the saddle points~:
\begin{align*}
\int_{\square}\sigma\sim& \int_{k2^{-N}}^{(k+1)2^{-N}}\int_{\ell 2^{-N}}^{(\ell+1)2^{-N}} (f^2+4\theta^2)^{-\frac{1}{2}}  \dd f\dd \theta\\
=&\int_{k }^{(k+1) }\int_{\ell}^{(\ell+1)} 2^N(f^2+4\theta^2)^{-\frac{1}{2}}  2^{-2N} \dd f\dd\theta =\mathcal{O}(2^{-N}).      
\end{align*}
\end{proof}
We introduce a discrete corona decomposition that will be crucial to define \emph{discrete versions of the weighted norms}. 
We should prove a scaled estimate of the following form.
We would like to control the area of those dyadic squares that approach very close to the saddle points where $\Psi_*\sigma$ is singular.
Consider a small dyadic square $\square(r,r+2^{-N},\theta,\theta+2^{-N})$ at a fixed distance from the saddle points. Assume w.l.o.g. that $(0,0) $ is our saddle point.
For instance all dyadic squares at distance between $ \frac{1}{2}$ and $1$ from the saddle point write
$\square(k2^{-N},(k+1)2^{-N},\ell 2^{-N},(\ell+1)2^{-N}) $ where $ 2^{N-1}\leqslant k,\ell\leqslant 2^{N}$, so $r=k2^{-N}$, $\theta=\ell 2^{-N}$.

So at every scale $j=1,\dots, N-1$, 
$\square(k2^{-N},(k+1)2^{-N},\ell 2^{-N},(\ell+1)2^{-N}) $ where $ 2^{N-j-1}\leqslant k,\ell\leqslant 2^{N-j}$. So all dyadic squares can be decomposed as a union
\begin{align*}
\bigcup_{j=1,\dots,N-1} \left( \bigcup_{2^{N-j-1}\leqslant k,\ell\leqslant 2^{N-j} }  \square(k2^{-N},(k+1)2^{-N},\ell 2^{-N},(\ell+1)2^{-N}) \right) 
\end{align*}
where for each scale $j$, a square $\square(k2^{-N},(k+1)2^{-N},\ell 2^{-N},(\ell+1)2^{-N})$ will be contained in a region (a shell) at distance 
$ \simeq 2^{-j} $, there are approximately $(2^{2(N-j)})$ such squares.

For such a square, we have an estimate of the form
\begin{align*}
\int_{\square(k2^{-N},(k+1)2^{-N},\ell 2^{-N},(\ell+1)2^{-N}) }\sigma\simeq 2^{-2N}2^{j}=2^{j-2N}.   
\end{align*}

\begin{lemma}\label{areaestim}
    There exists $C,\beta,\gamma>0$, such that for all $N$,
    \[\sup_{\theta \in \Theta^N_-}\sigma(0,2g+2,\theta,\theta_+) \leq C2^{-N\beta}\ \text{ and } \sup_{\substack{r\in R^N_- \\ \theta \in \Theta^N_-}}\sigma(r,r^+,\theta,\theta^+)\leq C2^{-N}\]
\end{lemma}
\begin{proof}
The second estimate was already proved.
The challenge is to give uniform bounds for area of very thin strips of dyadic thickness. If the angle $\theta$ is away from the singular angles then the bound is obvious. The difficult case is when strips intersect saddle points; see Figure \ref{strips}.
\begin{figure}
    \centering
    \includegraphics[width=1.2\linewidth]{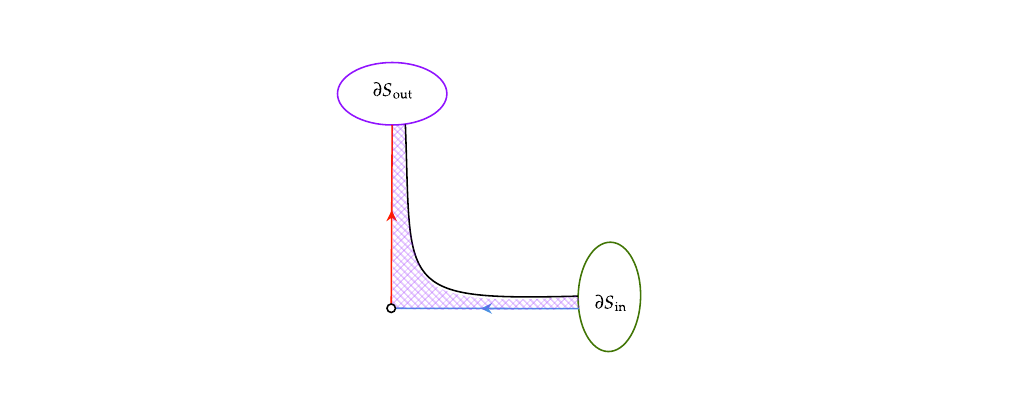}
    \caption{Strips near saddle points.}
    \label{strips}
\end{figure}

Because of the specific structure of Morse gradient flows, we only stay some \emph{finite time} outside hyperbolic boxed and everything reduces to bounding areas
\begin{align*}
&\int_0^1 \int_{k2^{-N}}^{(k+1)2^{-N}} \frac{\tilde{\sigma}(r,u)}{\sqrt{r^2+4u^2}}drdu  
=\int_0^1 \int_{k}^{(k+1)} \frac{\tilde{\sigma}(r,2^{-N}u)}{\sqrt{r^2+2^{-2N}4u^2}} 2^{-N} drdu\\
&= 2^{-N} \int_0^{2^N} \int_{k}^{(k+1)} \frac{\tilde{\sigma}( 2^{-N} r,2^{-N}u)}{\sqrt{r^2+4u^2}}  drdu\leqslant 2^{-N} \int_0^{2^N}  \frac{\tilde{\sigma}( 2^{-N} r,2^{-N}u)}{r}  dr=\mathcal{O}(N2^{-N}).
\end{align*}

So this proves the claim for all exponents $\beta<1$.
\end{proof}

\newcommand{\F}{\mathcal{F}}

\section{ A local limit theorem}\label{TwoCLT}

Let $G$ be a compact unitary Lie group with Lie algebra $\g$, that we equip with a bi invariant dot product $\langle\cdot,\cdot \rangle$. In this section, we would like to show two results of independent interest, that we will use in the sequel. The first one aims to study the convergence of the densities in the central limit theorem, and the second one can be seen as some type of a non-commutative Donsker theorem in which the random variables are not independent. Before we proceed, let us mention that Theorem~\ref{ConvSegalAmp} is a  direct corollary of Theorem~\ref{convck}.
\par Let us give a brief description of the result of this section. Let $(\mu_t)_{t> 0}$ be a family of probability measures such that for small $t$, $\mu_t$ is concentrated around the identity, in a way that will be made precise later. Suppose that the density $\rho_t$ of $\mu_t$ exists and enjoys some properties that we will discuss later. Then, we would like roughly to study the following convergence in $C^\infty(G)$ 
\[\underbrace{\rho_{\frac{1}{n}}\star \cdots \star\rho_{\frac{1}{n}}}_{n \text{ times}}\xrightarrow{n\to\infty} p_1,\]
where $p_1$ is the heat kernel at time $1$. The general idea is that we already know the convergence in $\mathcal{D}'(G)$, and therefore the only additional argument we need is the boundedness in $C^{\infty}(G)$. Let us now be more precise on the assumptions. 
\begin{definition}\label{def:CLTplus}
    Let $s>0$. A family of probability measures $(\mu_t)_{t>0}$ is said to have the property $C(s)$ if 
    \begin{enumerate}
    \item for each $t>0$, $\mu_t$ is ad-invariant and invariant by inversion,
    \item for all $v,w\in \g$, $\int_G \langle \log(g),v\rangle_{\mathfrak{g}} \langle \log(g),w\rangle_{\mathfrak{g}}  \mu_t(\dd g)=t\left\langle v,w \right\rangle_{\mathfrak{g}}  +o(t) $,
    \item $\int_G \vert \log g\vert^{3}\mu_t(\dd g)=O(t^{\frac{3}{2}}) $,
    \item For all $t>0$, $\mu_t$ has the following Fourier decay,
    \begin{equation}
\exists C>0; \forall \lambda\in \widehat{G},\forall t>0, \vert \widehat{\mu_t}(\lambda) \vert \leq \dim(V_\rho) C\left(1+\sqrt{tc_2(\rho)} \right)^{-s}
\end{equation}
where $c_2(\lambda)$ is the second Casimir number of the representation $\lambda$.
\end{enumerate}
\end{definition}

We would like to show the following theorem. This theorem and some of its generalizations are studied in a seperate note by the authors, Thibaut Lemoine, and Martin Vogel~\cite{DLNV}.
\begin{thm}\label{convck}
    Let  for each $n\in \N$,  $s^n_1,\dots, s^n_n \in \R_+$ such that $s^n_1+\cdots+ s^n_n=1$, and such that there exists $a>0$, and $A>1$ for which 
        \[\liminf_{n\to\infty} \frac{\left \vert T_n(a,A)\right \vert }{n} >0, \text{ with }\ T_n(a,A)\coloneq \left\{i=1,\cdots,n; \frac{a^2}{n}\leq s^n_i\leq \frac{A^2}{n}\right\}.\]
Then, if $(\mu_t\coloneq \rho_t\dd g)_{t>0}$ is a family of probability measures verifying the condition $C(s)$ for some $s>0$,  
\[\rho_{s^n_1}\star \cdots\star \rho_{s^n_n}\xrightarrow{n\to\infty} p_1 \text{ in } C^{\infty}(G),\]
where $p_1$ is the heat kernel. 
\end{thm}
\begin{proof}
    The central limit theorem tells us that 
    $\rho_{s^n_1}\star \cdots\star \rho_{s^n_n}\xrightarrow{n\to\infty} p_1 \text{ in } \mathcal{D}'(G).$
    The Proposition~\ref{boundedinCk}, whose proof is the subject of the remaining of the section, shows that for all $k\in \N$, the sequence  
    $(\rho_{s^n_1}\star \cdots\star \rho_{s^n_n})_{n\geq 1}$ is bounded in $C^k$. The result follows. 
\end{proof}

\begin{proposition}\label{boundedinCk}
    Using the same notations as in the previous theorem, for all $k\in \N$, the sequence  
    $(\rho_{s^n_1}\star \cdots\star \rho_{s^n_n})_{n\geq 1}$ is bounded in $C^k$.
\end{proposition}
The proof of the above proposition needs many intermediate steps. Let us start with some preparatory work. 
\begin{lemma}\label{Fourier_expansion}
    Using the same notations, we have for all $k$
    \[\|\rho_{s^n_1}\star \cdots\star \rho_{s^n_n}\|_{C^k}\leq \sum_{\lambda \in \widehat{G}} (1+ \vert c_2(\lambda)\vert) ^{\frac{k}{2}} \prod_{i=1}^n\left\vert\frac{\widehat{\rho}_{s^n_i}(\lambda)}{d_{\lambda}}\right\vert .\]
\end{lemma}
\begin{proof}
    Using Fourier analysis on $G$, we can write
\[\rho_{s^n_1}\star \cdots\star \rho_{s^n_n}= \sum_{\lambda \in \widehat{G}}\left(\prod_{i=1}^n\frac{\widehat{\rho}_{s^n_i}(\lambda)}{d_{\lambda}}\right)\chi_\lambda,
\text{ where } 
\widehat{\rho}_t(\lambda)=\int_{G}\chi_\lambda(g) \mu_t(\dd g).\]
Note that 
\[
    \|\rho_{s^n_1}\star \cdots\star \rho_{s^n_n}\|_{C^k}\coloneq \left\|(1+\Delta)^{\frac{k}{2}} \rho_{s^n_1}\star \cdots\star \rho_{s^n_n}\right\|_{\infty},
\]
and that 
\[(1+\Delta)^{\frac{k}{2}} \rho_{s^n_1}\star \cdots\star \rho_{s^n_n}=\sum_{\lambda \in \widehat{G}}(1+c_2(\lambda))^{\frac{k}{2}}\left(\prod_{i=1}^n\frac{\widehat{\rho}_{s^n_i}(\lambda)}{d_{\lambda}}\right)\chi_\lambda,\]
which gives
\[\|\rho_{s^n_1}\star \cdots\star \rho_{s^n_n}\|_{C^k}\leq \sum_{\lambda \in \widehat{G}}(1+ \vert c_2(\lambda)\vert) ^{\frac{k}{2}}\left\vert \prod_{i=1}^n\frac{\widehat{\rho}_{s^n_i}(\lambda)}{d_{\lambda}}\right\vert \|\chi_\lambda\|_{\infty} .\]
\end{proof}
Next, we will use the previous lemma to bound the sum using ideas from micro local analysis.  There will be three separate regimes in each of which we will bound the terms differently. 

\paragraph{The small $\sqrt{c_2(\lambda)/n}$ regime.}
\begin{lemma}\label{small}
There exists $M,\delta_M>0$ such that for all
pairs $(t,\lambda)\in \R_+\times \widehat{G}$ such that
$\sqrt{t c_2(\lambda) } \leq M $, we have
$
\vert \widehat{\mu_t}(\lambda)\vert \leq d_\lambda e^{-\delta_M c_2(\lambda)t }  
$.
\end{lemma}
\begin{proof}
First Taylor expand the character $\chi_\lambda$ near the identity using the pseudo-coordinates defined by 
$g=\exp\sum_{i}x^i(g)\xi_i$,
where $(\xi_i)_i$ is a basis of $\g$.
This reads~:
 \begin{align*}
 \chi_\lambda(x)=& \chi_\lambda(0) +x^i\partial_{x^i}\chi_\lambda(0) + \frac{ x^ix^j}{2}\partial^2_{x^ix^j}\chi_\lambda(0)+\mathcal{O}(\vert x\vert^3 \Vert \chi_\lambda\Vert_{C^3} ).
\end{align*}

Now we replace this Taylor expansion in the definition of $\widehat{\mu_t}(\lambda)$, this reads
\begin{align*}
  \widehat{\mu_t}(\lambda)=& \int_G (\chi_\lambda(0) +x^i\partial_{x^i}\chi_\lambda(0) + \frac{ x^ix^j}{2}\partial^2_{x^ix^j}\chi_\lambda(0)+\mathcal{O}(\vert x\vert^3 \Vert \chi_\lambda\Vert_{C^3} ))\mu_t
  \\
  =&\int_G \left( \chi_\lambda(0) +x^i\partial_{x^i}\chi_\lambda(0) - \vert x\vert^2c_2(\lambda) \left(\chi_\lambda\right)(0)+\mathcal{O}(\vert x\vert^3 \Vert \chi_\lambda\Vert_{C^3}  )  \right) \mu_t  \\
  =& d_\lambda(1-c_2(\rho)t+\mathcal{O}(c_2(\lambda)^{\frac{3}{2}}t^{\frac{3}{2}})),
\end{align*}
where we observe that all the crossed terms $x^i(g)x^j(g)$ for $i\neq j$ integrated against the measure $\mu_t$ contribute $o(t)$, then
we use the fact that our coordinates are related to the Casimir and the fact that characters are 
eigenfunctions of the Casimir--Laplace operator $\Delta_G$ on $G$ with eigenvalue $-c_2(\lambda)$ and 
we have used classical bounds on the H\"older--Besov norms of the characters.
This implies the desired bound by since we can find a $\delta_M$ such that
\[ d_\lambda(1-c_2(\lambda)t+\mathcal{O}(c_2(\lambda)^{\frac{3}{2}}t^{\frac{3}{2}}))\leq d_\lambda e^{-\delta_M c_2(\lambda)t} .\]
\end{proof}
\begin{lemma}\label{FirstSum}
    We have 
    \[\sup_{n\geq 0}\sum_{\substack{\lambda \in \widehat{G}\\ \sqrt{c_2(\lambda)/n}\leq \frac{M}{A}}}\vert c_2(\lambda)\vert ^{\frac{k}{2}} \prod_{i=1}^n\left\vert\frac{\widehat{\rho}_{s^n_i}(\lambda)}{d_{\lambda}}\right\vert  <\infty. \]
\end{lemma}
\begin{proof}
We have
\begin{eqnarray*}
   \sum_{\substack{\lambda \in \widehat{G}\\ \sqrt{c_2(\lambda)/n}\leq \frac{M}{A}}}\vert c_2(\lambda)\vert ^{\frac{k}{2}} \prod_{i=1}^n\left\vert\frac{\widehat{\rho}_{s^n_i}(\lambda)}{d_{\lambda}}\right\vert&\leq& \sum_{\substack{\lambda \in \widehat{G}\\ \sqrt{c_2(\lambda)/n}\leq \frac{M}{A}}}\vert c_2(\lambda)\vert ^{\frac{k}{2}} \prod_{i\in T_n(a.A)}\left\vert\frac{\widehat{\rho}_{s^n_i}(\lambda)}{d_{\lambda}}\right\vert \\
   &\leq& \sum_{\lambda \in \widehat{G}} \vert c_2(\lambda)\vert ^{\frac{k}{2}} \exp\left (-\frac{\vert T_n(a,A)\vert }{n}A^2\delta_Mc_2(\lambda)\right) <\infty. 
\end{eqnarray*}
   
\end{proof}
\paragraph{The big $\sqrt{c_2(\lambda)/n}$ regime.}
\begin{lemma}
    There exists $M_2$ such that for all
pairs $(t,\lambda)\in \R_+\times \widehat{G}$ such that
$\sqrt{t c_2(\lambda) } \geq M_2 $, we have
$
 \vert \widehat{\mu_t}(\lambda) \vert \leq d_\lambda \left(1+\sqrt{tc_2(\lambda)} \right)^{-{\frac{s}{2}}}.
$
\end{lemma}
\begin{proof}
By our Fourier's decay assumption,  
There exists $C>0$ and such that for all $\lambda \in \widehat{G}$,
\begin{align*}
 \vert \widehat{\mu_t}(\lambda) \vert \leq d_\lambda C\left(1+\sqrt{tc_2(\lambda)} \right)^{-s}\leq d_\lambda C\left(1+M_2 \right)^{-\frac{s}{2}}\left(1+\sqrt{tc_2(\lambda)} \right)^{-\frac{s}{2}}.
\end{align*}
The proof follows by taking $M_2$ as large as needed.
\end{proof}
\begin{lemma}\label{SecondSum}
    We have 
    \[\sup_{n\geq 0} \sum_{\substack{\lambda \in \widehat{G}\\ \sqrt{c_2(\lambda)/n}\geq \frac{M_2}{a}}}\vert c_2(\lambda)\vert ^{\frac{k}{2}} \prod_{i=1}^n\left\vert\frac{\widehat{\rho}_{s^n_i}(\lambda)}{d_{\lambda}}\right\vert <\infty. \]
\end{lemma}
\begin{proof}
    We have in this case, 
    \begin{eqnarray*}
      &&\sum_{\substack{\lambda \in \widehat{G}\\ \sqrt{c_2(\lambda)/n}\geq \frac{M_2}{a}}}\vert c_2(\lambda)\vert ^{\frac{k}{2}} \prod_{i=1}^n\left\vert\frac{\widehat{\rho}_{s^n_i}(\lambda)}{d_{\lambda}}\right\vert \leq \sum_{\substack{\lambda \in \widehat{G}\\ \sqrt{c_2(\lambda)/n}\geq \frac{M_2}{a}}}\vert c_2(\lambda)\vert ^{\frac{k}{2}} \prod_{i\in T_n(a,A)}\left(1+\sqrt{s^n_ic_2(\lambda)} \right)^{-{\frac{s}{2}}}  \\
      &\leq&\sum_{\substack{\lambda \in \widehat{G}\\ \sqrt{c_2(\lambda)/n}\geq \frac{M_2}{a}}}\vert c_2(\lambda)\vert ^{\frac{k}{2}} \left(1+A\sqrt{\frac{c_2(\lambda)}{n}} \right)^{-{\frac{s\vert T_n(a,A)\vert }{2}}}\leq \sum_{\substack{\lambda \in \widehat{G}\\ \sqrt{c_2(\lambda)/n}\geq \frac{M_2}{a}}}\frac{\vert c_2(\lambda)\vert ^{\frac{k}{2}}}{\left(1+\sqrt{\frac{c_2(\lambda)}{n}} \right)^{\frac{s\vert T_n(a,A)\vert}{2}}}\\
    &=&\sum_{j\geq \log_{[2]}\frac{M_2\sqrt{n}}{a}} \sum_{2^j\leq \sqrt{c_2(\lambda)}< 2^{j+1}}\frac{\vert c_2(\lambda)\vert ^{\frac{k}{2}}}{\left(1+\sqrt{\frac{c_2(\lambda)}{n}} \right)^{\frac{s\vert T_n(a,A)\vert}{2}}}\leq \sum_{j\geq \log_{[2]}\frac{M_2\sqrt{n}}{a}} \sum_{2^j\leq \sqrt{c_2(\lambda)}< 2^{j+1}}\frac{2^{k(j+1)}}{\left(1+\frac{2^j}{\sqrt{n}} \right)^{\frac{s\vert T_n(a,A)\vert}{2}}}\\
    &=&\sum_{j\geq \log_{[2]}\frac{M_2\sqrt{n}}{a}} \frac{2^{k(j+1)}}{\left(1+\frac{2^j}{\sqrt{n}} \right)^{\frac{s\vert T_n(a,A)\vert}{2}}}\vert \{\lambda\in \widehat{G}; 2^{2j}\leq c_2(\lambda)< 2^{j2+2}\}\vert.
    \end{eqnarray*}
    By Weyl's law, there exists $\kappa>0$ such that, \[\forall\ j\geq 0,  
    \vert \{\lambda\in \widehat{G}; 2^{2j}\leq c_2(\lambda)< 2^{j2+2}\}\vert\leq \kappa 2^{\dim(G)j}.\]
Assume now, without loss of generality that $2^{\gamma}\coloneq \sqrt{n}$ and $2^\alpha\coloneq \frac{M_2}{a}$, this gives 
    \[
    \sum_{\substack{\lambda \in \widehat{G}\\ \sqrt{c_2(\lambda)/n}\geq \frac{M_2}{a}}}\vert c_2(\lambda)\vert ^{\frac{k}{2}} \prod_{i=1}^n\left\vert\frac{\widehat{\rho}_{s^n_i}(\lambda)}{d_{\lambda}}\right\vert \leq \sum_{j\geq \gamma+\alpha } 2^{{kj+k+\dim(G)j}-\frac{s\vert T_n(a,A)\vert (j-
    \gamma)}{2}}.\]
   Choose $n_0$ big enough so that for all $n\geq n_0$, $k+\dim(G)-\frac{s\vert T_n(a,A)\vert }{2}\leq -1-k$, this gives 
        \[
    \sum_{\substack{\lambda \in \widehat{G}\\ \sqrt{c_2(\lambda)/n}\geq \frac{M_2}{a}}}\vert c_2(\lambda)\vert ^{\frac{k}{2}} \prod_{i=1}^n\left\vert\frac{\widehat{\rho}_{s^n_i}(\lambda)}{d_{\lambda}}\right\vert \leq 2^{(k+\dim (G))(\gamma+\alpha)+k-\frac{s\vert T_n(a,A)\vert \alpha}{2}},\]
    which is bounded uniformly in $n$ for $n$ large enough since 
    $\liminf_{n\to\infty}\frac{\vert T_n(a,A)\vert }{n} >0$.
\end{proof} 
\paragraph{The intermediate regime.}
\begin{lemma}
For all \emph{non trivial} representation $\lambda$, we have the strict inequality
$
\vert\widehat{\mu_t}(\lambda)\vert < d_\lambda 
$.
\end{lemma}
\begin{proof}
We use the inversion invariance of our measure $\mu_t$ and the unitarity of the representation $\lambda$~:
 \begin{align*}
     \int_G \chi_\rho \mu_t=\frac{1}{2}  \int_G \chi_\rho(g)+\chi_\rho(g^{-1}) \mu_t=\int_G \emph{Re} \{\chi_\lambda(g)\} \dd \mu_t(g).
 \end{align*} 
 Then the lemma follows from the fact that for $g$  outside a set of zero measure, we have
$-d_\lambda <\emph{Re} \{\chi_\lambda(g)\}<d_\lambda.$
\end{proof}
\begin{lemma}\label{ThirdSum}
    We have 
    \[\sup_{n\geq 0}\sum_{\substack{\lambda \in \widehat{G}\\ \frac{M_1}{a}\leq \sqrt{c_2(\lambda)/n}\leq \frac{M}{A}}}\vert c_2(\lambda)\vert ^{\frac{k}{2}} \prod_{i=1}^n\left\vert\frac{\widehat{\rho}_{s^n_i}(\lambda)}{d_{\lambda}}\right\vert <\infty. \]
\end{lemma}
\begin{proof}
    Set
\[ \delta\coloneq \sup_{\substack{\lambda\in \widehat{G} \\ M\leq \sqrt{c_2(\lambda)/n}\leq M_2 }} \frac{\vert \widehat{\mu_t}(\lambda) \vert}{d_\lambda}. \]
By the previous lemma, $0<\delta<1$.
Then, since  $\log \delta<0$,
\begin{eqnarray*}
    \sum_{\substack{\lambda \in \widehat{G}\\ \frac{M_1}{a}\leq \sqrt{c_2(\lambda)/n}\leq \frac{M}{A}}}\vert c_2(\lambda)\vert ^{\frac{k}{2}} \prod_{i=1}^n\left\vert\frac{\widehat{\rho}_{s^n_i}(\lambda)}{d_{\lambda}}\right\vert 
    &\leq&  \sum_{\substack{\lambda \in \widehat{G}\\ \frac{M_1}{a}\leq \sqrt{c_2(\lambda)/n}\leq \frac{M}{A}}} \vert c_2(\lambda)\vert ^{\frac{k}{2}} \delta^n\\
    &\leq& \sum_{\substack{\lambda \in \widehat{G}\\ \frac{M_1}{a}\leq \sqrt{c_2(\lambda)/n}\\\leq \frac{M}{A}}} \vert c_2(\lambda)\vert ^{\frac{k}{2}} \delta^{\frac{c_2(\lambda)a}{M_1}}\\
    &\leq&\sum_{\lambda \in \widehat{G}} \vert c_2(\lambda)\vert ^{\frac{k}{2}} \exp\left(\frac{c_2(\lambda)a}{M_1}\log \delta\right)<\infty.
\end{eqnarray*}
\end{proof}
\begin{proof}[Proof of Proposition\ref{boundedinCk}]
It is the direct consequence of the combination of Lemmas~\ref{Fourier_expansion},~\ref{FirstSum},~\ref{SecondSum}, and~\ref{ThirdSum}.
\end{proof}

\paragraph{Proof of Theorem \ref{ConvSegalAmp} about convergence of partition functions and Segal amplitudes.}

The convergence of the partition function is a particular case of the convergence
of Segal amplitudes for closed surfaces. The idea is to reduce to the proof of the previous local limit Theorem using ideas from Frobenius algebras~\cite[p.~32]{DerMar}, \cite[p.~8-9]{MnevBV}.
Recall that $(\mathcal{T}_n)_{n\geq 0}$ denotes a sequence of triangulation of some bordered surface.

Now we can conclude by the following Lemma.
\begin{lemma}
Define the coproduct acting on class functions
as $$f\in C^\infty(G)^G\mapsto \mathbf{\Delta} f(g_1,g_2):=\int_{G} f\left( g_1 hg_2h^{-1} \right)   dh .$$
Then the $k-1$--fold iterated coproduct is linear continuous from 
$C^\infty(G)^G\mapsto C^\infty(G^k)^G$.  
\end{lemma}

Define $\Omega:=\sum_{\lambda\in \widehat{G}} \frac{\chi_\lambda}{d_\lambda} $ this is a class distribution $\mathcal{D}^\prime(G)^G$, this is the Segal amplitude of the one holed torus as in~\cite[p.~32]{DerMar}.
Then note that the Segal amplitude  $Z_{\mathcal{T}_n}$ of the triangulated surface $\Sigma$ of genus $g$ and with $k$ boundary components can be written in terms of iterated coproducts , the iterated convolution $\underset{F\in \mathbb{F}}{\bigstar}\ \rho_{\sigma(F)}$ and the class distributions $\Omega$:
\begin{align*}
Z_{\mathcal{T}_n}(\cdot,\dots,\cdot)=\left\langle\mathbf{\Delta}^{g+k-1} \underset{F\in \mathbb{F}}{\bigstar}\ \rho_{\sigma(F)}, \Omega^{\otimes g} \right\rangle=    \sum_{\lambda\in\widehat{G}} 
\left( \prod_{F\in \mathbb{F}_n} \frac{\widehat{\mu}_{\sigma(F)}(\lambda)}{d_\lambda} \right)
\frac{\chi_\lambda(\cdot) \dots \chi_\lambda(\cdot)  }{d_\lambda^{\,2g-2+k}} \in C^\infty(G^k)  
\end{align*}
where the pairing $\left\langle\mathbf{\Delta}^{g+k-1} \underset{F\in \mathbb{F}}{\bigstar}\ \rho_{\sigma(F)}, \Omega^{\otimes g} \right\rangle$ is a partial contraction between a function in $C^\infty(G^{k+g})$ with some $g$--fold tensor product $\Omega^{\otimes g}$. We invite the reader to compare with ~\cite[formula p.~32 ]{DerMar} and \cite[formula for $Z$ on p.~9]{MnevBV}.
This concludes the proof of Theorem \ref{ConvSegalAmp} since the iterated convolution
$\underset{F\in \mathbb{F}}{\bigstar}\ \rho_{\sigma(F)}$ converges in $C^\infty(G)$ to the heat kernel $p_{\text{Area}(\Sigma)}\in C^\infty(G)$ by Theorem \ref{convck} and by continuity of the iterated coproducts and of the partial pairings.

\section{Technical estimates on Lie group valued random walks}\label{LieGroupRW}
In the next section, we want to study the scaling limit of some lattice gauge theories. Roughly speaking, a lattice gauge theory associated to a graph is a collection of random matrices, one for each face of the graph, and such that the law of each matrix depends on the area of the face it is associated to. 
\par Therefore, to be able to define lattice gauge theories, we need a  family of probability measures $(\mu_t)_{t>0}$, where $t$ is meant to be replaced later on by the area of the face.
\par The idea in the next section, is to show that a large class of lattice gauge theories, in the scaling limit, converge in some sense to the Yang--Mills measure. To do this, we will need some technical estimates on the family $(\mu_t)_{t>0}$ that we will do in the current section. 

\subsection{Assumption on the actions}\label{actions}
First, we will discuss the assumptions on the family $(\mu_t)_{t>0}$ . The set of assumptions is not optimal; however, we will stick to them to simplify the presentation. 
Consider a family $(\mu_t)_{t>0}$ of probability measures on $G$. We assume that 
\begin{enumerate}
\item For all $t>0$, $\mu_t$ is symmetric and ad-invariant.
    \item For all $t>0$, $\mu_t$ has a density $\rho_t$ with respect to the Haar measure on $G$.
    \item For all $v,w\in \g$, 
    \[\int_G \langle v, \log a \rangle \langle w,\log a\rangle  d\mu_t(a)=\langle v,w\rangle t+o_{v,w}(t).\]
    \item  For all integers $N$, there exists a constant $C_N>0$ such that
    \[ \sup_{g\in G}\left\{ \left(1+\frac{\mathrm{dist}(g,1_G)}{\sqrt{t}} \right)^{N} \rho_t(g) \right\}\leq C_Nt^{-\frac{\dim(G)}{2}}.   \]
    \item  For all $\beta>0$, there exists $C_\beta>0$ such that 
    \[\forall t\leq 1, \int_{G}\vert \log x\vert ^{2\beta} \rho_{t}(x)\dd x \leq C_{\beta} t^{\beta}.\]
\end{enumerate}
Such a one parameter family is said to have the property $(H)$. The letter $H$ stands for \emph{heat like}. 
Let us briefly comment on each of the assumptions.
\begin{itemize}
    \item The symmetry in the first item means that in the lattice gauge theory associated to $(\mu_t)_{t>0}$, the law of the holonomy around a face does not depend on the cyclical order of the face. The $\mathrm{Ad}$-invariance ensures the gauge invariance of the model.
\item The second item allows easier handling of the problem without a big loss. 
\item The third item is essential because it characterizes the leading term of the central limit theorem. 
\item The fourth item means we have super-polynomial concentration of the measure near $1_G$. It important because, as we will see in the next section, we need a concentration property for products of matrices picked under the family $(\mu_t)_{t>0}$. 
\item The fifth item is not essential but facilitates many results and is verified by a wide family of actions. It says that the moments of $(\mu_t)_{t>0}$ behave like those of a Gaussian.
\end{itemize}
The following lemma is particularly useful because it tells us that computing integrals in a small (but large enough) ball near $1_G$ is almost the same as computing the integral over all of $G$.  
\begin{lemma}\label{conv}
    For all $\eta < \frac 1 2$, for all $p>1$, there exists $C_{\eta,p}$ such that 
    \[ \mu_t(B(0,t^\eta)^c) \leq C_{\eta,p} t^p.\]
\end{lemma}
\begin{proof}
     Let $\delta>0$, we have, for all integers $N$,
    \begin{eqnarray*}
   \mu_t(B(0,\delta)^c)
    &\leq&  C_Nt^{-\frac{\dim(G)}{2}}\int_{B(0,\delta)^c}\frac{1}{\left(1+\frac{\mathrm{dist}(g,1_G)}{\sqrt{t}} \right)^{N}}dg\leq \frac{C_Nt^{-\frac{\dim(G)}{2}}}{\left(1+\frac{\delta}{\sqrt{t}} \right)^{N}}\leq C_Nt^{\frac{N-\dim(G)}{2}}\delta^{-N}.
    \end{eqnarray*}
    Now take $\delta = t^{\frac{1-\varepsilon}{2}}$, we get
    \[\mu_t(B(0,\delta)^c)\leq C_Nt^{\frac{N\varepsilon-\dim(G)}{2}}. \]
    It is enough to take $N$ big enough so that $N\varepsilon-\dim(G) >2p$.
\end{proof}

When $\rho_t = p_t$ is the heat kernel, the measure enjoys the semi-group property. This tells us that for two random variables distributed according to $p_{t_1}\dd g $ and $p_{t_2}\dd g$ respectively, their product follows $p_{t_1+t_2}\,dg$. This gives an automatic bound on the concentration of the product near $1_G$. This property is very important in studying lattice gauge theories, and appears when merging faces. It fails when $\rho$ is no longer the heat kernel. The next lemma, which relies essentially on the BDG inequality, aims to compensate for this loss in the case of a general measure.

\begin{lemma} \label{estim}
Under the $(H)-$property, for all $\beta>0$, there exists $C_{\beta}>0$ such that
\[  \forall k\geq 1, \forall s_1,\cdots,s_k\leq 1, \int_{G}\vert \log x\vert ^{2\beta} \rho_{s_1}\ast\cdots\ast \rho_{s_k} (x)\dd x \leq C_\beta(s_1+\cdots +s_k)^{\beta}.\]
\end{lemma}
\begin{proof}
    The proof is essentially the same as the proof of lemma 5.4 in~\cite{Chevyrev_2019}. 
\end{proof}

\subsection{Technical estimates} \label{techest}
In Section~\ref{s:CLT}, we will show that lattice gauge theories converge to a universal limit. We will do this using characteristic functions. Here we describe the exact technical expression we need to bound, without further motivation. The motivation, however, will be clear from the next section.   

Consider, for each $N\geq0$, a sequence of positive real numbers $(s^{(N)}_i)_{0\leq i\leq 2^N-1}$. Let also $(X^{(N)}_i)_{0\leq i\leq 2^N-1}$ be a sequence of independent  $G-$ valued random variables such that 
\[\forall 0\leq i\leq 2^N-1,\ \mathbb{P}(X^{(N)}_i\in \dd g)=\mu_{s^{(N)}_i}(\dd g).\] Let finally $(M^{(N)}_k)_{0\leq k\leq 2^N-1}$ be a family of elements in $\g$ such that 
\[\max_{1\leq k\leq 2^N-1} \|M^{(N)}_k\| \leq 1.\]
\par We would like to understand the limiting behavior of
\[E^{(N)}_{2^N-1}\coloneq \mathbb{E}\left[\prod_{k=0}^{2^N-1}\exp i\left \langle \log\left(\prod_{i=0}^{k+1} X^N_i\right)-\log\left(\prod_{i=0}^{k} X^N_i\right),M^{(N)}_k\right\rangle\right],
\]
under the assumptions on the family of measure, 
as it will be important in the identification of the limit of lattice gauge theories. We will proceed by induction.
\begin{lemma}\label{8.3}
    We have 
    \[E^{(N)}_{2^N-1}\coloneq \mathbb{E}\left[F_{s^{(N)}_{2^N-1}}\left(\prod_{i=0}^{2^N-2} X^N_i,M^{(N)}_{2^N-1}\right)\prod_{k=0}^{2^N-2}\exp i\left \langle \log\left(\prod_{i=0}^{k+1} X^N_i\right)-\log\left(\prod_{i=0}^{k} X^N_i\right),M^{(N)}_k \right\rangle\right],\]
    where 
    \[F_t(g,v)\coloneq \int_G \exp \left(i\left\langle v, \log(ag)- \log(g) \right\rangle \right) d\mu_t(a).\]
\end{lemma}
\begin{proof}
    By usual properties of the conditional expectation.
\end{proof}
The strategy will be to decompose $F_t(g,v)$ in the sum of two terms, a leading term, and an error term that we will bound tightly. This will induce the decomposition of  
\[F_{s^{(N)}_{2^N-1}}\left(\prod_{i=0}^{2^N-2} X^N_i,M^{(N)}_{2^N-1}\right)\]
in two terms as well :
\begin{itemize}
    \item a leading deterministic term that will happen to be
\[A^{(N)}_{2^N-1}\coloneq  1-\frac{\|M^{(N)}_{2^N-1}\|^2}{2}s^{(N)}_{2^N-1}.\]
\item an error random term $R^{(N)}$. The random part of the error term is a polynomial function  (with possible fractional powers)   of $\prod_{i=0}^{2^N-2} X^N_i$, which follows $\mu_{s^{(N)}_0}\star \cdots \star \mu_{s^{(N)}_{2^N-2}}$. Therefore, Lemma~\ref{estim} lets us easily bound it. 
\end{itemize}
This will tell us that \[E_{2^N-1}^{(N)}=\left(1-\frac{\|M^{(N)}_{2^N-1}\|^2}{2}s^{(N)}_{2^N-1}\right)E_{2^N-2}^{(N)}+\mathbb{E}[\vert R^{(N)}\vert],\]
and will let us start with the induction. 
\par The aim of the two following lemmas is to decompose $F_t(g,v)$.
\begin{lemma} \label{devlim}
Let $\delta>0$ be such that $\exp : B(1_G,\delta) \to \g$ 
is a diffeomorphism on its image.
Let $p$ be an integer greater than $2$. There exists $C_p >0$ such that if  $\vert X\vert , \vert Y\vert \leq \delta$, 
\begin{enumerate}
    \item  there exists $\Lambda_{Y}\in \mathrm{End}(\g)$, such that 
   \[\left \vert \log(\exp X \exp Y)- X- Y - \Lambda_{Y}(X)\right \vert  \leq C_p\left( \vert X \vert ^2\vert Y \vert+\vert X \vert \vert Y\vert^{p}\right); \]
   \item we have $\left \vert  \log(\exp X \exp Y)- X- Y\right \vert \leq C_p \vert X \vert \vert Y\vert .$
\end{enumerate}
\end{lemma}
\begin{proof}
It is a direct consequence of the BCH formula taken up to order $p$.
\end{proof}
\begin{remark}
Note that the presence of the linear function $\Lambda_Y$ can absorb as many higher-order commutator terms containing exactly one $X$, enabling us to increase the power $p$ in $Y$ on the right side as needed. This will be very useful because, thanks to symmetry,
 $\int_G \Lambda_Y(\log a) \, \mu_t(da) = 0,$ 
and we can therefore add this term whenever we need it, gaining powers in $Y$.
\end{remark}
\begin{lemma}\label{8.5}
    For each $p\geq 2$, there exists $C_p>0$ such that, for all $t\in [0,1]$, $g\in G$ and $v\in \g$, 
      \[F_t(g,v)=1-\frac{\|v\|^2}{2}t+R_t(g,v),\]
      where, 
      \[\vert R_t(g,v)\vert\leq  C_p((1+\|v\|^2)\vert \log g\vert t +\vert \log g\vert ^p\sqrt{t} +(\|v\|+\|v\|^2) t^p + \|v\|^3t^{\frac{3}{2}}+\|v\|^2\vert \log g\vert ^2 t).\]
\end{lemma}
\begin{remark}
      In the subsequent of this section, the notation $O(r)$ means a quantity whose absolute value is smaller than an absolute constant times $r$. An absolute constant is a constant that only depends on the fixed settings (such as the diameter of $G$).
\end{remark}
\begin{proof}
  
    We have 
    \begin{eqnarray*}
        F_t(g,v)&=&\int_G \exp {i\left\langle v, \log(ag)- \log(g) \right\rangle_{\mathfrak{g}} } d\mu_t(a) \\
        &=&\int_{B(0,t^\eta)} \exp {i\left\langle v, \log(ag)- \log(g) \right\rangle_{\mathfrak{g}} } d\mu_t(a)+\int_{B(0,t^\eta)^c} \exp {i\left\langle v, \log(ag)- \log(g) \right\rangle_{\mathfrak{g}} } d\mu_t(a).     
    \end{eqnarray*}
    Fix $\eta <\frac 12$. From Lemma~\ref{conv}, the second member of the second side verifies, for all $p$,
    \[\int_{B(0,t^\eta)^c} \exp {i\left\langle v, \log(ag)- \log(g) \right\rangle_{\mathfrak{g}} } d\mu_t(a)\leq \mu_t(B(0,t^\eta)^c)\leq C_p t^p.\]
    To treat the first member, let us develop the exponential. We have for all $g\in B(0,t^\eta)$, 
    \[ \exp {i\left\langle v, \log(ag)- \log(g) \right\rangle }=1+i\left\langle v, \log(ag)- \log(g) \right\rangle -\frac{\left\langle v, \log(ag)- \log(g) \right\rangle^2 }{2}+O(\|v\|^3\vert \log a\vert ^3).\]
   
    Now,
    \begin{eqnarray*}
        \int_{B(0,t^\eta)}1\dd \mu_t(a) = \int_{G}1\dd \mu_t(a)- \int_{B(0,t^\eta)^c}1\dd \mu_t(a)= 1 +O(C_p t^p).
    \end{eqnarray*}
    Moreover,
    \begin{eqnarray*}
        \int_{}\left\langle v, \log(ag)- \log(g) \right\rangle\dd \mu_t(a) &=&    \int_{G}\left\langle v, \log(ag)- \log(g) \right\rangle\dd \mu_t(a) +O(2\ \mathrm{diam}(G)\|v\|C_pt^p ).
    \end{eqnarray*}
    Now, using the fact that for $\Gamma \in \g^*$,
    $\int_G \Gamma(\log x)\mu_t(\dd x) = 0,$
    we have, from the first item of the Lemma~\ref{devlim}, 
    \begin{eqnarray*}
         \int_{G}\left\langle v, \log(ag)- \log(g) \right\rangle\dd \mu_t(a) &=&\int_{G}\left\langle v, \log(ag)- \log(g) -\log(a)-\Lambda_{\log(g)}(\log(a))\right\rangle\dd \mu_t(a)\\
         &=&O\left (C_p\int_G(\vert \log g\vert \vert \log a \vert^2 +\vert \log a\vert \vert \log g\vert ^p) \dd \mu_t(a)  \right),
    \end{eqnarray*}
    and from the fifth assumption on the family of measures,
    \begin{eqnarray*}
         \int_{G}\left\langle v, \log(ag)- \log(g) \right\rangle\dd \mu_t(a)
         &= & O(C_p \vert \log g\vert t +C_p\sqrt{t}\vert \log g\vert ^p).
    \end{eqnarray*}
Similarly, 
    \begin{eqnarray*}
        \int_{B(0,t^\eta)}\left\langle v, \log(ag)- \log(g) \right\rangle^2\dd \mu_t(a) &=&    \int_{G}\left\langle v, \log(ag)- \log(g) \right\rangle^2\dd \mu_t(a) +O(\|v\|^2 C_pt^p ).
    \end{eqnarray*}
    and 
    \begin{eqnarray*}
        \int_{G}\left\langle v, \log(ag)- \log(g) \right\rangle^2\dd \mu_t(a) 
        &= & \vert v\vert ^2t+ O(\|v\|^2\vert \log g\vert ^2 t+\|v\|^2\vert \log g\vert t).
    \end{eqnarray*}
    We conclude the proof by combining all the terms together.
\end{proof}
\begin{lemma}\label{EstimBruitBlanc}
Recall that $A^{(N)}_k\coloneq  1-\frac{\|M^{(N)}_k\|^2}{2}s^{(N)}_k.$
    For all $p>1$, we have 
    \begin{eqnarray*}
   \left  \vert E^{(N)}_{2^N-1}-\prod_{i=0}^{2^N-1}A^{(N)}_{i} \right \vert&\leq&  C_p\left (s_{N,1}^\frac 32 + s_{N,\frac 12}s_{N,1}^\frac p2+s_{N,p}+s_{N.\frac 32}+s_{N,1}^2\right )  
    \end{eqnarray*}
    where 
    $s_{N,\alpha}\coloneq \sum_{i=0}^{2^N-1}(s^{(N)}_i)^\alpha.$
\end{lemma}
\begin{proof}
    We have from Lemmas~\ref{8.3} and~\ref{8.5},
    \begin{align*}
        E^{(N)}_{2^N-1}&=A^{(N)}_{2^N-1} E^{(N)}_{2^N-2}\\
        +\ &\ \mathbb{E}\left[R_{s^{(N)}_{2^N-1}}\left(\prod_{i=0}^{2^N-2} X^N_i,M^{(N)}_{2^N-1}\right)\prod_{k=0}^{2^N-2}\exp i\left \langle \log\left(\prod_{i=0}^{k+1} X^N_i\right)-\log\left(\prod_{i=0}^{k} X^N_i\right),M^{(N)}_k \right\rangle\right].
    \end{align*}

    Therefore,
    \[\vert E^{(N)}_{2^N-1}-A^{(N)}_{2^N-1} E^{(N)}_{2^N-2}\vert \leq \mathbb{E}\left[\left\vert R_{s^{(N)}_{2^N-1}}\left(\prod_{i=0}^{2^N-2} X^N_i,M^{(N)}_{2^N-1}\right)\right \vert \right].\]
    We apply Lemma~\ref{estim} and the previous lemma, to get
    \begin{eqnarray*}
    \vert E^{(N)}_{2^N-1}-A^{(N)}_{2^N-1} E^{(N)}_{2^N-2}\vert&\lesssim& s^{(N)}_{2^N-1}\cdot (s^{(N)}_0+\cdots +s^{(N)}_{2^N-1})^\frac 12   +\left(s^{(N)}_{2^N-1}\right)^\frac 1 2\cdot (s^{(N)}_0+\cdots +s^{(N)}_{2^N-1})^\frac p2\\
    &+&\left(s^{(N)}_{2^N-1}\right)^p+\left(s^{(N)}_{2^N-1}\right)^\frac 3 2+s^{(N)}_{2^N-1}\cdot (s^{(N)}_0+\cdots +s^{(N)}_{2^N-1})
    \end{eqnarray*}
    where the constant in $C_p(\max (\|v\|^3,1))$. We get the desired result by triangular inequality.
\end{proof}

\subsection{Classical lattice gauge actions}
\label{ss:classicalactions}

In this subsection we discuss some historically important actions that were used in both, physicist and mathematics literature. We will prove that all these actions verify the $(H)-$property. This, together with Section~\ref{s:functspaces} conclude the proof of Theorem~\ref{thm:mainthmintro}.
In practice, the proof consists in reducing the computation to the Gaussian case on the Lie algebra with a small error, thanks to the following lemma.
\begin{lemma}\label{integ}
     Let $G\subset U(N)$ be a compact Lie group with lie algebra $\g$.  Let $U$ be a neighborhood of $1_G$ of diameter less than $1$, for which $\log$ is a diffeomorphism on its image. Then, there exists a constant $C$ such that for any positive bounded measurable function $f:U\subset G\rightarrow \mathbb{R}$,
    \[\left\vert \int_U f(g)\dd g - \int_{\log U} f(\exp X)dX\right\vert \leq C \int_{\g}\vert X\vert ^2 f(\exp X)dX. \]
\end{lemma}
\begin{proof}
   By a change of variables $g\mapsto \log g$, we have 
   \[ \int_U f(g)\dd g = \int_{\log U} f(\exp X)\vert J(X)\vert dX, \]
   where $\vert J(X)\vert =\det (\dd \exp (g)) $.
   However, we have 
  $ \vert 1-\vert J(X)\vert\vert \leq C \vert X\vert ^2.$
   This gives 
    \[\left\vert \int_U f(g)\dd g - \int_{\log U} f(\exp X)dX\right\vert \leq C \int_{\log U}\vert X\vert ^2 f(\exp X)dX \leq C \int_{\g}\vert X\vert ^2 f(\exp X)dX.\]
\end{proof}
\paragraph{The Villain action.}
The heat kernel. Let for all $t>0$, 
        $\mu_t(\dd g)\coloneq p_t(g)\dd g,$
        where $p_t$ is the heat kernel on $G$. The heat kernel verifies all the assumptions and it was the standard choice use by Lévy for the construction of the measure. 
\paragraph{The Manton action.}
Consider $t>0$ and define 
        \[\mu_t(\dd g) =\frac{1}{Z_m(t)}\exp\left(-\frac{d(g,1_G)^2}{2t}\right) \dd g, \text{ with } Z_m(t)\coloneq \int_G \exp\left(-\frac{d(g,1_G)^2}{2t}\right)\dd g.\] 
        
        \begin{proposition}
            We have 
            $Z_m(t)=\left(2\pi t\right)^{{\frac{\dim(G)}{2}}}+O(t^{{\frac{\dim(G)}{2}}+1})$.
        \end{proposition}
        \begin{proof}
        We have
        \begin{eqnarray*}
            Z_m(t)&=&\int_G \exp\left(-\frac{d(g,1_G)^2}{2t}\right)\dd g\\
            &=& \int_{B\left(0,t^{\frac{1-\varepsilon}{2}}\right)} \exp\left(-\frac{d(g,1_G)^2}{2t}\right)\dd g+\int_{B\left(0,t^{\frac{1-\varepsilon}{2}}\right)^c} \exp\left(-\frac{d(g,1_G)^2}{2t}\right)\dd g \\
            &=&\int_{B\left(0,t^{\frac{1-\varepsilon}{2}}\right)} \exp\left(-\frac{d(g,1_G)^2}{2t}\right)\dd g+O(\exp(-t^{-\varepsilon})).
        \end{eqnarray*}
        Now using Lemma~\ref{integ},
        \begin{eqnarray*}
        &&\int_{B\left(0,t^{\frac{1-\varepsilon}{2}}\right)} \exp\left(-\frac{d(g,1_G)^2}{2t}\right)\dd g= \int_{\log B\left(0,t^{\frac{1-\varepsilon}{2}}\right) }\exp\left(-\frac{\vert X\vert^2}{2t}\right)\vert J(X)\vert \dd X \\
            &=&\int_{\log B\left(0,t^{\frac{1-\varepsilon}{2}}\right) } \exp\left(-\frac{\vert X\vert^2}{2t}\right)(1+O(\vert X\vert ^2))\dd X = \int_{\log B\left(0,t^{\frac{1-\varepsilon}{2}}\right) } \exp\left(-\frac{\vert X\vert^2}{2t}\right)\dd X +O(t^{{\frac{\dim(G)}{2}}+1})\\
               &=&\int_{\g} \exp\left(-\frac{\vert X\vert^2}{2t}\right)\dd X +O(\exp(-t^{-\varepsilon}))+O(t^{{\frac{\dim(G)}{2}}+1})=\left(2\pi t\right)^{{\frac{\dim(G)}{2}}}+O(t^{{\frac{\dim(G)}{2}}+1}). 
        \end{eqnarray*}
        \end{proof}
        \begin{proposition}
            For all $v\in \g$, 
    \[\int_G \langle v, \log a \rangle ^2 d\mu_t(a)=\vert v\vert ^2t+o(t).\]
        \end{proposition}
        \begin{proof}
        Let $v\in \g$, we have 
        \begin{eqnarray*}
        \int_G \langle v, \log a \rangle ^2 d\mu_t(a) &=& \int_{B\left(0,t^{\frac{1-\varepsilon}{2}}\right)} \langle v, \log a \rangle ^2 d\mu_t(a)+\int_{B\left(0,t^{\frac{1-\varepsilon}{2}}\right)^c} \langle v, \log a \rangle ^2 d\mu_t(a)\\
        &=&\frac{1}{Z_m(t)}\int_{B\left(0,t^{\frac{1-\varepsilon}{2}}\right)} \langle v, X \rangle ^2 \exp\left(-\frac{\vert X\vert^2}{2t}\right)(1+O(\vert X\vert ^2)) \dd X+O(\exp(-t^{-\varepsilon})).
    \end{eqnarray*}
    Now,
    \begin{eqnarray*}
        &&\frac{1}{Z_m(t)}\int_{B\left(0,t^{\frac{1-\varepsilon}{2}}\right)} \langle v, X \rangle ^2 \exp\left(-\frac{\vert X\vert^2}{2t}\right)(1+O(\vert X\vert ^2)) \dd X\\
        &=&\frac{\left(2\pi t\right)^{{\frac{\dim(G)}{2}}}}{\left(2\pi t\right)^{{\frac{\dim(G)}{2}}}+O(t^{{\frac{\dim(G)}{2}}+1})}\left(\vert v\vert^2 t + o(t)\right) =(1-O(t))\left(\vert v\vert^2 t + o(t)\right)= \vert v\vert^2 t + o(t).
    \end{eqnarray*}
        \end{proof}
Therefore, we get the following.
\begin{proposition}
    The Manton family has the property $(H)$.
\end{proposition}

\paragraph{The Wilson action.}
The Wilson action. In this example, we assume $G=U(N)$. In this case, for $g\in U(N)$, we have, when $g$ is close to the identity, 
        \[d(g,1)^2=\vert \log g\vert^2 =\frac{1}{2}\mathrm{Tr} (\log g)(\log g)^*\approx \frac{1}{2}\mathrm{Tr} (1- g)(1- g)^*=-\frac{1}{2}\mathrm{Tr}( 1-g^*-g+gg^*),\]
        which gives 
        \[d(g,1)^2\approx\frac{1}{2}\mathrm{Tr}( 1-g^*)+\frac{1}{2}\mathrm{Tr}( 1-g)=\mathbf{Re}\mathrm{Tr} (1-g).\]
        This motivates the Wilson actions which is defined, for all $t>0$ as 
        \[\mu_t(\dd g)=\frac{1}{Z_w(t)} \exp\left(-\frac{\mathbf{Re}\mathrm{Tr} (1-g)}{t}\right),\text{ with } Z_w(t)\coloneq \int_G\exp\left(-\frac{\mathbf{Re}\mathrm{Tr} (1-g)}{t}\right) \dd g.\]
        We can show as well the following.
        \begin{proposition}
            We have 
                 $Z_w(t)=\left(2\pi t\right)^{{\frac{\dim(G)}{2}}}+O(t^{{\frac{\dim(G)}{2}}+1}),$ and 
                the Wilson family verifies conditions $(H)$.          
        \end{proposition}

\section{Scaling limit of lattice gauge theories on surfaces}\label{s:CLT}

Let $\Sigma$ be a compact surface and let $f : \Sigma \to \mathbb{R}$ be a Morse function satisfying the Smale condition. 
Blow up $\Sigma$ at the maximal critical point $\max f$ to obtain a surface with one boundary component, which we denote by $\mathcal{S}$. 
Let $G$ be a compact Lie group with Lie algebra $\mathfrak{g}$, equipped with a fixed bi-invariant inner product $\langle \cdot,\cdot\rangle$. 
On $G$ we consider a family of probability measures $(\mu_t)_{t>0}$ of the form
\[
\mu_t = \rho_t \,\dd g,
\]
where $\dd g$ is the Haar measure on $G$ and the densities $(\rho_t)_{t>0}$ satisfy the hypothesis $(H)$ introduced at the beginning of Section~\ref{LieGroupRW}. 

Let $\mathbb{G} = (\mathbb{V},\mathbb{E},\mathbb{F})$ be a finite graph embedded in $\mathcal{S}$, with vertex set $\mathbb{V}$, edge set $\mathbb{E}$ and face set $\mathbb{F}$, that we assume are all contractible. 
We choose once and for all an orientation of each edge that we call \emph{positive} and write $\mathbb{E}_+$ for the set of positively oriented edges. 
For $e \in \mathbb{E}_+$ we write $e_+$ for $e$ endowed with this reference orientation, and $e_-$ for the same geometric edge with the opposite orientation. 
We then set $\mathbb{E}_- \coloneqq \{e_- : e \in \mathbb{E}_+\}$ and regard $\mathbb{E}_+ \cup \mathbb{E}_-$ as the set of oriented edges.

\begin{definition}[Discrete connection on a graph]
A \emph{discrete connection} on $\mathbb{G}$ is a map
\[
g : \mathbb{E}_+ \cup \mathbb{E}_- \longrightarrow G
\]
such that for every $e \in \mathbb{E}_+$ we have
\[
g(e_+) = g(e_-)^{-1}.
\]
For an edge $e\in \mathbb{E}_+\cup \mathbb{E}_-$, we denote by $i(e)$ and $f(e)$ the vertices such that $e=(i(e),f(e))$. A path in $\mathbb{G}$ is a sequence of edges  $a_1,\dots,a_n$ such that $f(a_k)=i(a_{k+1})$ for all $k=1,\dots,n-1$. If $c = a_1 \cdots a_n$ is a path in $\mathbb{G}$, given as a concatenation of oriented edges $a_i \in \mathbb{E}_+ \cup \mathbb{E}_-$, we define the \emph{holonomy} of $g$ along $c$ by
\[
\mathrm{Hol}(g,c) \coloneqq g(a_1)\cdots g(a_n) \in G.
\]
\end{definition}

We next introduce the lattice Yang--Mills measure of Driver and Sengupta on the space of discrete connections.

\begin{definition}[Driver--Sengupta measure]
Let $(\rho_t)_{t>0}$ be as above and let $\mathbb{G}$ be a finite graph on $\mathcal{S}$. 
The \emph{Driver--Sengupta measure} associated with $(\rho_t)_{t>0}$ and $\mathbb{G}$ is the probability measure on $G^{\mathbb{E}_+}$, denoted by $\mathrm{DS}(\rho,\mathbb{G})$, whose density with respect to the product Haar measure $\dd g$ on $G^{\mathbb{E}_+}$ is given by
\[
\dd \mathrm{DS}(\rho,\mathbb{G})\big((g_e)_{e \in \mathbb{E}_+}\big) 
 = \frac{1}{Z_{\rho,\mathbb{G}}}
   \prod_{F \in \mathbb{F}} \rho_{\sigma(F)}\big(\mathrm{Hol}(g,\partial F)\big)\, \dd g.
\]
Here $\partial F$ is the oriented boundary of $F$, the time parameter $\sigma(F)>0$ is the area of $F$, and $Z_{\rho,\mathbb{G}}$ is the normalizing constant.
\end{definition}

We also record the natural notion of gauge transformation on the graph.

\begin{definition}[Gauge transformations]
A \emph{gauge transformation} on $\mathbb{G}$ is a map
\[
h : \mathbb{V} \longrightarrow G.
\]
Given a discrete connection $g$ on $\mathbb{G}$, its gauge transform $g^h$ is the discrete connection defined on positively oriented edges by
\[
g^h(e) \coloneqq h(\mathrm{t}(e))\, g(e)\, h(\mathrm{s}(e))^{-1}, \qquad e \in \mathbb{E}_+,
\]
and extended to negative edges by $g^h(e_-) \coloneqq g^h(e_+)^{-1}$. 
Here $\mathrm{s}(e)$ and $\mathrm{t}(e)$ denote the source and target vertices of $e$.
\end{definition}

One readily checks that the Driver--Sengupta measure is invariant under gauge transformations, and that holonomies of loops based at a fixed vertex transform by conjugation at that basepoint. 
In particular, traces of loop holonomies (“Wilson loops”) provide gauge-invariant observables.

\medskip

We now specialize to the sequence of Morse lattices introduced in Section~\ref{s:Morselattice} and state the main result of this subsection.

\begin{proposition}
There exists a sequence of random $\mathfrak{g}$-valued $1$-forms $(A_N)_{N\geq 0}$ on $\mathcal{S}$ such that the following properties hold, where $(\Lambda_N)_{N\geq 0}$ is the sequence of Morse lattices from Section~\ref{s:Morselattice}.
\begin{enumerate}
\item For each $N$, let $(g_e)_{e \in \mathbb{E}_N}$ be a discrete connection on $\Lambda_N$ with law $\mathrm{DS}(\rho,\Lambda_N)$. 
Then the family of loop holonomies of $A_N$ has the same law as the family of loop holonomies of $(g_e)_{e \in \mathbb{E}_N}$, in the sense that
\[
\big(\mathrm{Hol}(A_N,c)\big)_{c \in \mathrm{Loop}_o(\Lambda_N)}
\;\stackrel{\mathrm{law}}{=}\;
\big(\mathrm{Hol}\big((g_e)_{e\in\mathbb{E}_N},c\big)\big)_{c \in \mathrm{Loop}_o(\Lambda_N)}.
\]
\item There exists a random distributional $1$-form $A$ on $\mathcal{S}$ such that
\[
A_N \xrightarrow[N\to\infty]{\;\mathrm{law}\;} A
\]
in $\mathcal{D}^{'^2}(\mathcal{S},\mathfrak{g})$, where the test $1$-forms are taken to vanish to all orders on the boundary of $\mathcal{S}$. 
Moreover, the limit $A$ coincides with the random $1$-form constructed in the first part of the article.
\end{enumerate}
\end{proposition}

\subsection{Discrete Morse gauge}
The purpose of this section is to define a discrete analogue of the Morse gauge. Let us recall a few notation.
\begin{itemize}
    \item Call the graph obtained $\Lambda^N$, the set of its edges  $E^N$, and the set of its faces $F^N$.
    \item Denote by $H^N$ the set of its horizontal edges. This is the set of edges that are pieces of level sets of the Morse function.
    \item Denote by $V^N$ the set of vertical edges, i.e. that are pieces of flow lines of the gradient of the Morse function. 
\end{itemize}
We will assume that the edges are taken to be oriented in increasing in $r$ and clockwise in $\theta$. A configuration $M$ on $G^{E^N}$ can be indexed as  
\[(M_{H^N},M_{V^N})\coloneq \left(\left(M_{\theta\xrightarrow{r}\theta_+ }\right)_{\substack{\theta\in\Theta^N_-\\r\in R^N}},\left(M_{r\xrightarrow{\theta}{r_+} }\right)_{\substack{\theta\in\Theta^N\\r\in R^N_-}}\right).\]

A natural first attempt would be to set the connection to be trivial on all vertical edges. 
However, this cannot be achieved globally: the stable manifolds form closed loops (see Figure~\ref{MLatt}), and the holonomies along these loops define conjugacy classes in $G$ that are \emph{gauge-invariant}. 
If these holonomies were not already trivial, no gauge transformation can make them so. 
Thus, completely trivializing vertical edges would erase essential topological information. 
\begin{figure}
    \centering
    \includegraphics[width=\linewidth]{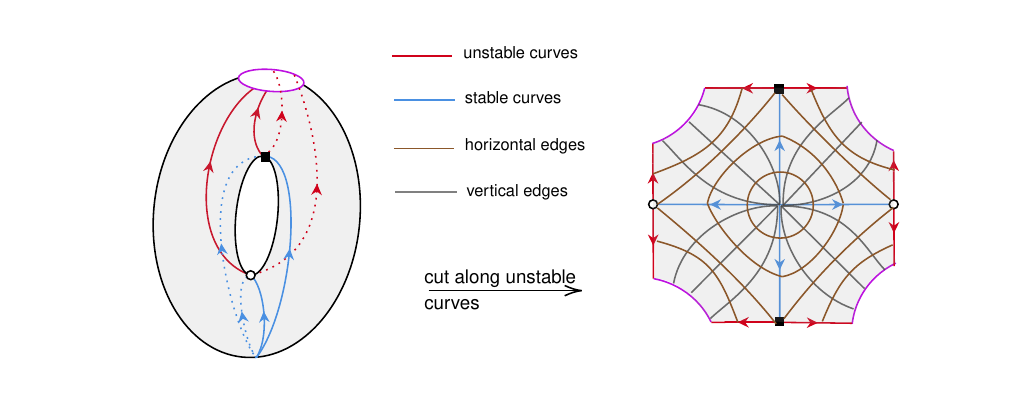}
    \caption{A Morse lattice.}
    \label{MLatt}
\end{figure}
This observation implies that a discrete connection in the Morse gauge should retain precisely two kinds of data:
\begin{enumerate}
    \item \emph{Face variables}, encoding the holonomy around each face of the Morse lattice. Since there are as many faces as horizontal edges, these holonomies can be concentrated on the horizontal edges $H^N$.
    \item \emph{Stable variables}, encoding the (conjugacy classes of) holonomies along the $2g$ closed stable curves.
\end{enumerate}
We get therefore the following definition of a discrete connection in the Morse gauge.
\begin{definition}[Discrete connection in Morse gauge]
    A configuration in the discrete Morse gauge on $\mathcal{S}$ is an element $(M,U)\in G^{H^N}\times G^{\mathrm{St}}$ where $H_N$ is the set of horizontal edges of $\Lambda_N$, and $\mathrm{St}$ is the set of the $2g$ closed stable curves. 
\end{definition}
The first component $M$ encodes local curvature data, while $U$ captures the global holonomy along the nontrivial cycles of the surface.
We now explain how such a pair $(M,U)$ determines a full discrete connection on $\Lambda^N$. 
Instead of assigning group elements to edges directly, it is convenient to view a discrete connection as an assignment of group elements to all closed loops in $\Lambda^N$ based at a fixed base point $0$. 
A natural basis for the free group $\mathrm{Loop}_0(\Lambda^N)$ consists of:
\begin{itemize}
    \item face loops $(l_F)_{F\in F^N}$, one around each face,
    \item stable loops $(s_i)_{1\le i\le 2g}$, one for each independent stable curve.
\end{itemize}
Given $(M,U)$, we define a family $H(M,U) := (h_c)_{c\in \mathrm{Loop}_0(\Lambda^N)}$ as follows: 
if a loop $c$ is expressed as a word in the basic loops, $h_c$ is the same word in the corresponding group elements (with $M$ assigned to face loops and $U$ to stable loops).

The main result of this section, which will motivate the definition of $A_N$ in the next section, is the following.

\begin{proposition}\label{DriverSenguptaDiscrete}
Let $M$ be a family of random variables, one for each horizontal edges with joint law
\[
\dd \mathbb{M}_N\left(\left(M_{\theta\xrightarrow{r}\theta_+ }\right)_{\substack{\theta\in\Theta^N_-\\r\in R^N}}\right)
\coloneq \frac{1}{Z_N}\prod_{\substack{\theta\in \Theta^N_-\\r\in R^N_-}}\rho_{\sigma\left(\square\left(r,r_+,\theta,\theta_+\right)\right)}\left(M_{\theta\xrightarrow{r}\theta_+}^{-1}M_{\theta\xrightarrow{r_+}\theta_+}\right)\dd M.
\]
Let $(U_i)_{1\le i\le 2g}$ be independent, uniformly distributed random variables in $G$, independent of $M$. 
Then $H(M,U)$ generates the same lattice gauge theory as the Driver--Sengupta measure.
\end{proposition}
\begin{proof}
    We need to calculate the joint law of 
    \[\mathrm{Hol}(l_F), \mathrm{Hol}(s_i); \ F\in\mathbb{F}_N, i=1,\cdots,2g,\]
    under the Driver--Sengupta measure.
    Let $f:G^{\mathbb{F_N}}\times G^{2g}\rightarrow \R$ be a central measurable function. From the Driver--Sengupta formula, we have
    \begin{align*}
        &\mathbb{E}\left[f\left(\mathrm{Hol}(l_F), \mathrm{Hol}(s_i); \ F\in\mathbb{F}_N, i=1,\cdots,2g\right)\right] \\
        =&\int_{G^\mathbb{E}}f\left (\prod_{e\in\partial F}g_e,\prod_{e\subset \mathrm{st}}g_e; \  F\in\mathbb{F}_N, \mathrm{st}\in\mathrm{St} \right) \prod_{F\in \mathbb{F}}\rho_{\sigma(F)}\big(\mathrm{Hol}(g,\partial F)\big) \, \dd g.
    \end{align*}
 We do the following change of variable:
 \[M_{\theta\xrightarrow{r_+} \theta_+}=\prod_{e\in \partial \square(r,r_+,\theta,\theta_+)}g_e \cdot M_{\theta\xrightarrow{r} \theta_+},\]
 and we get 
     \begin{align*}
        &\mathbb{E}\left[f\left(\mathrm{Hol}(l_F), \mathrm{Hol}(s_i); \ F\in\mathbb{F}_N, i=1,\cdots,2g\right)\right] \\
        =&\int_{G^{\mathbb{F_N}}\times G^{2g}}f\left (\left(M_{\theta\xrightarrow{r}\theta_+ }\right)_{\substack{\theta\in\Theta^N_-\\r\in R^N}},U_i; \  F\in\mathbb{F}_N, 1\leq i\leq 2g \right) \prod_{F\in \mathbb{F}}\rho_{\sigma\left(\square\left(r,r_+,\theta,\theta_+\right)\right)}\left(M_{\theta\xrightarrow{r}\theta_+}^{-1}M_{\theta\xrightarrow{r_+}\theta_+}\right)\, \dd M\dd U,
    \end{align*}
    which concludes the proof.
\end{proof}

\subsection{Fixed $N$ analysis}

Now let us see how to define $A_N$. Let $(M,U)$ be a configuration in Morse gauge on the lattice of resolution $N$. We want to define a $1-$form $A^{(N)}_{(M,U)}$ on the surface such that its holonomies agree with the Yang--Mills measure. 
\par We will start by decomposing $A^{(N)}_{(M,U)}$ as the sum of two one forms: one that lives on the bulk of the surface $A^{(N)}_{M,\Sigma}$, and one that lives on the unstable lines $A^{(N)}_{U,\Sigma}$ giving 
\[A^{(N)}_{(M,U)}=A^{(N)}_{M}+A^{(N)}_{U}.\]
The idea behind $A^{(N)}_{U}$ is that it encodes the random flat connection. It does not depend on $N$ or on $Q$, because as seen in Proposition \ref{DriverSenguptaDiscrete}, $A^{(N)}_{U}$ should set the conjugacy class of the holonomy around the stable lines. The easiest way to settle this is to put
\[A^{(N)}_{U}=\sum_{\mathrm{unst}\in\mathrm{Unst}}[U_\mathrm{unst}]\log g_{\mathrm{st}},\]
where the $g_i$ are independent Haar distributed $G$-valued random variables.
\par Then, we start by defining $A^{(N)}_{M}$ which is supposed to encode the noise. Start by defining it on the skeleton of the graph by putting 
\[A^{(N)}_{M}(x)\coloneq \left(\sum_{e \in \mathrm{edges}}\frac{\log M_e}{\mathrm{length}(e)} 1_{x\in e}\right)\dd \theta,\]
This gives, for all  $r\in R^N$, a piecewise constant function $\theta \mapsto A^{(N)}_{M}(r,\gamma)$ which we interpolate in a piecewise affine function in $r\in R^N$. This gives
\[A^{(N)}_{M}(u,\gamma)=2^N\sum_{r\in R^N_-}\left((u-r)A^{(N)}_{M}(r_+,\gamma) + (r_+-u)A^{(N)}_{M}(r,\gamma)\right)1_{[r,r_+]}(u).\]
It is clear that the holonomy induced by $A^{(N)}_{(M,U)}$ on the horizontal edges is exactly the configuration $(M,U)$. Now, consider the mapping 
\[D_N:(M,U)\mapsto A^{(N)}_{(M,U)}.\]
The push forward $(D_N)_*\mathbb{M}_{N}$ of the discrete Morse gauge fixed Yang--Mills measure gives rise to a random connection  in $\Omega^1(\Sigma,\g)$.  We would like to study the convergence of this sequence of random connections to the Yang--Mills measure defined in the formula in (\ref{YMFormula}). 
\subsection{Limit in law}
The idea is to study the limit of 
$A^{(N)}_{(M,U)}=A^{(N)}_{M}+A^{(N)}_{U}$, seen as a sequence of probability measure in the set of $\g$-valued currents of degree $1$ on $\Sigma$. This set denoted by $\mathcal{D}^{',1}(\Sigma,\g)$, is the topological dual of the set $\Omega_0^1(\Sigma,\g)$ of smooth $\g$-valued $1-$forms on $\Sigma$, vanishing up to all orders at $\partial \Sigma$, equipped with the $C^\infty-$topology. Roughly speaking, currents have the same properties as the Schwartz distribution, and are differential forms with distributional coefficients. 
\par Since the two terms of the decomposition are (statistically) independent, they can be studied separately from the point of view of convergence in law.
\par It is clear that since
$A^{(N)}_{U}$ is independent of $N$, it converges almost surely to the unstable currents, the leftmost part of (\ref{YMFormula}). The remaining part is to study the convergence of first term. Indeed, we will show the following theorem in what remains of this subsection. 
\begin{thm}\label{convAn}
    The sequence of $1-$forms $(A^{(N)}_{M})_{N\geq 1}$ converges in  $\mathcal{M}\left(\mathcal{D}^{'1}(\mathrm{Cyl},\g)\right)$ to the random $1-$form 
  $  \left(\partial_\theta\left\langle \xi, 1_{\square(r,\theta)} \right\rangle \right) \dd \theta$,
    defined in (\ref{YMFormula}).
\end{thm}
The proof requires several steps. The first one is to note that in the continuum, the gauge fixed Yang--Mills measure was nothing but a de Rham primitive of a white noise. There is a discrete analogue of this fact which will make the computations easier. In fact, let
\begin{eqnarray*}
\xi_N(u,\gamma)&=&2^{2N}\sum_{\substack{r\in R^N_-\\\theta \in \Theta^N_-}}\left(\log\left(M_{\theta\xrightarrow {r_+} \theta_+}  \right) - \log\left(M_{\theta\xrightarrow {r} \theta_+}  \right)\right)1_{[\theta,\theta_+]}(\gamma)1_{[r,r_+]}(u)  ,\\
\end{eqnarray*}
which is a discrete approximation of the white noise. We have \
\[\left(\int_0^t\xi_N(u,\gamma)\dd u \right) \dd \gamma= A^{(N)}_{M}(t,\gamma). \]
This means that  $A^{(N)}_{U}=\iota_V\mathcal{L}_V^{-1}\left(\xi_N \dd u\wedge \dd \gamma \right)$,
where $\iota_V\mathcal{L}_V^{-1}$ is just integration w.r.t. the $r$ variable in our systems of pseudo-coordinates, and $\mathcal{L}_V^{-1}$ is the inverse of the Lie derivative operator $\mathcal{L}_V$, where $V=\nabla f$ is the gradient of the Morse flow. Therefore, the problem is reduced to study the convergence of $\xi_N$ and to prove the continuity of $\iota_V\mathcal{L}_V^{-1}$. Let's start with the continuity.

\begin{lemma}\label{continuity_ivlv}
Let $V=\partial_r$ be the radial vector field in the cylinder $\mathbf{Cyl}$.
    The linear function 
    \[\iota_V\mathcal{L}_V^{-1}: \mathcal{D}^{',2}(\mathbf{Cyl},\g) \rightarrow  \mathcal{D}^{',1}(\mathbf{Cyl},\g)\] is continuous, 
    where similarly to $\mathcal{D}^{',1}(\mathbf{Cyl},\g)$, the set $\mathcal{D}^{',2}(\mathbf{Cyl},\g)$ is the topological dual of smooth $2-$forms on $\mathbf{Cyl}$ vanishing up to all orders at $\partial \mathbf{Cyl}$.
\end{lemma}
\begin{proof}
  The definition of $\iota_V\mathcal{L}_V^{-1}$ is just by duality, and the continuity is straightforward.
\end{proof}
\begin{remark}
In fact we can tell a bit more.
For the moment, the operator $\iota_V\mathcal{L}_V^{-1}$ is defined only on distributions that are dual to test functions that do not see the boundary of the cylinder since they vanish at infinite order. We can extend $\iota_V\mathcal{L}_V^{-1}$ to our anisotropic spaces $\mathcal{W}^{\alpha,\alpha-1,p}_{loc}(\mathbf{Cyl}\setminus \Psi(\mathrm{Crit}(f)_1))$, $\alpha\in (0,\frac{1}{2})$ of distributions that feel the boundary. We prove in the appendix the fact that this space injects continuously in $\mathcal{C}^{\beta}_{loc}(\mathbf{Cyl}\setminus \Psi(\mathrm{Crit}(f)_1) )$ for some $\beta>-1$.
    For any $\beta>-1$, the linear map 
    \[\iota_V\mathcal{L}_V^{-1}: \mathcal{C}^\beta_{loc} ( \mathbf{Cyl}\setminus \Psi(\mathrm{Crit}(f)_1) ) \rightarrow  \mathcal{C}_{loc}^{\beta} (\mathbf{Cyl}\setminus \Psi(\mathrm{Crit}(f)_1)) \] is linear continuous.
    Let $T$ be a distribution on $\mathbb{R}$ which is in the H\"older--Besov space $\mathcal{C}^{\beta}_{loc}$ for $\beta>-1$. Then we 
can define the distribution $T 1_{\mathbb{R}\geq 0}(r)$ by restriction to test functions in $C^\infty_c(\mathbb{R}_{\geq 0})$ then  by extension procedure by $0$ on $\mathbb{R}_{>0}$, in such a way that the restriction plus extension map
is linear continuous.
The proof is just a consequence of Lemma~\ref{lem:pairing} in the Appendix.
\end{remark}

The second step, is to study the convergence of $(\xi_N)_{N\geq 1}$ when seen as a sequence of measures in $\mathcal{D}^{'2}(\Sigma,\g)$. To show this, we will use Fernique's theorem~\cite[Théorème III.6.5]{Fernique}. This theorem states that a sequence of random distribution converges in law --in the topology of Schwartz distributions -- if and only if the sequence of characteristic functions converges to a continuous functional. Therefore, we need to study the convergence of 
\[ \mathbb{E}\left[\exp(i\langle \xi_N,\psi \rangle\right)],\]
    for some test function $\psi$. This is what we will do in the remaining of this section, and we will use the results of Section~\ref{techest}. Let us first prove a key lemma explaining the general form of $\mathbb{E}\left[\exp(i\langle \xi_N,\psi \rangle\right)]$.
\begin{lemma} 
Let $\psi\in C^\infty(\mathrm{Cyl},\g)$. Define 
   \[\forall r_1,r_2,\theta_1,\theta_2, \ \ \ K_{r_1,r_2,\theta_1,\theta_2}\psi\coloneq \frac{1}{(r_2-r_1)(\theta_2-\theta_1)}\int_\Sigma \psi(u,\gamma)     1_{[\theta_1,\theta_2]}(\gamma)1_{[r_1,r_2]}(u) \dd u\dd \gamma .\]
    We have 
        \begin{eqnarray*}
        \mathbb{E}\left[\exp(i\langle \xi_N,\psi \rangle\right)]
        &=&\prod_{\theta\in \Theta^N_-}\left\{\prod_{r\in R^N_-}\left(1-\frac{\vert K_{r,r_+,\theta,\theta_+}\psi\vert ^2}{2}\sigma(\square(r,r_+,\theta,\theta_+))\right)+O_{N,\theta}\right\},
    \end{eqnarray*}
    where 
    $\sum_{\theta \in \Theta^N_-}\vert O_{N,\theta }\vert  \xrightarrow[N\to \infty]{}0$.
\end{lemma}
\begin{proof}
    We have 
    \[\left \langle \xi_N,\psi \right\rangle =\sum_{\substack{r\in R^N_-\\\theta \in \Theta^N_-}}\left\langle\log\left(M_{\theta\xrightarrow {r_+} \theta_+}  \right) - \log\left(M_{\theta\xrightarrow {r} \theta_+}  \right),K_{r,r_+,\theta,\theta_+}\psi\right\rangle.\]
 
    Let us calculate the characteristic function. We have from the independence in $\theta$, and from Lemma~\ref{EstimBruitBlanc}, 
    \begin{eqnarray*}
        \mathbb{E}\left[\exp(i\langle \xi_N,\psi \rangle\right)]&=&\prod_{\theta\in \Theta^N_-}\mathbb{E}\left[\prod_{r\in R^N_-}\exp \left(i\left\langle\log\left(M_{\theta\xrightarrow {r_+} \theta_+}  \right) - \log\left(M_{\theta\xrightarrow {r} \theta_+}  \right),K_{r,r_+,\theta,\theta_+}\psi\right\rangle\right)\right]\\
        &=&\prod_{\theta\in \Theta^N_-}\left\{\prod_{r\in R^N_-}\left(1-\frac{\vert K_{r,r_+,\theta,\theta_+}\psi\vert ^2}{2}\sigma(\square(r,r_+,\theta,\theta_+))\right)+O_{N,\theta}\right\},
    \end{eqnarray*}
    where 
    \begin{eqnarray*}
        O_{N,\theta}&= &O\Bigg\{\Big(\sum_{r\in R^N_-}\sigma(r,r_+,\theta,\theta_+)\Big)^{\frac 32} + \sum_{r\in R^N_i}\sigma(r,r_+,\theta,\theta_+)^{\frac 12}\Big(\sum_{r\in R^N_-}\sigma(r,r_+,\theta,\theta_+)\Big)^p\\
        &+& \sum_{r\in R^N_-}\sigma(r,r_+,\theta,\theta_+)^p+\sum_{r\in R^N_-}\sigma(r,r_+,\theta,\theta_+)^{\frac 32}+\Big(\sum_{r\in R^N_-}\sigma(r,r_+,\theta,\theta_+)\Big)^{2}\Bigg\}
    \end{eqnarray*}
    We have 
     \begin{eqnarray*}
        O_{N,\theta}&= &O\Bigg\{\max_{\theta \in \Theta^N_-} \left(\sigma(0,2g+2,\theta,\theta_+)\right)^{\frac 12}\Big(\sum_{r\in R^N_-}\sigma(r,r_+,\theta,\theta_+)\Big) +2^{-\frac{N}{2}}2^{-N\beta p} 2^N\\
        &+& \big( \sup_{\substack{r\in R^N_- \\ \theta \in \Theta^N_-}}\sigma(r,r^+,\theta,\theta^+)\big)^{p-1}\sigma(0,2g,\theta,\theta_+)+\big( \sup_{\substack{r\in R^N_- \\ \theta \in \Theta^N_-}}\sigma(r,r^+,\theta,\theta^+)\big)^{\frac 12}\sigma(0,2g,\theta,\theta_+)\\
        &+&\max_{\theta \in \Theta^N_-} \left(\sigma(0,2g+2,\theta,\theta_+)\right)\Big(\sum_{r\in R^N_-}\sigma(r,r_+,\theta,\theta_+)\Big)\Bigg\},
    \end{eqnarray*}
    where for the second line, we have used Lemma~\ref{areaestim}.
    Now, 
         \begin{eqnarray*}
        \sum_{\theta\in \Theta^N_-}\vert O_{N,\theta}\vert &\lesssim & \max_{\theta \in \Theta^N_-} \left(\sigma(0,2g+2,\theta,\theta_+)\right)^{\frac 12}\sigma(\Sigma) + 2^{-\frac{N}{2}}2^{-N\beta p} 2^N+ \big( \sup_{\substack{r\in R^N_- \\ \theta \in \Theta^N_-}}\sigma(r,r^+,\theta,\theta^+)\big)^{p-1}\sigma(\Sigma)\\
        &+&\big( \sup_{\substack{r\in R^N_- \\ \theta \in \Theta^N_-}}\sigma(r,r^+,\theta,\theta^+)\big)^{\frac 12}\sigma(\Sigma)+\max_{\theta \in \Theta^N_-} \left(\sigma(0,2g+2,\theta,\theta_+)\right) \sigma(\Sigma),
    \end{eqnarray*}
    which for $p$ chosen big enough, and using Lemma~\ref{areaestim}, gives 
    $\sum_{\theta \in \Theta^N_-}\vert O_{N,\theta }\vert  \xrightarrow[N\to \infty]{}0$.
\end{proof}

We still need one more lemma before being ready to prove the convergence of $\xi_N$.
 \begin{lemma}
    There exists a constant $C>0$ that depends on $\psi$ such that for all $N\geq 1$, all $\theta\in \Theta^N_-$, 
    \[\prod_{r\in R^N_-}\left(1-\frac{\vert K_{r,r_+,\theta,\theta_+}\psi\vert ^2}{2}\sigma(\square(r,r_+,\theta,\theta_+))\right)\geq C\]
\end{lemma}
\begin{proof}
    We have 
    $\vert K_{r,r_+,\theta,\theta_+}\psi\vert\leq \|\psi\|_{\infty},$
    and 
    \[\sigma(\square(r,r_+,\theta,\theta_+)\leq \sup_{r\in R^N_-,\theta\in\Theta^N_-}\sigma(\square(r,r_+,\theta,\theta)\leq C2^{-N}.\]

    Therefore,
\[\prod_{r\in R^N_-}\left(1-\frac{\vert K_{r,r_+,\theta,\theta_+}\psi\vert ^2}{2}\sigma(\square(r,r_+,\theta,\theta_+))\right)\geq \prod_{r\in R^N_-}\left(1-\frac{C\|\psi\|_{\infty}^2}{2}2^{-N}\right). \]
However,
\begin{eqnarray*}
    \log \prod_{r\in R^N_-}\left(1-\frac{C\|\psi\|_{\infty}^2}{2}2^{-N}\right)&=&\sum_{r\in R^N_-} \log\left(1-\frac{C\|\psi\|_{\infty}^2}{2}2^{-N}\right) =(2g+1)2^N\log\left(1-\frac{C\|\psi\|_{\infty}^2}{2}2^{-N}\right)\\
    &\xrightarrow[N\to\infty]{} &\exp\left(-\frac{C(2g+1)\|\psi\|_{\infty}^2}{2}\right) >0.
\end{eqnarray*}
\end{proof}

We are able now to conclude the proof of the convergence.

\begin{proposition}\label{whitenoiseconv}
    The sequence $(\xi_N \dd r\wedge \dd \theta )_{N\geq 1}$ converges in $\mathcal{M}\left(\mathcal{D}^{'2}(\mathrm{Cyl},\g)\right)$ to the white noise, seen as a $2-$form.
\end{proposition}
\begin{proof} 
    By the previous lemmas,
    \begin{eqnarray*}
          \log \mathbb{E}\left[\exp(i\langle \xi_N,\psi \rangle\right)]&=&\sum_{\theta\in \Theta^N_-}\log\left\{\prod_{r\in R^N_-}\left(1-\frac{\vert K_{r,r_+,\theta,\theta_+}\psi\vert ^2}{2}\sigma(\square(r,r_+,\theta,\theta_+))\right)\right\}\\
          &+&\sum_{\theta\in \Theta^N_-}\log\left (1+\frac {O_{N,\theta}}{\prod_{r\in R^N_-}\left(1-\frac{\vert K_{r,r_+,\theta,\theta_+}\psi\vert ^2}{2}\sigma(\square(r,r_+,\theta,\theta_+))\right)}\right)\\
          &\xrightarrow[N\to\infty]{}&-\frac{\|\psi\|^2_{L^2(\Sigma,\sigma)}}{2}
    \end{eqnarray*}
   Since the second term converges to $0$ and 
    \[\log\left(1-\frac{\vert K_{r,r_+,\theta,\theta_+}\psi\vert ^2}{2}\sigma(\square(r,r_+,\theta,\theta_+))\right)=-\frac{\vert K_{r,r_+,\theta,\theta_+}\psi\vert ^2}{2}\sigma(\square(r,r_+,\theta,\theta_+))+O(\sigma(\square(r,r_+,\theta,\theta_+))^2)\]
    and therefore, 
    \begin{align*}
       & \lim_{N\to\infty}\sum_{\substack{r\in R^N_-\\\theta\in \Theta^N_-}}\log\left(1-\frac{\vert K_{r,r_+,\theta,\theta_+}\psi\vert ^2}{2}\sigma(\square(r,r_+,\theta,\theta_+))\right)=\lim_{N\to\infty} -\sum_{\substack{r\in R^N_-\\\theta\in \Theta^N_-}}\frac{\vert K_{r,r_+,\theta,\theta_+}\psi\vert ^2}{2}\sigma(\square(r,r_+,\theta,\theta_+)),
    \end{align*}
    
    which by classical martingale convergence arguments, is equal to $-\frac{\|\psi\|^2_{L^2(\mathrm{Cyl},\sigma)}}{2}=-\frac{\|\psi\|^2_{L^2(\Sigma,\sigma)}}{2}$.
\end{proof}
\begin{proof}[proof of Theorem~\ref{convAn}]
It is a combination of Lemma~\ref{continuity_ivlv}, Proposition~\ref{whitenoiseconv} and Fernique's theorem.
\end{proof}

\section{Convergence in functional spaces}\label{s:functspaces}
Consider again a surface $\Sigma$ with one outgoing boundary component. We will study the convergence of the free lattice Yang--Mills measures in more tailored functional spaces. 
\par In the previous section, we showed that for each $N$, $A^{(N)}_{(M,U)}$ can be decomposed into the sum of two independent functions $A^{(N)}_{M}$ and $A^{(N)}_{U}$: one living on the bulk of the surface, and the other one living on the unstable curves. Since the latter is constant, it converges, as $N$ goes to infinity, almost surely to the unstable currents associated to unstable curves. These currents belong to $\mathcal{C}^{-1-\epsilon}$, and we can say that the convergence of $A^{(N)}_{U}$ happens as well in these spaces. It remains to study the convergence of $(A^{(N)}_{M})_{N\geq 0}$ in functional spaces. To alleviate the notations in this section, we will denote $A^{(N)}_{M}$ simply by $A_{N}$.

\par Since we have already shown convergence in law, i.e. we identified the limit, the only missing argument is tightness. As scales of functional spaces enjoy compact embeddings from smaller into bigger spaces (which we will be showed for our special case in Appendix~\ref{C}), we only need to show that the sequence of the norms of $A_N$ is bounded independently of $N$. It is therefore crucial to identify functional spaces that reflects the anisotropic regularity of the limit connection on the one hand, and enables easy computations on the other hand. 

\subsection{Discretizing the spaces of distributional connections}
\begin{figure}
    \centering
    \includegraphics[width=1.1\linewidth]{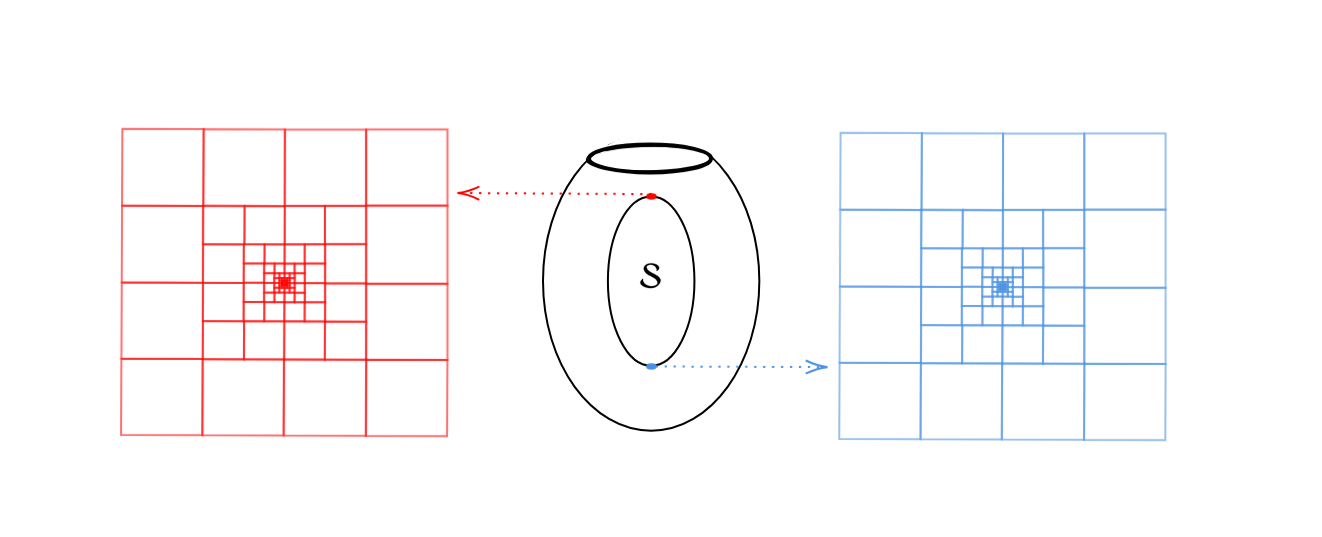}
    \caption{Corona formed by a collection of squares.}
    \label{fig:cor}
\end{figure}

In Section \ref{DistSpaces}, we introduced some norms on smooth connections, and defined spaces of distributional connections by completion of smooth $1$-forms with respect to these norms. The goal of this section is to explain how can one compute these norms for the special case of a connection that is piecewise constant in $\theta$ and piecewise affine in $r$, as is the case for $A_N$ introduced in the previous section. Let us first recall that for $\tilde{A}\in \Omega^1(\Sigma,\g)$, 
\[W(\tilde{A})(r,\theta)=\int_0^\theta A(0\xrightarrow r \theta)\dd \theta.\]
and 
\begin{equation}\label{eq:weightedHolder}
\Vert \tilde{A}\Vert_{\mathcal{W}^{\alpha,\alpha-1,p,s}_{r,\theta}}\coloneq   \left(\sum_{a\in\mathrm{Crit}(f)}\sum_{I_a\times J_a} \sum_{n=0}^\infty \left(2^{n(s-\frac{2}{p})}\Vert \mathcal{S}^{2^{-n} *}_{(r_0,\theta_0)} W(\tilde{A})\Vert_{\mathcal{W}^{\alpha,\alpha;p}_{r,\theta}(I\times J) }\right)^p \right)^{\frac{1}{p}}, 
 \end{equation}
where we take a sum over all critical points, and for each critical point $a$, the sum over a \emph{finite cover} of the form $I_a^r\times J_a^\theta$ of a certain corona of the form $\{ m\in \Sigma; \mathrm{dist}(m,\Psi(a))\in [1,2] \}$ centered near a singular point $\Psi(a)=(r_0,\theta_0)$. The factor $S^{-n}$ means that we are testing closer and closer to the critical point. The Figure \ref{fig:cor} explains the situation: the centers of the grids represent the critical points, and the blue and red lattices show the base corona formed of $12$ squares that we scale closer and closer to the critical point.

It is therefore enough to study $\Vert A_N\Vert_{\mathcal{W}^{\alpha,\alpha-1,p,s}_{r,\theta}}$ near an arbitrary  fixed saddle point, that we will assume has coordinates $(0,0)$ and for a fixed square  that we will call $I\times J$. Recall that $W(A_N)$ is piecewise affine in both directions, on a grid of mesh $2^{-N}$. To study its norm on a square of mesh $2^{-n}$, we should separate two cases: 

\begin{align*}
    \sum_{n=0}^\infty &\left(2^{n(s-\frac{2}{p})} \Vert \mathcal{S}^{2^{-n} *}_{(0,0)} W\Vert_{\mathcal{W}^{\alpha,\alpha;p}_{r,\theta}(I\times J) }\right)^p \\
    =&\sum_{n\leq N} \left(2^{n(s-\frac{2}{p})}\Vert \mathcal{S}^{2^{-n} *}_{(0,0)} W\Vert_{\mathcal{W}^{\alpha,\alpha;p}_{r,\theta}(I\times J) }\right)^p \\
    &+\sum_{n> N} \left(2^{n(s-\frac{2}{p})}\Vert \mathcal{S}^{2^{-n} *}_{(0,0)} W\Vert_{\mathcal{W}^{\alpha,\alpha;p}_{r,\theta}(I\times J) }\right)^p,
\end{align*}

\paragraph{The case $n<N$. } In this case, the problem reduces to bound the $\mathcal{W}^{\alpha,\alpha;p}_{r,\theta} $ norm of a piecewise linear function in both direction. In fact, since the $N$ resolution is larger than the $n$ resolution, the situation is as in Figure \ref{fig:case1}, where the red square is one of the squares of the corona, and the black squares are the region in which $W(A_N)$ is affine. 
\begin{figure}
    \centering
    \includegraphics[width=0.5\linewidth]{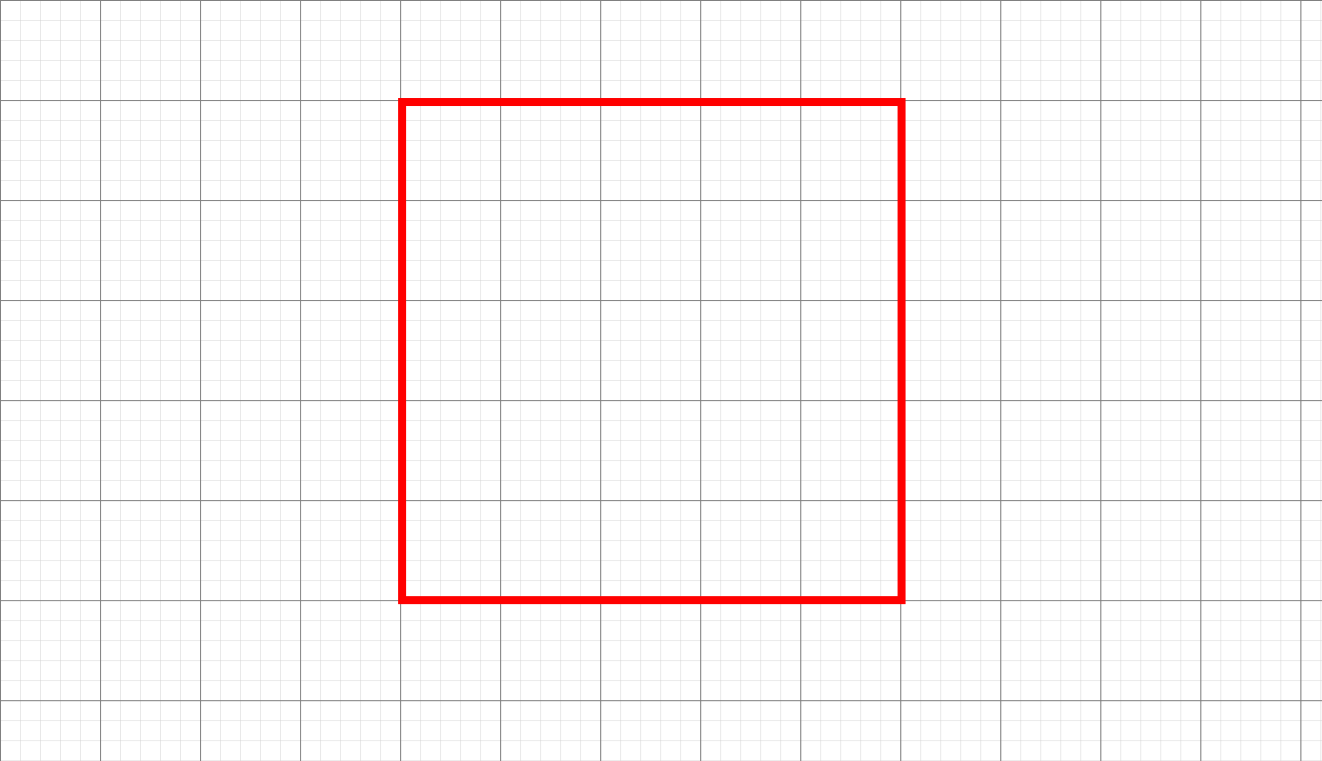}
    \caption{Resolution $N$ (black) bigger than resolution $n$ (red).}
    \label{fig:case1}
\end{figure}
We have that
\[\Vert \mathcal{S}^{2^{-n} *}_{(0,0)} W\Vert_{\mathcal{W}^{\alpha,\alpha;p}_{r,\theta}(I\times J) }=\int_{I_1\times I_1}\int_{I_1\times I_2}\frac{\left\vert A(2^{-n}\theta \xrightarrow{2^{-n}r'}2^{-n}\theta')-A(2^{-n}\theta \xrightarrow{2^{-n}r'}2^{-n}\theta')\right\vert ^p}{\vert r-r'\vert ^{1+p\alpha}\vert \theta-\theta'\vert ^{1+p\alpha}}\dd r\dd r'\dd \theta \dd \theta',\]
and by a change of variables, we get 
\[\Vert \mathcal{S}^{2^{-n} *}_{(0,0)} W\Vert_{\mathcal{W}^{\alpha,\alpha;p}_{r,\theta}(I\times J) }=2^{2n(1-p\alpha)}\int_{2^{-n}I_1\times 2^{-n}I_1}\int_{2^{-n}I_2\times 2^{-n}I_2}\frac{\left\vert A(\theta \xrightarrow{r'}\theta')-A(\theta \xrightarrow{r'}\theta')\right\vert ^p}{\vert r-r'\vert ^{1+p\alpha}\vert \theta-\theta'\vert ^{1+p\alpha}}\dd r\dd r'\dd \theta \dd \theta'.\]
Therefore, the problem reduces to understand a $1-$dimensional analogue: the behavior of a classical one dimensional Gagliardo norms for a piecewise affine function. This is done in the next lemma.
\begin{lemma} \label{Kepsilon}
   Let $M\in \mathbb{N}$. Let $f:[0,1]\rightarrow \g$ be  piecewise affine function such that $f$ is linear between $f(\frac{k}{M})$ and $f(\frac{k+1}{M})$ for all $k=0,\dots,M-1$. Then,  
\[\|f\|^p_{s,p} \coloneq \int_{[0,1]^2}\frac{\vert f(t)-f(s)\vert_\g ^p}{\vert t-s\vert ^{1+ps}}\dd t\dd s\leq \frac{12^p}{M^{1-ps}}\sum_{0\leq k < l \leq M}\frac{\left\vert f(\frac{k}{M})-f(\frac{l}{M})\right\vert^p}{(k-l)^{1+ps}}.\]

\end{lemma}
\begin{proof}
   
    We have 
    \begin{eqnarray*}
        &&\|f\|^p_{s,p}=\sum_{0\leq k,l\leq 2^N-1}\int_{\left[\frac{k}{M},\frac{k+1}{M}\right]\times \left[\frac{l}{M},\frac{l+1}{M}\right]}\frac{|f(x)-f(y)|^p}{|x-y|^{1+ps}}dxdy \\
        &=&\sum_{0\leq k\leq M-1}\int_{\left[\frac{k}{M},\frac{k+1}{M}\right]^2}\frac{|f(x)-f(y)|^p}{|x-y|^{1+ps}}dxdy +  \sum_{0\leq p\neq q\leq M-1}\int_{\left[\frac{k}{M},\frac{k+1}{M}\right]\times \left[\frac{l}{M},\frac{l+1}{M}\right]}\frac{|f(x)-f(y)|^p}{|x-y|^{1+ps}}dxdy.
    \end{eqnarray*}
    For $x\in \left[\frac{k}{M},\frac{k+1}{M}\right]$, and $y\in \left[\frac{l}{M},\frac{l+1}{M}\right]$, $$f(x)=M(a_{k+1}-a_k)\left(x-\frac{k}{M}\right) + a_k \text{ and } f(y)=M(a_{l+1}-a_l)\left(y-\frac{l}{M}\right) + a_l.$$
    Now, consider the three different cases below : 
    \begin{itemize}
        \item If $k=l$, then 
        $\vert f(x)-f(y)\vert =M \vert a_{k+1}-a_k \vert \vert x-y\vert ,$
        and 
        $$\frac{\vert f(x)-f(y)\vert^p}{\vert x-y\vert ^{1+ps }} \leq M^p \vert a_{k+1}-a_k\vert ^p\vert x-y\vert ^{p-1-ps} \leq M^{1+ps}\vert a_{k+1}-a_k\vert^p ,$$
        and therefore 
        $$\sum_{0\leq k\leq M-1}\int_{\left[\frac{k}{M},\frac{k+1}{M}\right]^p}\frac{|f(x)-f(y)|^p}{|x-y|^{1+ps}}dxdy\leq M^{ps-1}\sum_{0\leq k\leq M-1} \vert a_{k+1}-a_k\vert^2 \leq K.$$
        \item If not then 
        \begin{eqnarray*}
            \vert f(x)-f(y)\vert &\leq& \vert f(x)-a_{k+1}\vert + \vert a_{k+1}-a_{l} \vert + \vert a_{l}-f(y) \vert \\
            &\leq& \vert a_{k+1}-a_k \vert +\vert a_{k+1}-a_{l} \vert  +   \vert a_{l+1}-a_{l} \vert 
        \end{eqnarray*}
        and 
        $$\vert f(x)-f(y)\vert ^p \leq 3^p\left(\vert a_{k+1}-a_k\vert^2+\vert a_{l+1}-a_l\vert^2 + \vert a_{k+1}-a_{l} \vert^2 \right)$$
        since $\vert x-y\vert > (l-(k+1))$, then
        $$\frac{\vert f(x)-f(y)\vert ^p }{\vert x-y\vert ^{1+ps}} \leq 3^p\left(M^{1+ps}\vert a_{p+1}-a_p\vert^p+M^{1+ps}\vert a_{q+1}-a_q\vert^p + M^{1+ps}\frac{\vert a_{k+1}-a_{l} \vert^p}{\left( l-(k+1)\right)^{1+ps}} \right). $$
      Thus, 
      $$\sum_{0\leq p\neq q\leq 2^N-1}\int_{\left[\frac{p}{2^N},\frac{p+1}{2^N}\right]\times \left[\frac{q}{2^N},\frac{q+1}{2^N}\right]}\frac{|f(x)-f(y)|^2}{|x-y|^{1+2s}}dxdy \leq 3\times 3^pK,$$
    \end{itemize}
    and finally $\|f\|_{s,p} \leq 4\times 3^pK$. 
\end{proof}
\begin{lemma} \label{Membre1}
    For $N\geq  n$, we have 
    \begin{align*}
        \int_{2^{-n}I_1\times 2^{-n}I_1}&\int_{2^{-n}I_2\times 2^{-n}I_2}\frac{\left\vert A(\theta \xrightarrow{r'}\theta')-A(\theta \xrightarrow{r'}\theta')\right\vert ^p}{\vert r-r'\vert ^{1+p\alpha}\vert \theta-\theta'\vert ^{1+p\alpha}}\dd r\dd r'\dd \theta \dd \theta'\\
        &\leq \frac{12^{2p}}{2^{2N(1-ps)}}\sum_{\substack{2^{N-n}\leq k \leq l\leq 2^{N-n+1}\\2^{N-n}\leq p \leq q\leq 2^{N-n+1}}} \frac{\left\vert A_N(k2^{-N}\xrightarrow{q2^{-N}} l2^{-N})-A_N(k2^{-N}\xrightarrow{p2^{-N}} l2^{-N})\right\vert^p }{(l-k)^{1+p\alpha}(q-p)^{1+p\alpha}}.
    \end{align*}
    
\end{lemma}
\begin{proof}
We apply Lemma $\ref{Kepsilon}$ two times, once in $r$ and once in $\theta$.
\end{proof}
\paragraph{The case $n\geq N$.} In this case, the resolution $N$ is bigger than the resolution $n$, as depicted in Figure \ref{fig:Case2}, where the red squares are squares of the corona, and the black square is a region in which $W(A_N)$ is affine. This case reduces to computing a classical $1$-dimensional Gagliardo norm for (not anymore piecewise) affine function. 
\begin{figure}
    \centering
    \includegraphics[width=0.5\linewidth]{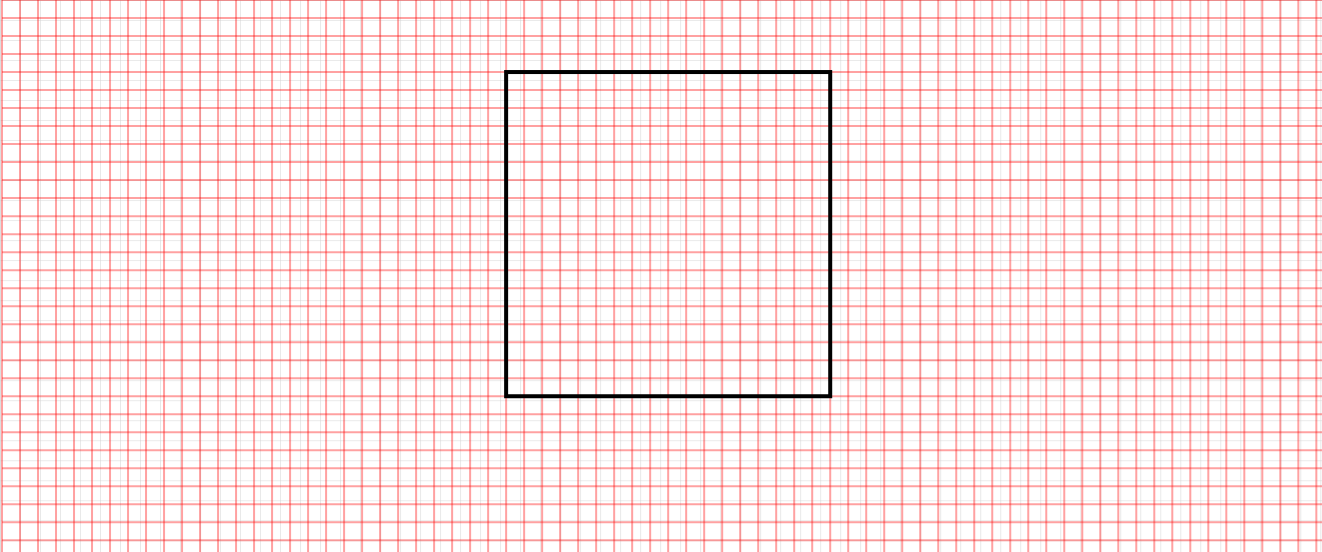}
    \caption{Resolution $N$ (black) bigger than resolution $n$ (red).}
    \label{fig:Case2}
\end{figure}
In this case, we have the following proposition.
\begin{proposition} \label{Membre2}
    We have    
    \begin{align*}
        \int_{2^{-n}I_1\times 2^{-n}I_1}&\int_{2^{-n}I_2\times 2^{-n}I_2}\frac{\left\vert A(\theta \xrightarrow{r'}\theta')-A(\theta \xrightarrow{r'}\theta')\right\vert ^p}{\vert r-r'\vert ^{1+p\alpha}\vert \theta-\theta'\vert ^{1+p\alpha}}\dd r\dd r'\dd \theta \dd \theta'\\
        & \lesssim  2^{2p(N-n)}\left\vert A(0\xrightarrow {2^{-N}} 2^{-N})-A(0\xrightarrow 0 2^{-N}) \right\vert^{p}.
    \end{align*}
\end{proposition}
 The proof consist in applying two times in a row the next lemma. 
 \begin{lemma}
     Let $f:[0,1]\rightarrow \R$ be a piecewise affine function, linear on each interval $[k2^{-N}, (k+1)2^{-N}]$. Then, for $n\geq N$, 
     \[\|S^{2^{-n}}f\|_{\mathcal{W}^{\alpha;p}}\lesssim 2^{p(N-n)}\vert f(2^{-N})-f(0)\vert ^p \]
 \end{lemma}
\begin{proof}
We have 
\begin{align*}
    \|S^{2^{-n}}f\|_{\mathcal{W}^{\alpha;p}([\frac{1}{2},1])}&=\int_{[\frac{1}{2},1]\times [\frac{1}{2},1]}\frac{\vert f_N(2^{-n}t)-f_N(2^{-n}s)\vert ^p}{\vert t-s\vert ^{1+p\alpha } }\dd s \dd t\\
    &=2^{p(N-n)}\vert f(2^{-N})-f(0) \vert ^p \int_{[\frac{1}{2},1]\times [\frac{1}{2},1]}\vert t-s\vert ^{p-1-\alpha p} \dd t \dd s .
\end{align*}
\end{proof}
Finally, we obtain the desired bound in the following proposition.  
\begin{proposition} \label{MainBound}
    For all $N\geq 0$, we have 
    \begin{align*}\|A_N&\|_{\mathcal{W}^{\alpha;\alpha-1;p;s}(\mathrm{Cyl},\g)} \\
    &\leq 
  \sum_{0\leq n \leq N }  \frac{2^{np(s-2\alpha)}}{2^{2N(1-p\alpha)}}\sum_{\substack{ 2^{-n+N-2} \leq n < m \leq 2^{-n+N-1}\\  2^{-n+N-2}\leq k , l \leq 2^{-n+N-1}  }}\frac{\left\vert A_N(k2^{-N}\xrightarrow{m2^{-N}} l2^{-N})-A_N(k2^{-N}\xrightarrow{n2^{-N}} l2^{-N})\right\vert^p}{\vert l-k \vert ^{1+p\alpha}(m-n)^{1+p\alpha}}    \\
  &\ \ \  +2^{2pN(s-2\alpha)}\left\vert A(0\xrightarrow {2^{-N}} 2^{-N})-A(0\xrightarrow 0 2^{-N}) \right\vert^{p}.
\end{align*}
\end{proposition}
\begin{proof}
    Combination of Lemmas \ref{Membre1} and \ref{Membre2}.
\end{proof}
\subsection{Probabilistic bounds}
Based on Proposition \ref{MainBound}, to be able to bound $\mathbb{E}[\|A_N\|_{\mathcal{W}^{\alpha;\alpha-1;p;s}(\mathrm{Cyl},\g)}]$, we need to understand how to bound 
\[\mathbb{E}\left[  \left\vert A_N(k2^{-N}\xrightarrow{m2^{-N}} l2^{-N})-A_N(k2^{-N}\xrightarrow{n2^{-N}} l2^{-N})\right\vert^p  \right].\]
Recall from the definition of $A_N$ that 
\begin{align*}
    &A_N(k2^{-N}\xrightarrow{m2^{-N}} l2^{-N})-A_N(k2^{-N}\xrightarrow{n2^{-N}} l2^{-N}) \\=&\sum_{i=k}^l\log\left(M_{ i2^{-N}\xrightarrow{m2^{-N}} (i+1)2^{-N}}  \right)-\log\left(M_{ i2^{-N}\xrightarrow{n2^{-N}} (i+1)2^{-N}}  \right),
\end{align*}
where $M$ is sampled under $\mu^N$. Therefore, we should understand, for $n < m\in \N$, $l\in \N$, and independent random variables $X_1,\dots,X_n, Y_1,\dots,Y_{m-n}$ such that 
$X_i \sim \mu_{s_i} \text{, and } Y_i\sim \mu_{t_i}$
the quantity 
\begin{align*}
\mathbb{E}\left[\left\vert \sum_{j=1}^l \log(X_1\cdots X_n Y_1\cdots Y_{m-n})-\log(X_1\cdots X_n)  \right\vert^{p}\right]. 
\end{align*}
In fact, we have the following estimate. 
\begin{proposition}\label{mainestimate}
    Consider a family of probability measures $(\mu_t)_{t>0}$ verifying  $M_\beta$ for some $\beta>0$. Let moreover $\varepsilon>0$. 
    \par There exists a constant $C_\beta$ such that for integers $n < m$ and $l$, and independent random variables $(X^j_u)_{\substack{1\leq j\leq l \\ 1\leq u \leq n  }}$ and $(Y^j_v)_{{\substack{1\leq j\leq l \\ 1\leq v \leq m-n  }}}$ such that 
\[\mathbb{P}(X^j_u\in \dd g)= \mu_{s^j_u}(\dd g) \text{ and } \mathbb{P}(Y^j_v\in \dd g)= \mu_{t^j_v}(\dd g), \]
 where 
 \[\max\Bigg (\max_{\substack{1\leq j\leq l \\ 1\leq u \leq n  }} s^j_u, \max_{{\substack{1\leq j\leq l \\ 1\leq v \leq m-n  }}} t^j_u\Bigg )\leq \varepsilon,\]
we have 
\begin{align*}
\mathbb{E}\left[\left\vert \sum_{j=1}^l \log(X^j_1\cdots X^j_n Y^j_1\cdots Y^j_{m-n})-\log(X^j_1\cdots X^j_n)  \right\vert^{2\beta}\right] \leq C_\beta l^{\beta} \left( (m-n)^\beta \varepsilon^\beta +n^{2\beta-}\varepsilon^{2\beta}\right).
\end{align*}
\end{proposition}
We need several lemmas for the proof.

\begin{lemma}\label{esperance}
     Under the $M_\beta$ condition, there exists a constant $C>0$ such that for  all $0
       \leq s_1,\dots,s_n,t_1,\dots,t_m\leq1$, we have 
       \begin{align*}
           \int_{G^{n+m}} \vert \log(x_1\cdots x_ny_1\cdots y_m)-\log(y_1\cdots y_m)\vert ^{2\beta}\prod_{i=1}^n\mu_{s_i}(\dd x_i)\prod_{j=1}^m\mu_{t_j}(\dd y_j)\\
           \leq C_p((s_1+\cdots+s_n)^{\beta}&+(t_1+\cdots+t_m) ^{2\beta}).
       \end{align*}
      
\end{lemma}
\begin{proof}
    By the Baker--Campbell--Hausdorff formula, there exists a constant $C>0$ such that for all $x,y\in G$, 
    \[\vert \log(xy)-\log(x)-\log(y)\vert \leq C(\vert \log x\vert ^2+\vert \log y\vert ^2).\]
    This gives
     \[\vert \log(xy)-\log(x)\vert^{2\beta} \leq 3^{2\beta}C^{2\beta}(\vert \log y\vert ^{2\beta}+\vert \log x\vert ^{4\beta}+\vert \log y\vert ^{4
     \beta}).\]
    By taking the integral, we get 
\begin{eqnarray*}
    &&\int_{G^{n+m}} |\log(x_1\cdots x_ny_1\cdots y_m)-\log(y_1\cdots y_m)|^{2\beta}\prod_{i=1}^n\mu_{s_i}(\dd x_i)\prod_{j=1}^m\mu_{t_j}(\dd y_j) \\
    &\lesssim&\int_{G^{n+m}} |\log(x_1\cdots x_n)|^{2\beta}\prod_{i=1}^n\mu_{s_i}(\dd x_i)\prod_{j=1}^m\mu_{t_j}(\dd y_j)\\
    &+&\int_{G^{n+m}} |\log(x_1\cdots x_n)|^{4\beta}\prod_{i=1}^n\mu_{s_i}(\dd x_i)\prod_{j=1}^m\mu_{t_j}(\dd y_j)\\
    &+&\int_{G^{n+m}} |\log(y_1\cdots y_m)|^{4\beta}\prod_{i=1}^n\mu_{s_i}(\dd x_i)\prod_{j=1}^m\mu_{t_j}(\dd y_j)\\
    &\lesssim&(s_1+\cdots+s_n)^{\beta}+(t_1+\cdots+t_m) ^{2\beta},
\end{eqnarray*}
where the $2\beta^-$ comes from Lemma~\ref{estim}.
\end{proof}

\begin{lemma}
Using the same notations as Proposition~\ref{mainestimate}, let us define, for $1\leq k \leq l$ ,
\[Z_k \coloneq  \sum_{j=1}^k\log(X^j_1\cdots X^j_n Y^j_1\cdots Y^j_{m-n})-\log(X^j_1\cdots X^j_n).\]
    The sequence $(Z_k)_{1\leq k\leq l}$ is a martingale.
\end{lemma}
\begin{proof}
    For fixed $k$, $Z_k$ is the sum of independent random variables. We only need to check that $Z_k$ is centered. Since the convolution of measures that are invariant by conjugation and inversion is again invariant by conjugation and inversion, we only need to show that 
    $\int_{G^2} (\log (ab)-\log(a)) \mu(\dd a)\nu(\dd b) =0$.
    for two measures $\mu$ and $\nu$ invariant by conjugation and inversion. However, we have 
    \begin{eqnarray*}
       && \int_{G^2} (\log (ab)-\log(a)) \mu(\dd a)\nu(\dd b)=\int_{G^2} (\log (ab^{-1})-\log(a)) \mu(\dd a)\nu(\dd b)\\
       &=&\int_{G^2} (\log ((ba^{-1})^{-1})-\log(a)) \mu(\dd a)\nu(\dd b)=-\int_{G^2} (\log (ba^{-1})-\log(a^{-1})) \mu(\dd a)\nu(\dd b)\\
       &=&-\int_{G^2} (\log (ba)-\log(a)) \mu(\dd a)\nu(\dd b)=-\int_{G^2} (\log (aba^{-1}a)-\log(a)) \mu(\dd a)\nu(\dd b)\\
              &=&-\int_{G^2} (\log (ab)-\log(a)) \mu(\dd a)\nu(\dd b).\\
    \end{eqnarray*}
    Therefore $Z_k$ is a martingale as the sum of centered independent random variables.  
\end{proof}
Recall below one form of the BDG inequality.
\begin{lemma}[BDG inequality for vector valued martingales] \label{BDG}
    Let $E$ be an euclidean space and $(Z_k)_{k\geq 0}$ be a $E-$valued martingales. Then, for all $p>1$, there exists a constant $c_p$ such that 
    \[\forall n \in \N, \mathbb{E}\left[\sup_{k\leq n}\|Z_k\|^p\right]\leq c_{p}\mathbb{E}\left[\Big(\sum_{k\leq n} \|Z_k-Z_{k-1}\|^2\Big)^{\frac{p}{2}}\right].\]
\end{lemma}
\begin{proof}[proof of Proposition~\ref{mainestimate}]
By Lemma~\ref{BDG}, we get 
    \begin{eqnarray*}
        \mathbb{E}[\vert Z_l\vert ^{2\beta}]& \leq& C_\beta \mathbb{E}\left[\left(\sum_{i=1}^{l-1}\vert Z_{i+1}-Z_{i}\vert ^2\right)^{\beta}\right] \leq C_\beta l^{\beta-1} \sum_{i=1}^{l-1}\mathbb{E}\left[\vert Z_{i+1}-Z_{i}\vert^{2\beta}\right] \\
        &\leq& C_\beta l^{\beta} \sup_{1\leq j\leq k}\mathbb{E}\left[\left\vert \log(X^j_1\cdots X^j_n Y^j_1\cdots Y^j_{m-n})-\log(X^j_1\cdots X^j_n)  \right\vert^{2\beta}\right] \\
        &\leq& C_\beta l^{\beta} \left((m-n)^\beta \varepsilon^\beta +n^{2\beta-}\varepsilon^{2\beta-}\right),
    \end{eqnarray*}
    where the last step comes from Lemma~\ref{esperance}.
\end{proof}
We finally show the following main proposition of this subsection. 
\begin{proposition}\label{EstimImpo}
    We have, for all integers $k \leq n \leq m \leq l$,
    \begin{align*}
        &\mathbb{E}\Big[
            \big|
                A_N(k2^{-N} \xrightarrow{m2^{-N}} l2^{-N})
                - A_N(k2^{-N} \xrightarrow{n2^{-N}} l2^{-N})
            \big|^p
        \Big] \\
        &\qquad\leq
        \sup_{\substack{k \leq j \leq l \\ n \leq i \leq m}}
        \sigma\big(\square(i2^{-N},(i+1)2^{-N},j2^{-N},(j+1)2^{-N})\big)^{\beta}
        (l-k)^{\beta}(m-n)^{\beta} \\
        &\qquad\quad+
        \sup_{\substack{k \leq j \leq l \\ n \leq i \leq m}}
        \sigma\big(\square(i2^{-N},(i+1)2^{-N},j2^{-N},(j+1)2^{-N})\big)^{2\beta}
        (l-k)^{\beta} n^{2\beta}.
    \end{align*}
\end{proposition}
\begin{proof}
    It is a direct application of Proposition \ref{mainestimate}.
\end{proof}

\subsection{Tightness result} \label{tightness}
In this section, we will show the tightness result. The first observation is the following.
\begin{lemma}\label{BoundAnSpace}
If for some parameters $s$ and $\alpha$, 
 \[\sup_{N\geq 0 } \mathbb{E}[\|A_N\|_{\mathcal{W}^{\alpha;\alpha-1;p;s}(\mathrm{Cyl},\g)}] < \infty,\]
   Then $(A_N)_{N\geq 0}$ is tight in $\mathcal{W}^{\alpha';\alpha'-1;p;s}(\mathrm{Cyl},\g)$ for any $\alpha'<\alpha$.
\end{lemma}
\begin{proof}
 We show in the Proposition~\ref{prop:compactembeddingweighted} from the appendix that for $\alpha>\alpha '$, we have the compact embedding of $\mathcal{W}^{\alpha;\alpha-1;p;s}(\mathrm{Cyl},\g)$ in $\mathcal{W}^{\alpha';\alpha'-1;p;s}(\mathrm{Cyl},\g)$. Therefore, the ball $B_{\mathcal{W}^{\alpha;\alpha-1;p;s}(\mathrm{Cyl},\g)}(0,c)$, seen as a subset in $\mathcal{W}^{\alpha';\alpha'-1;p's}(\mathrm{Cyl},\g)$ is compact, and is of measure arbitrarily close to $1$, provided we choose $c$ big enough. 
 \end{proof}

\begin{lemma}\label{sCondition}
    For $s<\frac 1 2 +2\alpha$, We have 
    \[\sup_{N\geq 0 }2^{pN(s-2\alpha)}\mathbb{E}\left\vert A(0\xrightarrow {2^{-N}} 2^{-N})-A(0\xrightarrow 0 2^{-N}) \right\vert^{p} <\infty \]
\end{lemma}
\begin{proof}
    We have from Proposition \ref{EstimImpo},
    \[\mathbb{E}\left\vert A(0\xrightarrow {2^{-N}} 2^{-N})-A(0\xrightarrow 0 2^{-N}) \right\vert^{p}\lesssim \sigma(\square(0,2^{-N},0,2^{-N}))^{\frac p 2} \leq 2^{-\frac{Np}{2}}\]
    which concludes the result.
\end{proof}

\begin{lemma}
    We have 
    \[\sup_{ 2^{N-n-2}\leq k,l \leq 2^{N-n-1}}\sigma(\square(k2^{-N},(k+1)2^{-N}, l2^{-N},(l+1)2^{-N})) \lesssim 2^{-2N+n}.\]
\end{lemma}
\begin{proof}
    It is a direct consequence of Lemma \ref{lem:dyadiccorona}.
\end{proof}

\begin{lemma}
Let $u<0$, and $\alpha \in (0,1)$. Then, for any integers $A<B$, we have
\begin{itemize}
    \item if $u>-1$, there exists a constant $C(u)>0$ such that
    $
        \sum_{A \leq p < q \leq B} (q-p)^u \leq C(u)\,(B-A)^{u+2},
    $
    \item if $u \leq -1$, there exists a constant $c(u)>0$ such that
    $
        \sum_{A \leq p < q \leq B} (q-p)^u \geq c(u)\,(B-A).
    $
    \item there exists $C(u,\alpha)>0$ such that \[
\sum_{A\leq p<q\leq B}\frac{p^u}{(q-p)^{1+\alpha u}}
\leq C(u,\alpha)\,(B-A)^{1+u-\alpha u}.
\]
\end{itemize}
\end{lemma}

\begin{thm}\label{Bounded}
For $0<\alpha <\frac 12$ and $s<\frac 12$, the sequence $(A_N)_{N\geq 0}$ is tight in the weighted space $\mathcal{W}^{\alpha;\alpha-1;p;s}(\mathrm{Cyl},\g)$.    
\end{thm}
\begin{proof}
The lemma \ref{sCondition} shows that the second term of \ref{MainBound} is uniformly bounded in $N$. In fact, condition $0<\alpha<\frac 1 2$ and $s<\frac 1 2$ implies condition $s<2\alpha+\frac 1 2$. 
\par It remains to show that the first term of \ref{MainBound} is uniformly bounded in $N$. 
    We have 
    \begin{align*}
         \sum_{0\leq n \leq N }&  \frac{2^{np(s-2\alpha)}}{2^{2N(1-p\alpha)}}\sum_{\substack{ 2^{-n+N-2} \leq n < m \leq 2^{-n+N-1}\\  2^{-n+N-2}\leq k , l \leq 2^{-n+N-1}  }}\frac{\mathbb{E}\left\vert A_N(k2^{-N}\xrightarrow{m2^{-N}} l2^{-N})-A_N(k2^{-N}\xrightarrow{n2^{-N}} l2^{-N})\right\vert^p}{\vert l-k \vert ^{1+p\alpha}(m-n)^{1+p\alpha}} \\
         \leq &\sum_{0\leq n \leq N }  \frac{2^{np(s-2\alpha)}}{2^{2N(1-p\alpha)}} \sum_{\substack{ 2^{-n+N-2} \leq n < m \leq 2^{-n+N-1}\\  2^{-n+N-2}\leq k , l \leq 2^{-n+N-1}  }} 2^{\frac{p}{2}(n-2N)}(k-l)^{\frac p 2-1-p\alpha}(m-n)^{\frac p 2-1-p\alpha} \\
`         &+\sum_{0\leq n \leq N }  \frac{2^{np(s-2\alpha)}}{2^{2N(1-p\alpha)}} \sum_{\substack{ 2^{-n+N-2} \leq n < m \leq 2^{-n+N-1}\\  2^{-n+N-2}\leq k , l \leq 2^{-n+N-1}  }} 2^{p(n-2N)}(k-l)^{ \frac p2 -1-p\alpha}\frac{n^{p}}{(m-n)^{1+p\alpha}}.
    \end{align*}
    Using the previous lemma, we have that for $0<\alpha <\frac  12$, 
    the first term is bounded by 
    \[\sum_{n\geq 0} 2^{n(ps-\frac p 2 -2)},\]
    which is finite for $s<\frac 12$
 and 
    the second goes to $0$ as $N\to\infty$.  Now using Lemma \ref{BoundAnSpace} we conclude that the sequence $(A_N)_{N\geq 0}$ is tight in $\mathcal{W}^{\alpha;\alpha-1;p;s}(\mathrm{Cyl},\g)$.
\end{proof} 

\section{Proof of the main theorem}\label{ProveMainTH}

\subsection{From the cylinder back to the surface}
\begin{proof}[proof of Proposition \ref{prop:boundarycase}] We already managed the singular part of $A$ formed of unstable currents. We only need to explain what happens in the bulk part. We identified the limit in Section  \ref{s:CLT}, and we showed the tightness in $\mathcal{W}^{\alpha;\alpha;p}(\mathrm{Cyl};\g)$ in Proposition \ref{BoundAnSpace}. Since $\mathcal{W}^{\alpha;\alpha;p}(\mathrm{Cyl};\g)$ is continuously embedded in $\mathcal{C}^{\alpha-\frac{1}{p};\alpha-\frac{1}{p};p}(\mathrm{Cyl};\g)$, we have 
\[A_N\xrightarrow[N\to \infty]{\mathcal{C}^{\alpha-\frac{1}{p};\alpha-\frac{1}{p};p}(\mathrm{Cyl};\g)}  \partial_\theta\left\langle \xi, 1_{\square(r,\theta)} \right\rangle  \dd \theta.\]
Therefore, taking $p$ large enough, we get from Proposition \ref{ContPullBack},
\[\Psi^*A_N\xrightarrow[N\to \infty]{\mathcal{C}^{\frac 1 2 -}(S)'} \Psi^*\left(\partial_\theta\left\langle \xi, 1_{\square(r,\theta)} \right\rangle  \dd \theta\right).\]
 Note that the space $\mathcal{C}^{\frac 1 2 -}(S)'$ is just a global space. However, locally, functional spaces look exactly like on the cylinder, and we have better local convergence properties.    
\end{proof}
\subsection{Closing the Surface}\label{ss:closing}
 In the previous section, we worked with the free boundary Yang--Mills measure. This means we considered a surface with one outgoing boundary, with no prescribed boundary conditions. In the current section, we would like to see how we can adapt this construction to the case of a closed surface. Following \cite{BCDRT}, the main idea  is to condition the holonomy on the boundary to be equal to $1_G$.
 \par Consider a closed surface $\Sigma$, and a Morse function on $\Sigma$ verifying the Morse--Smale condition. Blow-up the surface at the maximum to get a new surface, that we will call $\mathcal{S}$, with boundary $\partial \mathcal{S}$. Note that $\partial \mathcal{S}$ is now exactly the maximum set of the morse function. The idea is that the Yang--Mills measure on the closed surface is the Yang--Mills measure on the blown-up surface, in which we condition the holonomy on $\partial \Sigma$ to be $1_G$; this is summarized in Figure~\ref{fig:conditionning}.In this section, we will study this conditioning. 
 
 \begin{figure}[t]    
     \centering
     \includegraphics[width=\linewidth,trim={0cm 3cm 0cm 4cm}]{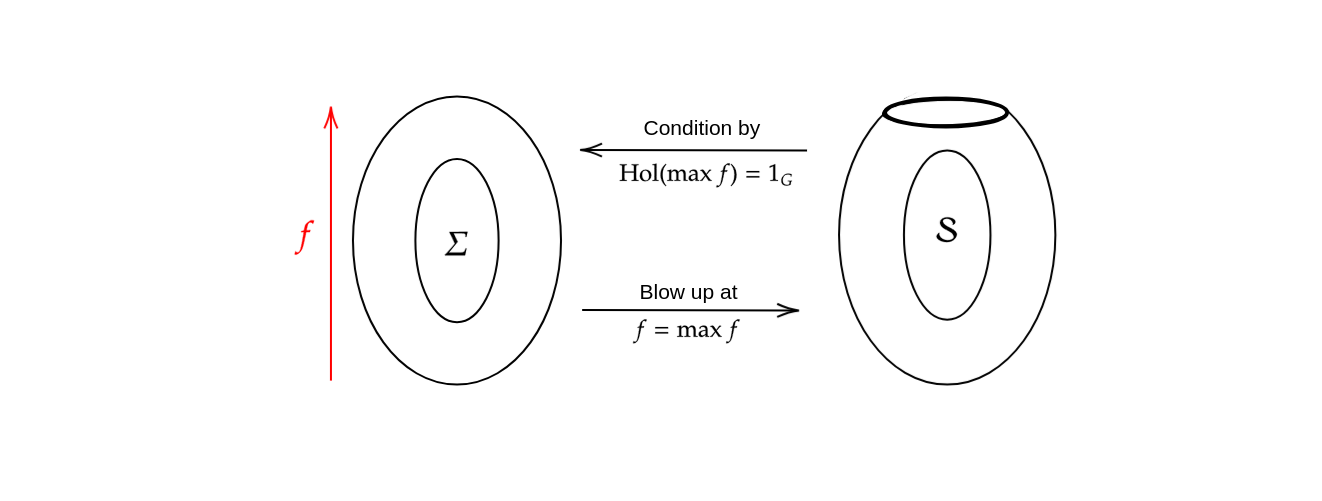}
     \caption{Measure on the closed surface.}
     \label{fig:conditionning}
 \end{figure}
 \par We will start in the lattice, where conditioning is much easily handled, and then show that the sequence of conditioned measure converges. 
We will more precisely prove the following theorem. 
 \par     Let $(\mu_t)_{t}$ be a family of probability measures verifying the conditions $(H)$. Recall the definition of the sequence of random connections  $(A^{(N)}_{(M,U)})_{N\geq 0}$ introduced in Section~\ref{s:functspaces}. Let $\mathcal{B}$ be a Banach space of distributions in which $(A^{(N)}_{(M,U)})_{N\geq 0}$ converges.
\begin{thm}\label{condThm}
     For each $N\geq 1$, let $(\mathbb{M}_N(\cdot | [g]))_{[g]\in G\backslash  \mathrm{Ad} (G)}$ be the disintegration of the law of $A^{(N)}_{(M,U)}$ with respect to $\mathrm{Hol}(A^{(N)}_{(M,U)},\partial \Sigma)$, seen as a random variable valued in $G\backslash  \mathrm{Ad} (G)$. Then, in $\mathcal{M}(\mathcal{B})$,
     \[\forall g\in G, \mathbb{M}_N(\cdot | [g]) \xrightarrow[N\to\infty]{} \mathbb{M}(\cdot |[g]).\]
     
\end{thm} 
 
\subsubsection{Disintegrating the lattice measure}

\par Consider again the Morse lattice on $\Sigma$ of resolution $2^{-N}$. Consider a Morse-Gauge fixed Yang--Mills measure associated to a family $\mu$ of probability measure. This means that we are dealing with two kinds of independent random variables:
\begin{itemize}
    \item the variables $M_{\theta \xrightarrow r \theta_+}$ for $\theta\in \Theta^N_-$ and $r\in R^N$,
    \item the $2g$ uniform random variables $(U_b)_{b=1}^{2g}$ in $G$.
\end{itemize}

\paragraph{The strategy using a Bayes type formula.}

We mimick in the discrete case the proof from~\cite{BCDRT} done in the continuum.
One of the key
idea to construct the Yang--Mills measure on the closed surface is to start from the free boundary lattice measure and to disintegrate the free boundary lattice measure with respect to the holonomy $\mathrm{Hol}(\max f)$. The problem is that applying the standard disintegration theorem yields a family of conditional measures defined only \emph{almost everywhere}, which is insufficient since we need the conditioned measure at precisely one conditioning value: $1_G$. Without further justification, there is absolutely no reason for this conditional family to be well defined at $1_G$. To overcome this, we establish a Bayes-type formula: the law of any observable depending on the bonds below $f = \max f - \epsilon$, conditioned on $\mathrm{Hol}(\max f) = 1_G$, equals the average over $g$ of the conditional expectation of the observable given $\mathrm{Hol}(\max f - \epsilon) = g$, multiplied by a transition kernel of the form $\mathbb{P}(\mathrm{Hol}(\max f) = 1_G \mid \mathrm{Hol}(\max f - \epsilon) = g)$. Since we now only require an average over $g$, the almost-everywhere disintegration of the observable along the level sets $\mathrm{Hol}(\max f - \epsilon) = g$ becomes sufficient. In the sequel we will give a way of formalizing this idea.

\paragraph{Some natural $\sigma$--algebras of the lattice gauge theory.}

\par Let us fix a $j\in \Theta_-^N$, with $j\geq 2g+1$ (just to mean that we consider a height higher than the last critical point of $f$), and introduce 
\[H_{N;j}\coloneq \mathrm{Hol}(0\xrightarrow{j} 2\pi).\]
For $r\in R^N$, let us define the sigma algebras 
\[\mathcal{F}^N_{\leq r}\coloneq \sigma\left(\left\{M_{\theta \xrightarrow{\rho}\theta_+}; \theta \in \Theta^N_-, \rho\leq r
\right\} \bigcup\{U_b;b\leq 2g\}\right)\]
\[ \mathcal{F}^N_{\geq r}\coloneq \sigma\left(\left\{M_{\theta \xrightarrow{\rho}\theta_+}; \theta \in \Theta^N_-, \rho\geq r
\right\}\bigcup \{U_b;b\geq r\}\right),\]
and 
\[\mathcal{F}^N_{r}\coloneq \sigma\left(\left\{M_{\theta \xrightarrow{r}\theta_+}; \theta \in \Theta^N_-
\right\}\bigcup \{U_b;b\leq 2g\}\right).\]
The sigma algebra $\mathcal{F}^N_{\leq r}$ (resp. $\mathcal{F}^N_{\geq r}, resp \mathcal{F}^N_{r}$) contains the information on all the bonds of the lattice $\Lambda^N$ located below the level $f=r$ (rep. above $f=r$, resp. exactly at $f=r$), in addition to \emph{all} uniform random variables associated to unstable curves. This is why $r$ is always supposed to be above the last critical point. 

\paragraph{The Markov property.}

\par First, we would like to understand what is the law of the holonomy of the circle $0\xrightarrow {k2^{-N}} 2\pi $ at dyadic level $k2^{-N}$ knowing all the bonds situated below level $j2^{-N}$. This is the purpose of the following lemma. 
\begin{lemma}[Causal Markov property]
    Let $f:G\rightarrow \R$ be a measurable central function, and let $j\leq k$. Then 
    \[\mathbb{E}[f(H_{N;k})|\mathcal{F}_{\leq j}]=\int_Gf(g)p_N(H_{N;j}^{-1}g)\dd g, \text{ where } p_N(x)=\sum_{\lambda \in \widehat{G}} \Bigg(\prod_{\substack{j\leq r\leq k \\ \theta\in\Theta^N_-}} \frac{\widehat{\mu}_{\sigma(\square(r,r_+,\theta,\theta_+))}}{d_\lambda} \Bigg) \chi_\lambda(x).\]
\end{lemma}

The Lemma tells us that all the law of the holonomy of the circle $0\xrightarrow {k2^{-N}} 2\pi $ at dyadic level $k2^{-N}$ conditionally on all the bonds located below level $j2^{-N}$ only depends on the knowledge of the holonomy at dyadic level $j2^{-N}$. This is exactly a Markov property of the discrete lattice gauge theory.

\begin{proof}
    Let $j\leq k$, we have $H_{N;k}=\prod_{b=0}^{2g-1} U_b  M_{I_b} U_{b}^{-1}M_{I'_b}$, 
where each $M_{I_b}$ is a product of 
\[M_{\theta \xrightarrow{k}\theta_+}\left(M_{\theta \xrightarrow{j}\theta_+}\right)^{-1} M_{\theta \xrightarrow{j}\theta_+}.\]
The Figure~\ref{fig:holcond} represents $U_1$, $M_{I_1}$, and $M_{I_1'}$.
\begin{figure}
    \centering

\tikzset{every picture/.style={line width=0.75pt}} 

\begin{tikzpicture}[x=0.75pt,y=0.75pt,yscale=-1,xscale=1]

\draw   (152,63.29) -- (243.43,63.29) -- (243.43,255.29) -- (152,255.29) -- cycle ;
\draw [color={rgb, 255:red, 255; green, 0; blue, 0 }  ,draw opacity=1 ]   (121.43,256.29) -- (121.43,66.29) ;
\draw [shift={(121.43,64.29)}, rotate = 90] [color={rgb, 255:red, 255; green, 0; blue, 0 }  ,draw opacity=1 ][line width=0.75]    (10.93,-3.29) .. controls (6.95,-1.4) and (3.31,-0.3) .. (0,0) .. controls (3.31,0.3) and (6.95,1.4) .. (10.93,3.29)   ;
\draw    (151.43,138.29) -- (243.43,138.29) ;
\draw    (170.43,62.71) -- (172.43,255.29) ;
\draw    (191.43,64.29) -- (192.43,255.29) ;
\draw    (220.43,64.29) -- (221.43,255.29) ;
\draw    (151.43,238.29) -- (172.43,238.29) ;
\draw    (152.43,218.29) -- (173.43,218.29) ;
\draw    (152.43,119.29) -- (173.43,119.29) ;
\draw    (172.43,238.29) -- (193.43,238.29) ;
\draw    (172.43,218.29) -- (193.43,218.29) ;
\draw    (221.43,239.29) -- (242.43,239.29) ;
\draw    (222.43,219.29) -- (243.43,219.29) ;
\draw    (171.43,119.29) -- (192.43,119.29) ;
\draw    (221.43,120.29) -- (242.43,120.29) ;
\draw   (262,63.29) -- (353.43,63.29) -- (353.43,255.29) -- (262,255.29) -- cycle ;
\draw    (261.43,138.29) -- (353.43,138.29) ;
\draw    (301.43,64.29) -- (302.43,255.29) ;
\draw    (330.43,64.29) -- (331.43,255.29) ;
\draw    (261.43,238.29) -- (282.43,238.29) ;
\draw    (262.43,218.29) -- (283.43,218.29) ;
\draw    (262.43,119.29) -- (283.43,119.29) ;
\draw    (282.43,238.29) -- (303.43,238.29) ;
\draw    (282.43,218.29) -- (303.43,218.29) ;
\draw    (331.43,239.29) -- (352.43,239.29) ;
\draw    (332.43,219.29) -- (353.43,219.29) ;
\draw    (281.43,119.29) -- (302.43,119.29) ;
\draw    (331.43,120.29) -- (352.43,120.29) ;
\draw  [color={rgb, 255:red, 255; green, 0; blue, 0 }  ,draw opacity=1 ] (243.43,63.29) -- (262.43,63.29) -- (262.43,255.29) -- (243.43,255.29) -- cycle ;
\draw [color={rgb, 255:red, 255; green, 0; blue, 0 }  ,draw opacity=1 ]   (244.43,138.29) -- (261.43,138.29) ;
\draw   (419,256) .. controls (423.67,255.98) and (425.99,253.64) .. (425.97,248.97) -- (425.77,208.32) .. controls (425.74,201.65) and (428.05,198.31) .. (432.72,198.29) .. controls (428.05,198.31) and (425.7,194.99) .. (425.67,188.32)(425.68,191.32) -- (425.47,147.68) .. controls (425.45,143.01) and (423.11,140.69) .. (418.44,140.71) ;
\draw  [color={rgb, 255:red, 255; green, 0; blue, 0 }  ,draw opacity=1 ] (353.43,63) -- (372.43,63) -- (372.43,255) -- (353.43,255) -- cycle ;
\draw [color={rgb, 255:red, 255; green, 0; blue, 0 }  ,draw opacity=1 ]   (354.43,138) -- (371.43,138) ;
\draw [color={rgb, 255:red, 255; green, 0; blue, 0 }  ,draw opacity=0.68 ]   (162.43,272.29) -- (162.33,46) ;
\draw [shift={(162.38,155.94)}, rotate = 89.98] [fill={rgb, 255:red, 255; green, 0; blue, 0 }  ,fill opacity=0.68 ][line width=0.08]  [draw opacity=0] (8.04,-3.86) -- (0,0) -- (8.04,3.86) -- (5.34,0) -- cycle    ;
\draw [color={rgb, 255:red, 208; green, 2; blue, 27 }  ,draw opacity=1 ] [dash pattern={on 0.84pt off 2.51pt}]  (162.33,46) -- (162.33,20) ;
\draw [color={rgb, 255:red, 208; green, 2; blue, 27 }  ,draw opacity=1 ] [dash pattern={on 0.84pt off 2.51pt}]  (183.33,297) -- (183.33,271) ;
\draw [color={rgb, 255:red, 255; green, 0; blue, 0 }  ,draw opacity=0.59 ]   (183.33,272) -- (183.24,45.71) ;
\draw [shift={(183.28,155.06)}, rotate = 89.98] [fill={rgb, 255:red, 255; green, 0; blue, 0 }  ,fill opacity=0.59 ][line width=0.08]  [draw opacity=0] (8.93,-4.29) -- (0,0) -- (8.93,4.29) -- (5.93,0) -- cycle    ;
\draw [color={rgb, 255:red, 208; green, 2; blue, 27 }  ,draw opacity=1 ] [dash pattern={on 0.84pt off 2.51pt}]  (183.24,45.71) -- (183.24,19.71) ;
\draw [color={rgb, 255:red, 255; green, 0; blue, 0 }  ,draw opacity=0.59 ]   (234.33,269) -- (231.37,48) ;
\draw [shift={(231.33,45)}, rotate = 89.23] [fill={rgb, 255:red, 255; green, 0; blue, 0 }  ,fill opacity=0.59 ][line width=0.08]  [draw opacity=0] (8.93,-4.29) -- (0,0) -- (8.93,4.29) -- (5.93,0) -- cycle    ;
\draw [color={rgb, 255:red, 208; green, 2; blue, 27 }  ,draw opacity=1 ] [dash pattern={on 0.84pt off 2.51pt}]  (234.33,298.33) -- (234.33,269) ;
\draw    (280.43,62.71) -- (282.43,255.29) ;
\draw    (261.43,238.29) -- (282.43,238.29) ;
\draw    (282.43,238.29) -- (303.43,238.29) ;
\draw    (331.43,239.29) -- (352.43,239.29) ;
\draw    (281.43,119.29) -- (302.43,119.29) ;
\draw [color={rgb, 255:red, 255; green, 0; blue, 0 }  ,draw opacity=0.68 ]   (272.43,272.29) -- (272.33,46) ;
\draw [shift={(272.38,155.94)}, rotate = 89.98] [fill={rgb, 255:red, 255; green, 0; blue, 0 }  ,fill opacity=0.68 ][line width=0.08]  [draw opacity=0] (8.04,-3.86) -- (0,0) -- (8.04,3.86) -- (5.34,0) -- cycle    ;
\draw [color={rgb, 255:red, 208; green, 2; blue, 27 }  ,draw opacity=1 ] [dash pattern={on 0.84pt off 2.51pt}]  (272.33,46) -- (272.33,20) ;
\draw [color={rgb, 255:red, 208; green, 2; blue, 27 }  ,draw opacity=1 ] [dash pattern={on 0.84pt off 2.51pt}]  (293.33,297) -- (293.33,271) ;
\draw [color={rgb, 255:red, 255; green, 0; blue, 0 }  ,draw opacity=0.59 ]   (293.33,272) -- (293.24,45.71) ;
\draw [shift={(293.28,155.06)}, rotate = 89.98] [fill={rgb, 255:red, 255; green, 0; blue, 0 }  ,fill opacity=0.59 ][line width=0.08]  [draw opacity=0] (8.93,-4.29) -- (0,0) -- (8.93,4.29) -- (5.93,0) -- cycle    ;
\draw [color={rgb, 255:red, 208; green, 2; blue, 27 }  ,draw opacity=1 ] [dash pattern={on 0.84pt off 2.51pt}]  (293.24,45.71) -- (293.24,19.71) ;
\draw [color={rgb, 255:red, 255; green, 0; blue, 0 }  ,draw opacity=0.59 ]   (344.33,269) -- (341.37,48) ;
\draw [shift={(341.33,45)}, rotate = 89.23] [fill={rgb, 255:red, 255; green, 0; blue, 0 }  ,fill opacity=0.59 ][line width=0.08]  [draw opacity=0] (8.93,-4.29) -- (0,0) -- (8.93,4.29) -- (5.93,0) -- cycle    ;
\draw [color={rgb, 255:red, 208; green, 2; blue, 27 }  ,draw opacity=1 ] [dash pattern={on 0.84pt off 2.51pt}]  (344.33,298.33) -- (344.33,269) ;

\draw (103,144.38) node [anchor=north west][inner sep=0.75pt]  [color={rgb, 255:red, 255; green, 0; blue, 0 }  ,opacity=1 ] [align=left] {$\displaystyle f$};
\draw (375,129) node [anchor=north west][inner sep=0.75pt]  [font=\footnotesize] [align=left] {$\displaystyle j2^{-N}$};
\draw (376,53) node [anchor=north west][inner sep=0.75pt]  [font=\footnotesize] [align=left] {$\displaystyle k2^{-N}$};
\draw (193,179) node [anchor=north west][inner sep=0.75pt]   [align=left] {$\displaystyle \cdots $};
\draw (303,179) node [anchor=north west][inner sep=0.75pt]   [align=left] {$\displaystyle \cdots $};
\draw (244.44,253.3) node [anchor=north west][inner sep=0.75pt]  [font=\scriptsize,color={rgb, 255:red, 255; green, 0; blue, 0 }  ,opacity=1 ,rotate=-270.2] [align=left] {Zero area band $\displaystyle U_{1}$};
\draw (444,188) node [anchor=north west][inner sep=0.75pt]   [align=left] {$\displaystyle \mathcal{F}_{j}$};
\draw (354.44,253.01) node [anchor=north west][inner sep=0.75pt]  [font=\scriptsize,color={rgb, 255:red, 255; green, 0; blue, 0 }  ,opacity=1 ,rotate=-270.2] [align=left] {Zero area band $\displaystyle ( U_{1})^{-1}$};
\draw (220,20) node [anchor=north west][inner sep=0.75pt]  [font=\footnotesize] [align=left] {$\displaystyle M_{I_{1}}{}$};
\draw (329,20) node [anchor=north west][inner sep=0.75pt]  [font=\footnotesize] [align=left] {$\displaystyle M_{I_{1}} '$};
\draw (303,179) node [anchor=north west][inner sep=0.75pt]   [align=left] {$\displaystyle \cdots $};

\end{tikzpicture}
    \caption{A part of the lattice wich represents $U_1$, $M_{I_1}$, and $M_{I_1'}$.}
    \label{fig:holcond}
\end{figure}
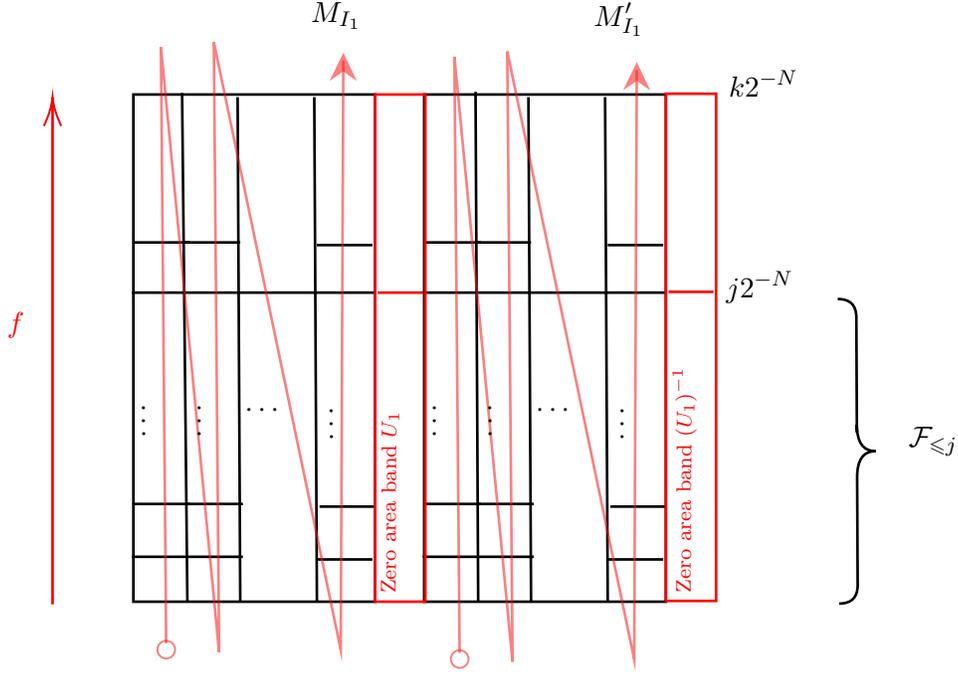
Therefore, $H_{N;k}$ is a word $\mathbf{w}$ of some variables that are $F_{\leq j}$ measurable, and others that are independent of $\mathcal{F}_{\leq j}$. Therefore,

\[\mathbb{E}[f(H_{N;k})|\mathcal{F}_j]= \Lambda\left(\left(M_{\theta\xrightarrow{r}\theta_+}\right)_{\substack{r\leq j\\\theta\in\Theta^N_-}},(U_b)_{0\leq b\leq 2g-1}\right),\]
where 
\[\Lambda\left(\left(m_{\theta\xrightarrow{r}\theta_+}\right)_{\substack{r\leq j\\\theta\in\Theta^N_-}},(u_b)_{0\leq b\leq 2g-1}\right)=\mathbb{E}[f(\mathbf{w}')],\]
where $\mathbf{w}'$ is the same word as $\mathbf{w}$ where we have replaced the $\mathcal{F}_{\leq j}-$measurable variables by deterministic constants represented by small letter.
By Ad-invariance of the law of $M_{\theta \xrightarrow{k}\theta_+}\left( M_{\theta \xrightarrow{j}\theta_+}\right)^{-1}$, we have 

\[\Lambda\left(\left(M_{\theta\xrightarrow{r}\theta_+}\right)_{\substack{r\leq j\\\theta\in\Theta^N_-}},(U_b)_{0\leq b\leq 2g-1}\right)=\mathbb{E}\left[f\left(h_j \prod_{\theta \in \Theta^N_-} M_{\theta \xrightarrow{k}\theta_+}\left( M_{\theta \xrightarrow{j}\theta_+}\right)^{-1}\right)\right]=\int_{G}f(g)p_N(h_j^{-1}g)\dd g,\]
where $p_N$ is the density of the convolution of all the $M_{\theta \xrightarrow{k}\theta_+}\left( M_{\theta \xrightarrow{j}\theta_+}\right)^{-1}$, and therefore given in Fourier representation by 
\begin{equation*}
p_N(x)=\sum_{\lambda \in \widehat{G}} \Bigg(\prod_{\substack{j\leq r\leq k \\ \theta\in\Theta^N_-}} \frac{\widehat{\mu}_{\sigma(\square(r,r_+,\theta,\theta_+))}}{d_\lambda} \Bigg) \chi_\lambda(x).     
\end{equation*}

\end{proof}

As mentioned above, the important observation of the previous lemma is that the law of $H_{N;k}$ knowing $\mathcal{F}_{\leq j}$ is only a function of $H_{N;j}$, not of all the bonds situated under the level set $j2^{-N}$ which is nothing but a causal Markov property. Therefore, we conclude the following lemma. 
\begin{lemma}\label{lem:preabelianization}
    We have, for any central function,
    $\mathbb{E}[f(H_{N;k})|\mathcal{F}_j]=\mathbb{E}[f(H_{N;k})|H_{N;j}].$
\end{lemma}
\begin{proof}
Direct consequence of the previous lemma.
\end{proof}

\paragraph{From joint law to the conditional law.}

Now, we want to compute the conditional law of $H_{N;j}$ with respect to $H_{N;k}$. A first step is to understand the joint law of these two random variables. 
\begin{lemma}
    We have, for all measurable functions $f$ and $g$
    \[\mathbb{E}[f(H_{N;j})g(H_{N;k})]=\int_{G^2}f(x)g(yx)f_{H_{N;j}}(x)p_{N}(y)\dd x \dd y,\]
    where 
    \[p_{N}(y)=\sum_{\lambda \in \widehat{G}} \Bigg(\prod_{\substack{j\leq r\leq k \\ \theta\in\Theta^N_-}} \frac{\widehat{\mu}_{\sigma(\square(t,t_+m\theta,\theta_+))}}{d_\lambda} \Bigg) \chi_\lambda(y). \]
\end{lemma}
As a result, we get the desired conditional law. 
\begin{lemma}\label{lem:abel2}
      We have, for all measurable functions $f$
      \[\mathbb{E}[f(H_{N;j})|H_{N;k}]=\int_Gf(x) \kappa_N (H_{N;k},x)\dd x, \text{ where } \kappa_n(g,x)\coloneq \frac{f_{H_{N;j}}(x)p_{N}(gx^{-1})}{f_{H_{N;k}}(g)}.\]
\end{lemma}

Before stating the main proposition of this subsection, let us first define precisely the support of a functional in $C_b(\mathcal{B},\R)$, we refer the reader to~\cite[Def III.1 p.~8]{BDLR} for more information.
\begin{definition}
    Let $F\in C_b(\mathcal{B},\R)$. 
 The support of $F$ is the smallest closed subset $\text{supp}(F)\subset \Sigma$ of $\Sigma$ with the following property~:
    \[\forall \varphi\in C^\infty(\Sigma), \,\, F(\varphi+h)=F(\varphi), \forall h\in C^\infty_c(\Sigma \setminus \text{supp}(F) ).\]
\end{definition}

\begin{proposition}\label{lem:keyLemmacondition}
Let $F\in C_b(\mathcal{B})$ such that the support of $F$ is included in $\{f\leq j\}$. For $k>j$, we have
We have 
\begin{equation*}
\mathbb{E}\left[F(A_N)|H_{N;k}=1_G \right]= 
\int_G \mathbb{E}\left[F(A_N)|H_{N;j}=x \right]\kappa_N(1_G,x)\dd x.
\end{equation*}    
\end{proposition}

\begin{proof}
For a central function $f$,
    \begin{eqnarray*}
        \mathbb{E}[F(A_N)f(H_{N;k})]&=&\mathbb{E}\left[\mathbb{E}[F(A_N)f(H_{N;k})|\mathcal{F}_{\geq j}]\right]=\mathbb{E}\left[\mathbb{E}[F(A_N)|\mathcal{F}_{\geq j}]f(H_{N;k})\right],
    \end{eqnarray*}
    where we have used that $H_{N;k}$ is $\mathcal{F}_{\geq j}-$measurable.
    Since $F(A_N)$ is $\mathcal{F}_{\leq j}-$measurable, then 
        \begin{eqnarray*} 
        \mathbb{E}[F(A_N)f(H_{N;k})]&=&\mathbb{E}\big[\mathbb{E}[F(A_N)|\mathcal{F}_{j}]f(H_{N;k})\big]=\mathbb{E}\Big[\mathbb{E}\big[\mathbb{E}[F(A_N)|\mathcal{F}_{j}]f(H_{N;k})\big| \mathcal{F}_j\big]\Big]\\
        &=&\mathbb{E}\left[\mathbb{E}[F(A_N)|\mathcal{F}_{j}]\mathbb{E}[f(H_{N;k})|\mathcal{F}_j]\right]=\mathbb{E}\left[\mathbb{E}[F(A_N)|\mathcal{F}_{j}]\mathbb{E}[f(H_{N;k})|H_{N;j}]\right]\\
        &&\text{by classical properties of conditional expectation and Lemma \ref{lem:preabelianization}} \\
        &=&\mathbb{E}\left[\mathbb{E}[\mathbb{E}[F(A_N)|\mathcal{F}_{j}]|H_{N;j}]\mathbb{E}[f(H_{N;k})|H_{N;j}]\right]\\
        &=&\mathbb{E}\left[\mathbb{E}[F(A_N)|H_{N;j}]\mathbb{E}[f(H_{N;k})|H_{N;j}]\right]
    \end{eqnarray*}

    again by classical properties of conditional expectations and since $H_{N;j}$ is $\mathcal{F}_j$ measurable.  
    By Lemma~\ref{lem:abel2}, we can now write 
\[\mathbb{E}[f(H_{N;k})|H_{N;j}]=\int_Gf(g)p_N(H_{N;j}^{-1}g)\dd g,\] 
which gives 
\begin{eqnarray*}
        \mathbb{E}[F(A_N)f(H_{N;k})] &=&\mathbb{E}\left[\mathbb{E}[F(A_N)|H_{N;j}]\int_Gf(g)p_N(H_{N;j}^{-1}g)\dd g\right]\\
        &=&\int_G\int_G\mathbb{E}[F(A_N)|H_{N;j}=x]f(g)p_N(x,g)\mathbb{P}(H_{N;j}\in \dd x)\dd g \\
        &=&\int_Gf(g)\mathbb{P}(H_{N;k}\in \dd g)\int_G\mathbb{E}[F(A_N)|H_{N;j}=k]\kappa_N(g,k)\dd k.
 \end{eqnarray*}       
 Therefore, 
 \[\mathbb{E}[F(A_N)|H_{N;k}=1_G]= \int_G\mathbb{E}[F(A_N)|H_{N;j}=x]\kappa_{N}(1_G,x)\dd x.\]
\end{proof}

\subsubsection{Convergence of the conditioned measures}
\paragraph{A sequence of measures on $\mathcal{B}\times G$.}
The main intermediate results to prove the convergence of the conditioned measures is to study the convergence of the pair
$\left(A^{(N)}_{(M,U)},\mathrm{Hol}(A^{(N)}_{(M,U)},\partial \Sigma)\right)$. 
\par Indeed, from the free boundary measure $\mu_N$ together with the holonomy map $\mathrm{Hol}(A_N,\partial\Sigma )$ (which is a random matrix when $A_N$ is chosen randomly under the probability measure $\mu_N$), we define a measure $\mathbf{m}_N$
on $\mathcal{B}\times G$. This measure describes the \emph{joint law} of the pair 
$\left(A_N,\mathrm{Hol}_{\partial\mathcal{S}_{\mathrm{out}} }(A_N)  \right)\in \mathcal{B}\times G $.

What we must prove is 
the following statement~:
\begin{lemma}\label{jointMeas}
The sequence of measures $\mathbf{m}_N$ converges to a limiting measure $\mathbf{m}$ describing the joint law of 
$\left(A,\mathrm{Hol}_{\partial\mathcal{S}_{\mathrm{out}} }(A)  \right)\in \mathcal{B}\times G $
where $A$ is chosen randomly under the limiting measure $\mu$.
\end{lemma}
\begin{proof}
 For every $N$, there is a well--defined holonomy map $\mathrm{Hol}_{\partial\mathcal{S}_{\mathrm{out}}}$ which is described as the solution of
 $\dd U_N=U_N \dd W_N,$ 
 where 
 \[\dd W_N=\sum_{\theta\in \Theta_-^N} \log\left (M_{ \theta \xrightarrow{\max f} \theta_+^N }\right )  2^N 1_{[\theta,\theta_+]}(\theta)\dd\theta,\]
 and $W_N$ is an affine interpolation of a $\g-$valued random walk that converges as $N\rightarrow +\infty$ to a re-parametrized $\mathfrak{g}$--valued Brownian motion. 
 Since we have 
 \[\forall p\in \mathbb{N},\,\,  \sup_{\substack{N\geq 0 \\\theta \in \Theta_-^N}}\mathbb{E}\left[ \left\vert 2^{N/2} \log\left (M_{ \theta \xrightarrow{\max f}\theta_+^N }\right )\right\vert^p \right] <+\infty,\] 
 we deduce that $\mathbf{W}_N \rightarrow \mathbf{W}^{\mathfrak{g}}$ in the rough path topology for variations of order $2$ by the proof of  Bayer--Friz~\cite[Thm 3.3 p.~269]{Bayer-Friz}. Indeed, this reference proves a Donsker Theorem in the rough path setting for walks which are independent but not necessarily identically distributed, itself being a generalization of ~\cite[Thm 1 p.~3489]{BreuillardRW}.

Let us recall the proof. Consider the geometric rough path \[\mathbf{W}_N\coloneq \left(W_N(s,t) \in \mathfrak{g} ; \int_s^t W_N\otimes \dd W_N \in \mathfrak{g}\otimes \mathfrak{g} \right).\]
Let 
 \[\mathbf{W}^{\mathfrak{g}}=\left(W_{s,t}, \int_s^t W \otimes \circ \dd W\right)\]
 denote the $\mathfrak{g}$--valued enhanced Brownian motion in the sense of Stratonovich. We would like to show that the \emph{pair} $(A_N,\mathbf{W}_N)$ converges in law in $\mathcal{B}\times \mathcal{C}^\alpha$ where the second space is the space of rough paths of H\"older regularity $\alpha\in (0,\frac{1}{2})$. Let us start by the easy part. The pair $A_N,W_N\coloneq \int_{\{\max(f)\}\times[0,\theta]} A_N$ converges in law where
\[\dd W_N=\sum_{\theta\in \Theta_-^N} \log(M_{ \theta \overset{\max(f)}{\longrightarrow} \theta_+^N })  2^N 1_{[\theta,\theta_+^N]}(\theta)\dd \theta \]
and $W_N$ is an affine interpolation of a $\g-$valued random walk that converges to some re-parametrized $\mathfrak{g}$--valued Brownian motion as $N\rightarrow +\infty$. To show the convergence in law of the pair $A_N,U_N(2\pi)$ using the continuity of the RDE, we need to enhance  $W_N$ to the rough path $\mathbf{W}_N$.
The strategy is to reduce everything to the proof of the usual central limit Theorem and prove some kind of tightness in H\"older spaces of rough paths. 
The first idea is to consider the interpretation of rough paths as Lie group valued paths. Start from the Lie algebra $\mathfrak{g}$, and consider any rough path as some element of the tensor algebra~:
\[\mathbf{W}\coloneq (1,W_{s,t},\mathbb{W}_{s,t})\in \mathbb{R}\oplus \mathfrak{g} \oplus \mathfrak{g}\otimes \mathfrak{g}=T^2\mathfrak{g}.\] This forms a non-commutative tensor algebra whose unit reads $(1,0,0)$~\cite[section 2.3 p.~17]{friz_hairer}. 
The elements of the form $(1,b,c)\in T^{(2)}(\mathfrak{g})$ form a group denoted by $T_1^{(2)}(\mathfrak{g})$ sitting inside $T^{(2)}(\mathfrak{g})$. This group is called the step--$2$ nilpotent Lie group. Note that $\mathbf{W}_{s,t}=\mathbf{W}_{0,t}\circ \mathbf{W}_{0,s}^{-1}$ is therefore a rough path and can be interpreted as a path valued in the Lie group $T_1^{(2)}(\mathfrak{g})$; and $\mathbf{W}_{s,t}$ is just an increment of the above path. This group has a natural Carnot--Carathéodory distance $\mathbf{d}_C$ and geometric rough paths of H\"older regularity $\alpha$ can be identified with H\"older maps of regularity $\alpha$ valued in $T_1^{(2)}(\mathfrak{g})$ for the distance  $\mathbf{d}_C$~\cite[]{friz_hairer}.
The idea is to prove some sort of Donsker theorem and the corresponding CLT for random walks valued into the nilpotent step $2$ group.
If we have some centered, independent $\mathfrak{g}$--valued random variables $\left(\xi^{n}_i, 1\leq i\leq n\right)$, we denote by
$e^{\xi_i^{(n)}}\coloneq (1,\xi_i,\xi_i\otimes \xi_i)$ the corresponding element in $T_1^{(2)}(\mathfrak{g})$. Assume that, for all $v,w\in \g$,
\[\mathbb{E}\left( \left\langle v,\xi_i^{(n)} \right\rangle_{\mathfrak{g}} \left\langle w,\xi_i^{(n)} \right\rangle_{\mathfrak{g}} \right)=a_i^{(n)} \left\langle v,w\right\rangle_{\mathfrak{g}} +o(\frac{1}{n}), \text{ with } a_i^{(n)}\sim \frac{1}{n},  \sum_{i=1}^n a_i^{(n)}=1.\]
Then by the CLT proved in~\cite[Lemma 4.1 p.~271]{Bayer-Friz},
we know that 
$e^{\xi_1^{(n)}}\otimes \dots \otimes e^{\xi_n^{(n)}}$ converges in law to the time $1$ of the Brownian motion on $T_1^{(2)}(\mathfrak{g})$ whose infinitesimal generator is the natural left-invariant sub--Laplacian $\sum_{i=1}^{\dim(\mathfrak{g})} X_i^2$ in  $T_1^{(2)}(\mathfrak{g})$. Then the next step consists in reducing to  
some combinatorial estimate. 
One key idea is that a geodesic $t\xi$ in $\mathfrak{g}$ lifts uniquely to a geodesic $e^{t\xi}$ in $T_1^{(2)}(\mathfrak{g})$ for the Carnot--Carathéodory distance. So linear (geodesic) interpolation on $\mathfrak{g}$ lifts functorially to geodesic interpolation in $T_1^{(2)}(\mathfrak{g})$ for the Carnot--Carathéodory distance~\cite[3.6 p.~40]{friz_hairer}. 

The hard part of the proof is the control of tightness in the H\"older space $\mathcal{C}^\alpha$ which is essential for rough differential equations.
By independence of increments and left invariance of the Carnot--Carathéodory distance,  
it is enough to prove that
\begin{equation}
\sup_n\mathbb{E}\left(d_C\left( \mathbf{W}_u,\mathbf{W}_v \right)^a  \right) \leq c \vert u-v \vert^{1+b}    
\end{equation}
for the ratio $\frac{b}{a}$ close to $\frac{1}{2}$ and $[u,v]=[0,\sum_{i=1}^k a_i^{(n)}]$. This rewrites
\begin{align*}
\mathbb{E}\left( \Vert e^{\xi_1^{(n)}}\otimes \dots \otimes e^{\xi_k^{(n)}}  \Vert^a\right) \leq \left(\sum_{i=1}^n a_i^{(n)}\right)^{1+b}    
\end{align*}
that we need to prove.
By changing the scaling again, we are reduced to proving an estimate of the form~:
\begin{align*}
\mathbb{E}\left( \Vert e^{\xi_1}\otimes \dots \otimes e^{\xi_k} \Vert^{2p}  \right) \lesssim  \left(\mathbb{E}\left(\sum_{i=1}^k \Vert\xi_i\Vert_{\mathfrak{g}}^2 \right) \right) ^{4p}    
\end{align*}
for any sequence $\xi_1,\dots, \xi_n,\dots$ of $\mathfrak{g}$--valued independent centered random variables, 
for all $p\geq 2$ and $k$ large enough.
The proof of this bound is exactly the content of~\cite[Prop 4.3 p.~273]{Bayer-Friz} and this tells us the sequence $(A_N,\mathbf{W}_N)\in \mathcal{B}\times \mathcal{C}^\alpha$ converges in law to the pair $(A,\mathbf{W})\in \mathcal{B}\times \mathcal{C}^\alpha$ where the second factor $\mathbf{W}_N\rightarrow \mathbf{W}$ as $\mathcal{C}^\alpha$ geometric rough path for $\alpha\in (\frac{1}{3},\frac{1}{2})$. Since $U_N$ solves the SDE $\dd U_N=U_N\circ \dd W_N$ in Stratonovich form, we can deduce by Ito-Lyons continuity of the RDE in the driving signal, that the pair $(A_N,U_N)$ converges in law to some pair $(A,U)\in \mathcal{B}\times G$.  
\end{proof}
\begin{proof}[Proof of the first part of Theorem~\ref{condThm}]
    Let $\epsilon>0$, and $F\in C_b(\mathcal{B})$, such that $\mathrm{supp\ }  F \subset \{f\leq \max f-\epsilon \}$. Choose $N$ large enough, and pick a dyadic level $j$ in $\{\max f -2\epsilon \leq f \leq \max f-\epsilon \}$. We have 
    \begin{eqnarray*}
        \mathbb{E}[F(A_N)|H_{N;\max f}=1_G]&=& \frac{1}{f_{H_{N;\max f}}(1_G)}\int_Gp_{N}(x) \mathbb{E}[F(A_N)|H_{N;j}=x]\mathbb{P}(H_{N;j}\in \dd x) \\
        &=&\frac{1}{f_{H_{N;\max f}}(1_G)}\int_G(p_{N}(x)-p_{\infty}(x)) \mathbb{E}[F(A_N)|H_{N;j}=x]\mathbb{P}(H_{N;j}\in \dd x)\\
        &+&\frac{1}{f_{H_{N;\max f}}(1_G)}\int_Gp_{\infty}(k) \mathbb{E}[F(A_N)|H_{N;j}=x ]\mathbb{P}(H_{N;j}\in \dd x).
    \end{eqnarray*}
   The first term  
   \[\left\vert\int_G(p_{N}(x)-p_{\infty}(x)) \mathbb{E}[F(A_N)|H_{N;j}=x]\mathbb{P}(H_{N;j}\in \dd x)\right\vert \leq \|p_{N}-p_{\infty}\|_{\infty} \|F\|_{\infty}\to 0,\]
   as a direct consequence of Theorem~\ref{convck},
   and   by Lemma~\ref{jointMeas}, the second term 
   \[\int_Gp_{\infty}(x) \mathbb{E}[F(A_N)|H_{N;j}=x]\mathbb{P}(H_{N;j}\in \dd x)=\mathbb{E}[F(A_N)p_{\infty}(H_{N;j})]\to  \mathbb{E}[F(A)p_{\infty}(H_j)],\]
\end{proof}

\subsubsection{Control on H\"older--Besov norms}
\label{sss:regimeofregularity}

We would like to control the H\"older--Besov regularities of the conditioned random connections.
The first direct result is that 
\[A^{(N)}_{(M,U)}|\mathrm{Hol}(A^{(N)}_{(M,U)},\partial \Sigma)=[g] \ \text{ converges in } \ \mathcal{C}^{-\beta-2}_{loc}(\Sigma\setminus \max(f),\g), \forall \beta+\alpha-1>0. \]

Moreover, we have precisely four regimes~:
\begin{enumerate}
    \item outside the union $\overline{\cup_{a\in \text{Crit}(f)_1} W^u(a)}$ of unstable curves, the sequence $A^{(N)}_{(M,U)}|\mathrm{Hol}(A^{(N)}_{(M,U)},\partial \Sigma)=[g] $ converges in $\mathcal{C}^\alpha_{r,\mathrm{loc}}\mathcal{C}^{\alpha-1}_\theta$,
    \item near unstable curves but away from saddle points, it converges in $\mathcal{C}^{-1-\varepsilon}_{loc}(\Sigma\setminus \max(f))$,  
    \item near saddle points, the sequence \[A^{(N)}_{(M,U)}|\mathrm{Hol}(A^{(N)}_{(M,U)},\partial \Sigma)=[g] \] converges in $\mathcal{C}^{-\beta-2}_{loc}(\Sigma\setminus \max(f),\g), \forall \beta+\alpha-1>0$,
    \item finally, near the north pole, bookkeeping the proof of conditioning, the reader can verify that it converges in some weighted H\"older space $\mathbf{d}(.,\max(f))^{-\dim G-\varepsilon}\mathcal{C}^{-1-\varepsilon}(\Sigma)$ with singular weight  $-\dim(G)-\varepsilon$, $\forall \varepsilon>0$ at $\max(f)$ .
\end{enumerate}

We refer the reader to figure \ref{fig:regregions} for the pictures of these four different regions.

\begin{figure}
    \centering
    \includegraphics[width=1.1\linewidth]{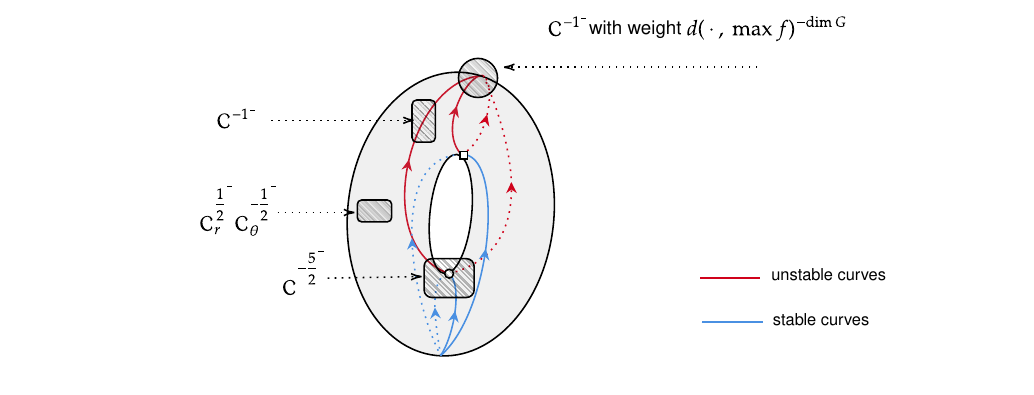}
    \caption{Four regions with different regularities for limiting connection.}
    \label{fig:regregions}
\end{figure}

Now, we will control the convergence in  $\mathcal{C}^\alpha_{r,\mathrm{loc}}\mathcal{C}^{\alpha-1}_\theta$, for all $\alpha\in(0,\frac{1}{2})$ outside the union $\overline{\cup_{a\in \text{Crit}(f)_1} W^u(a)}$. This finishes the proof of Theorem~\ref{thm:mainthmintro}.

\begin{proof}[Proof of Theorem~\ref{thm:mainthmintro}]
We only discuss in detail the regularity in a flowbox outside the union of unstable curves. 
Set $\alpha\in (0,\frac{1}{2})$, and let $\psi$ be any function in $C^\infty_c( \Sigma\setminus \partial\Sigma)$. Then we choose $m$ big enough, and $j=2^m\max(f)-1$ so that 
$F\coloneq  \mathcal{A}\in \mathcal{B}\mapsto  \Vert \mathcal{A}\psi\Vert_{C^\alpha_r C^{\alpha -1}_\theta}\in \mathbb{R}$
be $\mathcal{F}^{(N)}_{\leq j}$ measurable, for all $N\geq m$. From Theorem~\ref{BoundAnSpace}, we have
\[\sup_{N \geq 0} \mathbb{E}\left [F(A^{(N)}_{(M,U)})\right ] \leq \sup_{N \geq 0} \mathbb{E}\left [\|A^{(N)}_{(M,U)}\|\right ]   <\infty,\]
where we have used the continuity of the multiplication by $\psi$ in $C^\alpha_r C^{\alpha -1}_\theta$. Therefore, 
\[\int_G \mathbb{E}\left [F(A^{(N)}_{(M,U)})|H^{(N)}_j=g \right ] \mathbb{P}(H^{(N)}_1\in \dd g) \leq C \]
which gives, using the $C^\infty$  convergence of $ \mathbb{P}(H^{(N)}_1\in \dd g)$ to $p_{\sigma(\Sigma)}(g)\dd g$ (direct consequence of Theorem~\ref{convck}),which is bounded below by a non-zero positive number,
\[\int_G \mathbb{E}\left [F(A^{(N)}_{(M,U)})|H^{(N)}_j=g \right ]\dd g \leq C'. \]
Now, this concludes from  Proposition~\ref{lem:keyLemmacondition} that 
\[\sup_{N\geq 0} \mathbb{E}\left [F\left(A^{(N)}_{(M,U)}\right)\Big |\mathrm{Hol}(A^{(N)}_{(M,U)},\partial \Sigma)=1_G\right ] <\infty. \]
This means that for every test function $\psi\in C^\infty_c( \Sigma\setminus \partial\Sigma )$, the sequence $A^{(N)}_{(M,U)}\psi$ is tight in ${\mathcal{C}^\alpha_{r,\mathrm{loc}}\mathcal{C}^{\alpha-1}_\theta}$. Since the uniqueness of the limit has already been established in the previous section, we have finally that the sequence $A^{(N)}_{(M,U)}\psi$ converges in ${\mathcal{C}^\alpha_{r,\mathrm{loc}}\mathcal{C}^{\alpha-1}_\theta}$.
The other cases are left to the reader but they parallel the discussion of subsection~\ref{ss:singpullback} on singular pull--backs. The regularity estimates near 
unstable curves come from Lemma~\ref{lem:regunstablecurrents} in the appendix on the H\"older regularity of unstable currents. 
\end{proof}

\appendix

\section{Classical results in harmonic analysis}
\subsection{Path-wise integral of some rough differential forms}
In this appendix, we record a lemma that we occasionally use, which allows for the path-wise integration of rough differential forms. We also take this opportunity to note that the integration of differential forms is currently an active area of research, with many recent and interesting developments. We mention, in particular, the works of Z\"ust~\cite{Zust}, Alberti--Stepanov--Trevisan\cite{AlbertiStepanovTrevisan} ; and very recently Chandra--Singh~\cite{ChandraSingh}, and Jaffard~\cite{Jaffard}. 
\begin{lemma}\label{lem:pairing}
Let $M$ be some $C^\infty$ manifold.
Let $T\in \mathcal{D}^\prime(M)$ be a top degree current whose coefficients are in the 
H\"older--Besov space $\mathcal{C}^{\alpha}$ for $\alpha>-1$. Then for every 
domain $\Omega\subset M$ which has smooth boundary, 
the pairing
$$ \int_\Omega T=\int_M T 1_\Omega  $$
is well--defined. It is bilinear continuous when we endow the indicator function $1_\Omega$ with the topology of $\mathcal{B}^{-\alpha}_{1,1}$.
\end{lemma}
\begin{proof}
Without loss of generality using partitions of unity, the invariance of H\"older--Besov spaces by diffeomorphisms and adapted charts, we reduce to some half--space $ \{ (x_1,\dots,x_n); x_1\geqslant 0 \} $ in $\mathbb{R}^n$ and we want to give meaning to
$ \int_{\mathbb{R}^n} T\chi \Theta(x_1) $ where $\Theta$ is the Heaviside function.
The delta distribution $\delta_{\{0\}}^{\mathbb{R}}$ in $1d$ belongs to the Besov space $ \mathcal{B}^0_{1,\infty} $ hence in $\mathcal{B}^{-\varepsilon}_{1,1}$ for all $\varepsilon>0$. Therefore by taking primitive, the Heaviside function should belong to $\mathcal{B}^{1-\varepsilon}_{1,1, loc}$ for all $\varepsilon>0$. It follows that 
the product $ \chi \Theta(x_1)$ belongs to the Besov space $\mathcal{B}^{1-\varepsilon}_{1,1}$ for all $\varepsilon>0$ since $\chi\in C^\infty_c$ is a multiplier for Besov spaces. Then we conclude using $\mathcal{C}^\alpha=\mathcal{B}_{\infty,\infty}^\alpha$ and by the duality in Besov spaces $\left(\mathcal{B}^{-\alpha}_{1,1}\right)^\prime=\mathcal{B}^{\alpha}_{\infty,\infty}$. Beware that the duality is only in the sense stated, we do not have the converse :$\left(\mathcal{B}^{\alpha}_{\infty,\infty}\right)^\prime\neq \mathcal{B}^{-\alpha}_{1,1} $.
\end{proof}

\subsection{Restriction in Besov spaces}

We would like to make the following remark.
For $\alpha\in (0,\frac{1}{2})$, when we are given a distribution $T\in \mathcal{C}^{\alpha-1}(\mathbb{R})$ then Lemma ~\ref{lem:pairing} tells us that for any interval $I$, we can make sense of
the product $T1_{I}$ which is the distribution $T$ localized on the interval $I$. Intuitively, the product $T1_{I}$ is the unique distribution supported on $I$ which coincides with $T$ when acting on $C^\infty_c(I)$. However, some analysis using the Bony decomposition tells us that the product $T1_{I}$ can no longer be expected to belong to the H\"older space 
$\mathcal{C}^{\alpha-1}(\mathbb{R})$ because of the singularities of the indicator function $1_I$.

\begin{lemma}[Restriction Lemma in Besov regularity]\label{lem:restrictionBesov}
Under the assumption of Lemma~\ref{lem:pairing}, the product $T1_{\Omega}$ belongs to the Besov space $\mathcal{B}^{\alpha}_{1,1}$.  \end{lemma}

Note that there is a serious loss in regularity since the space $\mathcal{B}^{\alpha}_{1,1}$ is much larger than the H\"older space $\mathcal{B}^{\alpha}_{\infty,\infty}$.

\begin{proof}
We use $1_\Omega\in \mathcal{B}^{1-\varepsilon}_{1,1}$.
We use the product decomposition of Bony 
$ T1_\Omega= T\succ 1_{\Omega}+T\prec 1_{\Omega}+ T\circ 1_{\Omega} $ where $T\succ 1_{\Omega}\in \mathcal{B}^\alpha_{1,1}$, $T\prec 1_{\Omega}\in \mathcal{B}^{1-\varepsilon}_{1,1}  $,  $ T\circ 1_{\Omega}\in \mathcal{B}^{\alpha+1-\varepsilon}_{1,1} $ by 
the usual rules on the paraproduct and resonant products acting on Besov spaces as proved in  \cite[Lemma 2.1.34 p.~40]{Martin} from the thesis of Martin which generalizes certain statements from~\cite{BCD11}. This allows to conclude. 
\end{proof}

\subsection{Young products of distributions and duality}

Since we are often using it in our work, we need to recall the Young criterion for multiplying H\"older--Besov distributions with H\"older functions. 

\begin{lemma}[Young product of H\"older functions and distributions]\label{lem:young}
The product of smooth functions  
$$ (f_1,f_2)\in C^\infty(\mathbb{R}^n)\times C^\infty(\mathbb{R}^n) \mapsto f_1f_2\in C^\infty(\mathbb{R}^n) $$ extends uniquely as a bilinear continuous map
from $\mathcal{C}^\alpha(\mathbb{R}^n)\times \mathcal{C}^\beta(\mathbb{R}^n) \mapsto \mathcal{C}^{\inf(\alpha,\beta)}(\mathbb{R}^n) $ for $\alpha+\beta>0$.    

A consequence of the above is that for all compact set $K\subset \mathbb{R}^n$, the set $\mathcal{C}^\alpha_K(\mathbb{R}^n)$ of H\"older distributions of regularity $\alpha<0$ and supported by $K$ injects continuously in the dual of $\mathcal{C}^\beta(\mathbb{R}^n)$ provided $\alpha+\beta>0$.
\end{lemma}
\begin{proof}
The proof can be found in \cite{GIP} and follows immediately from \cite[Lemma 2.1 p.~11]{GIP}, see also~\cite[Thm 2.52]{BCD11} and \cite[thm 13.16 p.~201]{friz_hairer}.    
\end{proof}

We also recall some useful result on the topological dual $\mathcal{C}^\beta(\mathcal{S})^\prime$ to the H\"older space $\mathcal{C}^\beta(\mathcal{S})$ whenever $\beta\in (0,1)$. 
\begin{lemma}\label{lem:continjectiondualHolder}
On a smooth surface $\mathcal{S}$,
for every $\beta\in (0,1)$, the topological dual $\mathcal{C}^{\beta}(\mathcal{S})^\prime$ injects continuously in $\mathcal{C}^{-\beta-2-\varepsilon}(\mathcal{S})$ and $H^{-\beta-1-\varepsilon}(\mathcal{S})$ for all $\varepsilon>0$. 
\end{lemma}
\begin{proof}
The proof follows from the continuous injections $H^{\beta+1+\varepsilon}\hookrightarrow \mathcal{C}^{\beta}$ and $\mathcal{B}^{\beta+2+\varepsilon}\hookrightarrow \mathcal{C}^\beta$ and by duality in Besov spaces.      
\end{proof}

\subsection{Notion of wave front set of a distribution or current}
\label{ss:WF}

We recall the notion of wave front set which is used in the present paper to construct the global angular variable. This notion measures singularities of distributions in $\mathcal{D}^\prime(M)$ in \textbf{phase space}, which is the cotangent space $T^*M$ of $M$ and therefore refines the notion of singular support of a distribution in $\mathcal{D}^\prime(M)$ which is a closed subset of the base space $M$.
We refer to~\cite{BDHsmooth} for a pedagogical introduction to the notion of wave front set.

\begin{definition}[The wave front set of a distribution]
Let $U\subset \mathbb{R}^n$ be some open subset of $\mathbb{R}^n$.
Given a distribution $T\in \mathcal{D}^\prime(U)$, the wave front set $WF(T)\subset T^*U\setminus \underline{0} $ is the closed conical subset with the following property:
an element $(x_0;\xi_0)\notin WF(T)$ if there exists a neighborhood $U_0$ of $x_0$, a closed conic neighborhood  $V_0\subset \mathbb{R}^{n*}$ of $\xi_0$ such that for all $\chi\in C^\infty_c(U_0)$, for all $N$~:
\begin{eqnarray*}
\vert \widehat{T\chi}(\xi) \vert \leqslant C_{\chi,N,V_0} \left(1+\vert \xi\vert \right)^{-N} 
 \end{eqnarray*}
 uniformly in $\xi\in V_0$.
\end{definition}

By an important Theorem of H\"ormander~\cite[Thm 8.2.4 p.~263]{HormanderI}, the wave front set behaves functorially by pull--back by diffeomorphisms. More precisely,
for $f:U_1\mapsto U_2$ a diffeomorphism from open subsets of $\mathbb{R}^n$, given a distribution $T\in \mathcal{D}^\prime(U_2)$, the wave front set of the pulled back distribution $f^*T$ is given by
\begin{equation}
WF\left(f^*T \right)\subset \{(x;\xi); (f(x);\eta)\in WF(T), \xi=\eta\circ df(x)   \}.    
\end{equation}

Therefore, the notion of wave front set makes sense on smooth manifolds~\cite[p.~265]{HormanderI} and to bound the wave front set of some distribution $T\in \mathcal{D}^\prime(M)$, $M$ being some smooth manifold, it suffices to bound it on local charts. Now for a closed conical set $\Gamma\subset T^*M\setminus \underline{0}$, we will denote by $\mathcal{D}^\prime_\Gamma(M)$ the space of distributions whose wave front set is contained in the cone $\Gamma$. 

The above notion extends immediately to currents, given a current $T\in \mathcal{D}^{\prime,k}(M)$ of degree $k$, $U\subset M$ an open chart, the wave front set $WF(T)\cap T^*U$ over the chart $U$ is defined
as the union $\cup_\alpha WF(T_\alpha)$ of the wave front sets of its coefficients $(T_\alpha)_\alpha$, $\alpha$ are multi--indices~:
$$ T= \sum_{\vert\alpha\vert=k} T_\alpha dx^\alpha$$
when we express $T$ in local coordinates $(x^i)_{i=1}^n$ defined near $U$. The above definition does not depend on the choice of local coordinates used to write $T$.

The vector space $\mathcal{D}^\prime_\Gamma(M) $ can be endowed with a structure of locally convex topological vector space as explained in detail in~\cite[p.~204]{BDH} see also ~\cite[Def 7.2 and 7.3 p.~1842]{DRWitten} for quick recollection.
In the case of $M=\mathbb{R}^n$ and $\Gamma\subset T^*\mathbb{R}^n\setminus \underline{0}$, it is enough to consider the topology defined by the seminorms
\begin{enumerate}
    \item $\Vert T\Vert_{N,V,\chi}:=\sup_{\xi\in V} \vert \left(1+\vert \xi\vert\right)^N\widehat{T\chi}(\xi) \vert$ for all $\chi\in C^\infty_c(\mathbb{R}^n)$ and closed cone $V\subset \mathbb{R}^{n*}$ s.t. $\left(\text{supp}(\chi)\times V\right)\cap \Gamma=\emptyset$. These are the \textbf{continuous seminorms} that probe the microlocal regularity of $T$ outside the closed conic set $\Gamma$. 
    \item $ \sup_{\varphi\in B} \vert \left\langle T,\varphi \right\rangle \vert $
    where $B$ is a bounded set in $C^\infty(M)$. These are the continuous seminorms of the strong topology on $\mathcal{D}^\prime(\mathbb{R}^n)$. 
\end{enumerate}
The manifold case follows immediately from the case of $\mathbb{R}^n$ again by the continuity 
properties of the pull--back by a diffeomorphism.

Moreover, we recall some useful result on the wedge product of distributions with transverse wave front set~:
\begin{lemma}\label{lem:prodWF}
Let $M$ be a smooth manifold and $\Gamma_1,\Gamma_2$ be two closed conic sets in $T^*M\setminus \underline{0}$,   assume that the convex sum $\Gamma_1+\Gamma_2:=\{(x;\xi_1+\xi_2); (x;\xi_1)\in \Gamma_1,(x;\xi_2)\in \Gamma_2\}$ does not meet the zero section $\underline{0}$. 
Then the wedge product
$$ T_1,T_2\in \mathcal{D}^\prime_{\Gamma_1
 } \times \mathcal{D}^\prime_{\Gamma_2} \mapsto T_1\wedge T_2 $$
 is bilinear hypocontinuous. In particular, it is sequentially continuous.
\end{lemma}

\subsection{H\"older regularity of currents of integration}

We will also use the following
\begin{lemma}[H\"older and Sobolev regularity of unstable currents]\label{lem:regunstablecurrents}
Let $M$ be a smooth compact manifold and $Y\subset M$ a smooth hypersurface. Assume both $M$ and $Y$ are oriented, then the current of integration $[Y]$ on the submanifold $Y$ belongs to the H\"older--Besov space $\mathcal{C}^{-1-\varepsilon}(M)$ and the Sobolev space $H^{-\frac{1}{2}-\varepsilon}(M)$ , $\forall \varepsilon>0$. 
In particular, the unstable currents $U_a, a\in \text{Crit}(f)_1$ have H\"older regularity
$-1-\varepsilon$ for all $\varepsilon>0$.
\end{lemma}
\begin{proof}
The claim on the Sobolev regularity is quick to prove. By Sobolev trace theorem, one can restrict any $f\in H^{\frac{1}{2}+\varepsilon}(M)$ to the hypersurface $Y$, therefore  
by duality of Sobolev spaces the current of integration belongs to $H^{-\frac{1}{2}-\varepsilon}(M)$ for all $\varepsilon>0$. 

For the H\"older regularity, in a system $(x_1,\dots,x_n)$ of local coordinates on the open subset $U\subset M$ where $Y\cap U$ is locally given by $Y\cap U=\{x_1=0\}$, then $[Y]|_U=\pm\delta_{\{0\}}^{\mathbb{R}}(x_1)dx_1$. Now we use the fact that 
$\delta_{\{0\}}^{\mathbb{R}}$ belongs to $\mathcal{C}^{-1-\varepsilon}, \forall \varepsilon>0$ since $\delta_{\{0\}}^{\mathbb{R}}$ scales like 
$\delta_{\{0\}}^{\mathbb{R}}(\lambda.)=\lambda^{-1}\delta_{\{0\}}^{\mathbb{R}}(.)$, $\forall\lambda\in \mathbb{R}_{>0}$
and by the usual characterization of H\"older--Besov spaces by scaling that we recall in equation~(\ref{eq:Holderscaling}). By diffeomorphism invariance, we deduce that $[Y]|_U$ belongs to $\mathcal{C}^{-1-\varepsilon}_{loc}(U)$. Then by gluing with partition of unity and using that $\mathcal{C}^{-1-\varepsilon}(M)$ is stable by multiplication with $C^\infty$, we deduce the global result $[Y]\in \mathcal{C}^{-1-\varepsilon}(M)$. 
\end{proof}

\section{Recollection on the spectral analysis of Morse--Smale flows}

Let us collect the needed tools from~\cite{DR16}.
Our spectral analysis of Morse--Smale flows is very much inspired by the microlocal approach developped by Faure--Sj\"ostrand~\cite{FaureSjostrand} and Dyatlov--Zworski~\cite{DZ16}
to study the correlation spectrum of Anosov flows, which contain geodesic flows on negatively
curved manifolds. We briefly describe the general strategy. We will denote by $(x;\xi)$ elements of the cotangent $T^*\Sigma$ where $x$ (resp) denotes position (resp momentum). For a given function $m(x;\xi)$
in $S^0(T^*\Sigma)$, sometimes also called symbol of order $0$, 
we define the following Sobolev space of distributions of variable order 
\begin{align*}
\mathcal{H}^m\coloneq  \mathbf{Op}\left( (1+\vert \xi\vert^2_{g(x)})^{m(x;\xi)}  \right)L^2(\Sigma)   
\end{align*}
where $\mathbf{Op}$ means a quantization of symbols on the cotangent of space $T^*\Sigma$ of the surface $\Sigma$, it is chosen in such a way that $\mathbf{Op}\left( (1+\vert \xi\vert^2_{g(x)})^{m(x;\xi)}  \right)$ is a formally self--adjoint operator with principal symbol $(1+\vert \xi\vert^2_{g(x)})^{m(x;\xi)} $.   

Denote by $H_{V}\coloneq \left\langle\xi,V(x) \right\rangle \in C^\infty(T^*\Sigma) $ the Hamiltonian derived from the vector field $V=\nabla f$ and by $\left(\Phi^t_{H_V}\right)_{t\in \mathbb{R}}:T^*\Sigma \mapsto T^*\Sigma$ the corresponding Hamiltonian flow.
The order function $m$ is constructed in such a way that away from the zero section, it has good decay property along the forward Hamiltonian flow, there exists $R>0$ s.t. for all $\vert \xi \vert_{g(x)} \geqslant R $~:
$$  X_{H_V} \log(1+\vert \xi\vert^2_{g(x)})m(x;\xi)\leqslant -c<0 .$$

For our purpose, we shall need to extend the above definition on distributions to currents. This extension relies on tools from Hodge theory.
We consider the bundle $\Lambda^kT^*\Sigma \mapsto \Sigma$ of $k$ forms. Given some order function $m$ as above, we define $(1+\vert \xi\vert^2_{g(x)})^{m(x;\xi)} \mathbf{Id}\in C^\infty(T^*\Sigma, \mathrm{End}(\Lambda^k T^*\Sigma)  )$ which is the product of some symbol on $T^*\Sigma$ with the identity section $\mathbf{Id}\in C^\infty \left( \mathrm{End}(\Lambda^k T^*\Sigma) \right) $ acting on the bundle of $k$--forms. Using the Hodge star $\star$, we can define a scalar product on  $C^\infty(\Lambda^kT^*\Sigma)$ as
$\left\langle\alpha,\beta \right\rangle\coloneq \int_\Sigma \alpha\wedge \star\beta $,
the completion for this scalar product defined the space $L^2(\Lambda^kT^*\Sigma )$, or one could use the induced metric on the bundle $\Lambda^kT^*\Sigma$ of $k$ forms itself. Then the Sobolev space of currents of order $m$ and degree $k$ is defined as
\begin{equation}
\mathcal{H}^{m;k}(\Sigma)\coloneq  \mathbf{Op}\left(  (1+\vert \xi\vert^2_{g(x)})^{-m(x;\xi)} \mathbf{Id} \right)    L^2  \left(\Lambda^kT^*\Sigma \right)
\end{equation}
where $ \mathbf{Op}\left(  (1+\vert \xi\vert^2_{g(x)})^{-m(x;\xi)} \mathbf{Id} \right) $ is formally self--adjoint with principal symbol $ (1+\vert \xi\vert^2_{g(x)})^{-m(x;\xi)} \mathbf{Id} $, we refer to ~\cite[Appendix C.1]{DZ16} for a recollection on pseudodifferential operators acting on bundles.

One can state the main results of \cite{DR16} as follows:
\begin{thm}
Assume $\Sigma$ is a smooth closed compact surface, $f$ a Morse function on $\Sigma$ with adapted metric $g$ such that $V=\nabla f$ satisfies the Smale transversality condition and $1$ is the only Lyapunov exponent of $V=\nabla f$. 
For any $L>0$, there exists an order function $m\in C^\infty(T^*\Sigma)$ in the previous sense such that for all $k\in \{0,1,2\}$,
the resolvent $$\left(\mathcal{L}_V+z\right)^{-1}:\mathcal{H}^{m,k}\left( \Sigma\right)\mapsto \mathcal{H}^{m,k}\left(\Sigma \right)$$
has a meromorphic extension to the half--space $\Re(z)>-L$. Moreover for $\Re(z)>0$, the resolvent $\left(\mathcal{L}_V+z\right)^{-1}$ acting on $C^\infty(\Lambda^kT^*\Sigma)$ is given by the formula
\begin{equation}
  \left(\mathcal{L}_V+z\right)^{-1}\psi:=\int_0^\infty e^{-tz} \left(\varphi_f^{-t*} \psi\right) dt 
\end{equation}
for all $\psi\in C^\infty(\Lambda^kT^*\Sigma)$.

For every pair $\psi_1\in \mathcal{H}^{-m,k}, \psi_2\in \mathcal{H}^{m,2-k}$, we have the following asymptotic expansion~:
\begin{equation}
\int_\Sigma \psi_1\wedge \varphi_f^{-t*}\psi_2=\sum_{a\in \text{Crit}(f)} \left(\int_{\Sigma} \psi_1 \wedge U_a \right)   \left(\int_{\Sigma} S_a \wedge \psi_2 \right)+O\left( e^{-t} \Vert \psi_1 \Vert_{\mathcal{H}^{-m,k}} \Vert \psi_2 \Vert_{\mathcal{H}^{m,2-k}} \right)     
\end{equation}
where $(U_a)_{a\in \text{Crit}(f)}$ (resp $(S_a)_{a\in \text{Crit}(f)}$) are anisotropic Sobolev currents supported on the unstable manifolds $\overline{W^u(a)},a\in \text{Crit}(f)$ (resp stable manifolds $\overline{W^s(a)},a\in \text{Crit}(f)$).
\end{thm}

Let us also give a proof of a convergence result 
that we used in the construction of polar coordinates.

\begin{lemma}\label{lem:convergencecycles}
Let $\gamma$ be a smooth, oriented, closed curve in $\Sigma\setminus \min(f)$ transverse to $V$. We denote by $[\gamma]$ the corresponding current of integration induced in the blow--up surface $\mathcal{S}$. Then 
\begin{align*}  
\varphi_f^{t*}[\gamma]\rightarrow  \sum_{a\in \mathrm{Crit}(f)_1} \left(\int_\Sigma U_a\wedge [\gamma]\right) S_a +\left( \int_{\Sigma}[\gamma_0]\wedge [\gamma] \right) [\partial\mathcal{S}_{in}]     
\end{align*}
when $t\rightarrow +\infty$ where the convergence on the r.h.s. holds in the sense of currents in the blow up surface $\mathcal{S}$. 
\end{lemma}

We greatfully acknowledge Nguyen Bac Dang for explaining to us how to use cohomological arguments to prove convergence of dynamical correlators.

\begin{proof}
In fact we already know from the work~\cite{DR16} that
$\varphi_f^{t*}[\gamma]\rightarrow  \sum_{a\in \mathrm{Crit}(f)} \left(\int_\Sigma U_a\wedge [\gamma]\right) S_a \in \mathcal{D}^\prime(\mathrm{int}(\mathcal{S})) $
where the convergence holds in the sense of currents in $\mathcal{D}^\prime(\mathrm{int}(\mathcal{S}))$, these are dual to test forms which are supported in the \emph{interior of the surface}.
In fact in the interior, we know more: convergence takes place in suitable anisotropic Sobolev spaces of currents.

What we need to control is what happens at the boundary $\partial\mathcal{S}$.
We consider homology of currents on $\mathcal{S}$ which is \emph{relative} w.r.t. the ingoing boundary $\partial\mathcal{S}_{in}$ and \emph{absolute} w.r.t. the outgoing boundary $\partial\mathcal{S}_{out}$. Concretely, it means that for us 
a current $T\in \mathcal{D}^{\prime,\bullet}(\mathcal{S})$ is closed if $\text{supp}\left(\partial T\right)\subset \partial\mathcal{S}_{in}$.

For this we use a topological argument, first $[\gamma]$ defines a closed current of degree $1$, so we can decompose the corresponding cohomology class in the basis 
spanned by $([S_a],[\partial\mathcal{S}_{in}])_{a\in \mathrm{Crit}(f)_1} $. A dual basis for the Lefschetz pairing is
$(U_a,[\gamma_0] ])_{a\in \mathrm{Crit}(f)_1} $.
By a cohomological argument based on duality, we use the fact that de Rham cohomology is finite dimensional and Lefschetz pairing is non degenerate, yields that the closed current
$\sum_{a\in \mathrm{Crit}(f)_1} \left(\int_\Sigma U_a\wedge [\gamma]\right) S_a +\left( \int_{\Sigma}[\gamma_0]\wedge [\gamma] \right) [\partial\mathcal{S}_{in}] $
belongs to the same cohomology class as $[\gamma]$. 

From this, we deduce that $[\gamma]-\left( \sum_{a\in \mathrm{Crit}(f)_1} \left(\int_\Sigma U_a\wedge [\gamma]\right) S_a +\left( \int_{\Sigma}[\gamma_0]\wedge [\gamma] \right) [\partial\mathcal{S}_{in}]  \right)=dT$ is an exact current. Now we make the following elementary observation, if we remove $\gamma\cup \cup_{a\in \mathrm{Crit}(f)_1} W^s(a)$ from $\mathcal{S}$, we see that $dT=0$ in $\mathcal{D}^\prime(\mathrm{int}(\mathcal{S}) \setminus (\gamma\cup \cup_{a\in \mathrm{Crit}(f)_1} W^s(a))  )$. By the constancy Theorem, this means that $T$ is locally constant on each connected component of $\mathrm{int}(\mathcal{S}) \setminus (\gamma\cup \cup_{a\in \mathrm{Crit}(f)_1} W^s(a)) $ with bounded value. Since $[\gamma]-\left( \sum_{a\in \mathrm{Crit}(f)_1} \left(\int_\Sigma U_a\wedge [\gamma]\right) S_a +\left( \int_{\Sigma}[\gamma_0]\wedge [\gamma] \right) [\partial\mathcal{S}_{in}]  \right)$ is an integral current of finite mass and $T$ is defined uniquely up to constant, we know that up to subtracting a constant we may assume that $T=0$ in the connected compoment of $\partial \mathcal{S}_{out}$.    

Now $T$ is a bounded integer valued function which is vanishing near $\partial\mathcal{S}_{out}$ hence $\varphi^{t*}T \rightarrow 0 $ in $\mathcal{D}^\prime (\mathcal{S})$ when $t\rightarrow +\infty$. There are many ways to prove such result. Observe that  $\varphi^{t*}T, t\geqslant 0$
forms a bounded family of $L^2$ functions, hence by compactness of $L^2$ for the weak* topology, we may assume that we have a convergent subsequence. Denote by $T_\infty\in L^2$ any weak limit, its support should be invariant by the backward flow and does not contain a neighborhood of $\partial\mathcal{S}_{out}$, so $\mathrm{supp}(T_\infty)$ should be contained in $\cup_{a\in \mathrm{Crit}(f)_1} W^s(a)\cup \partial\mathcal{S}_{in}$ which has measure $0$, $T_\infty\in L^2$ is a measurable function hence $T_\infty=0$.  

We can conclude by using the commutation between pull--back and $d$ that~:
\begin{align*}
&\varphi^{t*}_f[\gamma]=\varphi_f^{t*}dT+ \varphi_f^{t*}\left(\sum_{a\in \mathrm{Crit}(f)_1} \left(\int_\Sigma U_a\wedge [\gamma]\right) S_a +\left( \int_{\Sigma}[\gamma_0]\wedge [\gamma] \right) [\partial\mathcal{S}_{in}] \right) \\
&\underset{t\rightarrow +\infty}{\rightarrow}  \sum_{a\in \mathrm{Crit}(f)_1} \left(\int_\Sigma U_a\wedge [\gamma]\right) S_a +\left( \int_{\Sigma}[\gamma_0]\wedge [\gamma] \right) [\partial\mathcal{S}_{in}]     
\end{align*}
since $\varphi_f^{t*} dT=d \varphi_f^{t*}T\rightarrow 0$ when $t\rightarrow +\infty$ and the current \[ \sum_{a\in \mathrm{Crit}(f)_1} \left(\int_\Sigma U_a\wedge [\gamma]\right) S_a +\left( \int_{\Sigma}[\gamma_0]\wedge [\gamma] \right) [\partial\mathcal{S}_{in}]     \] is invariant.
\end{proof}

\section{Continuous and compact injections in anisotropic spaces} \label{C}

In this appendix, we prove the necessary results on continuous and compact injections in our anisotropic spaces of currents.

We start by the continuous injections.

\begin{proposition}\label{prop:continuousinjection}

Let $\alpha\in (0,\frac{1}{2})$.
    If $\frac{1}{p}<\alpha$, we have the following continuous injection
    \[\mathcal{W}^{\alpha,\alpha-1,p}_{r,\theta}(\mathbf{Cyl},\g)\subset \mathcal{C}^{\alpha-\frac{1}{p},\alpha-1-\frac{1}{p}}_{r,\theta}(\mathbf{Cyl},\g)\]
    when we restrict to functions which are $0$ at $r=0$.
\end{proposition}

\begin{proof}
    The proof is a simple consequence of 
    the classical continuous injections in $1$ dimension.
    Without loss of generality, we may 
prove the continuous injection in the case
 \[\mathcal{W}^{\alpha,\alpha,p}_{r,\theta}(\mathbf{Cyl},\g)\subset \mathcal{C}^{\alpha-\frac{1}{p},\alpha-\frac{1}{p}}_{r,\theta}(\mathbf{Cyl},\g).\]
Then taking derivatives in the sense of distributions in the $\theta$ variable recovers the result we want to establish.

Also without loss of generality because we can glue things with partitions of unity, we will restrict to nice rectangles of the form $I_1\times I_2$ and prove the stronger injection
\[\mathcal{W}^{\alpha,\alpha,p}_{r,\theta}(I_1\times I_2,\g)\subset \mathcal{C}^{\alpha-\frac{1}{p},\alpha-\frac{1}{p}}_{r,\theta}(I_1\times I_2,\g).\]

The second observation is that if we restrict to distributions $f(.,.)$ on $\mathbf{Cyl}$ that vanish at $r=0$, we can drop the term $\Vert f\Vert_{L^p(I_1\times I_2)}$ in the definition of the norm and only work with the Sobolev semi-norm $ \left(
\int_{I_1^2} \frac{\Vert W(r_1,.)-W(r_2,.) \Vert^p_{W^{\alpha,p}_{\theta,I_2}}}{\vert r_1-r_2\vert^{1+\alpha p}} \dd r_1\dd r_2 \right)^{\frac{1}{p}} $.
For any smooth function $f$ on some interval $I=[a,b]$ which vanishes at $a$, we always have the a priori estimates, for all $\alpha\in (0,1)$,
\begin{align*}
\Vert f \Vert_{W^{s,p}(I)} = \left(\int_{I^2} \frac{\vert f(s_1)-f(s_2) \vert^p}{\vert s_1-s_2\vert^{1+\alpha p}}\dd s_1\dd s_2 \right)^{\frac{1}{p}} \geq C_I \sup_{s_1\neq s_2} \frac{\vert f(s_1)-f(s_2) \vert}{\vert s_1-s_2\vert^{\alpha -  \frac{1}{p}  } }
\end{align*}    
 where the constant $C_I$ does not depend on $f$. This estimate always holds true by the injection of homogeneous Sobolev spaces into their homogeneous H\"older counterparts~\cite[Thm 7.23 p.~275]{Leoni}. 

 Let us apply this estimate twice,
\begin{align*}
\Vert f\Vert^p_{\mathcal{W}^{\alpha,\alpha;p}_{r,\theta}(I_1\times I_2)}&=\int_{I_1^2} \frac{\Vert f(r_1,.)-f(r_2,.) \Vert^p_{W^{\alpha,p}_{\theta}(I_2)} }{ \vert r_1-r_2 \vert^{1+\alpha p}  }\dd r_1\dd r_2 \\
 &\geq C_{I_2} \int_{I_1^2}  \sup_{\theta_1\neq \theta_2} \frac{\vert f(r_1,\theta_1)-f(r_2,\theta_1)-f(r_1,\theta_2)+f(r_2,\theta_2) \vert^p }{ \vert r_1-r_2 \vert^{1+\alpha p} \vert \theta_1-\theta_2 \vert^{\alpha p-1} } \dd r_1\dd r_2 \\
 &\geq \frac{C_{I_2}}{\vert \theta_1-\theta_2 \vert^{\alpha p-1}} \int_{I_1^2}   \frac{\vert f(r_1,\theta_1)-f(r_1,\theta_2)-f(r_2,\theta_1)+f(r_2,\theta_2) \vert^p }{ \vert r_1-r_2 \vert^{1+\alpha p}  } \dd r_1\dd r_2
\end{align*} 
for all $\theta_1\neq \theta_2 $. Therefore,
\begin{align*}
\Vert f\Vert^p_{\mathcal{W}^{\alpha,\alpha;p}_{r,\theta}(I_1\times I_2)}&\geq \frac{C_{I_2}}{\vert \theta_1-\theta_2 \vert^{\alpha p-1}}
\Vert f(.,\theta_1)-f(.,\theta_2) \Vert_{W^{s,p}_{r,I_1}}^p\\
&\geq \sup_{r_1\neq r_2}  \frac{C_{I_1}C_{I_2}\vert f(r_1,\theta_1)-f(r_1,\theta_2)-(f(r_2,\theta_1)-f(r_2,\theta_2)) \vert^p}{\vert \theta_1-\theta_2 \vert^{\alpha p-1}\vert r_1-r_2 \vert^{\alpha p-1}  }\\
&\geq
\frac{C_{I_1}C_{I_2}\vert f(r_1,\theta_1)-f(r_1,\theta_2)-(f(r_2,\theta_1)-f(r_2,\theta_2)) \vert^p}{\vert \theta_1-\theta_2 \vert^{\alpha p-1}\vert r_1-r_2 \vert^{\alpha p-1}  }
\end{align*}
for all $r_1\neq r_2$, $\theta_1\neq \theta_2$. Then taking again the $\sup$ over $\theta$ then the sup over $r$ yields the final estimates
\begin{align*}
\Vert f\Vert^p_{\mathcal{W}^{\alpha,\alpha;p}_{r,\theta}(I_1\times I_2)}\geq C_{I_1}C_{I_2} \Vert f\Vert^p_{\mathcal{C}^{\alpha,\alpha;p}_{r,\theta}(I_1\times I_2)} .
\end{align*}
Once we glue together the estimates thanks to  cut off functions in the $\theta$  variable,
this concludes the proof of continuous injections.
\end{proof}

We prove another continuous injection which compares the anisotropic spaces 
$\mathcal{C}^{\alpha,\alpha-1}$ with the classical non anisotropic H\"older--Besov
spaces.
\begin{lemma}\label{lem:aniso_inj_holder}
For every $\alpha\in (0,1)$, any element $T\in \mathcal{C}^{\alpha,\alpha-1}_{x,y}(I\times J)$ for every $\chi\in C^\infty_c(J)$, the product $T1_I\chi$ which is well--defined
also belongs to the H\"older space $\mathcal{C}^{\alpha-1}(I\times J)$.     
\end{lemma}
\begin{proof}
We will use the following characterization of the H\"older--Besov space of negative regularity by scalings~\cite[p.~201 especially remark 13.15]{friz_hairer}.
$T\chi\in \mathcal{C}^{\alpha-1}$ if and only if for all compact balls $B$,
\begin{equation}\label{eq:Holderscaling}
\sup_{\vert \varphi\vert_{C^1_c(B)}=1} \sup_{a\in B}\sup_{n\geqslant 1} \vert 2^{n(\alpha-1)} \left\langle \mathcal{S}^{2^{-n}*}_{a} \left(T\chi\right),\varphi \right\rangle  \vert  \leqslant C_B.
\end{equation}

The fact that $T\chi$ belongs to the anisotropic space already tells us that for all $\varphi\in C^1_c(J)$~: 
$$\sup_{x\in I} \sup_{y_0\in J}\vert \left\langle T\chi(x, 2^{-n} (.-y_0),\varphi \right\rangle \vert\leqslant C 2^{-n(\alpha-1)}\Vert \varphi\Vert_{C^1} . $$

Therefore for  $a=(x_0,y_0)$, 
\begin{align*}
&\vert\left\langle \mathcal{S}^{2^{-n}*}_{a} \left(T\chi\right),\varphi \right\rangle \vert = 
\vert\left\langle  \left(T\chi\right)(2^{-n}(.-x_0), 2^{-n}(.-y_0)),\varphi \right\rangle \vert\\
&=\vert\int_{x\in I} \left\langle  \left(T\chi\right)(2^{-n}(x-x_0), 2^{-n}(.-y_0)),\varphi(x,.) \right\rangle_{y} dx \vert \\
&\leqslant \int_{x\in I}\vert \left\langle  \left(T\chi\right)(2^{-n}(x-x_0), 2^{-n}(.-y_0)),\varphi(x,.) \right\rangle_{y}\vert dx \\
&\leqslant C_B\vert I\vert  2^{n(1-\alpha)} \sup_{x\in I } \Vert\varphi(x,.) \Vert_{C^1(J)}
\end{align*}
where the constant on the r.h.s. does not depend on $a\in B$ which concludes the proof of the Lemma.

 For the second claim, we use the fact that $T$ is a continuous function of $x$ valued in distributions of $y$ of regularity $\alpha-1$ hence Lemma~\ref{lem:pairing}
allows us to make sense of $T1_J$ which is still continuous of $x$ valued in distributions of $y$ of regularity $\alpha-1$.
\end{proof}

We next prove compact injections of the anisotropic Sobolev and H\"older spaces.
\begin{proposition}
For $\varepsilon>0$,
the injection $\mathcal{W}^{\alpha+\varepsilon,\alpha-1+\varepsilon,p}(I\times \mathbb{S}^1) \hookrightarrow \mathcal{W}^{\alpha,\alpha-1,p}(I\times \mathbb{S}^1) $ is compact. Moreover for every pair of intervals $(I,J)$, for every $\chi\in C^\infty_c(J)$, the map
$T\in \mathcal{W}^{\alpha+\varepsilon,\alpha-1+\varepsilon,p}(I\times J) \mapsto \left(T\chi \right)\in \mathcal{W}^{\alpha,\alpha-1,p}(I\times J) $ is also compact.
\end{proposition}

One can also formulate the second claim by taking strictly smaller interval $J^\prime\subsetneq J$.

\begin{proof}
Recall that given a bounded sequence $(\varphi_n)_n$ in $L^p(\mathbb{S}^1)$ 
such that $\mathrm{supp}(\widehat{\varphi_n})\subset B(0,R)$ then this sequence has convergent subsequence in $L^p(\mathbb{S}^1)$, this is trivial since we deal with finite dimensional vector spaces. 

 Without loss of generality, for arbitrary $\chi\in C^\infty_c(I_2)$, we may 
prove the map
\[ T\in \mathcal{W}^{\alpha+\varepsilon ,\alpha+\varepsilon,p}_{r,\theta}(I_1\times I_2,\g)\mapsto T\chi \in \mathcal{W}^{\alpha,\alpha}_{r,\theta}(I_1\times I_2,\g)\]
is compact.
Then taking derivatives in $\theta$ and multiplying with another cut--off function in $\theta$ recovers the desired compact injection.

First principle, when we multiply $T$ defined on $I\times J$  with a cut--off function $\chi\in C^\infty_c(J)$ where $\chi$ depends only on the second variable, the function $T\chi$ can be considered
as living on $I\times\mathbb{S}^1$. So we can Fourier expand in the second variable $\theta$.
\begin{align*}
 \Vert T\chi\Vert_{\mathcal{W}^{\alpha,\alpha,p}}:= \int_{I^2} \frac{ \Vert T\chi(t_1,.)-T\chi(t_2,.) \Vert^p_{W^{\alpha,p}_\theta} }{\vert t_2-t_1\vert^{1+\alpha p}}    \dd t_1\dd t_2 
\end{align*}
where 
\begin{align*}
 \Vert T\chi(t_1,.)-T\chi(t_2,.) \Vert^p_{W^{\alpha,p}_\theta}=\sum_{j=0}^\infty 2^{jp\alpha} \Vert \psi_j(\sqrt{-\Delta}) \left( T\chi(t_1,.)-T\chi(t_2,.) \right)  \Vert_{L^p}^p    
\end{align*}
where we used the definition of $W^{\alpha,p}$ spaces in terms of Littlewood--Paley 
decomposition.

Assume that we have a bounded sequence $(T_n)_n$ in the anisotropic space $\mathcal{W}^{\alpha+\varepsilon,\alpha+\varepsilon,p}$, then note that 
$\left( t\mapsto \widehat{T_n\chi}(t,k) \right)_{k\in \mathbb{Z}} $ forms a bounded sequence of functions in $W^{\alpha+\varepsilon,p}(I)$ by the a priori bound

$$  \int_{I^2} \frac{ \vert \widehat{T\chi}(t_1,k)-\widehat{T\chi}(t_2,k) \vert^p }{\vert t_2-t_1\vert^{1+\alpha p}}    \dd t_1\dd t_2 \lesssim  \Vert T\chi\Vert_{\mathcal{W}^{\alpha,\alpha,p}} \left\langle k \right\rangle^{-\alpha-\varepsilon} $$
since
$$ \vert \left\langle T_n\chi(t_1,.)-T_n\chi(t_2,.), e^{ik.}\right\rangle \vert \leqslant C \Vert T_n\chi(t_1,.)-T_n\chi(t_2,.)\Vert_{W^{\alpha+\varepsilon,p}_\theta} \left\langle k \right\rangle^{-\alpha-\varepsilon}  $$
where $p^{-1}+q^{-1}=1$.

Assume $\alpha-\frac{1}{p}>0$.
By a diagonal extraction argument using the compact injection $W^{\alpha+\varepsilon,p}(I)\hookrightarrow W^{\alpha,p}(I)$ where $I$ a compact interval~\cite[item 4) Thm 2.1]{Vasquez} \footnote{In fact, it is a bit indirect to get the form we need. So first given a function $u\in W^{s,p}(I)$ extend it as $\tilde{u}\in W^{s,p}(\mathbb{R})$, the extension is linear continuous by \cite[Thm 5.4 p.~548]{Nezzafractional} and then apply  \cite[item 4) Thm 2.1]{Vasquez} to get the compact injection } ,  
we can extract a subsequence 
so that for every $k$, $\widehat{T_n\chi}(.,k)$ converges in $W^{\alpha+\varepsilon,p}$ to a limit $T_\infty(.,k)$. We need to prove that this is a Cauchy sequence in $W^{\alpha,\alpha,p}$.

By boundedness of the sequence $(T_n\chi)_n $ in $\mathcal{W}^{\alpha+\varepsilon,\alpha+\varepsilon,p}$
\begin{align*}
\sup_n\int_{I^2} \frac{\sum_{j=0}^\infty 2^{jp(\alpha+\varepsilon)} \Vert \psi_j(\sqrt{-\Delta})( T_n\chi(t_1,.)-T_n\chi(t_2,.))  \Vert_{L^p}^p }{\vert t_2-t_1\vert^{1+(\alpha+\varepsilon) p}}    \dd t_1\dd t_2 \leqslant C,     
\end{align*}
we deduce that for every integer $N\geqslant 1$, $\delta>0$, for all $n$:
\[
\int_{I^2 } \frac{\sum_{j=N}^\infty 2^{jp\alpha} \Vert \psi_j(\sqrt{-\Delta}) (T_n\chi(t_1,.)-T_n\chi(t_2,.) ) \Vert_{L^p}^p }{\vert t_2-t_1\vert^{1+\alpha p}}    \dd t_1\dd t_2 \leqslant C 2^{-Np\varepsilon},\]
and 
\[\int_{ \vert t_2-t_1\vert \geqslant \delta } \frac{\sum_{j=0}^\infty 2^{jp\alpha} \Vert \psi_j(\sqrt{-\Delta}) (T_n\chi(t_1,.)-T_n\chi(t_2,.))  \Vert_{L^p}^p }{\vert t_2-t_1\vert^{1+\alpha p}}    \dd t_1\dd t_2 \leqslant C \delta^{\varepsilon p}.\]

Now using the above two bounds we can bound the following difference term
\begin{align*}
 &\Vert (T_{n_1}-T_{n_2})\chi\Vert_{\mathcal{W}^{\alpha,\alpha,p}}\\
 &= \int_{I^2} \frac{ \sum_{j=0}^\infty 2^{jp\alpha} \Vert \psi_j(\sqrt{-\Delta})( T_{n_1}\chi(t_1,.)-T_{n_1}\chi(t_2,.)-T_{n_2}\chi(t_1,.)+T_{n_2}\chi(t_2,.)  )\Vert_{L^p}^p }{\vert t_2-t_1\vert^{1+\alpha p}}    \dd t_1\dd t_2 \\
 &\leqslant \int_{I^2} \frac{ \sum_{j=0}^N 2^{jp\alpha} \Vert \psi_j(\sqrt{-\Delta})( T_{n_1}\chi(t_1,.)-T_{n_1}\chi(t_2,.)-T_{n_2}\chi(t_1,.)+T_{n_2}\chi(t_2,.) ) \Vert_{L^p}^p }{\vert t_2-t_1\vert^{1+\alpha p}}    \dd t_1\dd t_2  +2C  2^{-Np\varepsilon}. 
\end{align*}
Now by choosing $N$ large enough, we can make the error term $2C  2^{-Np\varepsilon} $ as small as we want.
Since the partial sum $\sum_{j=0}^N 2^{jp\alpha} \Vert \psi_j(\sqrt{-\Delta})( T_{n_1}\chi(t_1,.)-T_{n_1}\chi(t_2,.)-T_{n_2}\chi(t_1,.)+T_{n_2}\chi(t_2,.) ) \Vert_{L^p}^p$  \emph{depends only on a finite number of Fourier modes}, letting $n_1,n_2\rightarrow +\infty$, we have the term $$ t\in I\mapsto \Vert \psi_j(\sqrt{-\Delta})\left(T_{n_1}\chi(t,.)-T_{n_2}\chi(t,.) \right)\Vert_{L^p_\theta}^p $$ 
goes to $0$ in $W_t^{\alpha,p}(I)$ hence we conclude since the term
\[\int_{I^2} \frac{ \sum_{j=0}^N 2^{jp\alpha} \Vert \psi_j(\sqrt{-\Delta})( T_{n_1}\chi(t_1,.)-T_{n_1}\chi(t_2,.)-T_{n_2}\chi(t_1,.)+T_{n_2}\chi(t_2,.) ) \Vert_{L^p}^p }{\vert t_2-t_1\vert^{1+\alpha p}}    \dd t_1\dd t_2 \] goes to zero. 

Let us make a remark on the above proof if we wanted not to use the cut--off by $\chi$.
Given any $T$ in $\mathcal{W}^{\alpha+\varepsilon ,\alpha+\varepsilon,p}_{r,\theta}(I_1\times I_2,\g)$, we can first extend it in the variable $\theta$ in a linear continuous way as a function 
$\tilde{T}\in \mathcal{W}^{\alpha+\varepsilon ,\alpha+\varepsilon,p}_{r,\theta}(I_1\times \mathbb{R},\g)$ by \cite[Thm 5.4 p.~548]{Nezzafractional}, now up to multiplication with a cut--off function $\chi\in C^\infty_c(0,2\pi)$ that equals $1$ on $I_2$, 
we may assume that the extension $\tilde{T}$ belongs to $\mathcal{W}^{\alpha+\varepsilon ,\alpha+\varepsilon,p}_{r,\theta}(I_1\times \mathbb{S}^1,\g)$ and coincides with $T$ on $I_1\times I_2$. Then repeat the previous proof based on Fourier decomposition in the variable $\theta$ yields in fact a compact injection
\[\mathcal{W}^{\alpha+\varepsilon ,\alpha+\varepsilon,p}_{r,\theta}(I_1\times I_2,\g)\subset \mathcal{W}^{\alpha,\alpha}_{r,\theta}(I_1\times I_2,\g).\]
\end{proof}

We also need to prove some compact injections in the weighted spaces since we would like to go from local to global.
\begin{proposition}[Compact injections weighted versions.]\label{prop:compactembeddingweighted}
For $\alpha\in (0,1)$, $\varepsilon>0$ s.t. $\alpha+\varepsilon\in (0,1)$,
we have the compact injections of the weighted spaces
$\mathcal{C}^{\alpha,\alpha-1,s} \hookrightarrow \mathcal{C}^{\alpha+\varepsilon,\alpha-1+\varepsilon,s-\varepsilon} $ is compact.
\end{proposition}
It is important that on the r.h.s. 
the weight has decreased.
\begin{proof}
Recall that a larger scaling exponent $s$ in the weighted Sobolev norms means your distribution has better decay near the singular points.  
The proof is an easy consequence of the above compactness embedding just taking the weights into account. 
\end{proof}

\section{Symbolic index} \label{symb}

We collect in this appendix commonly used symbols of the article, together
with their meaning and, if relevant, the page where they first occur.

 \begin{center}
\renewcommand{\arraystretch}{1.1}
\begin{longtable}{lll}
\toprule
Symbol & Meaning & Page\\
\midrule
\endfirsthead
\toprule
Symbol & Meaning & Page\\
\midrule
\endhead
\bottomrule
\endfoot
\bottomrule
\endlastfoot
$G$ & A compact Lie group &\\
$1_G$& Identity on G&\\
$\log$& A local inverse of $\exp$ extended to all $G$\\
$[g]$& Conjugacy class of the element $g\in G$\\
$\widehat{G}$ & The set of equivalence classes of irreps of $G$ &\\
$\lambda$ & An equivalent class of an irrep &\\
$\chi_\lambda$ & The character of $\lambda$ &\\
$d_\lambda$ & The dimension of any representation in $\lambda$ &\\

$\g$ & Lie algebra og $G$&\\
$\Sigma, M$&  A compact surface & \\
$\sigma$& Smooth area form \\
$\partial \Sigma$ & Boundary of $\Sigma$\\
$A$ & A $\g-$valued $1-$form on $\Sigma$&\\
$\mathrm{Tr}$ & A trace&\\
$\Omega^1(M,\g)$ & Space of $g-$valued $1-$forms on $M$\\
 $\mathrm{Hol}(A,c)$ & Parallel transport, or holonomy of $A$ along the curve $c$& \pageref{hol} \\
 $S_{YM}(A)$& Yang--Mills action of $A$& \pageref{YMaction} \\
 $\mathcal{G}$ & The gauge group $C^{\infty}(M,G)$&\pageref{GaugeGroup}\\
 $g\cdot A$ & The gauge transform of $A$
  by $g\in \G$ & \pageref{GaugeTransform}\\
  $Z_{\mathcal{T}}$&Segal amplitude&\pageref{SegalAmp}\\
  $\dd A$ & Exterior differential of $A$&\\
  $A\wedge A$ & Exterior product of $A$ and $A$ for the Lie bracket&\\
  $\xi$& A white noise& \\
  $C^\infty(.,*)$ & Smooth $*-$valued functions on $.$&\\
  $\mathcal{D}'(.,*)$& Topological dual of  $C^\infty(.,*)$ \\
  $\theta_1 \xrightarrow{r}\theta_2$ & The segment $\{(r,\theta_1+t); 0\leq t \leq \theta_2-\theta_1\}$ of a level set& \\
  $\circ \dd$ & Stratonovich differential & \pageref{Strato}\\
  $\square(r,r',\theta,\theta')$& The region of $\Sigma$ bounded by the flowlines $\theta$, $\theta'$ and the level sets $r$ and $r'$ \\
  $B^s_{p,q}$& Besov regularity scale \\
$\mathcal{W}^{s,p}$& Sobolev regularity scale \\
$\mathcal{C}^{s}$& H\"older--Besov regularity scale \\
$\mathbb{P}(X\in \dd x), f_{X}$ & Probability density of $X$ with respect to $\dd x$\\
 &  Morse-Smale  &\pageref{MorseSmale} \\
 &  Morse Chart  &\pageref{MorseChart} \\
 &  Adapted Metric  &\pageref{AdaptedMetric} \\
$W^s(a),W^u(a)$ & Stable, Unstable manifolds of $a$ &    \pageref{StableUnstable} \\
$U_a$& Unstable current & \pageref{UnstableCurrents}\\
$\mathcal{M}(\star)$ & Set of probability measures on $\star$ equipped with weak topology & \\

 \end{longtable}
 \end{center}

\end{document}